%% file: more-stable-stems.tex
\documentclass{memo-l}

\usepackage{amsrefs,amsmath,amsthm,amssymb}            
\usepackage{enumerate,lscape,calc}
\usepackage[all]{xy}   
\usepackage{longtable,booktabs}
\usepackage{color,colortbl}
\usepackage{tikz}

\newtheorem{thm}{Theorem}[chapter]
\newtheorem{prop}[thm]{Proposition}
\newtheorem{cor}[thm]{Corollary}
\newtheorem{lemma}[thm]{Lemma}
\newtheorem{conj}[thm]{Conjecture}

\theoremstyle{definition}
\newtheorem{defn}[thm]{Definition}
\newtheorem{ex}[thm]{Example}

\theoremstyle{remark}
\newtheorem{remark}[thm]{Remark}

\numberwithin{section}{chapter}
\numberwithin{equation}{chapter}

\newcommand{\M}{\mathbb{M}}
\newcommand{\F}{\mathbb{F}}
\newcommand{\C}{\mathbb{C}}
\newcommand{\R}{\mathbb{R}}
\newcommand{\Z}{\mathbb{Z}}
\newcommand{\D}{\Delta}

\newcommand{\map}{\rightarrow}
\newcommand{\tmf}{\textit{tmf}}
\newcommand{\mmf}{\textit{mmf}}
\newcommand{\ol}{\overline}
\newcommand{\sigmabar}{\ol{\sigma}}
\newcommand{\kappabar}{\ol{\kappa}}
\newcommand{\cl}{\mathrm{cl}}

\DeclareMathOperator{\Ext}{Ext}
\DeclareMathOperator{\Gr}{Gr}
\DeclareMathOperator{\Sq}{Sq}

\newcommand{\mydot}{\mathord{\cdot}}
\newcommand{\rev}[1]{#1} 
\newcommand{\revv}[1]{#1} 

\definecolor{red}{rgb}{1, 0, 0}
\definecolor{darkblue}{rgb}{0.1, 0.1, 0.7}
\definecolor{darkgreen}{rgb}{0.1, 0.8, 0.1}
\definecolor{midgray}{rgb}{0.5, 0.5, 0.5}
\definecolor{lightgray}{rgb}{0.8, 0.8, 0.8}

\definecolor{lightred}{rgb}{1, 0.7, 0.7}
\definecolor{lightblue}{rgb}{0.7, 0.7, 1}
\newcommand{\revrow}{} 
\newcommand{\revvrow}{} 
\newcommand{\revdeg}[1]{{\color{blue}$({#1})$}}

\begin{document}

\frontmatter

\title{\rev{Stable Homotopy Groups of Spheres: From Dimension 0 to 90}}

\author{Daniel C. Isaksen}
\address{
Department of Mathematics, 
Wayne State University, 
Detroit, MI 48202, USA}
\email{isaksen@wayne.edu}

\thanks{The first author was supported by NSF grants
DMS-1606290 and DMS-1904241.
The second author was supported by grant NSFC-11801082. 
The third author was supported by NSF grants DMS-1810638 and DMS-2105462.
Many of the associated machine computations were performed on the
Wayne State University Grid high performance computing cluster.}

\author{Guozhen Wang}
\address{Shanghai Center for Mathematical Sciences, Fudan University, Shanghai, China, 200433}
\email{wangguozhen@fudan.edu.cn}

\author{Zhouli Xu}
\address{Department of Mathematics, UC San Diego, La Jolla, CA 92093, USA}
\email{xuzhouli@ucsd.edu}



\date{}

\begin{abstract}
\rev{Using techniques in motivic homotopy theory, especially the theorem of Gheorghe, the second and the third author on the isomorphism between motivic Adams spectral sequence for $C\tau$ and the algebraic Novikov spectral sequence for $BP_*$,
we compute the classical and motivic stable homotopy groups of spheres 
from dimension 0 to 90},
except for
some carefully enumerated uncertainties.
\end{abstract}

\maketitle

%

\setcounter{page}{4}

\tableofcontents

\listoftables


\mainmatter

\include{more-stable-stems-intro}

\include{more-stable-stems-background}

\include{more-stable-stems-cofiber-tau}

\include{more-stable-stems-Massey}

\include{more-stable-stems-Adams-diff}

\include{more-stable-stems-Toda}

\include{more-stable-stems-hidden-extn}

\include{more-stable-stems-tables}


\backmatter

\bibliographystyle{amsalpha}



\end{document}

%% file: more-stable-stems-intro.tex
\chapter{Introduction}

The computation of stable homotopy groups of spheres is one of the most fundamental and important
problems in homotopy theory. It \revv{informs on} many topics in topology, such as the cobordism theory of framed manifolds, 
the classification of smooth structures on spheres, obstruction theory, the theory of topological modular forms, algebraic K-theory, motivic homotopy theory, and equivariant homotopy theory.

Despite their simple definition, which was available eighty years ago, these groups are notoriously hard to compute.  All known methods only give a complete answer through a range, \revv{but they eventually stall.
Further progress requires the introduction of a new method.}
The standard approach to computing stable stems is to use an 
\revv{Adams spectral sequence (based on a generalized cohomology theory
$E$) that converges} from algebra to homotopy.  In turn, to identify the algebraic $E_2$-pages, one needs algebraic spectral sequences that converge from simpler algebra to more complicated algebra. For any spectral sequence, difficulties arise in computing differentials and in solving extension problems. Different methods lead to trade-offs.  One method may compute some types of differentials and extension problems efficiently, but leave other types unanswered, perhaps even unsolvable by that technique. To obtain complete computations, one must be eclectic, applying and combining different methodologies. Even so, combining all known methods, there are eventually some problems that \revv{have not been} solved.  Mahowald's uncertainty principle states that no finite collection of methods can completely compute the stable homotopy groups of spheres.

Because stable stems are finite \revv{abelian} groups (except for the $0$-stem),
the computation is most easily
accomplished by working one prime at a time.  At odd primes,
the Adams-Novikov spectral sequence and the chromatic spectral sequence, which are based on complex cobordism and formal groups, have yielded a wealth of
data \cite{Ravenel86}. As the prime grows, so does the range of computation. For example, at the primes 3 and 5, we have complete knowledge up to around 100 and 1000 stems respectively \cite{Ravenel86}.

The prime $2$, being the smallest prime, remains the most difficult part of the computation.
\revv{This entire manuscript considers exclusively the $2$-completed
stable homotopy groups.}
In this case, the Adams spectral sequence is the most effective tool.
The manuscript \cite{Isaksen14c} presents a careful analysis of 
the Adams spectral sequence, in both the classical and $\C$-motivic
contexts, that is essentially complete through the 59-stem.
This includes a verification of the details in the classical literature
 \cite{BJM84} \cite{BMT70} \cite{Bruner84} 
  \cite{MT67}.
Subsequently,
the second and third authors computed the 60-stem
and 61-stem \cite{WangXu17}.

We also mention \cite{Kochman90}  \cite{KM93},
which take an entirely different approach to computing stable
homotopy groups. 
However, the computations in \cite{Kochman90}  \cite{KM93} are now known to contain several errors. 
See \cite{WangXu17}*{Section 2} for a more detailed discussion.

The goal of this manuscript is to continue the analysis of the Adams
spectral sequence into higher stems at the prime 2.  We will present information
up to the 90-stem.
While we have not
been able to resolve all of the possible differentials in this range,
we enumerate the handful of uncertainties explicitly
\rev{within Table \ref{tab:Adams-higher}}.

The charts in \cite{Isaksen14a} and \cite{IWX19}
are an essential companion to this manuscript. 
They present the same information in an easily interpretable 
graphical format.

Our analysis uses various methods and techniques, including 
machine-generated homological algebra computations, 
a deformation of homotopy theories that connects
$\C$-motivic and classical stable homotopy theory,
and the theory of motivic modular forms. Here is a quick summary of our approach:
\begin{enumerate}
\item
\label{item:machine-Adams}
Compute 
the cohomology of the $\C$-motivic Steenrod algebra by machine.  
These groups
serve as the input to the $\C$-motivic Adams spectral sequence.
\item
\label{item:machine-algNov}
Compute by machine
the algebraic Novikov spectral sequence that converges to the
cohomology of the Hopf algebroid $(BP_*, BP_* BP)$.
This includes all
differentials, and the multiplicative structure of the cohomology of
$(BP_*, BP_* BP)$.
\item
Identify
the $\C$-motivic Adams spectral sequence for the cofiber of $\tau$ 
with 
the algebraic Novikov spectral sequence
\cite{GWX18}.  This includes an identification of the
cohomology of $(BP_*, BP_* BP)$ with the
homotopy groups of the cofiber of $\tau$.
\item
Pull back and push forward Adams differentials for the cofiber of $\tau$
to Adams differentials for the $\C$-motivic sphere, along
the inclusion of the bottom cell and the projection to the top cell.
\item
\label{item:ad-hoc}
Deduce additional Adams differentials for the $\C$-motivic sphere with 
a variety of ad hoc arguments.  The most important methods are Toda bracket shuffles and 
comparison to the motivic modular forms spectrum $\mmf$ \cite{GIKR18}.
\item
Deduce hidden $\tau$
extensions in the $\C$-motivic Adams spectral sequence for the sphere, 
using a long exact sequence in homotopy groups.
\item
Obtain the classical Adams spectral sequence and the 
classical stable homotopy groups by inverting $\tau$.
\end{enumerate}

The machine-generated data that we obtain in steps (\ref{item:machine-Adams}) and (\ref{item:machine-algNov}) are available at 
\cite{Isaksen14a} and \cite{IWX19}.
See also \cite{Wang20} for a discussion of the implementation of the
machine computation.

Much of this process is essentially automatic.
The exception occurs in step (\ref{item:ad-hoc}) where ad hoc arguments come
into play.

This document describes the results of this systematic program through the
90-stem.  We anticipate that our approach will allow us to compute into 
even higher stems, especially towards the last unsolved Kervaire invariant problem in dimension 126.  However, we have not yet carried out a careful analysis.

\section{New Ingredients}

We discuss in more detail several new ingredients that allow us to carry out this program.

\subsection{Machine-generated algebraic data}

The Adams-Novikov spectral sequence has been used very successfully
to carry out computations at odd primes.  However, at the prime 2, its usage has not been fully exploited in stemwise computations. This is due to the difficulty of computing its $E_2$-page. The first author predicted in
\cite{Isaksen14c} that ``the next major breakthrough in computing stable stems will involve machine computation of the Adams-Novikov $E_2$-page."

The second author achieved this machine computation; the resulting data is available at \cite{IWX19}.  The process goes roughly like this.
Start with a minimal resolution that computes the cohomology of the  Steenrod algebra.  Lift this resolution to a resolution of $BP_*BP$.
Finally, use the Curtis algorithm to compute the homology of the resulting
complex, and to compute differentials in the associated algebraic
spectral sequences, such as 
the algebraic Novikov spectral sequence 
and the Bockstein spectral sequence.
See \cite{Wang20} for further details.

\subsection{Motivic homotopy theory}

The $\C$-motivic stable homotopy category
gives rise to new methods to compute stable stems.
These ideas are used in a critical way in \cite{Isaksen14c} to compute stable stems up to the $59$-stem.

The key insight of this article that distinguishes it
significantly from \cite{Isaksen14c} is that $\C$-motivic 
cellular stable homotopy theory is a deformation of classical stable
homotopy theory \cite{GWX18}.  
From this perspective, the ``generic fiber" of
$\C$-motivic stable homotopy theory is classical stable homotopy theory,
and the ``special fiber"
has an entirely algebraic description.  The special fiber is
  \revv{Hovey's stable derived} category of $BP_*BP$-comodules \cite{Hovey}, or equivalently,
the \revv{stable derived} 
category of quasicoherent sheaves on the moduli stack of 
1-dimensional formal groups.

\revv{
In more concrete terms, let $C\tau$ be the cofiber of the 
$\C$-motivic stable map $\tau$.
The cofiber sequence
$S^{0,0} \map C\tau \map S^{1,-1}$
induces maps
\[
\xymatrix{
E_2(S^{0,0}) \ar[r] \ar@{=>}[d] & E_2(C\tau) \ar[r] \ar@{=>}[d] & E_2(S^{1,-1}) \ar@{=>}[d] \\
\pi_{*,*}(S^{0,0}) \ar[r] & \pi_{*,*}(C\tau) \ar[r] & \pi_{*,*}(S^{1,-1})
}
\]
of spectral sequences, in which each vertical column represents
a $\C$-motivic Adams spectral sequence.

The homotopy category of $C\tau$-modules has an algebraic structure
\cite{GWX18}.
In particular,  the $\C$-motivic Adams spectral sequence for $C\tau$ is isomorphic to the algebraic Novikov spectral sequence 
that computes the $E_2$-page of the Adams-Novikov spectral
sequence for $BP_*$.  This means that the middle spectral sequence
in the above diagram
can be computed by machine.  Naturality then yields information
about the $\C$-motivic Adams spectral sequence
for the $\C$-motivic sphere spectrum in two different ways, 
since the latter spectral sequence appears on both the
left and right side of the diagram.
Finally,
the Betti realization functor produces differentials in the classical Adams spectral sequence.
}

Our use of $\C$-motivic stable homotopy theory appears to rely on
the fundamental computations, due to Voevodsky \cite{Voevodsky03b} \cite{Voevodsky10},
of the motivic cohomology of a point and of the motivic Steenrod
algebra.  
In fact, recent progress has determined that 
our results do not depend on this deep and difficult work.
There are now purely topological constructions of homotopy
categories that have identical computational properties to
the cellular stable $\C$-motivic homotopy category
\cite{GIKR18} \cite{Pstragowski18}.
In these homotopy categories, one can obtain from first principles
the fundamental computations of the cohomology of a point and
of the Steenrod algebra, using only well-known classical computations.
Therefore,
the material in this manuscript does not logically
depend on Voevodsky's work, even though the methods were very much
inspired by his groundbreaking computations.

\subsection{Motivic modular forms}

In classical chromatic homotopy theory, the theory of topological modular forms, introduced by Hopkins and Mahowald \cite{tmf14}, plays a central role in the computations of the $K(2)$-local sphere.

Using a topological model of the cellular stable $\C$-motivic homotopy category, 
one can construct a ``motivic modular forms" spectrum $\mmf$ \cite{GIKR18},
whose motivic \revv{$\F_2$-}cohomology is the quotient of the $\C$-motivic
Steenrod algebra by its subalgebra generated by
$\Sq^1$, $\Sq^2$, and $\Sq^4$.
Just as $\tmf$ plays an essential role in studies of the classical
Adams spectral sequence \rev{\cite{BMQ21} \cite{BBBCX}}, $\mmf$ is an essential tool for
motivic computations.  The $\C$-motivic Adams spectral sequence
for $\mmf$ can be analyzed completely \cite{Isaksen18}, and naturality
of Adams spectral sequences along the unit map of $\mmf$
provides much information about the behavior of the $\C$-motivic
Adams spectral sequence for the $\C$-motivic sphere spectrum.

\section{Main results}

We summarize our main results in the following theorem and corollaries.

\begin{thm}
\label{thm:main-Adams}
The $\C$-motivic Adams spectral sequence for the $\C$-motivic sphere 
spectrum
is displayed in the charts in \cite{Isaksen14a}, up to the $90$-stem.
\end{thm}

The proof of Theorem \ref{thm:main-Adams} consists of a series of
specific computational facts, which are verified throughout this
manuscript.

\begin{cor}
\label{cor:main-Adams}
The classical Adams spectral sequence for the sphere spectrum is 
displayed in the charts in \cite{Isaksen14a}, up to the $90$-stem.
\end{cor}

Corollary \ref{cor:main-Adams} follows immediately from
Theorem \ref{thm:main-Adams}.  One simply inverts $\tau$, or equivalently
ignores $\tau$-torsion.

\rev{
Theorem \ref{thm:main-Adams} could also be used to completely determine
the $E_2$-page and all differentials of the Adams-Novikov spectral
sequence for the sphere spectrum.}
 As described in
\cite{Isaksen14c}*{Chapter 6}, the Adams-Novikov spectral sequence
can be reverse-engineered from information about $\C$-motivic
stable homotopy groups.

\begin{cor}
\label{cor:cl-order}
Table \ref{tab:order} describes the stable homotopy groups $\pi_k$
for all values of $k$ up to $90$.
\end{cor}

We adopt the following
notation in Table \ref{tab:order}.  
An integer $n$ stands for the cyclic abelian group $\Z/n$; 
the symbol $\mydot$ by itself stands for the trivial group;
the expression $n\mydot m$ stands for the direct sum $\Z/n \oplus \Z/m$;
and $n^j$ stands for the direct sum of $j$ copies of $\Z/n$.
The horizontal line after dimension 61 indicates the range
in which our computations are new information.

Table \ref{tab:order} describes each group $\pi_k$ as the direct sum
of three subgroups: the $2$-primary $v_1$-torsion, the
odd primary $v_1$-torsion, and the $v_1$-periodic subgroups.

The last column of Table \ref{tab:order} describes the groups of homotopy spheres  that classify smooth structures on spheres in dimensions at least 5. See Section~\ref{section homotopy sphere} and Theorem~\ref{thm homotopy sphere} for more details.

Starting in dimension 84, there remain some uncertainties in the $2$-primary $v_1$-torsion.
In most cases, these uncertainties mean that the order of some stable
homotopy groups are
known only up to factors of $2$.  In a few cases, the additive group
structures are also undetermined.

These uncertainties have two causes.  First, there are a handful
of differentials that remain unresolved; they are listed in
\rev{Table \ref{tab:Adams-higher}.}
Second, there are some possible hidden $2$ extensions
that remain unresolved.

\begin{longtable}{l|llll}
\caption{Stable homotopy groups up to dimension $90$.
$n$ stands for $\Z/n$;
$n\mydot m$ stands for $\Z/n \oplus \Z/m$;
and $n^j$ stands for $(\Z/n)^j$.
\label{tab:order} 
} \\
\toprule
$k$ & $v_1$-torsion & $v_1$-torsion & $v_1$-periodic & group of  \\
& at the prime 2 & at odd & & smooth structures\\
& & primes \\
\midrule \endfirsthead
\caption[]{Stable homotopy groups up to dimension $90$.
$n$ stands for $\Z/n$;
$n\mydot m$ stands for $\Z/n \oplus \Z/m$;
and $n^j$ stands for $(\Z/n)^j$.
} \\
\toprule
$k$ & $v_1$-torsion & $v_1$-torsion & $v_1$-periodic & group of  \\
& at the prime 2 & at odd & & smooth structures\\
& & primes \\
\midrule \endhead
\bottomrule \endfoot
$1$ & $\mydot$ & $\mydot$ & $2$ & $\mydot$ \\
$2$ & $\mydot$ & $\mydot$ & $2$ & $\mydot$ \\
$3$ & $\mydot$ & $\mydot$ & $8\mydot3$ & $\mydot$ \\
$4$ & $\mydot$ & $\mydot$ & $\mydot$ & $?$ \\
$5$ & $\mydot$ & $\mydot$ & $\mydot$ & $\mydot$ \\
$6$ & $2$ & $\mydot$ & $\mydot$ & $\mydot$ \\
$7$ & $\mydot$ & $\mydot$ & $16\mydot3\mydot5$ & $\underline{b_2}$ \\
$8$ & $2$ & $\mydot$ & $2$ & $2$\\
$9$ & $2$ & $\mydot$ & $2^2$ & $\underline{2}\mydot 2^2$ \\
$10$ & $\mydot$ & $3$ & $2$ & $2\mydot3$ \\
$11$ & $\mydot$ & $\mydot$ & $8\mydot9\mydot7$ & $\underline{b_3}$ \\
$12$ & $\mydot$ & $\mydot$ & $\mydot$ & $\mydot$ \\
$13$ & $\mydot$ & $3$ & $\mydot$ & $3$ \\
$14$ & $2\mydot2$ & $\mydot$ & $\mydot$ & $2$ \\
$15$ & $2$ & $\mydot$ & $32\mydot3\mydot5$ & $\underline{b_4}\mydot2$ \\
$16$ & $2$ & $\mydot$ & $2$ & $2$ \\
$17$ & $2^2$ & $\mydot$ & $2^2$ & $\underline{2}\mydot2^3$ \\
$18$ & $8$ & $\mydot$ & $2$ & $2\mydot8$ \\
$19$ & $2$ & $\mydot$ & $8\mydot3\mydot11$ & $\underline{b_5}\mydot2$\\
$20$ & $8$ & $3$ & $\mydot$ & $8\mydot3$\\
$21$ & $2^2$ & $\mydot$ & $\mydot$ & $\underline{2}\mydot2^2$ \\
$22$ & $2^2$ & $\mydot$ & $\mydot$ & $2^2$ \\
$23$ & $2\mydot8$ & $3$ & $16\mydot9\mydot5\mydot7\mydot13$ & $\underline{b_6}\mydot2\mydot8\mydot3$ \\
$24$ & $2$ & $\mydot$ & $2$ & $2$ \\
$25$ & $\mydot$ & $\mydot$ & $2^2$ & $\underline{2}\mydot2$ \\
$26$ & $2$ & $3$ & $2$ & $2^2\mydot3$ \\
$27$ & $\mydot$ & $\mydot$ & $8\mydot3$ & $\underline{b_7}$ \\
$28$ & $2$ & $\mydot$ & $\mydot$ & $2$ \\
$29$ & $\mydot$ & $3$ & $\mydot$ & $3$ \\
$30$ & $2$ & $3$ & $\mydot$ & $3$ \\
$31$ & $2^2$ & $\mydot$ & $64\mydot3\mydot5\mydot17$ & $\underline{b_8}\mydot2^2$ \\
$32$ & $2^3$ & $\mydot$ & $2$ & $2^3$ \\
$33$ & $2^3$ & $\mydot$ & $2^2$ & $\underline{2}\mydot2^4$ \\
$34$ & $2^2\mydot4$ & $\mydot$ & $2$ & $2^3\mydot4$ \\
$35$ & $2^2$ & $\mydot$ & $8\mydot27\mydot7\mydot19$ & $\underline{b_9}\mydot2^2$ \\
$36$ & $2$ & $3$ & $\mydot$ & $2\mydot3$ \\
$37$ & $2^2$ & $3$ & $\mydot$ & $\underline{2}\mydot2^2\mydot3$ \\
$38$ & $2\mydot4$ & $3\mydot5$ & $\mydot$ & $2\mydot4\mydot3\mydot5$ \\
$39$ & $2^5$ & $3$ & $16\mydot3\mydot25\mydot11$ & $\underline{b_{10}}\mydot2^5\mydot3$ \\
$40$ & $2^4\mydot4$ & $3$ & $2$ & $2^4\mydot4\mydot3$ \\
$41$ & $2^3$ & $\mydot$ & $2^2$ & $\underline{2}\mydot2^4$ \\
$42$ & $2\mydot8$ & $3$ & $2$  & $2^2\mydot8\mydot3$\\
$43$ & $\mydot$ & $\mydot$ & $8\mydot3\mydot23$ & $\underline{b_{11}}$ \\
$44$ & $8$ & $\mydot$ & $\mydot$ & $8$ \\
$45$ & $2^3\mydot16$ & $9\mydot5$ & $\mydot$ & $\underline{2}\mydot2^3\mydot16\mydot9\mydot5$ \\
$46$ & $2^4$ & $3$ & $\mydot$ & $2^4\mydot3$ \\
$47$ & $2^3\mydot4$ & $3$ &  $32\mydot9\mydot5\mydot7\mydot13$ & $\underline{b_{12}}\mydot2^3\mydot4\mydot3$  \\
$48$ & $2^3\mydot4$ & $\mydot$ & $2$ & $2^3\mydot4$ \\
$49$ & $\mydot$ & $3$ & $2^2$ & $\underline{2}\mydot2\mydot3$ \\
$50$ & $2^2$ & $3$ & $2$ & $2^3\mydot3$ \\
$51$ & $2\mydot8$ & $\mydot$ & $8\mydot3$ & $\underline{b_{13}}\mydot2\mydot8$ \\
$52$ & $2^3$ & $3$ & $\mydot$ & $2^3\mydot3$ \\
$53$ & $2^4$ & $\mydot$ & $\mydot$ & $\underline{2}\mydot2^4$ \\
$54$ & $2\mydot4$ & $\mydot$ & $\mydot$ & $2\mydot4$ \\
$55$ & $\mydot$ & $3$ & $16\mydot3\mydot5\mydot29$ & $\underline{b_{14}}\mydot3$ \\
$56$ & $\mydot$ & $\mydot$ & $2$ & $\mydot$ \\
$57$ & $2$ & $\mydot$ & $2^2$ & $\underline{2}\mydot2^2$ \\
$58$ & $2$ & $\mydot$ & $2$  & $2^2$\\
$59$ & $2^2$ & $\mydot$ & $8\mydot9\mydot7\mydot11\mydot31$ & $\underline{b_{15}}\mydot2^2$ \\
$60$ & $4$ & $\mydot$ & $\mydot$ & $4$ \\
$61$ & $\mydot$ & $\mydot$ & $\mydot$  & $\mydot$\\
\hline
$62$ & $2^4$ & $3$ & $\mydot$ & $2^3\mydot3$ \\
$63$ & $2^2\mydot4$ & $\mydot$ & $128\mydot3\mydot5\mydot17$  & $\underline{b_{16}}\mydot2^2\mydot4$\\
$64$ & $2^5\mydot4$ & $\mydot$ & $2$  & $2^5\mydot4$\\
$65$ & $2^7\mydot4$ & $3$ & $2^2$  & $\underline{2}\mydot2^8\mydot4\mydot3$\\
$66$ & $2^5\mydot8$ & $\mydot$ & $2$  & $2^6\mydot8$\\
$67$ & $2^3\mydot4$ & $\mydot$ & $8\mydot3$  & $\underline{b_{17}}\mydot2^3\mydot4$\\
$68$ & $2^3$ & $3$ & $\mydot$ & $2^3\mydot3$\\
$69$ & $2^4$ & $\mydot$ & $\mydot$ & $\underline{2}\mydot2^4$\\
$70$ & $2^5\mydot4^2$ & $\mydot$ & $\mydot$ & $2^5\mydot4^2$\\
$71$ & $2^6\mydot4\mydot8$  & $\mydot$ & $16\mydot27\mydot5\mydot7\mydot13\mydot19\mydot37$ & $\underline{b_{18}}\mydot2^6\mydot4\mydot8$\\
$72$ & $2^7$ & $3$ & $2$ & $2^7\mydot3$\\
$73$ & $2^5$ & $\mydot$ & $2^2$& $\underline{2}\mydot2^6$ \\
$74$ & $4^3$ & $3$ & $2$ & $2\mydot4^3\mydot3$ \\
$75$ & $2$ & $9$ & $8\mydot3$ & $\underline{b_{19}}\mydot2\mydot9$\\
$76$ & $2^2\mydot4$ & $5$ & $\mydot$ & $2^2\mydot4\mydot5$\\
$77$ & $2^5\mydot4$ & $\mydot$ & $\mydot$ & $\underline{2}\mydot2^5\mydot4$\\
$78$ & $2^3\mydot4^2$ & $3$ & $\mydot$ & $2^3\mydot4^2\mydot3$\\
$79$ & $2^6\mydot4$ & $\mydot$ & $32\mydot3\mydot25\mydot11\mydot41$ & $\underline{b_{20}}\mydot2^6\mydot4$\\
$80$ & $2^8$ & $\mydot$ & $2$ & $2^8$\\
$81$ & $2^3\mydot4\mydot8$ & $3^2$ & $2^2$ & $\underline{2}\mydot2^4\mydot4\mydot8\mydot3^2$\\
\revrow
$82$ & $2^5\mydot8$ & $3\mydot7$ & $2$ & $2^6\mydot8\mydot3\mydot7$ or $2^4\mydot4\mydot8\mydot3\mydot7$ \\
\revrow
$83$ & $2^3\mydot8$ & $5$ & $8\mydot9\mydot49\mydot43$ & $\underline{b_{21}}\mydot2^3\mydot8\mydot5$ \\
$84$ & $2^6$ or $2^5$  & $3^2$ & $\mydot$ & $2^6\mydot3^2$ or $2^5\mydot3^2$  \\
\revrow
$85$ & $2^6\mydot4^2$ or $2^5\mydot4^2$ or & $3^2$ & $\mydot$ & $2^6\mydot4^2\mydot3^2$ or $2^5\mydot4^2\mydot3^2$  \\
	\revrow
	& $2^4\mydot4^3$ or $2^7\mydot4$ & & & or $2^4\mydot4^3\mydot3^2$ or $2^7\mydot4\mydot3^2$ \\
\revrow
$86$ &  $2^4\mydot8^2$ or $2^2\mydot4\mydot8^2$ & $3\mydot5$ & $\mydot$ &  $2^4\mydot8^2\mydot3\mydot5$ or $2^2\mydot4\mydot8^2\mydot3\mydot5$ \\
\revrow
$87$ & $2^5\mydot4$ & $\mydot$ & $16\mydot3\mydot5\mydot23$ & $\underline{b_{22}}\mydot2^5\mydot4$ \\
$88$ & $2^4\mydot4$  & $\mydot$ & $2$ & $2^4\mydot4$ \\
$89$ & $2^3$ & $\mydot$ & $2^2$ & $\underline{2}\mydot2^4$ \\
$90$ & $2^3\mydot8$ or $2^2\mydot8$ & $3$ & $2$ & $2^4\mydot8\mydot3$ or $2^3\mydot8\mydot3$
\end{longtable}

Figure \ref{fig:Hatcher} displays the $2$-primary stable homotopy groups 
in a graphical format that is
a modification by Allen Hatcher of Adams spectral sequence charts
\cite{Hatcher}.
Vertical chains of $n$ dots 
indicate $\Z/2^n$.  The non-vertical
lines indicate multiplications by $\eta$ and $\nu$.  
The blue dots represent the $v_1$-periodic subgroups.
The green dots are associated to the topological modular forms
spectrum $\tmf$.  These elements are detected by
the unit map from the sphere spectrum to $\tmf$, either in homotopy
or in the algebraic $\Ext$ groups that serve as Adams $E_2$-pages.

Finally, the red dots indicate uncertainties.  In addition, in higher
stems, there are possible extensions by $2$, $\eta$, and $\nu$
that are not indicated in Figure \ref{fig:Hatcher}.
See Tables \ref{tab:2-extn-possible}, \ref{tab:eta-extn-possible},
and \ref{tab:nu-extn-possible} for more details about these possible
extensions.

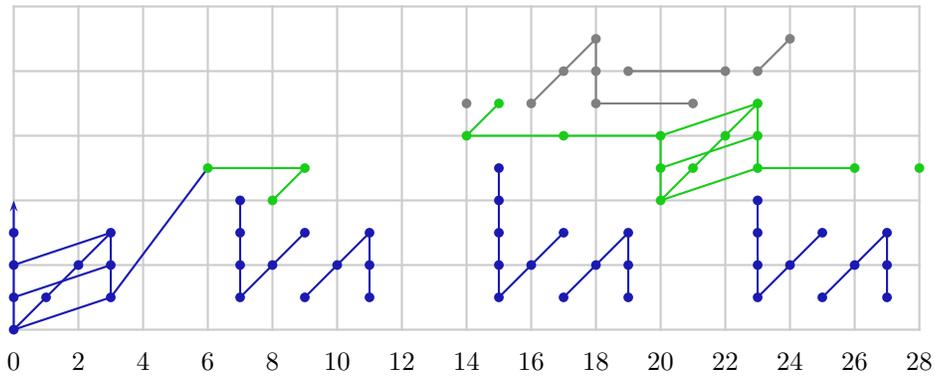
\begin{figure}[htbp!]

\caption{$2$-primary stable homotopy groups
\label{fig:Hatcher}}

\input{Hatcher0-28.tex}

\end{figure}

\begin{center}
\input{Hatcher30-60.tex}
\end{center}

\begin{center}
\input{Hatcher62-90.tex}
\end{center}

The orders of individual $2$-primary stable homotopy groups do not 
follow a clear pattern, with large increases and decreases seemingly
at random.  However, an empirically observed pattern emerges
if we consider the cumulative size of the groups, i.e., the product
of the orders of 
all $2$-primary stable homotopy groups from dimension $1$ to
dimension $k$.  

Our data strongly suggest that asymptotically, there is a linear relationship between $k^2$ 
and the logarithm of this product of orders.
In other words, the number of dots in Figure \ref{fig:Hatcher} in
stems $1$ through $k$ is linearly proportional to $k^2$.
Correspondingly,
the number of dots in the classical Adams $E_\infty$-page in
stems $1$ through $k$ is linearly proportional to $k^2$.
Thus, in extending from dimension 60 to dimension 90, the overall 
size of the computation more than doubles.
Specifically, through dimension 60, the cumulative rank of the Adams $E_\infty$-page is 199, and is 435 through
dimension 90.  Similarly, through dimension 60, the cumulative
rank of the Adams $E_2$-page is 488, and is 1,461 through dimension 90.

\begin{conj}
Let $f(k)$ be the product of the orders of the $2$-primary
stable homotopy groups in dimensions $1$ through $k$.
There exists a non-zero constant $C$ such that
\[
\lim_{k \map \infty} \frac{\log_2 f(k)}{k^2} = C.
\]
\end{conj}

One interpretation of this conjecture is that the expected value of
the logarithm of the order of the 2-primary component of $\pi_k$
grows linearly in $k$.  
We have only data to support the conjecture, and we have not formulated
a mathematical rationale.
It is possible that in higher stems, new phenomena occur that alter the
growth rate of the stable homotopy groups. 

By comparison, data indicates that the growth rate of the
Adams $E_2$-page is qualitatively greater than the growth rate of the
Adams $E_\infty$-page.  This apparent mismatch has implications 
for the frequency of Adams differentials.

\section{Remaining uncertainties}

Some uncertainties remain in the analysis of the first 90 stable stems.
All undetermined possible differentials in this range are mentioned within
Table \ref{tab:Adams-higher}.  
All of these uncertainties concern the Adams differentials $d_r$
for $r \geq 9$.
This means that
the orders of some of the stable homotopy groups are known only up
to factors of $2$.

In addition, there are some possible hidden extensions by $2$, $\eta$,
and $\nu$ that remain unresolved.  Tables \ref{tab:2-extn-possible},
\ref{tab:eta-extn-possible}, and \ref{tab:nu-extn-possible} 
summarize these possibilities.  The presence of
unknown hidden extensions by $2$ means that
the group structures of some stable homotopy groups are not known, 
even though their orders are known.

\section{Groups of homotopy spheres}\label{section homotopy sphere}

An important application of stable homotopy group computations 
is to the work of Kervaire and Milnor \cite{KervaireMilnor}
on the classification of smooth structures on spheres in dimensions at least 5. Let $\Theta_n$ be the group of $h$-cobordism classes of homotopy $n$-spheres. This group classifies the differential structures on $S^n$ for $n\geq5$. It has a subgroup $\Theta_n^{bp}$, which consists of homotopy spheres that bound parallelizable manifolds. The relation between $\Theta_n$ and the stable homotopy group $\pi_n$ 
is summarized in Theorem \ref{km}. See also \cite{Milnor} for a survey on this subject.

\begin{thm}\label{km}(Kervaire-Milnor \cite{KervaireMilnor})
Suppose that $n\geq 5$.
\begin{enumerate}
\item
The subgroup $\Theta_n^{bp}$ is cyclic, and has the following order:
\begin{equation*}
|\Theta_n^{bp}|=\left\{
\begin{split}
1 & , ~~\text{if~~}n\text{~~is even,} \\
1 \text{~~or~~} 2 & , ~~\text{if~~}n = 4k+1,\\
b_k & , ~~\text{if~~}n = 4k-1.
\end{split}
\right.
\end{equation*}
Here $b_k$ is $2^{2k-2}(2^{2k-1}-1)$ times the numerator of $8\zeta(1-2k)$,
where $\zeta$ is the Riemann zeta function.\\
\item
For $n \not\equiv 2 ~(mod ~4)$, there is an exact sequence
\begin{displaymath}
    \xymatrix{
   0 \ar[r] & \Theta_n^{bp} \ar[r] & \Theta_n \ar[r] & \pi_n/J \ar[r] & 0.
    }
\end{displaymath}
Here $\pi_n/J$ is the cokernel of the $J$-homomorphism.\\

\item
For $n \equiv 2 ~(mod ~4)$, there is an exact sequence
\begin{displaymath}
    \xymatrix{
   0 \ar[r] & \Theta_n^{bp} \ar[r] & \Theta_n \ar[r] & \pi_n/J \ar[r]^\Phi & \mathbb{Z}/2 \ar[r] & \Theta_{n-1}^{bp} \ar[r] & 0.
    }
\end{displaymath}
Here the map $\Phi$ is the Kervaire invariant.
\end{enumerate}
\end{thm}
The first few values, and then estimates, of the numbers $b_k$
(for $k \geq 2$)
are given by the sequence
$$28, \ 992, \ 8128, \ 261632, \ 1.45\times 10^9, \ 6.71\times 10^7, \ 1.94\times 10^{12}, \ 7.54\times 10^{14}, \ldots .$$

\begin{thm}\label{thm homotopy sphere}
The last column of Table~\ref{tab:order}
describes the groups $\Theta_n$ for $n \leq 90$, with the exception of $n=4$. 
The underlined symbols denote the contributions from $\Theta_n^{bp}$.
\end{thm}
 
The cokernel of the $J$-homomorphism is slightly
different than the $v_1$-torsion part of $\pi_n$ at the prime 2. 
In dimensions $8m+1$ and $8m+2$, there are classes detected by $P^m h_1$ 
and $P^m h_1^2$ in the Adams spectral sequence. 
These classes are $v_1$-periodic, in the sense that
they are detected by the $K(1)$-local sphere. 
However, they are also in the cokernel of the $J$-homomorphism.  

We restate the following conjecture from \cite{WangXu17}, which is based on the current knowledge of stable stems and a problem proposed by Milnor \cite{Milnor}.

\begin{conj}\label{conjecture homotopy sphere}
In dimensions greater than 4, the only spheres with unique smooth structures are $S^5$, $S^6$, $S^{12}$, $S^{56}$, and $S^{61}$.	
\end{conj}

Uniqueness in
dimensions 5, 6 and 12 was known to Kervaire and Milnor \cite{KervaireMilnor}. 
Uniqueness in dimension 56 is due to the first author \cite{Isaksen14c}, 
and uniqueness in dimension 61 is due to the second and the third authors \cite{WangXu17}. 

Conjecture~\ref{conjecture homotopy sphere} is equivalent to the claim
that the group $\Theta_n$ is not of order $1$ for dimensions greater than $61$. This conjecture has been confirmed in all odd dimensions by the second and the third authors \cite{WangXu17} based on the work of Hill, Hopkins, and Ravenel \cite{HHR}, and \rev{in more than half of the even dimensions by Behrens, Hill, Hopkins, Mahowald and Quigley \cite{Behrens}\cite{BMQ21}.}

\section{Notation}
\label{sctn:notation}

The cohomology of the Steenrod algebra is highly irregular, so
consistent naming systems for elements presents a challenge.
A list of multiplicative generators appears in Table \ref{tab:Adams-d2}.
To a large extent, we rely on the traditional names for elements,
as used in \cite{Bruner97}, \cite{Isaksen14c}, \cite{Tangora70a},
and elsewhere.  However, we have adopted some new conventions 
in order to partially systematize the names of elements.

First, we use the symbol $\Delta x$ to indicate an element that is
represented by $v_2^4 x$ in the May spectral sequence.  This use of
$\Delta$ is consistent with the role that $v_2^4$ plays in the homotopy
of $\tmf$, where it detects the discriminant element $\Delta$.
For example, instead of the traditional symbol $r$, we use the
name $\D h_2^2$.

Second, the symbol $M$ indicates the Massey product
operator $\langle -, h_0^3, g_2 \rangle$.  For example, instead of the
traditional symbol $B_1$, we use the name $M h_1$.

Similarly, the symbol $g$ indicates the Massey product
operator $\langle -, h_1^4, h_4 \rangle$. For example, we write
$h_2 g$ for the indecomposable element $\langle h_2, h_1^4, h_4 \rangle$.

Eventually, we encounter elements that neither have traditional names,
nor can be named using symbols such as $P$, $\Delta$, $M$, and $g$.
In these cases, we use arbitrary names of the form $x_{s,f}$, where
$s$ and $f$ are the stem and Adams filtration of the element.

The last column of Table \ref{tab:Adams-d2} gives alternative names,
if any, for each multiplicative generator.  These alternative names
appear in at least one of 
\cite{Bruner97} \cite{Isaksen14c} \cite{Tangora70a}.

\begin{remark}
One specific element deserves further discussion.
\revv{In the cohomology of the motivic Steenrod algebra},
we define $\tau Q_3$ to be the unique \revv{non-zero} element 
\revv{in degree $(67, 5, 35)$}
such that
$h_3 \cdot \tau Q_3 = 0$.
This choice is not compatible with the notation of \cite{Isaksen14c}.
The element $\tau Q_3$ from \cite{Isaksen14c} equals
the element $\tau Q_3 + \tau n_1$ in this manuscript.
\end{remark}

We shall also extensively study the Adams spectral sequence for
the cofiber of $\tau$.  See Section \ref{sctn:Ctau-naming} for
more discussion of the names of elements in this spectral sequence,
and how they relate to the Adams spectral sequence for the sphere.

Table \ref{tab:notation} gives some notation for elements in
$\pi_{*,*}$.  Many of these names follow standard usage, but we 
have introduced additional non-standard elements such as
$\kappa_1$ and $\kappabar_2$.  These elements are defined by
the classes in the Adams $E_\infty$-page that detect them.
In some cases, this style of definition leaves indeterminacy because
of the presence of elements in the $E_\infty$-page in higher filtration.
In some of these cases, Table \ref{tab:notation} provides additional
defining information.  Beware that this additional defining information
does not completely specify a unique element in $\pi_{*,*}$ in 
all cases.  For the purposes of our computations, these remaining
indeterminacies are not consequential.

\revv{Here is a list of the key notation that we use extensively:}

\begin{itemize}
\item
\revv{
Because we have completed at $2$, we have a map
$\tau: S^{0,-1} \map S^{0,0}$ \cite{HKO11}*{Lemma 25}.  We write
$C\tau$ for its cofiber.}
We can also write $S/\tau$ for this $\C$-motivic spectrum, but the latter
notation is more cumbersome.
\item
$\Ext = \Ext_\C$ is the cohomology of the $\C$-motivic Steenrod algebra.
It is graded in the form $(s,f,w)$, where $s$ is the stem (i.e.,
the total degree minus the Adams filtration), $f$ is the Adams
filtration (i.e., the homological degree), and $w$ is the motivic weight.
\item
$\Ext_\cl$ is the cohomology of the classical Steenrod algebra.
It is graded in the form $(s,f)$, where $s$ is the stem (i.e.,
the total degree minus the Adams filtration), and $f$ is the Adams
filtration (i.e., the homological degree).
\item
$\pi_{*,*}$ are the $2$-completed $\C$-motivic stable homotopy groups.
\item
$H^*(S; BP)$ is the Adams-Novikov $E_2$-page for the classical
sphere spectrum, i.e., $\Ext_{BP_* BP} (BP_*, BP_*)$.
\item
$H^*(S/2; BP)$ is the Adams-Novikov $E_2$-page for the classical
mod $2$ Moore spectrum, i.e., $\Ext_{BP_* BP} (BP_*, BP_*/2)$.
\end{itemize}

\section{How to use this manuscript}

The manuscript is oriented around a series of tables to be found
in Chapter \ref{ch:table}.  In a sense, the rest of the manuscript
consists of detailed arguments for establishing each of the 
computations listed in the tables.  We have attempted to give 
references and cross-references within these tables, so that the 
reader can more easily find the specific arguments pertaining to
each computation.

We have attempted to make the arguments accessible to users who
do not intend to read the manuscript in its entirety.  To some
extent, with an understanding of how the manuscript is structured,
it is possible to extract information about a particular homotopy
class in isolation.
\revv{
A secondary goal is to offer a 
a guide to the computational techniques in use in 
stable homotopy theory today.
}

We assume that the reader is also referring to the Adams charts in
\cite{Isaksen14a} and \cite{Isaksen18}.  These charts describe the
same information as the tables, except in graphical form.
\revv{Especially when there are multiple elements in a single
degree, the charts can be somewhat ambiguous.  In such cases,
we encourage readers to use the associated spreadsheets
\cite{Isaksen14a}.  These spreadsheets are more cumbersome than charts,
but they are entirely explicit.
}

\rev{The style of this manuscript is very much similar to \cite{Isaksen14c}.}  We will
frequently refer to discussions in \cite{Isaksen14c}, rather than
repeat that same material here in an essentially redundant way.
This is especially true for the first parts of Chapters 2, 3, and 4
of \cite{Isaksen14c}, which discuss respectively the general properties of
$\Ext$, the May spectral sequence, and Massey products; 
the Adams spectral sequence and Toda brackets;
and hidden extensions.

Chapter \ref{ch:background} provides some additional miscellaneous
background material not already covered in \cite{Isaksen14c}.
Chapter \ref{ch:AANSS} discusses the nature of
the machine-generated data that we rely on.  In particular, it
describes our data on the algebraic Novikov spectral sequence,
which is equal to the Adams spectral sequence for the cofiber of $\tau$.
Chapter \ref{ch:Massey} provides some tools for computing
Massey products in $\Ext$, and gives some specific computations.
Chapter \ref{ch:Adams} carries out a detailed analysis of Adams
differentials.
Chapter \ref{ch:Toda} computes some miscellaneous Toda brackets
that are needed for various specific arguments elsewhere.
Chapter \ref{ch:hidden} methodically studies hidden extensions
by $\tau$, $2$, $\eta$, and $\nu$ in the $E_\infty$-page of the
$\C$-motivic Adams spectral sequence.  This chapter also gives some
information about other miscellaneous hidden extensions.
Finally, Chapter \ref{ch:table} includes the tables that summarize
the multitude of specific computations that contribute to our study
of stable homotopy groups.

\section{Acknowledgements}

We thank Agn\`es Beaudry, Joey Beauvais-Feisthauer, Mark Behrens, Robert Bruner, Robert Burklund, Dexter Chua, Paul Goerss, Jesper Grodal, Lars Hesselholt, Mike Hopkins, Peter May, Haynes Miller, Christian Nassau, Doug Ravenel, and John Rognes for advice, suggestions, corrections, support, and encouragement.

%% file: Hatcher0-28.tex
		\begin{tikzpicture}
		[scale=0.43,
		>=stealth,
		every path/.style={line width = 0.645pt},
		dot/.style={circle,
			inner sep=0,
			minimum size=0.11cm},
		height0/.style={
			draw={red},
			fill={red}},
		height1/.style={
			draw={darkblue},
			fill={darkblue}},
		height2/.style={
			draw={darkgreen},
			fill={darkgreen}},
		height3/.style={
			draw={midgray},
			fill={midgray}},
		tower/.style={->}]
			\draw[step=2cm, black!20, thin] (0,0) grid (28,10);
			\foreach \x in {0,2,...,28}
				\draw (\x, -1) node{\x};
			\draw[, height1] (0,0) -- (0,1);
			\draw[, height1] (0,0) -- (1,1);
			\draw[, height1] (0,0) -- (3,1);
			\node[dot, height1] at (0,0) {};
			\draw[, height1] (7,1) -- (7,2);
			\draw[, height1] (7,1) -- (8,2);
			\node[dot, height1] at (7,1) {};
			\draw[, height1] (9,1) -- (10,2);
			\node[dot, height1] at (9,1) {};
			\draw[, height2] (14,6) -- (15,7);
			\draw[, height2] (14,6) -- (17,6);
			\node[dot, height2] at (14,6) {};
			\node[dot, height3] at (14,7) {};
			\draw[, height1] (15,1) -- (15,2);
			\draw[, height1] (15,1) -- (16,2);
			\node[dot, height1] at (15,1) {};
			\draw[, height3] (16,7) -- (17,8);
			\node[dot, height3] at (16,7) {};
			\draw[, height1] (17,1) -- (18,2);
			\node[dot, height1] at (17,1) {};
			\draw[, height3] (19,8) -- (22,8);
			\node[dot, height3] at (19,8) {};
			\draw[, height1] (23,1) -- (23,2);
			\draw[, height1] (23,1) -- (24,2);
			\node[dot, height1] at (23,1) {};
			\draw[, height3] (23,8) -- (24,9);
			\node[dot, height3] at (23,8) {};
			\draw[, height1] (25,1) -- (26,2);
			\node[dot, height1] at (25,1) {};
			\node[dot, height2] at (28,5) {};
			\draw[tower, height1] (0,3) -- +(0,1);
			\node[dot, height1] at (0,3) {};
			\draw[, height1] (0,1) -- (0,2);
			\draw[, height1] (0,1) -- (3,2);
			\node[dot, height1] at (0,1) {};
			\draw[, height1] (0,2) -- (0,3);
			\draw[, height1] (0,2) -- (3,3);
			\node[dot, height1] at (0,2) {};
			\draw[, height1] (1,1) -- (2,2);
			\node[dot, height1] at (1,1) {};
			\draw[, height1] (2,2) -- (3,3);
			\node[dot, height1] at (2,2) {};
			\draw[, height1] (3,1) -- (3,2);
			\draw[, height2] (3,1) -- (6,5);
			\node[dot, height1] at (3,1) {};
			\draw[, height1] (3,2) -- (3,3);
			\node[dot, height1] at (3,2) {};
			\node[dot, height1] at (3,3) {};
			\draw[, height2] (6,5) -- (9,5);
			\node[dot, height2] at (6,5) {};
			\draw[, height1] (7,2) -- (7,3);
			\node[dot, height1] at (7,2) {};
			\draw[, height1] (7,3) -- (7,4);
			\node[dot, height1] at (7,3) {};
			\node[dot, height1] at (7,4) {};
			\draw[, height1] (8,2) -- (9,3);
			\node[dot, height1] at (8,2) {};
			\draw[, height2] (8,4) -- (9,5);
			\node[dot, height2] at (8,4) {};
			\node[dot, height1] at (9,3) {};
			\node[dot, height2] at (9,5) {};
			\draw[, height1] (10,2) -- (11,3);
			\node[dot, height1] at (10,2) {};
			\draw[, height1] (11,1) -- (11,2);
			\node[dot, height1] at (11,1) {};
			\draw[, height1] (11,2) -- (11,3);
			\node[dot, height1] at (11,2) {};
			\node[dot, height1] at (11,3) {};
			\draw[, height1] (15,2) -- (15,3);
			\node[dot, height1] at (15,2) {};
			\draw[, height1] (15,3) -- (15,4);
			\node[dot, height1] at (15,3) {};
			\draw[, height1] (15,4) -- (15,5);
			\node[dot, height1] at (15,4) {};
			\node[dot, height1] at (15,5) {};
			\node[dot, height2] at (15,7) {};
			\draw[, height1] (16,2) -- (17,3);
			\node[dot, height1] at (16,2) {};
			\node[dot, height1] at (17,3) {};
			\draw[, height2] (17,6) -- (20,6);
			\node[dot, height2] at (17,6) {};
			\draw[, height3] (17,8) -- (18,9);
			\node[dot, height3] at (17,8) {};
			\draw[, height1] (18,2) -- (19,3);
			\node[dot, height1] at (18,2) {};
			\draw[, height3] (18,7) -- (18,8);
			\draw[, height3] (18,7) -- (21,7);
			\node[dot, height3] at (18,7) {};
			\draw[, height3] (18,8) -- (18,9);
			\node[dot, height3] at (18,8) {};
			\node[dot, height3] at (18,9) {};
			\draw[, height1] (19,1) -- (19,2);
			\node[dot, height1] at (19,1) {};
			\draw[, height1] (19,2) -- (19,3);
			\node[dot, height1] at (19,2) {};
			\node[dot, height1] at (19,3) {};
			\draw[, height2] (20,4) -- (20,5);
			\draw[, height2] (20,4) -- (21,5);
			\draw[, height2] (20,4) -- (23,5);
			\node[dot, height2] at (20,4) {};
			\draw[, height2] (20,5) -- (20,6);
			\draw[, height2] (20,5) -- (23,6);
			\node[dot, height2] at (20,5) {};
			\draw[, height2] (20,6) -- (23,7);
			\node[dot, height2] at (20,6) {};
			\draw[, height2] (21,5) -- (22,6);
			\node[dot, height2] at (21,5) {};
			\node[dot, height3] at (21,7) {};
			\draw[, height2] (22,6) -- (23,7);
			\node[dot, height2] at (22,6) {};
			\node[dot, height3] at (22,8) {};
			\draw[, height1] (23,2) -- (23,3);
			\node[dot, height1] at (23,2) {};
			\draw[, height1] (23,3) -- (23,4);
			\node[dot, height1] at (23,3) {};
			\node[dot, height1] at (23,4) {};
			\draw[, height2] (23,5) -- (23,6);
			\draw[, height2] (23,5) -- (26,5);
			\node[dot, height2] at (23,5) {};
			\draw[, height2] (23,6) -- (23,7);
			\node[dot, height2] at (23,6) {};
			\node[dot, height2] at (23,7) {};
			\draw[, height1] (24,2) -- (25,3);
			\node[dot, height1] at (24,2) {};
			\node[dot, height3] at (24,9) {};
			\node[dot, height1] at (25,3) {};
			\draw[, height1] (26,2) -- (27,3);
			\node[dot, height1] at (26,2) {};
			\node[dot, height2] at (26,5) {};
			\draw[, height1] (27,1) -- (27,2);
			\node[dot, height1] at (27,1) {};
			\draw[, height1] (27,2) -- (27,3);
			\node[dot, height1] at (27,2) {};
			\node[dot, height1] at (27,3) {};
		\end{tikzpicture}

%% file: Hatcher30-60.tex
		\begin{tikzpicture}
		[scale=0.40,
		>=stealth,
		every path/.style={line width = 0.6000000000000001pt},
		dot/.style={circle,
			inner sep=0,
			minimum size=0.10cm},
		height0/.style={
			draw={red},
			fill={red}},
		height1/.style={
			draw={darkblue},
			fill={darkblue}},
		height2/.style={
			draw={darkgreen},
			fill={darkgreen}},
		height3/.style={
			draw={midgray},
			fill={midgray}},
		tower/.style={->}]
			\draw[step=2cm, black!20, thin] (30,0) grid (60,16);
			\foreach \x in {30,32,...,60}
				\draw (\x, -1) node{\x};
			\draw[, height3] (30,7) -- (31,8);
			\draw[, height3] (30,7) -- (33,7);
			\node[dot, height3] at (30,7) {};
			\draw[, height1] (31,1) -- (31,2);
			\draw[, height1] (31,1) -- (32,2);
			\node[dot, height1] at (31,1) {};
			\draw[, height3] (31,9) -- (34,9);
			\node[dot, height3] at (31,9) {};
			\draw[, height2] (32,5) -- (33,6);
			\draw[, height2] (32,5) -- (35,5);
			\node[dot, height2] at (32,5) {};
			\draw[, height3] (32,8) -- (35,8);
			\node[dot, height3] at (32,8) {};
			\draw[, height3] (32,10) -- (33,11);
			\node[dot, height3] at (32,10) {};
			\draw[, height1] (33,1) -- (34,2);
			\node[dot, height1] at (33,1) {};
			\draw[, height3] (36,6) -- (39,6);
			\node[dot, height3] at (36,6) {};
			\draw[, height3] (37,9) -- (40,9);
			\node[dot, height3] at (37,9) {};
			\draw[, height1] (39,1) -- (39,2);
			\draw[, height1] (39,1) -- (40,2);
			\node[dot, height1] at (39,1) {};
			\draw[, height2] (39,5) -- (40,6);
			\draw[, height2] (39,5) -- (42,7);
			\node[dot, height2] at (39,5) {};
			\draw[, height3] (39,7) -- (40,8);
			\node[dot, height3] at (39,7) {};
			\node[dot, height3] at (39,10) {};
			\draw[, height3] (40,7) -- (41,8);
			\node[dot, height3] at (40,7) {};
			\draw[, height3] (40,11) -- (41,12);
			\node[dot, height3] at (40,11) {};
			\draw[, height1] (41,1) -- (42,2);
			\node[dot, height1] at (41,1) {};
			\draw[, height3] (44,14) -- (44,15);
			\draw[, height3] (44,14) -- (45,15);
			\draw[, height3] (44,14) -- (47,14);
			\node[dot, height3] at (44,14) {};
			\draw[, height2] (45,6) -- (46,7);
			\draw[, height2] (45,6) -- (48,7);
			\node[dot, height2] at (45,6) {};
			\draw[, height3] (45,12) -- (46,13);
			\draw[, height3] (45,12) -- (48,12);
			\node[dot, height3] at (45,12) {};
			\draw[, height1] (47,1) -- (47,2);
			\draw[, height1] (47,1) -- (48,2);
			\node[dot, height1] at (47,1) {};
			\draw[, height3] (47,13) -- (48,14);
			\node[dot, height3] at (47,13) {};
			\draw[, height1] (49,1) -- (50,2);
			\node[dot, height1] at (49,1) {};
			\draw[, height3] (50,7) -- (53,7);
			\node[dot, height3] at (50,7) {};
			\draw[, height2] (52,4) -- (53,5);
			\node[dot, height2] at (52,4) {};
			\node[dot, height3] at (52,8) {};
			\draw[, height2] (54,5) -- (54,6);
			\draw[, height2] (54,5) -- (57,5);
			\node[dot, height2] at (54,5) {};
			\draw[, height1] (55,1) -- (55,2);
			\draw[, height1] (55,1) -- (56,2);
			\node[dot, height1] at (55,1) {};
			\draw[, height1] (57,1) -- (58,2);
			\node[dot, height1] at (57,1) {};
			\node[dot, height3] at (58,6) {};
			\draw[, height2] (59,4) -- (60,5);
			\node[dot, height2] at (59,4) {};
			\node[dot, height3] at (59,6) {};
			\draw[, height1] (31,2) -- (31,3);
			\node[dot, height1] at (31,2) {};
			\draw[, height1] (31,3) -- (31,4);
			\node[dot, height1] at (31,3) {};
			\draw[, height1] (31,4) -- (31,5);
			\node[dot, height1] at (31,4) {};
			\draw[, height1] (31,5) -- (31,6);
			\node[dot, height1] at (31,5) {};
			\node[dot, height1] at (31,6) {};
			\node[dot, height3] at (31,8) {};
			\draw[, height1] (32,2) -- (33,3);
			\node[dot, height1] at (32,2) {};
			\node[dot, height1] at (33,3) {};
			\node[dot, height2] at (33,6) {};
			\draw[, height3] (33,11) -- (34,12);
			\node[dot, height3] at (33,11) {};
			\node[dot, height3] at (33,7) {};
			\draw[, height1] (34,2) -- (35,3);
			\node[dot, height1] at (34,2) {};
			\draw[, height2] (34,4) -- (35,5);
			\node[dot, height2] at (34,4) {};
			\draw[, height3] (34,11) -- (34,12);
			\node[dot, height3] at (34,11) {};
			\node[dot, height3] at (34,12) {};
			\node[dot, height3] at (34,9) {};
			\draw[, height1] (35,1) -- (35,2);
			\node[dot, height1] at (35,1) {};
			\draw[, height1] (35,2) -- (35,3);
			\node[dot, height1] at (35,2) {};
			\node[dot, height1] at (35,3) {};
			\node[dot, height2] at (35,5) {};
			\draw[, height3] (35,8) -- (38,8);
			\node[dot, height3] at (35,8) {};
			\draw[, height3] (37,7) -- (38,8);
			\node[dot, height3] at (37,7) {};
			\node[dot, height3] at (38,8) {};
			\draw[, height3] (38,4) -- (38,5);
			\draw[, height3] (38,4) -- (39,6);
			\node[dot, height3] at (38,4) {};
			\node[dot, height3] at (38,5) {};
			\draw[, height1] (39,2) -- (39,3);
			\node[dot, height1] at (39,2) {};
			\draw[, height1] (39,3) -- (39,4);
			\node[dot, height1] at (39,3) {};
			\node[dot, height1] at (39,4) {};
			\draw[, height3] (39,8) -- (40,9);
			\node[dot, height3] at (39,8) {};
			\node[dot, height3] at (39,6) {};
			\draw[, height1] (40,2) -- (41,3);
			\node[dot, height1] at (40,2) {};
			\draw[, height2] (40,5) -- (40,6);
			\draw[, height2] (40,5) -- (41,6);
			\node[dot, height2] at (40,5) {};
			\node[dot, height2] at (40,6) {};
			\node[dot, height3] at (40,9) {};
			\node[dot, height3] at (40,8) {};
			\node[dot, height1] at (41,3) {};
			\draw[, height2] (41,6) -- (42,7);
			\node[dot, height2] at (41,6) {};
			\node[dot, height3] at (41,8) {};
			\draw[, height3] (41,12) -- (42,13);
			\node[dot, height3] at (41,12) {};
			\draw[, height1] (42,2) -- (43,3);
			\node[dot, height1] at (42,2) {};
			\node[dot, height2] at (42,7) {};
			\draw[, height3] (42,11) -- (42,12);
			\draw[, height3] (42,11) -- (45,11);
			\node[dot, height3] at (42,11) {};
			\draw[, height3] (42,12) -- (42,13);
			\node[dot, height3] at (42,12) {};
			\node[dot, height3] at (42,13) {};
			\draw[, height1] (43,1) -- (43,2);
			\node[dot, height1] at (43,1) {};
			\draw[, height1] (43,2) -- (43,3);
			\node[dot, height1] at (43,2) {};
			\node[dot, height1] at (43,3) {};
			\draw[, height3] (44,15) -- (44,16);
			\node[dot, height3] at (44,15) {};
			\node[dot, height3] at (44,16) {};
			\draw[, height3] (45,15) -- (46,16);
			\node[dot, height3] at (45,15) {};
			\draw[, height3] (45,8) -- (45,9);
			\draw[, height3] (45,8) -- (46,9);
			\draw[, height3] (45,8) -- (48,9);
			\node[dot, height3] at (45,8) {};
			\draw[, height3] (45,9) -- (45,10);
			\draw[, height3] (45,9) -- (48,10);
			\node[dot, height3] at (45,9) {};
			\draw[, height3] (45,10) -- (45,11);
			\node[dot, height3] at (45,10) {};
			\node[dot, height3] at (45,11) {};
			\draw[, height2] (46,7) -- (47,8);
			\node[dot, height2] at (46,7) {};
			\node[dot, height3] at (46,16) {};
			\node[dot, height3] at (46,13) {};
			\draw[, height3] (46,9) -- (47,10);
			\node[dot, height3] at (46,9) {};
			\draw[, height1] (47,2) -- (47,3);
			\node[dot, height1] at (47,2) {};
			\draw[, height1] (47,3) -- (47,4);
			\node[dot, height1] at (47,3) {};
			\draw[, height1] (47,4) -- (47,5);
			\node[dot, height1] at (47,4) {};
			\node[dot, height1] at (47,5) {};
			\draw[, height2] (47,6) -- (47,8);
			\draw[, height2] (47,6) -- (48,7);
			\node[dot, height2] at (47,6) {};
			\node[dot, height2] at (47,8) {};
			\node[dot, height3] at (47,14) {};
			\node[dot, height3] at (47,10) {};
			\draw[, height1] (48,2) -- (49,3);
			\node[dot, height1] at (48,2) {};
			\node[dot, height2] at (48,7) {};
			\draw[, height3] (48,12) -- (51,12);
			\node[dot, height3] at (48,12) {};
			\draw[, height3] (48,9) -- (48,10);
			\draw[, height3] (48,9) -- (51,9);
			\node[dot, height3] at (48,9) {};
			\node[dot, height3] at (48,10) {};
			\node[dot, height3] at (48,14) {};
			\node[dot, height1] at (49,3) {};
			\draw[, height1] (50,2) -- (51,3);
			\node[dot, height1] at (50,2) {};
			\draw[, height3] (50,11) -- (51,12);
			\draw[, height3] (50,11) -- (53,12);
			\node[dot, height3] at (50,11) {};
			\draw[, height1] (51,1) -- (51,2);
			\node[dot, height1] at (51,1) {};
			\draw[, height1] (51,2) -- (51,3);
			\node[dot, height1] at (51,2) {};
			\node[dot, height1] at (51,3) {};
			\draw[, height3] (51,10) -- (51,11);
			\draw[, height3] (51,10) -- (52,11);
			\node[dot, height3] at (51,10) {};
			\draw[, height3] (51,11) -- (51,12);
			\node[dot, height3] at (51,11) {};
			\node[dot, height3] at (51,12) {};
			\draw[, height3] (51,9) -- (54,9);
			\node[dot, height3] at (51,9) {};
			\node[dot, height3] at (52,11) {};
			\node[dot, height2] at (53,5) {};
			\node[dot, height3] at (53,7) {};
			\node[dot, height3] at (53,12) {};
			\draw[, height3] (53,8) -- (54,9);
			\node[dot, height3] at (53,8) {};
			\node[dot, height2] at (54,6) {};
			\node[dot, height3] at (54,9) {};
			\draw[, height1] (55,2) -- (55,3);
			\node[dot, height1] at (55,2) {};
			\draw[, height1] (55,3) -- (55,4);
			\node[dot, height1] at (55,3) {};
			\node[dot, height1] at (55,4) {};
			\draw[, height1] (56,2) -- (57,3);
			\node[dot, height1] at (56,2) {};
			\node[dot, height1] at (57,3) {};
			\draw[, height2] (57,5) -- (60,5);
			\node[dot, height2] at (57,5) {};
			\draw[, height1] (58,2) -- (59,3);
			\node[dot, height1] at (58,2) {};
			\draw[, height1] (59,1) -- (59,2);
			\node[dot, height1] at (59,1) {};
			\draw[, height1] (59,2) -- (59,3);
			\node[dot, height1] at (59,2) {};
			\node[dot, height1] at (59,3) {};
			\draw[, height2] (60,4) -- (60,5);
			\node[dot, height2] at (60,4) {};
			\node[dot, height2] at (60,5) {};
		\end{tikzpicture}

%% file: Hatcher62-90.tex
		\begin{tikzpicture}
		[scale=0.43,
		>=stealth,
		every path/.style={line width = 0.645pt},
		dot/.style={circle,
			inner sep=0,
			minimum size=0.11cm},
		height0/.style={
			draw={red},
			fill={red}},
		height1/.style={
			draw={darkblue},
			fill={darkblue}},
		height2/.style={
			draw={darkgreen},
			fill={darkgreen}},
		height3/.style={
			draw={midgray},
			fill={midgray}},
		tower/.style={->}]
			\draw[step=2cm, black!20, thin] (62,0) grid (90,22);
			\foreach \x in {62,64,...,90}
				\draw (\x, -1) node{\x};
			\node[dot, height3] at (62,8) {};
			\draw[, height3] (62,12) -- (65,12);
			\node[dot, height3] at (62,12) {};
			\draw[, height3] (62,13) -- (63,14);
			\draw[, height3] (62,13) -- (65,13);
			\node[dot, height3] at (62,13) {};
			\draw[, height3] (62,19) -- (63,20);
			\draw[, height3] (62,19) -- (65,20);
			\node[dot, height3] at (62,19) {};
			\draw[, height1] (63,1) -- (63,2);
			\draw[, height1] (63,1) -- (64,2);
			\node[dot, height1] at (63,1) {};
			\draw[, height3] (64,9) -- (65,10);
			\node[dot, height3] at (64,9) {};
			\draw[, height3] (64,8) -- (67,8);
			\node[dot, height3] at (64,8) {};
			\draw[, height1] (65,1) -- (66,2);
			\node[dot, height1] at (65,1) {};
			\draw[, height2] (65,5) -- (66,6);
			\node[dot, height2] at (65,5) {};
			\draw[, height2] (65,7) -- (68,7);
			\node[dot, height2] at (65,7) {};
			\node[dot, height3] at (66,12) {};
			\draw[, height3] (67,15) -- (70,15);
			\node[dot, height3] at (67,15) {};
			\node[dot, height3] at (70,10) {};
			\draw[, height3] (70,14) -- (71,15);
			\draw[, height3] (70,14) -- (73,14);
			\node[dot, height3] at (70,14) {};
			\draw[, height1] (71,1) -- (71,2);
			\draw[, height1] (71,1) -- (72,2);
			\node[dot, height1] at (71,1) {};
			\draw[, height3] (71,8) -- (72,9);
			\node[dot, height3] at (71,8) {};
			\node[dot, height3] at (71,13) {};
			\draw[, height3] (71,16) -- (72,17);
			\node[dot, height3] at (71,16) {};
			\draw[, height3] (72,10) -- (73,11);
			\node[dot, height3] at (72,10) {};
			\draw[, height1] (73,1) -- (74,2);
			\node[dot, height1] at (73,1) {};
			\node[dot, height3] at (73,8) {};
			\draw[, height3] (74,8) -- (74,9);
			\node[dot, height3] at (74,8) {};
			\draw[, height3] (75,8) -- (76,9);
			\node[dot, height3] at (75,8) {};
			\draw[, height3] (76,7) -- (76,8);
			\draw[, height3] (76,7) -- (79,7);
			\node[dot, height3] at (76,7) {};
			\draw[, height3] (76,16) -- (77,17);
			\draw[, height3] (76,16) -- (79,16);
			\node[dot, height3] at (76,16) {};
			\draw[, height3] (77,8) -- (80,8);
			\node[dot, height3] at (77,8) {};
			\node[dot, height3] at (77,9) {};
			\draw[, height3] (77,13) -- (78,14);
			\draw[, height3] (77,13) -- (80,13);
			\node[dot, height3] at (77,13) {};
			\draw[, height3] (77,19) -- (78,20);
			\draw[, height3] (77,19) -- (80,19);
			\node[dot, height3] at (77,19) {};
			\draw[, height3] (78,10) -- (78,11);
			\draw[, height3] (78,10) -- (79,12);
			\draw[, height3] (78,10) -- (81,10);
			\node[dot, height3] at (78,10) {};
			\draw[, height1] (79,1) -- (79,2);
			\draw[, height1] (79,1) -- (80,2);
			\node[dot, height1] at (79,1) {};
			\draw[, height3] (79,17) -- (82,17);
			\node[dot, height3] at (79,17) {};
			\draw[, height3] (79,15) -- (80,16);
			\node[dot, height3] at (79,15) {};
			\draw[, height3] (79,14) -- (80,15);
			\node[dot, height3] at (79,14) {};
			\node[dot, height2] at (80,5) {};
			\draw[, height1] (81,1) -- (82,2);
			\node[dot, height1] at (81,1) {};
			\draw[, height3] (81,20) -- (84,20);
			\node[dot, height3] at (81,20) {};
			\node[dot, height3] at (81,12) {};
			\draw[, height3] (82,15) -- (85,15);
			\node[dot, height3] at (82,15) {};
			\node[dot, height3] at (82,11) {};
			\node[dot, height3] at (82,12) {};
			\draw[, height0] (83,16) -- (84,17);
			\node[dot, height3] at (83,16) {};
			\draw[, height3] (84,18) -- (87,18);
			\node[dot, height3] at (84,18) {};
			\node[dot, height2] at (85,5) {};
			\draw[, height3] (85,19) -- (85,20);
			\draw[, height3] (85,19) -- (86,21);
			\draw[, height3] (85,19) -- (88,19);
			\node[dot, height0] at (85,19) {};
			\node[dot, height3] at (85,11) {};
			\draw[, height3] (85,7) -- (86,8);
			\node[dot, height0] at (85,7) {};
			\draw[, height3] (86,5) -- (87,6);
			\node[dot, height3] at (86,5) {};
			\draw[, height1] (87,1) -- (87,2);
			\draw[, height1] (87,1) -- (88,2);
			\node[dot, height1] at (87,1) {};
			\node[dot, height3] at (87,7) {};
			\node[dot, height3] at (87,8) {};
			\draw[, height3] (87,9) -- (88,10);
			\node[dot, height3] at (87,9) {};
			\draw[, height3] (87,10) -- (87,11);
			\draw[, height3] (87,10) -- (88,11);
			\node[dot, height3] at (87,10) {};
			\draw[, height3] (88,8) -- (89,9);
			\node[dot, height3] at (88,8) {};
			\draw[, height3] (88,6) -- (89,7);
			\node[dot, height3] at (88,6) {};
			\draw[, height1] (89,1) -- (90,2);
			\node[dot, height1] at (89,1) {};
			\node[dot, height2] at (90,5) {};
			\node[dot, height0] at (90,9) {};
			\node[dot, height3] at (90,10) {};
			\draw[, height1] (63,2) -- (63,3);
			\node[dot, height1] at (63,2) {};
			\draw[, height1] (63,3) -- (63,4);
			\node[dot, height1] at (63,3) {};
			\draw[, height1] (63,4) -- (63,5);
			\node[dot, height1] at (63,4) {};
			\draw[, height1] (63,5) -- (63,6);
			\node[dot, height1] at (63,5) {};
			\draw[, height1] (63,6) -- (63,7);
			\node[dot, height1] at (63,6) {};
			\node[dot, height1] at (63,7) {};
			\draw[, height3] (63,17) -- (64,19);
			\draw[, height3] (63,17) -- (66,17);
			\node[dot, height3] at (63,17) {};
			\draw[, height3] (63,19) -- (63,20);
			\draw[, height3] (63,19) -- (64,20);
			\draw[, height3] (63,19) -- (66,20);
			\node[dot, height3] at (63,19) {};
			\node[dot, height3] at (63,20) {};
			\draw[, height3] (63,14) -- (64,15);
			\node[dot, height3] at (63,14) {};
			\draw[, height1] (64,2) -- (65,3);
			\node[dot, height1] at (64,2) {};
			\draw[, height3] (64,19) -- (65,20);
			\node[dot, height3] at (64,19) {};
			\draw[, height3] (64,20) -- (65,21);
			\node[dot, height3] at (64,20) {};
			\draw[, height3] (64,18) -- (67,18);
			\node[dot, height3] at (64,18) {};
			\draw[, height3] (64,14) -- (64,15);
			\draw[, height3] (64,14) -- (65,15);
			\node[dot, height3] at (64,14) {};
			\node[dot, height3] at (64,15) {};
			\node[dot, height1] at (65,3) {};
			\draw[, height3] (65,19) -- (65,20);
			\draw[, height3] (65,19) -- (66,20);
			\node[dot, height3] at (65,19) {};
			\node[dot, height3] at (65,20) {};
			\node[dot, height3] at (65,21) {};
			\draw[, height3] (65,10) -- (66,11);
			\node[dot, height3] at (65,10) {};
			\node[dot, height3] at (65,12) {};
			\draw[, height3] (65,13) -- (68,14);
			\node[dot, height3] at (65,13) {};
			\draw[, height3] (65,15) -- (66,16);
			\node[dot, height3] at (65,15) {};
			\draw[, height1] (66,2) -- (67,3);
			\node[dot, height1] at (66,2) {};
			\node[dot, height2] at (66,6) {};
			\node[dot, height3] at (66,17) {};
			\node[dot, height3] at (66,20) {};
			\draw[, height3] (66,10) -- (66,11);
			\node[dot, height3] at (66,10) {};
			\node[dot, height3] at (66,11) {};
			\draw[, height3] (66,9) -- (66,10);
			\draw[, height3] (66,9) -- (69,9);
			\node[dot, height3] at (66,9) {};
			\node[dot, height3] at (66,16) {};
			\draw[, height1] (67,1) -- (67,2);
			\node[dot, height1] at (67,1) {};
			\draw[, height1] (67,2) -- (67,3);
			\node[dot, height1] at (67,2) {};
			\node[dot, height1] at (67,3) {};
			\draw[, height3] (67,18) -- (70,18);
			\node[dot, height3] at (67,18) {};
			\node[dot, height3] at (67,8) {};
			\draw[, height3] (67,12) -- (67,13);
			\draw[, height3] (67,12) -- (68,14);
			\draw[, height3] (67,12) -- (70,12);
			\node[dot, height3] at (67,12) {};
			\draw[, height3] (67,13) -- (70,13);
			\node[dot, height3] at (67,13) {};
			\draw[, height2] (68,7) -- (71,7);
			\node[dot, height2] at (68,7) {};
			\draw[, height3] (68,16) -- (69,17);
			\node[dot, height3] at (68,16) {};
			\node[dot, height3] at (68,14) {};
			\draw[, height3] (69,17) -- (70,18);
			\node[dot, height3] at (69,17) {};
			\node[dot, height3] at (69,9) {};
			\draw[, height3] (69,8) -- (70,9);
			\node[dot, height3] at (69,8) {};
			\draw[, height3] (69,11) -- (70,13);
			\draw[, height3] (69,11) -- (72,12);
			\node[dot, height3] at (69,11) {};
			\node[dot, height3] at (70,18) {};
			\draw[, height2] (70,6) -- (71,7);
			\draw[, height2] (70,6) -- (73,7);
			\node[dot, height2] at (70,6) {};
			\draw[, height3] (70,17) -- (70,18);
			\draw[, height3] (70,17) -- (71,18);
			\node[dot, height3] at (70,17) {};
			\node[dot, height3] at (70,15) {};
			\draw[, height3] (70,12) -- (70,13);
			\draw[, height3] (70,12) -- (73,13);
			\node[dot, height3] at (70,12) {};
			\node[dot, height3] at (70,13) {};
			\draw[, height3] (70,9) -- (71,10);
			\node[dot, height3] at (70,9) {};
			\draw[, height1] (71,2) -- (71,3);
			\node[dot, height1] at (71,2) {};
			\draw[, height1] (71,3) -- (71,4);
			\node[dot, height1] at (71,3) {};
			\node[dot, height1] at (71,4) {};
			\draw[, height2] (71,5) -- (71,6);
			\draw[, height2] (71,5) -- (72,6);
			\draw[, height2] (71,5) -- (74,6);
			\node[dot, height2] at (71,5) {};
			\draw[, height2] (71,6) -- (71,7);
			\draw[, height2] (71,6) -- (74,7);
			\node[dot, height2] at (71,6) {};
			\node[dot, height2] at (71,7) {};
			\node[dot, height3] at (71,18) {};
			\node[dot, height3] at (71,15) {};
			\draw[, height3] (71,11) -- (72,12);
			\node[dot, height3] at (71,11) {};
			\draw[, height3] (71,9) -- (71,10);
			\draw[, height3] (71,9) -- (72,11);
			\node[dot, height3] at (71,9) {};
			\node[dot, height3] at (71,10) {};
			\draw[, height1] (72,2) -- (73,3);
			\node[dot, height1] at (72,2) {};
			\draw[, height2] (72,6) -- (73,7);
			\node[dot, height2] at (72,6) {};
			\node[dot, height3] at (72,9) {};
			\draw[, height3] (72,13) -- (73,14);
			\node[dot, height3] at (72,13) {};
			\node[dot, height3] at (72,17) {};
			\node[dot, height3] at (72,12) {};
			\draw[, height3] (72,11) -- (73,13);
			\node[dot, height3] at (72,11) {};
			\node[dot, height1] at (73,3) {};
			\node[dot, height2] at (73,7) {};
			\node[dot, height3] at (73,14) {};
			\node[dot, height3] at (73,13) {};
			\draw[, height3] (73,11) -- (74,12);
			\node[dot, height3] at (73,11) {};
			\draw[, height1] (74,2) -- (75,3);
			\node[dot, height1] at (74,2) {};
			\draw[, height2] (74,6) -- (74,7);
			\node[dot, height2] at (74,6) {};
			\node[dot, height2] at (74,7) {};
			\node[dot, height3] at (74,9) {};
			\draw[, height3] (74,11) -- (74,12);
			\node[dot, height3] at (74,11) {};
			\node[dot, height3] at (74,12) {};
			\draw[, height1] (75,1) -- (75,2);
			\node[dot, height1] at (75,1) {};
			\draw[, height1] (75,2) -- (75,3);
			\node[dot, height1] at (75,2) {};
			\node[dot, height1] at (75,3) {};
			\node[dot, height3] at (76,9) {};
			\node[dot, height3] at (76,8) {};
			\draw[, height3] (77,14) -- (77,15);
			\draw[, height3] (77,14) -- (78,15);
			\node[dot, height3] at (77,14) {};
			\node[dot, height3] at (77,15) {};
			\draw[, height3] (77,17) -- (78,18);
			\node[dot, height3] at (77,17) {};
			\draw[, height3] (78,11) -- (81,11);
			\node[dot, height3] at (78,11) {};
			\node[dot, height3] at (78,20) {};
			\node[dot, height3] at (78,14) {};
			\draw[, height3] (78,15) -- (79,16);
			\node[dot, height3] at (78,15) {};
			\draw[, height3] (78,17) -- (78,18);
			\draw[, height3] (78,17) -- (79,18);
			\node[dot, height3] at (78,17) {};
			\node[dot, height3] at (78,18) {};
			\node[dot, height3] at (79,12) {};
			\node[dot, height3] at (79,18) {};
			\draw[, height1] (79,2) -- (79,3);
			\node[dot, height1] at (79,2) {};
			\draw[, height1] (79,3) -- (79,4);
			\node[dot, height1] at (79,3) {};
			\draw[, height1] (79,4) -- (79,5);
			\node[dot, height1] at (79,4) {};
			\node[dot, height1] at (79,5) {};
			\draw[, height3] (79,6) -- (79,7);
			\draw[, height3] (79,6) -- (82,6);
			\node[dot, height3] at (79,6) {};
			\node[dot, height3] at (79,7) {};
			\node[dot, height3] at (79,16) {};
			\draw[, height1] (80,2) -- (81,3);
			\node[dot, height1] at (80,2) {};
			\draw[, height3] (80,7) -- (81,8);
			\node[dot, height3] at (80,7) {};
			\draw[, height3] (80,9) -- (81,11);
			\draw[, height3] (80,9) -- (83,10);
			\node[dot, height3] at (80,9) {};
			\draw[, height3] (80,15) -- (81,16);
			\node[dot, height3] at (80,15) {};
			\draw[, height3] (80,19) -- (83,19);
			\node[dot, height3] at (80,19) {};
			\node[dot, height3] at (80,8) {};
			\draw[, height3] (80,13) -- (83,13);
			\node[dot, height3] at (80,13) {};
			\node[dot, height3] at (80,16) {};
			\node[dot, height1] at (81,3) {};
			\draw[, height3] (81,8) -- (82,9);
			\node[dot, height3] at (81,8) {};
			\draw[, height3] (81,10) -- (81,11);
			\node[dot, height3] at (81,10) {};
			\node[dot, height3] at (81,11) {};
			\draw[, height3] (81,14) -- (81,15);
			\draw[, height3] (81,14) -- (84,14);
			\node[dot, height3] at (81,14) {};
			\draw[, height3] (81,15) -- (81,16);
			\node[dot, height3] at (81,15) {};
			\node[dot, height3] at (81,16) {};
			\node[dot, height3] at (82,17) {};
			\draw[, height1] (82,2) -- (83,3);
			\node[dot, height1] at (82,2) {};
			\draw[, height3] (82,7) -- (82,8);
			\draw[, height3] (82,7) -- (85,9);
			\node[dot, height3] at (82,7) {};
			\draw[, height3] (82,8) -- (82,9);
			\node[dot, height3] at (82,8) {};
			\node[dot, height3] at (82,9) {};
			\draw[, height3] (82,6) -- (85,6);
			\node[dot, height3] at (82,6) {};
			\draw[, height3] (83,11) -- (83,12);
			\draw[, height3] (83,11) -- (84,12);
			\draw[, height3] (83,11) -- (86,12);
			\node[dot, height3] at (83,11) {};
			\draw[, height1] (83,1) -- (83,2);
			\node[dot, height1] at (83,1) {};
			\draw[, height1] (83,2) -- (83,3);
			\node[dot, height1] at (83,2) {};
			\node[dot, height1] at (83,3) {};
			\draw[, height3] (83,10) -- (86,10);
			\node[dot, height3] at (83,10) {};
			\node[dot, height3] at (83,19) {};
			\draw[, height3] (83,12) -- (83,13);
			\draw[, height3] (83,12) -- (86,13);
			\node[dot, height3] at (83,12) {};
			\draw[, height3] (83,13) -- (86,14);
			\node[dot, height3] at (83,13) {};
			\node[dot, height0] at (84,17) {};
			\node[dot, height3] at (84,20) {};
			\node[dot, height3] at (84,14) {};
			\draw[, height3] (84,5) -- (85,6);
			\node[dot, height3] at (84,5) {};
			\draw[, height3] (84,12) -- (85,13);
			\node[dot, height3] at (84,12) {};
			\draw[, height3] (85,8) -- (85,9);
			\draw[, height3] (85,8) -- (86,10);
			\node[dot, height3] at (85,8) {};
			\node[dot, height3] at (85,9) {};
			\draw[, height3] (85,20) -- (88,20);
			\node[dot, height3] at (85,20) {};
			\node[dot, height3] at (85,6) {};
			\node[dot, height3] at (85,15) {};
			\draw[, height3] (85,13) -- (86,14);
			\node[dot, height3] at (85,13) {};
			\draw[, height3] (86,15) -- (86,16);
			\draw[, height3] (86,15) -- (87,18);
			\node[dot, height3] at (86,15) {};
			\draw[, height3] (86,16) -- (86,17);
			\node[dot, height3] at (86,16) {};
			\node[dot, height3] at (86,17) {};
			\node[dot, height3] at (86,10) {};
			\node[dot, height3] at (86,21) {};
			\draw[, height3] (86,12) -- (86,13);
			\draw[, height3] (86,12) -- (89,12);
			\node[dot, height3] at (86,12) {};
			\draw[, height3] (86,13) -- (86,14);
			\node[dot, height3] at (86,13) {};
			\node[dot, height3] at (86,14) {};
			\node[dot, height3] at (86,8) {};
			\draw[, height1] (87,2) -- (87,3);
			\node[dot, height1] at (87,2) {};
			\draw[, height1] (87,3) -- (87,4);
			\node[dot, height1] at (87,3) {};
			\node[dot, height1] at (87,4) {};
			\node[dot, height3] at (87,18) {};
			\node[dot, height3] at (87,6) {};
			\node[dot, height3] at (87,11) {};
			\draw[, height1] (88,2) -- (89,3);
			\node[dot, height1] at (88,2) {};
			\draw[, height3] (88,19) -- (88,20);
			\node[dot, height3] at (88,19) {};
			\node[dot, height3] at (88,20) {};
			\node[dot, height3] at (88,10) {};
			\node[dot, height3] at (88,11) {};
			\node[dot, height1] at (89,3) {};
			\node[dot, height3] at (89,9) {};
			\node[dot, height3] at (89,12) {};
			\draw[, height3] (89,7) -- (90,8);
			\node[dot, height3] at (89,7) {};
			\node[dot, height1] at (90,2) {};
			\draw[, height3] (90,6) -- (90,7);
			\node[dot, height3] at (90,6) {};
			\draw[, height3] (90,7) -- (90,8);
			\node[dot, height3] at (90,7) {};
			\node[dot, height3] at (90,8) {};
		\end{tikzpicture}

%% file: more-stable-stems-background.tex
\chapter{Background}
\label{ch:background}

\section{Associated graded objects}
\label{sctn:assoc-graded}

\begin{defn}
A filtered object $A$ consists of a finite chain
\[
A = F_0 A \supseteq F_1 A \supseteq F_2 A \supseteq \cdots
\supseteq F_{p-1} A \supseteq F_p A = 0
\]
of inclusions descending from $A$ to $0$.
\end{defn}

We will only consider finite chains because these are the examples
that arise in our Adams spectral sequences.  Thus we do not need
to refer to ``exhaustive" and ``Hausdorff" conditions on filtrations,
and we avoid subtle convergence issues associated with infinite filtrations.

\begin{ex}
\label{ex:14,8-filtered}
The $\C$-motivic stable homotopy group $\pi_{14,8} = \Z/2 \oplus \Z/2$
is a filtered object under the Adams filtration.  
The generators of this group are $\sigma^2$ and $\kappa$.
The subgroup $F_5$ is zero, the subgroup $F_3 = F_4$ is generated
by $\kappa$, and the subgroup $F_0 = F_1 = F_2$ is generated by
$\sigma^2$ and $\kappa$.
\end{ex}

\begin{defn}
Let $A$ be a filtered object.  The associated graded object
$\Gr A$ is
\[
\bigoplus_0^p F_i A / F_{i+1} A.
\]
\end{defn}

If $a$ is an element of $\Gr A$, then
we write $\{a\}$ for the set of elements of $A$ that are detected
by $a$.  In general, $\{a\}$ consists of more than one element of $A$,
unless $a$ happens to have maximal filtration.  
\revv{In other words}, the element $a$ is a coset $\alpha + F_{i+1} A$
for some $\alpha$ in $A$,
and $\{a\}$ is another name for this coset.
In this situation, we say that $a$ detects $\alpha$.

In this manuscript, the main example of a filtered object is
a $\C$-motivic homotopy group $\pi_{p,q}$, equipped with its Adams
filtration.  

\begin{ex}
\label{ex:14,8-assoc-graded}
Consider the $\C$-motivic stable homotopy group $\pi_{14,8}$
with its Adams filtration, as described in Example \ref{ex:14,8-filtered}.
The associated graded object is non-trivial only in degrees
$2$ and $4$, and it is generated by $h_3^2$ and $d_0$ respectively.
\end{ex}

\begin{defn}
Let $A$ and $B$ be filtered objects, perhaps with filtrations
of different lengths.
A map $f:A \map B$ is filtration preserving if
$f(F_i A)$ is contained in $F_i B$ for all $i$.
\end{defn}

Let $f: A \map B$ be a filtration preserving map of filtered objects.
We write $\Gr f: \Gr A \map \Gr B$ for the induced
map on associated graded objects.

\begin{defn}
\label{defn:hidden-value}
Let $a$ and $b$ be elements of $\Gr A$ and $\Gr B$ respectively.
We say that $b$ is the (not hidden) value of $a$ under $f$ if
$\Gr f(a) = b$.

We say that $b$ is the hidden value of $a$ under $f$ if:
\begin{enumerate}
\item
$\Gr f (a) = 0$.
\item
\revv{
there exists an element $\alpha$ of $\{a\}$ in $A$ such that:
\begin{enumerate}
\item
$f(\alpha)$ is contained in $\{b\}$ in $B$, and 
\item
there is no element $\gamma$ in filtration strictly higher than
$\alpha$ such that $f(\gamma)$ is contained in $\{b\}$.
\end{enumerate}
}
\end{enumerate}
\end{defn}

The motivation for condition (2b) may not be obvious.
The point is to avoid situations in which condition (2a) is satisfied
trivially.
Suppose that 
there is an element $\gamma$ 
such that $f(\gamma)$ is contained in $\{b\}$.
Let $a$ be any element of $\Gr A$ 
whose filtration is strictly less than the filtration of $\gamma$.
Now let $\alpha$ be any element of $\{ a\}$ such that $f(\alpha) = 0$.
(It may not be possible to choose such an $\alpha$ in general,
but sometimes it is possible.)
Then $\alpha + \gamma$ is another element
of $\{a\}$ such that $f(\alpha + \gamma)$ is contained in $\{b\}$.
Thus $f$ takes some element of $\{a\}$ into $\{b\}$, but only because
of the presence of $\gamma$.  
Condition (2b) is designed to exclude
this situation.

\begin{ex}
\label{ex:14,8-hidden-defn}
We illustrate the role of condition (2b) 
in Definition \ref{defn:hidden-value} with a specific example.
Consider the map $\eta: \pi_{14,8} \map \pi_{15,9}$.
The associated graded map $\Gr(\eta)$ takes $h_3^2$ to $0$ and 
takes $d_0$ to $h_1 d_0$.

The coset $\{h_3^2\}$ in $\pi_{14,8}$ consists of two elements
$\sigma^2$ and $\sigma^2 + \kappa$.  
One of these elements is non-zero after multiplying by $\eta$.
(In fact, $\eta \sigma^2$ equals zero, and 
$\eta (\sigma^2 + \kappa) = \eta \kappa$ is non-zero, but
that is not relevant here.)
Conditions (1) and (2a) of Definition \ref{defn:hidden-value}
are satisfied, but condition (2b) fails because of the presence
of $\kappa$ in higher filtration.
\end{ex}

Suppose that $b$ is the hidden value of $a$ under $f$.  
It is typically the case that $f(\alpha)$ is contained in $\{b\}$ for every
$\alpha$ in $A$.  However, an even more complicated situation can occur
in which this is not true.

Suppose that $b_0$ is the hidden value of $a_0$ under $f$, and suppose
that $b_1$ is the (hidden or not hidden) value of $a_1$ under $f$.
Moreover, suppose that the filtration of $a_0$ is strictly lower than
the filtration of $a_1$, and the filtration of $b_0$
is strictly greater than the filtration of $b_1$.
In this situation, we say that the value of $a_0$ under $f$ crosses
the value of $a_1$ under $f$.

The terminology arises from the usual graphical calculus, in which 
elements of higher filtration are drawn above elements of lower filtration,
and values of maps are indicated by line segments, as in 
Figure \ref{fig:crossing-values}.

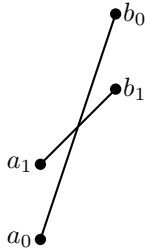
\begin{figure}[h!]
\caption{Crossing values}
\label{fig:crossing-values}

\begin{tikzpicture}[scale=0.5,
	every path/.style={line width=0.8pt},
	dot/.style={circle, inner sep=0, minimum size=0.13cm, draw=black,
				fill=black}]
	\node[dot, label=180:{$a_0$}] at (0,0) {};
	\draw (0,0) -- (2,6);
	\node[dot, label=180:{$a_1$}] at (0,2) {};
	\draw (0,2) -- (2,4);
	\node[dot, label=0:{$b_1$}] at (2,4) {};
	\node[dot, label=0:{$b_0$}] at (2,6) {};
	
\end{tikzpicture}

\end{figure}

\begin{ex}
For any map $X \map Y$ of $\C$-motivic spectra, naturality of the
Adams spectral sequence induces a filtration preserving map
$\pi_{p,q} X \map \pi_{p,q} Y$.
We are often interested in 
inclusion $S^{0,0} \map C\tau$ of the
bottom cell into $C\tau$, and in
projection $C\tau \map S^{1,-1}$ from $C\tau$ to the top cell.
We also consider the unit map
$S^{0,0} \map \mmf$.
\end{ex}

\subsection{Indeterminacy in hidden values}
\label{subsctn:hid-indet}

Definition \ref{defn:hidden-value} allows for the possibility of
some essentially redundant cases.  In order to avoid this redundancy,
we introduce indeterminacy into our definition.

Suppose, as in Definition \ref{defn:hidden-value},
that $b$ is the hidden value of $a$ under $f$, so
there exists some $\alpha$ in $\{a\}$ such that
$f(\alpha)$ is contained in $\{b\}$.
Suppose also that there is another element $a'$ in $\Gr A$ 
in degree strictly greater than the degree of $a$, such that
$f(\alpha')$ is contained in $\{b'\}$, where
$\alpha'$ is in $\{a'\}$ and $b'$ has the same degree as $b$.
Then $b+b'$ is also a hidden value of $a$ under $f$, since
$\alpha + \alpha'$ is contained in $\{a\}$ and
$f(\alpha + \alpha')$ is contained in $\{b+b'\}$.
In this case, we say that $b'$ belongs to the target
indeterminacy of the hidden value.

\begin{ex}
Consider the map $\eta: \pi_{63,33} \map \pi_{64,34}$.
The element $h_3 Q_2$ is a hidden value of $\tau h_1 H_1$
under this map.
This hidden value has target indeterminacy generated by 
$\tau h_1 X_2 = h_1 \cdot (\tau X_2 + \tau C')$.
\end{ex}

\subsection{Hidden extensions}

Let $\alpha$ be an element of $\pi_{a,b}$.  Then
multiplication by $\alpha$ induces a filtration preserving map
$\pi_{p,q} \map \pi_{p+a,q+b}$.
A hidden value of this map is precisely the same as a hidden
extension by $\alpha$ in the sense of \cite{Isaksen14c}*{Definition 4.2}.
For clarity, we repeat the definition here.

\begin{defn}
\label{defn:hidden}
Let $\alpha$ be an element of $\pi_{*,*}$ that is detected
by an element $a$ of the $E_\infty$-page of the $\C$-motivic Adams spectral sequence.
A hidden extension by $\alpha$ is a pair
of elements $b$ and $c$ of $E_\infty$ such that:
\begin{enumerate}
\item
$a b = 0$ in the $E_\infty$-page.
\item
There exists an element $\beta$ of $\{ b \}$ 
such that $\alpha \beta$ is contained in $\{ c \}$.
\item
If there exists an element $\beta'$ of $\{b'\}$ such
that $\alpha \beta'$ is contained in $\{ c\}$, then
the Adams filtration of $b'$ is less than or equal to the
Adams filtration of $b$.
\end{enumerate}
\end{defn}

A crossing value for the map 
$\alpha: \pi_{p,q} \map \pi_{p+a,q+b}$
is precisely the same
as a crossing extension in the sense of \cite{Isaksen14c}*{Examples
4.6 and 4.7}.

The discussion target indeterminacy
applies to the case of hidden extensions.
For example, the hidden $\eta$ extension from
$h_3 Q_2$ to $\tau h_1 H_1$ has target indeterminacy
generated by $\tau h_1 X_2$.

In later chapters, we will thoroughly explore hidden extensions
by $2$, $\eta$, and $\nu$.  We warn the reader that a complete
understanding of such hidden extensions does not necessarily
lead to a complete understanding of multiplication by
$2$, $\eta$, and $\nu$ in the $\C$-motivic stable homotopy groups.

For example, in the 45-stem, 
there exists an element $\theta_{4.5}$ that
is detected by $h_3^2 h_5$ such that $4 \theta_{4.5}$ is detected
by $h_0 h_5 d_0$.  This is an example of a hidden $4$ extension.
However, there is no hidden $2$ extension from
$h_0 h_3^2 h_5$ to $h_0 h_5 d_0$; condition (2b) of
Definition \ref{defn:hidden-value} is not satisfied.

In fact, a complete understanding of \emph{all} hidden
extensions leads to a complete understanding of the multiplicative
structure of the $\C$-motivic stable homotopy groups, but the process is
perhaps more complicated than expected.

For example, we mentioned in Example \ref{ex:14,8-hidden-defn}
that either $\eta (\sigma^2 + \kappa)$ or $\eta \sigma^2$ 
is non-zero, but these cases cannot be distinguished by a study
of hidden $\eta$ extensions.  However, we can express
that $\eta \sigma^2$ is zero by observing that there is
no hidden $\sigma$
extension from $h_1 h_3$ to $h_1 d_0$.

There are even further complications.  For example, 
the equation $h_2^3 + h_1^2 h_3 = 0$ does not prove that
$\nu^3 + \eta^2 \sigma$ equals zero because it could be detected
in higher filtration.  In fact, this does occur.  
Toda's relation \cite{Toda62} says that 
\[
\eta^2 \sigma + \nu^3 = \eta \epsilon,
\]
where $\eta \epsilon$ is detected by $h_1 c_0$.

We can express Toda's relation in terms of a ``matric hidden extension".
We have a map 
$[ \nu \quad \eta]: \pi_{6,4} \oplus \pi_{8,5} \map \pi_{9,6}$.
The associated graded map takes $(h_2^2, h_1 h_3)$ to zero,
but $h_1 c_0$ is the hidden value of $(h_2^2, h_1 h_3)$ under this 
map, in the sense of Definition \ref{defn:hidden-value}.

\section{Motivic modular forms}
\label{sctn:mmf}

Over $\C$, a ``motivic modular forms" spectrum $\mmf$
has recently been constructed \cite{GIKR18}.  
From our computational perspective, $\mmf$ is a ring spectrum
whose cohomology is $A//A(2)$, i.e., the quotient of the
$\C$-motivic Steenrod algebra by the subalgebra generated by
$\Sq^1$, $\Sq^2$, and $\Sq^4$.  By the usual change-of-rings
isomorphism, this implies that the homotopy groups of $\mmf$
are computed by an Adams spectral sequence whose $E_2$-page is
the cohomology of $\C$-motivic $A(2)$ \cite{Isaksen09}.  
The Adams spectral sequence for
$\mmf$ has been completely computed \cite{Isaksen18}.

By naturality, the unit map $S^{0,0} \map \mmf$ 
yields a map of Adams spectral sequences.  This map allows us
to transport information from the thoroughly understood
spectral sequence for $\mmf$ to the less well understood
spectral sequence for $S^{0,0}$.  This comparison technique
is essential at many points throughout our computations.

We rely on notation from \cite{Isaksen09} and \cite{Isaksen18}
for the Adams spectral sequence for $\mmf$, 
except that we use $a$ and $n$ instead of $\alpha$ and $\nu$
respectively.

For the most part, the map $\pi_{*,*} \map \pi_{*,*} \mmf$
is detected on Adams $E_\infty$-pages.  However, this map does
have some hidden values.

\begin{thm}
\label{thm:unit-mmf}
Through dimension 90, Table \ref{tab:unit-mmf} lists
all hidden values of the map $\pi_{*,*} \map \pi_{*,*} \mmf$.
\end{thm}

\begin{proof}
Most of these hidden values follow from hidden $\tau$ extensions
in the Adams spectral sequences for $S^{0,0}$ and for $\mmf$.
For example, for $S^{0,0}$, there is a hidden $\tau$ extension
from $h_1 h_3 g$ to $d_0^2$.  For $\mmf$, there is a hidden
$\tau$ extension from $c g$ to $d^2$.  This implies that
$c g$ is the hidden value of $h_1 h_3 g$.

A few cases are slightly more difficult.  
The hidden values of $\D h_1 h_3$ and $h_0 h_5 i$ follow from
the Adams-Novikov spectral sequences for $S^{0,0}$ and for $\mmf$.
These two values are detected on Adams-Novikov $E_\infty$-pages
in filtration $2$.

Next, the hidden value on $P h_2 h_5 j$ follows from 
multiplying the hidden value on $h_0 h_5 i$ by $d_0$.
Finally, the hidden values on $\D h_1^2 h_3$, $h_0 h_2 h_5 i$,
and $P h_5 j$ follow from already established hidden values,
relying on $h_1$ extensions and $h_2$ extensions.
\end{proof}

\begin{remark}
Through the 90-stem, there are no crossing values for the map
$\pi_{*,*} \map \pi_{*,*} \mmf$.  Moreover, in this range, there is only one hidden value that has target indeterminacy.
Namely, $\D^2 h_2 d$ is the hidden value of $P h_5 j$,
with target indeterminacy generated by $\tau^3 \D h_1 g^2$.
\end{remark}

\section{The cohomology of the $\C$-motivic Steenrod algebra}

We have implemented machine computations of
$\Ext$, i.e., 
the cohomology of the $\C$-motivic Steenrod algebra,
in detail through the 110-stem.  We take this computational data
for granted.  It is depicted graphically in the chart of the 
$E_2$-page shown in \cite{Isaksen14a}; the data is also available
there.
See \cite{Wang20} for a discussion of the implementation.

In addition to the additive structure of $\Ext$, we also have complete
information about multiplications by $h_0$, $h_1$, $h_2$, and $h_3$.
We do not have complete multiplicative information.  Occasionally we
must deduce some multiplicative information on an ad hoc basis.

Similarly, we do not have systematic machine-generated 
Massey product information about $\Ext$.  We deduce some of the 
necessary information about Massey products in Chapter \ref{ch:Massey}.

In the classical situation, Bruner has carried out extensive
machine computations of the cohomology of the classical Steenrod algebra
\cite{Bruner97}.  More recently, Bruner and Rognes have extended
these computations to total degree 184 \cite{BR20}.
This data includes complete primary multiplicative
information, but no higher Massey product structure.
We rely heavily on this information.
Our reliance on this data is so ubiquitous that we will not give
repeated citations.
\rev{Very recent work of Joey Beauvais-Feisthauer, Hood Chatham, Dexter Chua, and Weinan Lin
extends these machine computations of classical $\Ext$
to significantly higher stems.}

The May spectral sequence is the
key tool for a conceptual computation of $\Ext$.  
See \cite{Isaksen14c} for full details.  
In this manuscript, we use the May spectral sequence to
compute some Massey products that we need for various
specific purposes; see Remark \ref{rem:May-convergence}
for more details.

For convenience, we restate the following structural theorem
about a portion of $\Ext_\C$  \cite{Isaksen14c}*{Theorem 2.19}.

\begin{thm}
\label{thm:Chow-zero}
There is a highly structured isomorphism from 
$\Ext_\cl$ to the subalgebra of $\Ext$ consisting of elements in degrees 
$(s, f, w)$ with $s + f - 2w = 0$. 
This isomorphism takes classical elements of degree $(s,f)$ 
to motivic elements of degree $(2s+f,f,s+f)$.
\end{thm}

\section{Toda brackets}

Toda brackets are an essential computational tool for understanding
stable homotopy groups 
\revv{
\cite{Kochman90}*{Chapter 2}
\cite{Toda62} 
}.

Brackets appear throughout the various
stages of the computations, including in the analysis of Adams
differentials and in the resolution of hidden extensions.

It is well-known that the stable homotopy groups form a ring under
the composition product.  The higher Toda bracket structure is an
extension of this ring structure that is much deeper and more intricate.
Our philosophy is that the stable homotopy groups are not really
understood until the Toda bracket structure is revealed.

A complete analysis of all Toda brackets (even in a range)
is not a practical goal.  There are simply too many possibilities
to take into account methodically, especially when including
matric Toda brackets
(and possibly other more exotic non-linear types of brackets).
In practice, we compute only the Toda brackets that we need
for our specific computational purposes.

\subsection{The Moss Convergence Theorem}
\label{subsctn:Moss}

We next discuss the Moss Convergence Theorem \cite{Moss70}, which is
the essential tool for computing Toda brackets in stable homotopy
groups.  
\rev{See also \cite{BK21} for a modern proof of the Moss Convergence 
Theorem that applies not only in classical
stable homotopy theory but also to a wide variety of stable homotopy 
theories including $\C$-motivic stable homotopy theory.
}

In order to state the Moss Convergence Theorem precisely, 
we must clarify the
various types of bracket operations that arise.
First, the Adams $E_2$-page has Massey products 
arising from the fact that
it is the homology of the cobar complex, which is a 
differential graded algebra.  We typically refer to these
simply as ``Massey products", although we write
``Massey products in the $E_2$-page" for clarification when necessary.

Next, each higher $E_r$-page
also has Massey products, since it is the homology of the
$E_{r-1}$-page, which is a differential graded algebra.
We always refer to these as ``Massey products in the $E_r$-page"
to avoid confusion with the more familiar Massey products in
the $E_2$-page.  This type of bracket appears only occasionally 
throughout the manuscript.

Beware that the higher $E_r$-pages do not inherit Massey products
from the preceding pages.  For example, $\tau h_1^2$ equals the
Massey product $\langle h_0, h_1, h_0 \rangle$ in the $E_2$-page.
However, in the $E_3$-page, the bracket $\langle h_0, h_1, h_0 \rangle$
equals zero, since the product $h_0 h_1$ is already equal to zero
in the $E_2$-page before taking homology to obtain the $E_3$-page.

On the other hand, the Massey product $\langle h_1, h_0, h_3^2 \rangle$
is not a well-defined Massey product in the $E_2$-page since
$h_0 h_3^2$ is non-zero, while $\langle h_1, h_0, h_3^2 \rangle$
in the $E_3$-page equals $h_1 h_4$ because of the differential
$d_2(h_4) = h_0 h_3^2$.

Finally, we have Toda brackets in the stable homotopy groups
$\pi_{*,*}$.  The point of the Moss Convergence Theorem is to relate
these various kinds of brackets.

\revv{
\begin{defn}
\label{defn:crossing-diff}
Given $r$ and a degree $(s,f,w)$, 
a \textbf{crossing differential} is a nonzero differential
$d_{r+n}(x) = y$
in the $\C$-motivic Adams spectral sequence
 such that the stem of $y$ is $s$; the Adams filtration of $y$
is strictly greater than $f$ but strictly less than $f+n+1$;
and the weight of $y$ is $w$.
\end{defn}

\begin{remark}
In practice, we will never use Definition \ref{defn:crossing-diff}.
Rather, we will consider elements
$a$ and $b$ in the $E_{r}$-page of the $\C$-motivic Adams spectral sequence such that  $ab=0$.
We informally call a nonzero differential
$$d_{r+n}x=y$$
a crossing differential for the product 
$ab$ if it satisfies Definition \ref{defn:crossing-diff} for
the degree of $ab$.
In other words, 
the element $y$ has the same stem and motivic weight as the product $ab=0$; 
and the difference between the Adams filtration of $y$ and of
$ab$ is strictly greater than $0$ but strictly less than $n+1$.
\end{remark}
}

Figure \ref{fig:crossing-differentials} depicts the situation of 
a crossing differential in a chart for the $E_{r}$-page.
Typically, the product $ab$ is zero in the $E_{r}$-page because
it was hit by a $d_{r-1}$ differential, as shown by the dashed
arrow in the figure.  However, 
it may very well be the case that the product $ab$ is already zero in 
the $E_{r-1}$-page (or even in an earlier page),
in which case the dashed $d_{r-1}$ differential is actually 
$d_{r-1} (0) = 0$.

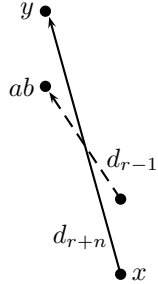
\begin{figure}[!ht]
\caption{Crossing differentials}
\label{fig:crossing-differentials}

\begin{tikzpicture}[scale=0.5,
	>=stealth,
	every path/.style={line width=0.8pt},
	dot/.style={circle, inner sep=0, minimum size=0.13cm, draw=black,
				fill=black}]
	\node[dot, label=0:{$x$}] at (2,0) {};
	\draw[->] (2,0) -- (0.1,6.85);
	\node at (1,1) {$d_{r+n}$};
	\node[dot] at (2,2) {};
	\draw[->, dashed] (2,2) -- (0.1,4.85);
	\node at (2.1,3.2) {$d_{r-1}$};
	\node[dot, label=180:{$ab$}] at (0,5) {};
	\node[dot, label=180:{$y$}] at (0,7) {};
	
\end{tikzpicture}

\end{figure}

\begin{thm}[Moss Convergence Theorem]
\label{thm:Moss}
Suppose that $a$, $b$, and $c$ are permanent cycles in the $E_r$-page of the $\C$-motivic Adams spectral sequence that detect
homotopy classes $\alpha$, $\beta$, and $\gamma$ in $\pi_{*,*}$ respectively. Suppose further that
\begin{enumerate}
\item the Massey product $\langle a, b, c \rangle$ is defined in the 
$E_r$-page, i.e., $ab =0$ and $bc= 0$ in the $E_r$-page.
\item 
the Toda bracket $\langle \alpha, \beta, \gamma \rangle$ is defined in 
$\pi_{*,*}$, i.e., $\alpha \beta = 0$ and $\beta \gamma = 0$.
\item there are no crossing differentials for the products $ab$ and $bc$ in the $E_r$-page.
\end{enumerate}
Then there exists an element $e$ contained in the Massey product $\langle a, b, c \rangle$ in the $E_r$-page, such that
\begin{enumerate}
\item the element $e$ is a permanent cycle.
\item the element $e$ detects a homotopy class  in the Toda bracket	$\langle \alpha, \beta, \gamma \rangle$.	
\end{enumerate}
\end{thm}

\begin{remark}
The homotopy classes $\alpha$, $\beta$, and $\gamma$ 
are usually not unique.  The presence of elements in higher 
Adams filtration implies that $a$, $b$, and $c$ detect more
than one homotopy class.  Moreover, it may be the case that 
$\langle \alpha, \beta, \gamma \rangle$ is defined for only
some choices of $\alpha$, $\beta$, and $\gamma$, 
while the Toda bracket is not defined for other choices.
\end{remark}

\begin{remark}
The Moss Convergence Theorem \ref{thm:Moss} says that
a certain Massey product $\langle a, b, c \rangle$ in the $E_r$-page
contains an element with certain properties.
The theorem does not claim that every element of
$\langle a, b, c \rangle$ has these properties.  In the presence
of indeterminacies, there can be elements in $\langle a, b, c\rangle$
that do not satisfy the given properties.
\end{remark}

\begin{remark}
Beware that the Toda bracket $\langle \alpha, \beta, \gamma \rangle$
may have non-zero indeterminacy.  In this case,
we only know that $e$ detects one element of the Toda bracket.
Other elements of the Toda bracket could possibly be detected by
other elements of the Adams $E_\infty$-page; these occurrences must
be determined by inspection.
\end{remark}

\begin{remark}
In practice, one computes a Toda bracket
$\langle \alpha, \beta, \gamma \rangle$ by first studying
its corresponding Massey product $\langle a, b, c \rangle$ 
in a certain page of the Adams spectral sequence. 
In the case that the Massey product
$\langle a, b, c \rangle$ equals zero in the $E_r$-page
in Adams filtration $f$, 
the Moss Convergence Theorem \ref{thm:Moss}
does not imply that the Toda bracket 
$\langle \alpha, \beta, \gamma \rangle$ contains zero.  Rather,
the Toda bracket contains an element (possibly zero) that is detected
in Adams filtration at least $f+1$.
\end{remark}

\begin{ex}
Consider the Toda bracket $\langle \nu, \eta, \nu \rangle$.
The elements $h_1$ and $h_2$ are permanent cycles that detect 
$\eta$ and $\nu$, and the product $\eta \nu$ is zero.
We have that $\langle h_2, h_1, h_2 \rangle$ equals $h_1 h_3$,
with no indeterminacy, in the $E_2$-page.
There are no crossing differentials for the product $h_1 h_2 = 0$ 
in the $E_2$-page, so the Moss Convergence Theorem \ref{thm:Moss}
implies that  $h_1 h_3$ detects a homotopy class
in $\langle \nu, \eta, \nu \rangle$.

Note that $h_1h_3$ detects the homotopy class $\eta\sigma$ because
$h_3$ is a permanent cycle that detects $\sigma$. 
However, we cannot conclude that
$\eta\sigma$ is contained in $\langle \nu, \eta, \nu \rangle$. 
The presence of the permanent cycle $c_0$
in higher filtration means that $h_1 h_3$ detects both
$\eta \sigma$ and $\eta \sigma +\epsilon$, where $\epsilon$
is the unique homotopy class that is detected by $c_0$.
The Moss Convergence Theorem \ref{thm:Moss} implies that
either $\eta \sigma$ or $\eta \sigma + \epsilon$ is contained in
the Toda bracket $\langle \nu, \eta, \nu \rangle$.
In fact, $\eta \sigma + \epsilon$ is contained in the Toda bracket,
but determining this requires further analysis.
\end{ex}

\begin{ex}
Consider the Toda bracket $\langle \sigma^2, 2, \eta \rangle$.
The elements $h_3^2$, $h_0$, and $h_1$ are permanent cycles
that detect $\sigma^2$, $2$, and $\eta$ respectively, and
the products $2 \sigma^2$ and $2 \eta$ are both zero.
Due to the Adams differential $d_2(h_4) = h_0h_3^2$, the Massey product 
$\langle h_3^2, h_0, h_1 \rangle$ equals $h_1 h_4$ in the
$E_3$-page, with zero indeterminacy. 
There are no crossing differentials for the products $h_0 h_3^2 = 0$ 
and $h_0 h_1 = 0$ in the $E_3$-page.
The Moss Convergence Theorem \ref{thm:Moss} implies that
$h_1 h_4$ detects a homotopy class in the Toda bracket 
$\langle \sigma^2, 2, \eta \rangle$.

The element $h_3^2$ also detects $\sigma^2 + \kappa$, 
where $\kappa$ is the unique homotopy class that is detected by $d_0$,
and the product $2 (\sigma^2 + \kappa)$ is zero.
The Moss Convergence Theorem \ref{thm:Moss} also implies that
$h_1 h_4$ detects a homotopy class in the Toda bracket 
$\langle \sigma^2 + \kappa, 2, \eta \rangle$. 
\end{ex}

\begin{ex}
Consider the Toda bracket $\langle \kappa, 2, \eta \rangle$.
The elements $d_0$, $h_0$, and $h_1$ are permanent cycles
that detect $\kappa$, $2$, and $\eta$ respectively, and
the products $2 \kappa$ and $2 \eta$ are both zero.
Due to the Adams differential $d_3(h_0h_4) = h_0d_0$, 
the Massey product $\langle d_0, h_0, h_1 \rangle$ equals
$h_0 h_4 \cdot h_1 = 0$ in Adams filtration 3 in the $E_4$-page,
with zero indeterminacy.
There are no crossing differentials for the products $h_0 d_0 = 0$ 
and $h_0 h_1 = 0$ in the $E_4$-page.
The Moss Convergence Theorem \ref{thm:Moss} implies that
the Toda bracket $\langle \kappa, 2, \eta \rangle$ either contains zero,
or it contains a non-zero element detected in Adams filtration
greater than $3$.

The only possible detecting element is $P c_0$. 
There is a hidden $\eta$ extension from $h_0^3h_4$ to $P c_0$, 
so $P c_0$ detects an element in the indeterminacy of 
$\langle \kappa, 2, \eta \rangle$.
Consequently, the Toda bracket is $\{0, \eta \rho_{15} \}$,
where $\rho_{15}$ is detected by $h_0^3h_4$.
\end{ex}

\begin{ex}
The Massey product $\langle h_2, h_3^2, h_0^2 \rangle$
equals $\{f_0, f_0 + h_0^2 h_2 h_4 \}$ in the $E_2$-page.
The elements $h_2$, $h_3^2$, and $h_0^2$ are permanent cycles
that detect $\nu$, $\sigma^2$, and $4$ respectively, and
the products $\nu \sigma^2$ and $4 \sigma^2$ are both zero.
However, the product $h_0^2 h_3^2$ has a crossing differential
$d_3(h_0 h_4) = h_0 d_0$.
The Moss Convergence Theorem \ref{thm:Moss} does not apply,
and we cannot conclude anything about the Toda bracket
$\langle \nu, \sigma^2, 4 \rangle$.  In particular,
we cannot conclude that $\{f_0, f_0 + h_0^2 h_2 h_4\}$ 
contains a permanent cycle.  In fact, both elements support
Adams $d_2$ differentials.
\end{ex}

\begin{remark}
\label{rem:Moss-4fold}
There is a version of the Moss Convergence Theorem \ref{thm:Moss}
for computing fourfold Toda brackets 
$\langle \alpha, \beta, \gamma, \delta\rangle$ in terms of fourfold
Massey products $\langle a, b, c, d \rangle$ in the $E_r$-page. 
In this case, the crossing differential condition applies not only to
the products $ab$, $bc$, and $cd$, but also to the subbrackets
$\langle a, b, c \rangle$ and $\langle b, c, d \rangle$.
\end{remark}

\begin{remark}
\label{rem:May-convergence}
Just as the Moss Convergence Theorem \ref{thm:Moss} is the key
tool for computing Toda brackets with the Adams spectral sequence,
the May Convergence Theorem is the key tool for computing
Massey products with the May spectral sequence.  The statement of
the May Convergence Theorem is entirely analogous to the
statement of the Moss Convergence Theorem, with 
Adams differentials replaced by  May differentials;
 Adams $E_r$-pages replaced by 
May $E_r$-pages;  $\pi_{*,*}$ replaced by
 $\Ext$; and  Toda brackets replaced
by Massey products.  An analogous crossing differential
condition applies.  See \cite{Isaksen14c}*{Section 2.2} \cite{May69}
for more details.  We will use the May Convergence Theorem to
compute various Massey products that we need for specific purposes.
\end{remark}

\subsection{Moss's higher Leibniz rule}
\label{subsctn:Leibniz}

Occasionally, we will use Moss's higher Leibniz rule \cite{Moss70}, 
which describes how Massey products in the $E_r$-page interact with
the Adams $d_r$ differential.  This theorem is a 
direct generalization of the usual Leibniz rule
$d_r(ab) = d_r(a) b + a d_r(b)$ for twofold products.

\begin{thm}
\label{thm:Leibniz}
\cite{Moss70}
Suppose that $a$, $b$, and $c$ are elements in the $E_r$-page 
of the $\C$-motivic Adams spectral sequence such that $ab=0$, $bc=0$,
$d_r(b) a = 0$, and $d_r(b) c = 0$.
Then
$$d_r \langle a, b, c \rangle \subseteq \langle d_r(a), b, c \rangle + \langle a, d_r(b), c \rangle + \langle a, b, d_r(c) \rangle,$$	
where all brackets are computed in the $E_r$-page.
\end{thm}

\begin{remark}
By the Leibniz rule, 
the conditions $d_r (b) a = 0$ and $d_r (b) c =0$ imply that 
$d_r (a) b = 0$ and $d_r (c) b =0$. Therefore, all of the Massey
products in Theorem \ref{thm:Leibniz} are well-defined.
\end{remark}

\begin{remark}
The Massey products in Moss's higher Leibniz rule \ref{thm:Leibniz} may have indeterminacies,
so the statement involves an inclusion of sets, rather than an equality.
\end{remark}

\begin{remark}
Beware that Moss's higher Leibniz rule \ref{thm:Leibniz}
cannot be applied
to Massey products in the $E_r$-page to study differentials in
higher pages.  For example, we cannot use it 
to compute the $d_3$ differential on a Massey product in the $E_2$-page.
In fact, there are versions of the higher Leibniz rule that
apply to higher differentials \cite{Kochman78}*{Theorem 8.2} 
\cite{May69}*{Theorem 4.3}, 
but these results have strong
vanishing conditions that often do not hold in practice.
\end{remark}

\begin{ex}
\label{ex:Leibniz-tD1h1^2}
Consider the element $\tau \Delta_1 h_1^2$, which 
was called $G$ in \cite{Tangora70a}.
Table \ref{tab:Adams-d2} shows that there is an Adams
differential $d_2(\tau \Delta_1 h_1^2) = M h_1 h_3$, which follows
by comparison to $C\tau$.  To illustrate Moss's higher Leibniz
rule \ref{thm:Leibniz}, we shall give an independent derivation
of this differential.

Table \ref{tab:Massey} shows that $\tau \Delta_1 h_1^2$ equals
the Massey product $\langle h_1, h_0, D_1 \rangle$,
with no indeterminacy.
By Moss's higher Leibiz rule \ref{thm:Leibniz}, the element
$d_2(\tau \Delta_1 h_1^2)$ is contained in
\[
\langle 0, h_0, D_1 \rangle + \langle h_1, 0, D_1 \rangle + \langle h_1, h_0, d_2(D_1) \rangle.
\]
By inspection, the first two terms vanish.  Also, Table \ref{tab:Adams-d2}
shows that $d_2(D_1)$ equals $h_0^2 h_3 g_2$.

Therefore,
$d_2(\tau \Delta_1 h_1^2)$ is contained in the bracket
$\langle h_1, h_0, h_0^2 h_3 g_2 \rangle$, which equals
$\langle h_1, h_0, h_0^2 g_2 \rangle h_3$.
Finally, Table \ref{tab:Massey} shows that
$\langle h_1, h_0, h_0^2 g_2 \rangle$ equals $M h_1$.
This shows that 
$d_2(\tau \D_1 h_1^2)$ equals $M h_1 h_3$.	
\end{ex}

\begin{ex}
Consider the element  $\tau e_0 g$ in the Adams $E_3$-page.
Because of the Adams differential $d_2(e_0) = h_1^2 d_0$, we have that
$\tau e_0 g$ equals $\langle d_0, h_1^2, \tau g \rangle$ 
in the Adams $E_3$-page.
The higher Leibniz rule \ref{thm:Leibniz} implies that 
$d_3(\tau e_0 g)$ is contained in
\[
\langle 0, h_1^2, \tau g \rangle + \langle d_0, 0, \tau g \rangle +
\langle d_0, h_2^2, 0 \rangle,
\]
which equals $\{0, c_0 d_0^2 \}$.
In this case, 
the higher Leibniz rule \ref{thm:Leibniz} does not help to determine
the value of $d_3(\tau e_0 g)$ because the indeterminacy is too large.
(In fact, $d_3(\tau e_0 g)$ does equal $c_0 d_0^2$, but we need a 
different argument.)
\end{ex}

\begin{ex}
Lemma \ref{lem:d3-Dh2^2h6} shows that 
$d_3(\D h_2^2 h_6)$ equals $h_1 h_6 d_0^2$. 
This argument uses that 
$\D h_2^2 h_6$ equals $\langle \D h_2^2, h_5^2, h_0 \rangle$
in the $E_3$-page, because of the Adams differential 
$d_2(h_6) = h_0 h_5^2$.
\end{ex}

\subsection{Shuffling formulas for Toda brackets}

Toda brackets satisfy various types of formal relations that we 
will use extensively.  The most important example of such a relation
is the shuffle formula
\[
\alpha \langle \beta, \gamma, \delta \rangle =
\langle \alpha, \beta, \gamma \rangle \delta,
\]
which holds whenever both Toda brackets are defined.
Note the equality of sets here; the indeterminacies of both
expressions are the same.

The following theorem states some formal properties of threefold
Toda brackets that we will use later.  
We apply these results so frequently that
we typically use them without further mention.

\begin{thm}
\label{thm:Toda-3fold}
\revv{\cite{Toda62}*{p.\ 33}}
Let $\alpha$, $\alpha'$, $\beta$, $\gamma$, and $\delta$ 
be homotopy classes in $\pi_{*,*}$. 
Each of the following relations involving threefold Toda brackets
holds up to \rev{signs}, whenever the Toda brackets are defined:
\begin{enumerate}
\item 
$ \langle \alpha + \alpha', \beta, \gamma\rangle \subseteq 
\langle \alpha, \beta, \gamma \rangle + 
\langle \alpha', \beta, \gamma\rangle.$
\item 
$ \langle \alpha, \beta, \gamma\rangle = 
\langle \gamma, \beta, \alpha \rangle.$
\item
 $\alpha \langle \beta, \gamma, \delta \rangle \subseteq 
 \langle \alpha\beta, \gamma, \delta \rangle.$
\item 
$\langle \alpha  \beta, \gamma, \delta \rangle \subseteq 
\langle \alpha, \beta\gamma, \delta \rangle.$
\item $\alpha \langle \beta, \gamma, \delta \rangle = \langle \alpha, \beta, \gamma \rangle \delta.$
\item $0 \in \langle \alpha, \beta, \gamma\rangle + 
\langle \beta, \gamma, \alpha\rangle + 
\langle \gamma, \alpha, \beta\rangle.$
\item 
\label{part:0}
\rev{$0 \in \langle \langle \alpha, \beta, \gamma \rangle, \delta, \epsilon \rangle + \langle \alpha, \langle \beta, \gamma, \delta \rangle, \epsilon \rangle + \langle  \alpha, \beta, \langle \gamma, \delta, \epsilon \rangle \rangle.$}
\end{enumerate}
\end{thm}

\rev{
Part (\ref{part:0}) of
Theorem \ref{thm:Toda-3fold}
requires some further explanation.
In the expression
$\langle \langle \alpha, \beta, \gamma \rangle, \delta, \epsilon \rangle$,
we have a set $\langle \alpha, \beta, \gamma \rangle$ as the 
first input to a threefold Toda bracket.
The expression $\langle \langle \alpha, \beta, \gamma \rangle, \delta, \epsilon \rangle$ is defined to be the union of all sets of the form
$\langle \zeta, \delta, \epsilon \rangle$, where
$\zeta$ ranges over all elements of 
$\langle \alpha, \beta, \gamma \rangle$.
Similar remarks apply to the other terms in Part (\ref{part:0}).
}

We next turn our attention to fourfold Toda brackets.  New complications
arise in this context.  If $\alpha \beta = 0$, $\beta \gamma = 0$,
$\gamma \delta = 0$, $\langle \alpha, \beta, \gamma \rangle$
contains zero, and $\langle \beta, \gamma, \delta \rangle$
contains zero, then the fourfold bracket
$\langle \alpha, \beta, \gamma, \delta \rangle$
is not necessarily defined.  Problems can arise when both
threefold subbrackets have indeterminacy.
See \cite{Isaksen14b} for a careful analysis of this problem in the
analogous context of Massey products.

However, when at least one of the threefold subbrackets is strictly
zero, then these difficulties vanish.  Every fourfold bracket that
we use has at least one threefold subbracket that is strictly zero.

Another complication with fourfold Toda brackets lies in the
description of the indeterminacy.  If at least one
threefold subbracket is strictly zero, then the indeterminacy of
$\langle \alpha, \beta, \gamma, \delta \rangle$
is the linear span of the sets
$\langle \alpha, \beta, \epsilon \rangle$,
$\langle \alpha, \epsilon, \delta \rangle$,
and $\langle \epsilon, \gamma, \delta \rangle$,
where $\epsilon$ ranges over all possible values in the appropriate
degree for which the Toda bracket is defined.

The following theorem states some formal properties of fourfold
Toda brackets that we will use later.  
We apply these results so frequently that
we typically use them without further mention.

\begin{thm}
\label{thm:Toda-4fold}
\revv{\cite{Kochman90}*{Chapter 2}}
Let $\alpha$, $\alpha'$, $\beta$, $\gamma$, $\delta$, and 
$\epsilon$ be homotopy classes in $\pi_{*,*}$. 
Each of the following relations involving fourfold Toda brackets
holds up to \rev{signs}, whenever the Toda brackets are defined:
\begin{enumerate}
\item $ \langle \alpha + \alpha', \beta, \gamma, \delta\rangle \subseteq  \langle \alpha, \beta, \gamma, \delta\rangle + \langle \alpha', \beta, \gamma, \delta\rangle.$
\item $ \langle \alpha, \beta, \gamma, \delta\rangle = \langle \delta, \gamma, \beta, \alpha \rangle.$
\item $\alpha \langle \beta, \gamma, \delta, \epsilon \rangle \subseteq \langle \alpha\beta, \gamma, \delta, \epsilon \rangle.$
\item $\langle \alpha  \beta, \gamma, \delta, \epsilon \rangle \subseteq \langle \alpha, \beta\gamma, \delta, \epsilon \rangle.$
\item $\alpha \langle \beta, \gamma, \delta, \epsilon \rangle = \langle \alpha, \beta, \gamma, \delta \rangle \epsilon.$
\item 
\label{part:a<b,c,d,e>}
$\alpha \langle \beta, \gamma, \delta, \epsilon \rangle \subseteq \langle \langle \alpha, \beta, \gamma \rangle, \delta, \epsilon \rangle.$
\end{enumerate}
\end{thm}

\rev{
As in Part (\ref{part:0}) of Theorem \ref{thm:Toda-3fold},
Part (\ref{part:a<b,c,d,e>}) of
Theorem \ref{thm:Toda-4fold}
requires some further explanation.
The expression $\langle \langle \alpha, \beta, \gamma \rangle, \delta, \epsilon \rangle$ is defined to be the union of all sets of the form
$\langle \zeta, \delta, \epsilon \rangle$, where
$\zeta$ ranges over all elements of 
$\langle \alpha, \beta, \gamma \rangle$.
}

We will make occasional use of matric Toda brackets. 
We will not describe their shuffling properties in detail, except to observe that
they obey analogous matric versions of the properties
in Theorems \ref{thm:Toda-3fold} and \ref{thm:Toda-4fold}. 
\revv{These properties can be proved with the same techniques that apply
to matric Massey product \cite{May69}; see 
\cite{Kochman78} \cite{Kochman96} for examples of this style of argument.}

%% file: more-stable-stems-cofiber-tau.tex
\chapter{The algebraic Novikov spectral sequence}
\label{ch:AANSS}

Consider the cofiber sequence
\[
\xymatrix@1{
S^{0,-1} \ar[r]^\tau & S^{0,0} \ar[r]^i & C\tau \ar[r]^-p & S^{1,-1},
}
\]
where $C\tau$ is the cofiber of $\tau$.
The inclusion $i$ of the bottom cell and projection $p$ to the top cell
are tools for comparing the $\C$-motivic Adams spectral sequence
for $S^{0,0}$ to the $\C$-motivic Adams spectral sequence for
$C\tau$.
In \cite{Isaksen14c}, the first author analyzed both spectral sequences simultaneously,
playing the structure of each against the other in order to obtain
more detailed information about both.
Then the structure of the homotopy of $C\tau$ was used
to reverse-engineer the structure of the classical Adams-Novikov spectral
sequence.

In this manuscript, we use $C\tau$ in a different, much more powerful
way, because we have a deeper understanding of the connection
between the homotopy of $C\tau$ and the structure of the
classical Adams-Novikov spectral sequence.
Namely, the $\C$-motivic spectrum $C\tau$ is an $E_\infty$-ring spectrum
\cite{Gheorghe18}.
\revv{
Here we are referring to the classical $E_\infty$-operad that 
parametrizes homotopy-coherent commutative multiplications.}

Moreover, 
the homotopy category of $C\tau$-modules is equivalent
to \revv{Hovey's stable derived category} of $BP_*BP$-comodules  \cite{GWX18} \cite{Krause18}.
By considering endomorphisms of unit objects, this comparison
of homotopy categories gives a structured explanation for the
identification of the homotopy of $C\tau$ and the
classical Adams-Novikov $E_2$-page.

From a computational perspective, there is an even better connection.
Namely, the algebraic Novikov spectral sequence for computing
the Adams-Novikov $E_2$-page 
\cite{Novikov67} \cite{Miller75} 
is identical to the $\C$-motivic Adams
spectral sequence for computing the homotopy of $C\tau$ \cite{GWX18}.
This rather shocking, and incredibly powerful, 
identification of spectral sequences allows us to transform
purely algebraic computations directly into information about
Adams differentials for $C\tau$.
Finally, naturality along the inclusion $i$ of the bottom cell
and along the projection $p$ to the top cell allows us to
deduce information about Adams differentials for $S^{0,0}$.

Due to the large quantity of data, we do not explicitly 
describe the structure of the Adams spectral sequence
for $C\tau$ in this manuscript.  We refer the interested
reader to the charts in \cite{IWX19}, which provide details
in a graphical form.

\section{Naming conventions}
\label{sctn:Ctau-naming}

Our naming convention for elements of the algebraic Novikov
spectral sequence (and for elements of the Adams-Novikov
spectral sequence) differs from previous approaches.
Our names are chosen to respect the inclusion $i$ of the
bottom cell and the projection $p$ to the top cell.
Specifically, if $x$ is an element of the
$\C$-motivic Adams $E_2$-page for $S^{0,0}$, then we use the same
letter $x$ to indicate its image 
$i_*(x)$ in the Adams $E_2$-page for $C\tau$.
It is certainly possible that $i_*(x)$ is zero, but
we will only use this convention in cases where 
$i_*(x)$ is non-zero, i.e., when $x$ is not a multiple of $\tau$.

On the other hand, if $x$ is an element of the $\C$-motivic
Adams $E_2$-page for $S^{0,0}$ such that $\tau x$ is zero,
then we use the symbol $\overline{x}$ to indicate an element
of $p_*^{-1} (x)$ in the Adams $E_2$-page for $C\tau$.
There is often more than one possible choice for $\overline{x}$,
and the indeterminacy in this choice equals the image of
$i_*$ in the appropriate degree.  We will not usually be explicit
about these choices.  However, potential confusion can arise
in this context.  For example, it may be the case that
one choice of $\overline{x}$ supports an $h_1$ extension,
while another choice of $\overline{x}$ supports an $h_2$ extension,
but there is no possible choice of $\overline{x}$ that simultaneously
supports both extensions.  (The authors dwell on this point
because this precise issue has generated confusion about specific
computations.)

\section{Machine computations}
\label{sctn:AANNSS-machine}

We have analyzed the algebraic Novikov spectral sequence
by computer in a large range.  Roughly speaking, our algorithm
computes a Curtis table for a minimal resolution. 
Significant effort went into optimizing the linear algebra
algorithms to complete the computation in a reasonable amount of time.
The data is available at \cite{IWX19}.
See \cite{Wang20} for a discussion of the implementation.

Our machine computations give us a full description of the
additive structure of the algebraic Novikov $E_2$-page, together
with all $d_r$ differentials for $r \geq 2$.  
It thus yields a full description of the additive structure
of the algebraic Novikov $E_\infty$-page.

Moreover, the data also gives full information about
multiplication by $2$, $h_1$, and $h_2$ in the 
Adams-Novikov $E_2$-page for the classical sphere spectrum,
which we denote by $H^*(S; BP)$.

We have also conducted machine computations of the
Adams-Novikov $E_2$-page for the classical cofiber of $2$, which
we denote by $H^*(S/2; BP)$.
Note that $H^*(S; BP)$ is the homology of a differential
graded algebra (i.e., the cobar complex) that is free as a $\Z_2$-module.
Therefore, $H^*(S/2; BP)$ 
is the homology of this differential graded algebra modulo $2$.
We have computed this homology by machine, including full information
about multiplication by $h_1$, $h_2$, and $h_3$.
These computations are related by a long exact sequence 
\[
\xymatrix@1{
\cdots \ar[r] & H^*(S; BP) \ar[r]^j & H^*(S/2; BP) \ar[r]^q &
H^*(S; BP) \ar[r] & \cdots 
}
\]

Because $h_2^2$, $h_3^2$, $h_4^2$, and $h_5^2$ are annihilated by
$2$ in $H^*(S; BP)$, there are classes
$\widetilde{h_2^2}$, $\widetilde{h_3^2}$, $\widetilde{h_4^2}$, and 
$\widetilde{h_5^2}$ in $H^*(S/2; BP)$ such that
$q(\widetilde{h_i^2})$ equals $h_i^2$ for $2 \leq i \leq 5$.
We also have full information about multiplication by
$\widetilde{h_2^2}$, $\widetilde{h_3^2}$, $\widetilde{h_4^2}$, and 
$\widetilde{h_5^2}$ in $H^*(S/2; BP)$

This multiplicative information allows use to determine some of the
Massey product structure in the Adams-Novikov $E_2$-page for the
sphere spectrum.  There are several cases to consider.

First, let $x$ and $y$ be elements of $H^*(S;BP)$.
If the product $j(x) j(y)$ is non-zero in $H^*(S/2;BP)$, then 
$x y$ must also be non-zero in $H^*(S; BP)$.

In the second case, let $x$ be an element of $H^*(S; BP)$, and let
$\widetilde{y}$ be an element of $H^*(S/2; BP)$ such that
$q(\widetilde{y}) = y$.
If the product $x \cdot \widetilde{y}$ is non-zero in
$H^*(S/2; BP)$ and equals $j(z)$ for some $z$ in
$H^*(S; BP)$, then $z$ belongs to the Massey product
$\langle 2, y, x \rangle$.  This follows immediately from the
relationship between Massey products and the multiplicative
structure of a cofiber, as discussed in \cite{Isaksen14c}*{Section 3.1.1}.

Third, let $\widetilde{x}$ and $\widetilde{y}$ be elements of
$H^*(S/2; BP)$ such that 
$q(\widetilde{x}) = x$, $q(\widetilde{y}) = y$, and
$q(\widetilde{x} \cdot \widetilde{y}) = z$.
Then $z$ belongs to the Massey product
$\langle x, 2, y \rangle$ in $H^*(S; BP)$.  This follows
immediately from the multiplicative snake lemma \ref{lem:mult-snake}.

\begin{ex}
Computer data shows that 
the product $\widetilde{h_4^2} \cdot \widetilde{h_5^2}$ does not equal
zero in $H^*(S/2;BP)$.  This implies that the
Massey product $\langle h_4^2, 2, h_5^2 \rangle$ does not contain zero in
$H^*(S; BP)$, which in turn implies that
the Toda bracket $\langle \theta_4, 2, \theta_5 \rangle$
does not contain zero in $\pi_{93,48}$.
\end{ex}

\begin{remark}
let $\widetilde{x}$ and $\widetilde{y}$ be elements of
$H^*(S/2; BP)$ such that 
$q(\widetilde{x})$ and $q(\widetilde{y})$ equal $x$ and $y$, and
such that
$\widetilde{x} \cdot \widetilde{y}$
equals $j(z)$ for some $z$ in $H^*(S; BP)$.
It appears that $z$ has some relationship to the fourfold Massey
product $\langle 2, x, 2, y \rangle$, but we have not made this precise.
\end{remark}

\begin{lemma}[Multiplicative Snake Lemma]
\label{lem:mult-snake}
Let $A$ be a differential graded algebra that has no $2$-torsion, 
and let $H(A)$ be its homology.  Also let $H(A/2)$ be the homology
of $A/2$, and let $\delta: H(A/2) \map H(A)$ be the boundary
map associated to the short exact sequence
\[
\xymatrix@1{
0 \ar[r] & A \ar[r]^2 & A \ar[r] & A/2 \ar[r] & 0.
}
\]
Suppose that $a$ and $b$ are elements of $H(A/2)$ such that 
$2 \delta(a) = 0$ and $2 \delta(b) = 0$ in $H(A)$.
Then the Massey product $\langle \delta(a), 2, \delta(b) \rangle$
in $H(A)$
contains $\delta(a b)$.
\end{lemma}

\begin{proof}
We carry out a diagram chase in the spirit of the snake lemma.
Write $\partial$ for the boundary operators in $A$ and $A/2$.

Let $x$ and $y$ be cycles in $A/2$ that represent $a$ and $b$
respectively.
Let $x'$ and $y'$ be elements in $A$ that reduce to $x$ and $y$.
Then $\partial x'$ and $\partial y'$ reduce to zero in $A/2$
because $x$ and $y$ are cycles.
Therefore, $\partial x' = 2 \widetilde{x}$ and
$\partial y' = 2 \widetilde{y}$ 
for some $\widetilde{x}$ and $\widetilde{y}$ in $A$.

By definition of the boundary map,
$\delta(a)$ and $\delta(b)$ are represented by $\widetilde{x}$
and $\widetilde{y}$.  By the definition of Massey products,
the cycle
$\widetilde{x} y' + x' \widetilde{y}$ is contained in 
$\langle \delta(a), 2, \delta(b) \rangle$.

Now we compute $\delta(a b)$.
Note that $x' y'$ is an element of $A$ that reduces to $a b$.
Then 
\[
\partial (x' y') = \partial(x') y' + x' \partial(y') =
2 (\widetilde{x} y' + x' \widetilde{y}).
\]
This shows that $\delta(a b)$ is represented by
$\widetilde{x} y' + x' \widetilde{y}$.
\end{proof}

\section{$h_1$-Bockstein spectral sequence}
\label{sctn:h1-Bockstein}

The charts in \cite{IWX19} show graphically 
the algebraic Novikov spectral sequence,
i.e., the Adams spectral sequence for $C\tau$.  Essentially all of
the information in the charts can be read off from machine-generated data.
This includes hidden extensions in the $E_\infty$-page.

One aspect of these charts requires further explanation.
The $\C$-motivic Adams $E_2$-page for $C\tau$ contains a large
number of $h_1$-periodic elements, i.e.,
elements that support infinitely many $h_1$ multiplications.
The behavior of these elements is entirely understood
\cite{GI15}, at least up to many multiplications by $h_1$,
i.e., in an $h_1$-periodic sense.

On the other hand, it takes some work to \revv{``delocalize"}
this information.  For example, we can immediately deduce from
\cite{GI15} that $d_2(h_1^k e_0) = h_1^{k+2} d_0$ for large
values of $k$, but that does not necessarily determine
the behavior of Adams differentials for small values of $k$.

The behavior of these elements is a bit subtle in another sense, 
as illustrated
by Example \ref{ex:_c0e0}.

\begin{ex}
\label{ex:_c0e0}
Consider the $h_1$-periodic element $\ol{c_0 e_0}$
\revv{in the algebraic Novikov spectral sequence}.
Machine computations tell us that this element supports a 
$d_2$ differential, but there is more than one possible value
for $d_2(\ol{c_0 e_0})$ because of the presence of
both $h_1^2 \ol{c_0 d_0}$ and $P e_0$.

In fact, $d_2(\ol{c_0 d_0})$ equals $P d_0$, and
$d_2(P e_0)$ equals $P h_1^2 d_0$.  Therefore,
$P e_0 + h_1^2 \ol{c_0 d_0}$ is the only non-zero $d_2$ cycle,
and it follows that 
$d_2(\ol{c_0 e_0})$ must equal $P e_0 + h_1^2 \ol{c_0 d_0}$.

\revv{
The careful reader will note that $d_2(\ol{c_0 e_0})$ is
not shown on the algebraic Novikov chart in \cite{IWX19}.
As discussed in \cite{IWX19}*{Section 4}, the $h_1$-periodic
differentials are not shown for legibility.
Instead, the differential is shown in the $h_1$-Bockstein
spectral sequence chart of \cite{IWX19}, up to higher
powers of $h_1$.
}
\end{ex}

In higher stems, it becomes more and more difficult to determine
the exact values of the Adams $d_2$ differentials
on $h_1$-periodic classes.  Eventually, these complications
become unmanageable because they involve sums of many monomials.

Fortunately, we only need concern ourselves with the Adams
$d_2$ differential in this context.
The $h_1$-periodic $E_3$-page equals the $h_1$-periodic
$E_\infty$-page, and the only non-zero classes are well-understood
$v_1$-periodic families running along the top of the Adams chart.

Our solution to this problem, as usual, is to introduce a filtration
that hides the \revv{higher order terms}.
In this case, we filter by powers of $h_1$.  The effect is that
terms involving higher powers of $h_1$ are ignored, and the
formulas become much more manageable.

This $h_1$-Bockstein spectral sequence starts with an $E_0$-page,
because there are some differentials that do not increase
$h_1$ divisibility.  For example, we have
Bockstein differentials $d_0(\ol{h_1^2 e_0}) = \ol{h_1^4 d_0}$
and $d_0(\ol{c_0 d_0}) = P d_0$, reflecting the Adams
differentials $d_2(\ol{h_1^2 e_0}) = \ol{h_1^4 d_0}$
and $d_2(\ol{c_0 d_0}) = P d_0$.

There are also plenty of higher $h_1$-Bockstein differentials,
such as $d_2(e_0) = h_1^2 d_0$, and
$d_7(e_0^2 g) = M h_1^8$.

\begin{remark}
Beware that filtering by powers of $h_1$ changes the multiplicative
structure in perhaps unexpected ways.
For example, $P h_1$ and $d_0$ are not $h_1$-multiples, so their
$h_1$-Bockstein filtration is zero.  One might expect their
product to be $P h_1 d_0$, but the 
$h_1$-Bockstein filtration of this element is $1$.
Therefore, $P h_1 \cdot d_0$ equals $0$ in the
$h_1$-Bockstein spectral sequence.

But not all $P h_1$ multiplications are trivial in the
$h_1$-Bockstein spectral sequence.  For example, we have
$P h_1 \cdot \ol{c_0 d_0} = \ol{P h_1 c_0 d_0}$ because
the $h_1$-Bockstein filtrations of all three elements are zero.
\end{remark}

In Example \ref{ex:_c0e0}, we explained that there is an
Adams differential $d_2(\ol{c_0 e_0}) = P e_0 + h_1^2 \ol{c_0 d_0}$.
When we throw out higher powers of $h_1$, we obtain
the $h_1$-Bockstein differential $d_0(\ol{c_0 e_0}) = P e_0$.
We also have an $h_1$-Bockstein differential
$d_0(\ol{c_0 d_0}) = P d_0$.

The first four charts in \cite{IWX19} show graphically how this 
$h_1$-Bockstein spectral sequence plays out in practice.
The main point is that the $h_1$-Bockstein
$E_\infty$-page reveals which (formerly) $h_1$-periodic classes
contribute to the Adams $E_3$-page for $C\tau$.

%% file: more-stable-stems-Massey.tex
\chapter{Massey products}
\label{ch:Massey}

The purpose of this chapter is to provide some general tools,
and to give some specific computations, of Massey products in
$\Ext$.  This material contributes to Table \ref{tab:Massey},
which lists a number of Massey products in $\Ext$ that we need
for various specific purposes.  Most commonly, these Massey
products yield information about Toda brackets via
the Moss Convergence Theorem \ref{thm:Moss}.

We begin with a $\C$-motivic version of a classical theorem of
Adams about symmetric Massey products.

\begin{thm}
\label{thm:Massey-symmetric}
\mbox{}
\begin{enumerate}
\item
If $h_0 x$ is zero, then $\langle h_0, x, h_0 \rangle$ contains 
$\tau h_1 x$.
\item
If $n \geq 1$ and $h_n x$ is zero, then 
$\langle h_n, x, h_n \rangle$ contains $h_{n+1} x$.
\end{enumerate}
\end{thm}

\revv{
\begin{proof}
The element $\Sq^0(h_n) x$ is contained in $\langle h_n, x, h_n \rangle$
\cite{Hirsch55}, where $\Sq^0$ is an algebraic Steenrod operation \cite{May70}.  We compute that $\Sq^0(h_0)$ equals $\tau h_1$
and $\Sq^0(h_n)$ equals $h_{n+1}$ for $n \geq 1$.  These motivic 
computations follow from the analogous classical computations
\cite{Adams60}*{Lemma 2.5.4}.
\end{proof}
}

\revv{
\begin{remark}
The two parts of Theorem \ref{thm:Massey-symmetric} are less different
than they appear.  Because of the specific values of the motivic weights,
we have a factor of $\tau$ in $\Sq^0(h_0)$, while no $\tau$ appears
in $\Sq^0(h_n)$ for $n \geq 1$.
\end{remark}
}

\section{The operator $g$}

The projection map $p:A_* \map A(2)_*$ induces a map
$p_*: \Ext_\C \map \Ext_{A(2)}$.
Because $\Ext_{A(2)}$ is completely known \cite{Isaksen09},
this map is useful for detecting structure in
$\Ext_\C$.  Proposition \ref{prop:g-Massey} provides a tool
for using $p_*$ to compute certain types of Massey products.

\begin{prop}
\label{prop:g-Massey}
Let $x$ be an element of $\Ext_\C$ such that $h_1^4 x = 0$.
Then $p_* \left( \langle h_4, h_1^4, x \rangle \right)$
equals the element $g p_*(x)$ in $\Ext_{A(2)}$.
\end{prop}

\begin{proof}
The idea of the proof is essentially the same as in
\cite{Isaksen19}*{Proposition 3.1}.
The $\Ext_\C$-module $\Ext_{A(2)}$ is a ``Toda module", 
in the sense that Massey products
$\langle x, a, b \rangle$ are defined for all
$x$ in $\Ext_{A(2)}$ and all $a$ and $b$ in $\Ext_\C$ such
that $x \cdot a = 0$ and $ab = 0$.
In particular, the bracket $\langle 1, h_4, h_1^4 \rangle$
is defined in $\Ext_{A(2)}$.
We wish to compute this bracket.

We use the May Convergence Theorem in order to compute the
bracket.  The crossing differentials condition on the theorem
is satisfied because there are no possible differentials that 
could interfere.

The key point is the May differential
$d_4(b_{21}^2) = h_1^4 h_4$.  
This shows that $g$ is contained in $\langle 1, h_4, h_1^4 \rangle$.
Also, the bracket has no indeterminacy by inspection.

Now suppose that $x$ 
is an element of $\Ext_\C$ such that $h_1^4 x = 0$.
Then
\[
p_* \left( \langle h_4, h_1^4, x \rangle \right) =
1 \cdot \langle h_4, h_1^4, x \rangle =
\langle 1, h_4, h_1^4 \rangle \cdot x =
g p_*(x).
\]
\end{proof}

\begin{ex}
\label{ex:g-Massey}
We illustrate the practical usefulness of
Proposition \ref{prop:g-Massey} with a specific
example.  Consider the Massey product
$\langle h_1^3 h_4, h_1, h_2 \rangle$.
The proposition says that
\[
p_* \left( \langle h_1^3 h_4, h_1, h_2 \rangle \right) = h_2 g
\]
in $\Ext_{A(2)}$.  This implies that
$\langle h_1^3 h_4, h_1, h_2 \rangle$ equals $h_2 g$ in
$\Ext_\C$.
\end{ex}

\begin{remark}
The Massey product computation in Example \ref{ex:g-Massey}
is in relatively low dimension, and it can be computed using
other more direct methods.
Table \ref{tab:Massey} lists additional examples, including
some that cannot be determined by more elementary methods.
\end{remark}

\section{The Mahowald operator}
\label{subsctn:Mahowald-operator}

We recall some results from \cite{Isaksen19} about the
Mahowald operator.  The Mahowald operator is defined to be
$M x = \langle x, h_0^3, g_2 \rangle$ for all $x$ such that $h_0^3 x$
equals zero.  As always, one must be cautious about indeterminacy
in $M x$.

There exists a subalgebra $B$ of the $\C$-motivic Steenrod algebra whose
cohomology
$\Ext_B(\M_2, \M_2)$ equals $\M_2[v_3] \otimes_{\M_2} \Ext_{A(2)}$.
The inclusion of $B$ into the $\C$-motivic Steenrod algebra induces a map
$p_*: \Ext_\C \map \Ext_B$.

\begin{prop} \cite{Isaksen19}*{Theorem 1.1}
\label{prop:M-Massey}
The map $p_*: \Ext_\C \map \Ext_B$ takes $M x$ to the product
$(e_0 v_3^2 + h_1^3 v_3^3) p_*(x)$, whenever $M x$ is defined.
\end{prop}

Proposition \ref{prop:M-Massey} is useful in practice for detecting
certain Massey products of the form $\langle x, h_0^3, g_2 \rangle$.
For example, if $x$ is an element of $\Ext_\C$ such that
$h_0^3 x$ equals zero and $e_0 p_*(x)$ is non-zero in $\Ext_{A(2)}$,
then $\langle x, h_0^3, g_2 \rangle$ is non-zero.

\begin{ex}
\label{ex:Mh1}
Proposition \ref{prop:M-Massey} shows that
$\langle h_1, h_0, h_0^2 g_2 \rangle$ is non-zero.  There is
only one non-zero element in the appropriate degree, so we have
identified the Massey product.  We give this element the name
$M h_1$.
\end{ex}

\begin{ex}
\label{ex:M^2h1}
Expanding on Example \ref{ex:Mh1}, 
Proposition \ref{prop:M-Massey} also shows that 
$\langle M h_1, h_0, h_0^2 g_2 \rangle$ is non-zero.
Again, there is only one non-zero element in the appropriate
degree, so we have identified the Massey product.  We give this
element the name $M^2 h_1$.
\end{ex}

\section{Additional computations}

\begin{lemma}
\label{lem:h1^2,h4^2,h1^2,h4^2}
\revdeg{66, 6, 36}
The Massey product $\langle h_1^2, h_4^2, h_1^2, h_4^2 \rangle$
equals $\D_1 h_3^2$.
\end{lemma}

\begin{proof}
Table \ref{tab:Massey} shows that $\D h_2^2$ equals the Massey product
$\langle h_0^2, h_3^2, h_0^2, h_3^2 \rangle$.
Recall the isomorphism between classical $\Ext$ groups
and $\C$-motivic $\Ext$ groups in degrees satisfying $s+f-2w=0$,
as described in Theorem \ref{thm:Chow-zero}.
This shows that 
$\D_1 h_3^2$ equals $\langle h_1^2, h_4^2, h_1^2, h_4^2 \rangle$.
\end{proof}

\rev{
\begin{lemma}
\label{lem:A',h1,h2}
\revdeg{66, 7, 35}
The Massey product $\langle A', h_1, h_2 \rangle$
equals $\tau G_0$.
\end{lemma}

\begin{proof}
Consider the shuffle
\[
A' \langle h_1, h_2, h_1 \rangle =
\langle A', h_1, h_2 \rangle h_1.
\]
Table \ref{tab:Massey} shows that the left side equals
$h_2^2 A'$, which equals $h_1 \cdot \tau G_0$.
This implies that $\langle A', h_1, h_2 \rangle$
contains $\tau G_0$.

The indeterminacy is zero by inspection.
\end{proof}
}

\begin{lemma}
\label{lem:h1^3h4,h1,tgn}
\revdeg{71, 13, 40}
The Massey product $\langle h_1^3 h_4, h_1, \tau g n \rangle$
equals $\tau g^2 n$, with indeterminacy generated by
$M h_0 h_2^2 g$.
\end{lemma}

\begin{proof}
We start by analyzing the indeterminacy.
The product $M c_0 \cdot h_1^3 h_4$ equals
\[
\langle g_2, h_0^3, c_0 \rangle h_1^3 h_4 = 
\langle g_2, h_0^3, h_1^3 h_4 c_0 \rangle =
\langle g_2, h_0^3, h_0 h_2 \cdot h_2 g \rangle =
\langle g_2, h_0^3, h_2 g \rangle h_0 h_2,
\]
which equals $M h_0 h_2^2 g$.
The equalities hold because the indeterminacies are zero, and
the first and last brackets in this computation are given by
Table \ref{tab:Massey}.
This shows that $M h_0 h_2^2 g$ belongs to the indeterminacy.

Table \ref{tab:Massey} shows that
\[
\langle h_2, h_1^3 h_4, h_1 \rangle = 
\langle h_2, h_1, h_1^3 h_4 \rangle
\]
equals $h_2 g$.
Then 
\[
h_2 \langle h_1^3 h_4, h_1, \tau g n \rangle =
\langle h_2, h_1^3 h_4, h_1 \rangle \tau g n = \tau h_2 g^2 n.
\]
This implies that
$\langle h_1^3 h_4, h_1, \tau g n \rangle$
contains either $\tau g^2 n$ or \revv{$\D h_3 g^2$}.
However, the shuffle
\[
h_1 \langle h_1^3 h_4, h_1, \tau g n \rangle =
\langle h_1, h_1^3 h_4, h_1 \rangle \tau g n = 0
\]
eliminates \revv{$\D h_3 g^2$}.
\end{proof}

\begin{lemma}
\label{lem:h3,p',h2}
\revdeg{80, 5, 42}
The Massey product $\langle h_3, p', h_2 \rangle$ equals
$h_0 e_2$, with no indeterminacy.
\end{lemma}

\begin{proof}
We have
\[
\langle h_3, p', h_2 \rangle h_4^2 = 
h_3 \langle p', h_2, h_4^2 \rangle = 
p' \langle h_2, h_4^2, h_3 \rangle.
\]
Table \ref{tab:Massey} shows that the last Massey product
equals $c_2$.
Observe that $p' c_2$ equals $h_0 h_4^2 e_2$.

\revv{Since $h_4^2 \cdot h_6 e_0$ is zero (as usual, we rely
on complete information about classical products in a large range
\cite{Bruner97} \cite{BR20}),}
this shows that $\langle h_3, p', h_2 \rangle$ equals
either $h_0 e_2$ or $h_0 e_2 + h_6 e_0$.
However, shuffle to obtain
\[
\langle h_3, p', h_2 \rangle h_1 = 
h_3 \langle p', h_2, h_1 \rangle,
\]
which must equal zero
\revv{because multiplication by $h_3$ is zero in the appropriate degree.}
Since $h_1(h_0 e_2 + h_6 e_0)$ is non-zero, it cannot equal
$\langle h_3, p', h_2 \rangle$.

The indeterminacy is zero by inspection.
\end{proof}

\begin{remark}
\label{rem:h1^3h4,h1,tgn}
The Massey product of Lemma \ref{lem:h1^3h4,h1,tgn}
cannot be established with Proposition \ref{prop:g-Massey}
because $p_*(\tau g n) = 0$ in $\Ext_{A(2)}$.
\end{remark}

\begin{lemma}
\label{lem:De1+C0,h1^3,h1h4}
\revdeg{82, 12, 45}
The Massey product $\langle \D e_1 + C_0, h_1^3, h_1 h_4 \rangle$
equals $(\D e_1 + C_0) g$, with no indeterminacy.
\end{lemma}

\begin{proof}
Consider the Massey product
$\langle \tau (\D e_1 + C_0), h_1^4, h_4 \rangle$.
By inspection, this Massey product has no indeterminacy. 
Therefore,
\[
\langle \tau (\D e_1 + C_0), h_1^4, h_4 \rangle =
(\D e_1 + C_0) \langle \tau, h_1^4, h_4 \rangle.
\]
Table \ref{tab:Massey} shows that the latter bracket equals $\tau g$,
so the expression equals $\tau (\D e_1 + C_0) g$.

On the other hand, it also equals
$\tau \langle \D e_1 + C_0, h_1^3, h_1 h_4 \rangle$.
Therefore, the bracket
$\langle \D e_1 + C_0, h_1^3, h_1 h_4 \rangle$ must contain
$(\D e_1 + C_0) g$.  Finally, the indeterminacy can be computed
by inspection.
\end{proof}

\begin{lemma}
\label{lem:t^3gG0,h0h2,h2}
\revdeg{93, 13, 49}
The Massey product $\langle \tau^3 g G_0, h_0 h_2, h_2 \rangle$
has indeterminacy generated by $\tau M^2 h_2$, and it either
contains zero or $\tau e_0 x_{76,9}$.  In particular, it does
not contain any linear combination of $\D^2 h_1 g_2$
with other elements.
\end{lemma}

\begin{proof}
The indeterminacy can be computed by inspection.

The only possible elements in
the Massey product $\langle \tau^2 g G_0, h_0 h_2, h_2 \rangle$
are linear combinations of $e_0 x_{76,9}$ and $M^2 h_2$.
The inclusion
\[
\tau \langle \tau^2 g G_0, h_0 h_2, h_2 \rangle \subseteq
\langle \tau^3 g G_0, h_0 h_2, h_2 \rangle
\]
gives the desired result.
\end{proof}

%% file: more-stable-stems-Adams-diff.tex
\chapter{Adams differentials}
\label{ch:Adams}

The goal of this chapter is to describe the values of the
Adams differentials in the motivic Adams spectral sequence.
These values are given in Tables \ref{tab:Adams-d2}, \ref{tab:Adams-d3},
\ref{tab:Adams-d4}, \ref{tab:Adams-d5}, and \ref{tab:Adams-higher}.
See also the Adams charts in \cite{Isaksen14a} for a graphical
representation of the computations.
\revv{
For easy reference, the many lemmas in this chapter are labelled with degrees that match the degrees given in the tables.}

\section{The Adams $d_2$ differential}
\label{sctn:Adams-d2}

\revv{
Table \ref{tab:Adams-d2} lists all of the multiplicative generators 
of the Adams $E_2$-page through the 95-stem.  The third column indicates
the value of the $d_2$ differential, if it is non-zero.
A blank entry in the third column indicates that the $d_2$ differential
is zero.
The fourth column indicates the proof. A blank entry in the fourth column
indicates that there are no possible values for the differential.
The fifth column gives alternative names for the element, as used 
in \cite{Tangora70a}, \cite{Bruner97}, and \cite{Isaksen14c}.
}

\begin{thm}
\label{thm:Adams-d2}
Table \ref{tab:Adams-d2} lists the values of the
Adams $d_2$ differential on all multiplicative generators
through the 95-stem.
\end{thm}

\rev{
\begin{remark}
\label{rem:d2}
A previous version of this manuscript left uncertain the value of
$d_2$ on three multiplicative generators.  These three values
have since been determined by Dexter Chua \cite{Chua21}.
We have included those values here, but we defer to \cite{Chua21}
and \cite{BF21} for their proofs.
Also, we are grateful to
Joey Beauvais-Feisthauer \cite{BF21} for discovering an error
in our previous calculation of $d_2(x_{85,6})$.
\end{remark}
}

\begin{proof}
The fourth column of Table \ref{tab:Adams-d2} gives information on the
proof of each differential.
Most follow immediately by comparison to
the Adams spectral sequence for $C\tau$
\revv{\cite{IWX19}}.
A few additional differentials follow by comparison to the classical
Adams spectral sequence for $\tmf$ \revv{\cite{BR21}}.

If an element is listed in the fourth column of Table \ref{tab:Adams-d2}, 
then the corresponding differential can be deduced from a straightforward
argument using a multiplicative relation.  For example, it is possible
that $d_2(\D h_1 h_3)$ equals $\tau d_0 e_0$.  However,
$h_0 \cdot \D h_1 h_3$ is zero, while $h_0 \cdot \tau d_0 e_0$ is non-zero.
Therefore, $d_2(\D h_1 h_3)$ must equal zero.

In some cases, it is necessary to combine these different techniques to
establish the differential.

The remaining more difficult computations are carried out in the
following lemmas.  We refer to \cite{Chua21} and \cite{BF21} in
a few cases.
\end{proof}

\begin{lemma}
\label{lem:d2-Dx}
\revdeg{61, 9, 32}
$d_2(\D x) = h_0^2 B_4 + \tau M h_1 d_0$.
\end{lemma}

\begin{proof}
We have a differential $d_2(\D x) = h_0^2 B_4$ in the
Adams spectral sequence for $C\tau$.
Therefore, $d_2(\D x)$ equals either
$h_0^2 B_4$ or $h_0^2 B_4 + \tau M h_1 d_0$.

We have the relation
$h_1^2 \cdot \D x = P h_1 \cdot \tau \D_1 h_1^2$
\revv{(as usual, we rely on complete information about classical products in a large range \cite{Bruner97} \cite{BR20})}, so
$h_1^2 d_2(\D x) = P h_1 d_2(\tau \D_1 h_1^2) = 
P h_1 h_3 \cdot M h_1 = \tau M h_1^3 d_0$.
Therefore, $d_2(\D x)$ must equal
$h_0^2 B_4 + \tau M h_1 d_0$.
\end{proof}

\begin{remark}
The proof of \cite{Isaksen14c}*{Lemma 3.50}
is incorrect.  We claimed that
$h_1^2 \cdot \D x$ equals $h_3 \cdot \D^2 h_1 h_3$,
when in fact
$h_1^2 \cdot \D x$ equals $\tau h_3 \cdot \D^2 h_1 h_3$.
\end{remark}

\begin{lemma}
\label{lem:d2-x77,7}
\revdeg{77, 7, 40}
$d_2(x_{77,7}) = \tau M h_1 h_4^2$.
\end{lemma}

The following proof was suggested to us by Dexter Chua.

\begin{proof}
This follows from the interaction between algebraic squaring
operations and classical Adams differentials \cite{Bruner84}*{Theorem 2.2},
applied to the element $x$ in the $37$-stem.
The theorem says that
\[
d_* \Sq^2 x = \Sq^3 d_2 x \dotplus h_0 \Sq^3 x.
\]
The notation means that there is an Adams differential on $\Sq^2 x$
hitting either $\Sq^3 d_2 x = 0$ or $h_0 \Sq^3 x$,
depending on which element has lower Adams filtration.
Therefore
$d_2 \Sq^2 x = h_0 \Sq^3 x$.

Next, observe from \cite{Bruner04} that
$\Sq^3 x = h_0^2 x_{76,6} + \tau^2 d_1 g_2$, so
\[
h_0 \Sq^3 x = h_0^3 x_{76,6} = \tau M h_1 h_4^2.
\]
Therefore, there is a $d_2$ differential whose value is
$\tau M h_1 h_4^2$, and the possibility is that
$d_2(x_{77,7})$ equals $\tau M h_1 h_4^2$.
\end{proof}

\begin{lemma}
\label{lem:d2-tB5g}
\revdeg{86, 14, 47}
$d_2(\tau B_5 g) = \tau M h_0^2 g^2$.
\end{lemma}

\begin{proof}
We use the Mahowald operator methods of 
Section \ref{subsctn:Mahowald-operator}.
\revv{According to \cite{Isaksen19}*{Table 1}},
the map $p_*: \Ext_\C \map \Ext_B$ takes
$d_0 \cdot \tau B_5 g$ to
$\tau h_0 a g^3 v_3^2$, which is non-zero.
We deduce that the product $d_0 \cdot \tau B_5 g$ is non-zero
in $\Ext$.
\revv{By inspection of motivic weights},
the only possibility is that it equals $\tau M g \cdot h_0 m$.

Now $d_2(\tau M g \cdot h_0 m)$ equals
$\tau M g \cdot h_0^2 e_0^2$, which we also know is non-zero
since it maps to the non-zero element
$\tau h_0^2 d e g^2 v_3^2$ of $\Ext_B$
\revv{by \cite{Isaksen19}*{Theorem 1}}.
It follows that $d_2(\tau B_5 g)$ 
\revv{is non-zero.  By inspection of motivic weights, the only
possibility is $\tau M h_0^2 g^2$.}
\end{proof}

\section{The Adams $d_3$ differential}
\label{sctn:Adams-d3}

\revv{
Table \ref{tab:Adams-d3} lists the multiplicative generators 
of the Adams $E_3$-page through the 95-stem
whose $d_3$ differentials are non-zero, or whose $d_3$ differentials
are zero for non-obvious reasons.
}

\begin{thm}
\label{thm:Adams-d3}
Table \ref{tab:Adams-d3} lists some values of the
Adams $d_3$ differential on multiplicative generators.
Through the 95-stem, 
the Adams $d_3$ differential is zero on all multiplicative
generators not listed in the table.
\end{thm}

\rev{
\begin{remark}
A previous version of this manuscript left uncertain the value of
$d_3$ on several multiplicative generators.  These values
have since been determined by Dexter Chua \cite{Chua21}.
We have included those values here, but we defer to \cite{Chua21}
for their proofs.
We are also grateful to Dexter Chua for correcting a few mistakes
in the values of the $d_3$ differential.
\end{remark}
}

\begin{proof}
The $d_3$ differential on many multiplicative generators is zero.
A few of these multiplicative generators appear
in Table \ref{tab:Adams-d3} because their proofs require further
explanation.
For the remaining majority of such multiplicative generators,
the $d_3$ differential is zero
because there are no possible non-zero
values, because of comparison to the Adams spectral sequence for $C\tau$,
or because the element is already known to be a permanent cycle
as shown in Table \ref{tab:Adams-perm}.
These cases do not appear in Table \ref{tab:Adams-d3}.

The last column of Table \ref{tab:Adams-d3} gives information on the
proof of each differential.
Most follow immediately by comparison to
the Adams spectral sequence for $C\tau$.  
A few additional differentials follow by comparison to the classical
Adams spectral sequence for $\tmf$, or by comparison to the
$\C$-motivic Adams spectral sequence for $\mmf$.

If an element is listed in the last column of Table \ref{tab:Adams-d3}, 
then the corresponding differential can be deduced from a straightforward
argument using a multiplicative relation.  For example, 
\[
d_3( h_1 \cdot \tau P d_0 e_0 ) =
P h_1 \cdot d_3(\tau d_0 e_0) =
P^2 h_1 c_0 d_0,
\]
so $d_3(\tau P d_0 e_0)$ must equal $P^2 c_0 d_0$.

If a $d_4$ differential is listed in the last column of 
Table \ref{tab:Adams-d3}, then
the corresponding differential is forced by consistency with
that later differential.
In each case,
a $d_3$ differential on an element $x$
is forced by the existence of a later
$d_4$ differential on $\tau x$.
For example, Table \ref{tab:Adams-d4} shows that
there is a differential
$d_4(\tau^2 e_0 g) = P d_0^2$.
Therefore, $\tau e_0 g$ cannot survive to the $E_4$-page.
It follows that $d_3(\tau e_0 g) = c_0 d_0^2$.

In some cases, it is necessary to combine these different techniques to
establish the differential.

The remaining more difficult computations are carried out in the
following lemmas.  We refer to \cite{Chua21} in
a few cases.
\end{proof}

\begin{prop}
\label{prop:Adams-perm}
Some permanent cycles in the
$\C$-motivic Adams spectral sequence are shown in 
Table \ref{tab:Adams-perm}.
\end{prop}

\begin{proof}
The third column of the table gives information on the proof
for each element.  If a Toda bracket is given in the third
column, then
the Moss Convergence Theorem \ref{thm:Moss} implies that the
element must survive to detect that Toda bracket
(see Table \ref{tab:Toda} for more information on how each
Toda bracket is computed).
If a product is given in the third column, then
the element must survive to detect that product
(see Table \ref{tab:misc-extn} for more information
on how each product is computed).
In a few cases, the third column refers to a specific lemma
that gives a more detailed argument.
\end{proof}

\begin{lemma}
\label{lem:d3-h2h5}
\mbox{}
\begin{enumerate}
\item
\revdeg{34, 2, 18}
$d_3(h_2 h_5) = \tau h_1 d_1$.
\item
\revdeg{74, 6, 38}
$d_3(P h_2 h_6) = \tau h_1 h_4 Q_2$.
\end{enumerate}
\end{lemma}

\begin{proof}
In the Adams spectral sequence for $C\tau$, there is an
$\eta$ extension from $h_2 h_5$ to $\ol{h_1^2 d_1}$.
The element $\ol{h_1^2 d_1}$ maps to $h_1^2 d_1$ under projection
from $C\tau$ to the top cell, so
$h_2 h_5$ must also map non-trivially under projection from
$C\tau$ to the top cell.
The only possibility is that $h_2 h_5$ maps to $h_1 d_1$.
Therefore, $\tau h_1 d_1$ must be hit by a differential.
This establishes the first differential.

The proof for the second differential is identical, using that
there is an $\eta$ extension from $P h_2 h_6$ to $\ol{h_1^2 h_4 Q_2}$
in the Adams spectral sequence for $C\tau$.
\end{proof}

\begin{lemma}
\label{lem:d3-t^2D1h1^2}
\revdeg{54, 6, 28}
$d_3(\tau^2 \D_1 h_1^2) = \tau M c_0$.
\end{lemma}

\begin{proof}
The element $M P$ maps to zero under inclusion of the bottom cell into 
$C \tau$. 
Therefore,
$M P$ is either hit by a differential, or it is the target of a 
hidden $\tau$ extension.  
\revv{If it is the target of a hidden $\tau$ extension},
then the only possibility is that $\tau M c_0$ is zero in the
$E_\infty$-page, and that there is a hidden
$\tau$ extension from $M c_0$ to $M P$.

\revv{It remains to show that $M P$ cannot be hit by a differential.
The only possibility is that $d_4(\tau^2 \D_1 h_1^2)$ might
equal $M P$. 
Note that $P h_1 \cdot \tau^2 \D_1 h_1^2$ 
equals $h_1 ( \tau h_1 \cdot \Delta x)$ in the $E_4$-page.
This means that 
$d_4(P h_1 \cdot \tau^2 \D_1 h_1^2)$ cannot equal $M P^2 h_1$
since $M P^2 h_1$ is not divisible by $h_1$.
In turn, $d_4(\tau^2 \D_1 h_1^2)$ cannot equal $M P$.
}
\end{proof}

\begin{lemma}
\label{lem:d3-tDg2.h0^3}
\revdeg{68, 11, 35}
$d_3(\tau h_0^3 \cdot \D g_2) = \tau^3 \D h_2^2 e_0 g$.
\end{lemma}

\begin{proof}
Table \ref{tab:unit-mmf} shows that the element
$h_0 h_5 i$ maps to $\D^2 h_2^2$ in
the Adams spectral sequence for $\tmf$.

Now $\D^2 h_2^2 d_0$ is not zero and not divisible by $2$ in $\tmf$.
Therefore, $\kappa \{ h_0 h_5 i\}$ must be non-zero and 
not divisible by $2$ in $\pi_{68,36}$.
The only possibility is that
$\kappa \{h_0 h_5 i\}$ is detected by 
$P h_2 h_5 j = d_0 \cdot h_0 h_5 i$, and that
$P h_2 h_5 j$ is not an $h_0$ multiple in the $E_\infty$-page.
Therefore, $\tau \D g_2 \cdot h_0^3$ cannot survive to the
$E_\infty$-page.
\end{proof}

\rev{
\begin{lemma}
\label{lem:d3-tD3'}
\revdeg{69, 8, 36}
$d_3(\tau D'_3) = \tau^2 M h_2 g$.
\end{lemma}
}

\begin{proof}
Table \ref{tab:Toda} shows that
the Toda bracket $\langle 2, 8 \sigma, 2, \sigma^2 \rangle$
contains $\tau \nu \kappabar$, which is detected by $\tau^2 h_2 g$.
Table \ref{tab:nu-extn} shows that
$M h_2$ detects $\nu \alpha$ for some
$\alpha$ in $\pi_{45,24}$ detected by $h_3^2 h_5$.
(Beware that there is a crossing extension, $M h_2$ does not
detect $\nu \alpha$ for every $\alpha$ that is detected by
$h_3^2 h_5$.)
It follows that
$\tau^2 M h_2 g$ detects
$\langle 2, 8 \sigma, 2, \sigma^2 \rangle \alpha$.

This expression is contained in
$\langle 2, 8 \sigma, \langle 2, \sigma^2, \alpha \rangle \rangle$.
Lemma \ref{lem:2,sigma^2,theta4.5} shows that
the inner bracket equals
$\{0, 2 \tau \kappabar^3 \}$.

The Toda bracket
$\langle 2, 8 \sigma, 0 \rangle$ 
in $\pi_{68,36}$
consists entirely of multiples of $2$.
The Toda bracket
$\langle 2, 8 \sigma, 2 \tau \kappabar^3 \rangle$ contains
$\langle 2, 8 \sigma, 2 \rangle \tau \kappabar^3$.
This last expression equals zero because 
\[
\langle 2, 8 \sigma, 2 \rangle = \tau \eta \cdot 8 \sigma = 0
\]
by Corollary \ref{cor:2-symmetric}.
Therefore, 
$\langle 2, 8 \sigma, 2 \tau \kappabar^3 \rangle$ equals
its indeterminacy, which
consists entirely of multiples of $2$ in $\pi_{68,36}$.

We conclude that $\tau^2 M h_2 g$ is either hit by a differential, or
is the target of a hidden $2$ extension.  
Lemma \ref{lem:2-h3A'} shows that
there is no hidden $2$ extension from $h_3 A'$ to $\tau^2 M h_2 g$,
and there are no other possible extensions to $\tau^2 M h_2 g$.

Therefore, $\tau^2 M h_2 g$ must be hit by 
a differential, and the only possible source of this differential
is $\tau D'_3$.
\end{proof}

\begin{lemma}
\label{lem:d3-t^2Mh0l}
\revdeg{77, 14, 40}
$d_3(\tau^2 M h_0 l) = \D^2 h_0 d_0^2$.
\end{lemma}

\begin{proof}
Table \ref{tab:Adams-d4} shows that
$d_4(\tau^2 d_0 e_0 + h_0^7 h_5)$ equals $P^2 d_0$, so
$d_4(\tau^3 M h_1 d_0 e_0)$ equals $\tau M P^2 h_1 d_0$.
We have the relation $h_0 \cdot \tau^2 M h_0 l = \tau^3 M h_1 d_0 e_0$,
but the element $\tau M P^2 h_1 d_0$ is not divisible by $h_0$.
Therefore, $\tau^2 M h_0 l$ cannot survive to the $E_4$-page.

By comparison to the Adams spectral sequence for $\tmf$,
the value of $d_3(\tau^2 M h_0 l)$ cannot be
$\tau^3 \D h_1 e_0^3 + \D^2 h_0 d_0^2$ or
$\tau^3 \D h_1 e_0^3$.
The only remaining possibility is that
$d_3(\tau^2 M h_0 l)$ equals
$\D^2 h_0 d_0^2$.
\end{proof}

\begin{lemma}
\label{lem:d3-h0^3x78,10}
\revdeg{78, 13, 40}
$d_3(h_0^3 x_{78,10}) = \tau^6 e_0 g^3$.
\end{lemma}

\begin{proof}
Suppose that $h_0^3 x_{78,10}$ were a permanent cycle.
Then it would map under inclusion of the bottom cell
to the element $h_0^3 x_{78,10}$ in the 
Adams $E_\infty$-page for $C\tau$.

There is a hidden $\nu$ extension from $h_0^3 x_{78,10}$
to $\D^3 h_1^2 h_3$ in the Adams $E_\infty$-page for $C\tau$.
Then $\D^3 h_1^2 h_3$ would also have to be in the image
of inclusion of the bottom cell.  
The only possible pre-image is the element $\D^3 h_1^2 h_3$
in the Adams spectral sequence for the sphere, but this element
does not survive by Lemma \ref{lem:d4-D^3h1^2h3}.

By contradiction, we have shown that $h_0^3 x_{78,10}$
must support a differential.  The only possibility is that
$d_3(h_0^3 x_{78,10})$ equals $\tau^6 e_0 g^3$.
\end{proof}

\begin{lemma}
\label{lem:d3-x1}
\revdeg{79, 5, 42}
$d_3(x_1) = \tau h_1 m_1$.
\end{lemma}

\begin{proof}
This follows from the interaction between algebraic squaring
operations and classical Adams differentials \cite{Bruner84}*{Theorem 2.2}.
The theorem says that
\[
d_* \Sq^1 e_1 = \Sq^3 d_3 e_1 \dotplus h_1 \Sq^3 e_1.
\]
The notation means that there is an Adams differential on $\Sq^1 e_1$
hitting either $\Sq^3 d_3 e_1$ or $h_1 \Sq^3 e_1$,
depending on which element has lower Adams filtration.
Therefore
$d_3 \Sq^1 e_1 = h_1 \Sq^3 e_1$.

Finally, we observe from \cite{Bruner04} that
$\Sq^1 e_1 = x_1$ and $\Sq^3 e_1 = m_1$.
\end{proof}

\begin{lemma}
\label{lem:perm-D^2d1}
\revdeg{80, 12, 42}
The element $\D^2 d_1$ is a permanent cycle.
\end{lemma}

\begin{proof}
The element $\D^2 d_1$ in the Adams $E_\infty$-page for $C\tau$
must map to zero under the projection from $C\tau$ to the top cell.
The only possible value in sufficiently high filtration is
$\tau^2 \D h_1 e_0^2 g$.  However, comparison to $\mmf$
shows that this element is not annihilated by $\tau$, and therefore
cannot be in the image of projection to the top cell. 

Therefore, $\D^2 d_1$ 
must be in the image of the inclusion of the bottom cell into $C\tau$.
The element $\D^2 d_1$ is the 
only possible pre-image in the Adams $E_\infty$-page for the sphere 
in sufficiently low filtration.
\end{proof}

\rev{
\begin{lemma}
\label{lem:d3-D^3h1h3}
\revdeg{80, 14, 41}
$d_3(\D^3 h_1 h_3)$ equals either
$\tau^4 \D h_1 e_0^2 g$, $\tau \D^2 h_0 d_0 e_0$, or
$\tau^4 \D h_1 e_0^2 g + \tau \D^2 h_0 d_0 e_0$; and it is not equal to
$d_3(\tau^2 d_0 B_5)$.
\end{lemma}

\begin{proof}
There is a relation $P h_1 \cdot \D^3 h_1 h_3 = \tau \D^3 h_1^3 d_0$
in the Adams $E_2$-page.  Because of the differential
$d_2(\D^3 h_1 e_0) = \D^3 h_1^3 d_0 + \tau^5 e_0^2 g m$,
we have the relation
$P h_1 \cdot \D^3 h_1 h_3 = \tau^6 e_0^2 g m$ in
the $E_3$-page.

There is a differential $d_4(\tau^6 e_0^2 g m) = \tau^4 d_0^4 l$.
But $\tau^4 d_0^4 l$ is not divisible by $P h_1$,
so $\tau^6 e_0^2 g m$ cannot be divisible by $P h_1$
in the $E_4$-page.  Therefore,
$d_3(\D^3 h_1 h_3)$ must be non-zero.

The same argument shows that 
$d_3(\D^3 h_1 h_3 + \tau^3 d_0 B_5)$ must also be non-zero.
\end{proof}

\begin{remark}
In fact, Chua has determined that $d_3(\D^3 h_1 h_3)$ equals
$\tau^4 \D h_1 e_0^2 g$ \cite{Chua21}.
\end{remark}
}

\rev{
\begin{lemma}
\label{lem:d3-t^2d0B5}
\revdeg{80, 14, 42}
$d_3(\tau^2 d_0 B_5)$ equals either $\D^2 h_0 d_0 e_0$
or $\D^2 h_0 d_0 e_0 + \tau^3 \D h_1 e_0^2 g$.
\end{lemma}

\begin{proof}
The element $\D^2 h_0 d_0 e_0$ is a permanent cycle because
there are no possible differentials that it could support.
Moreover, 
it must map to zero under the inclusion of the bottom cell into $C\tau$ because there are no elements in the
Adams $E_\infty$-page for $C\tau$ of sufficiently high filtration.
Therefore, $\D^2 h_0 d_0 e_0$ 
is either hit by a differential, or it is the target of a 
hidden $\tau$ extension, or it is the target of a non-hidden
$\tau$ extension.

The only possible hidden $\tau$ extension has source $h_1^3 x_{76,6}$.
However, Table \ref{tab:hid-proj} shows that
$h_1^3 x_{76,6}$ is in the image of projection from $C\tau$ to the
top cell.  Therefore, it cannot support a hidden extension.

We now know that $\D^2 h_0 d_0 e_0$ must be hit by a differential, or
it is $\tau$-divisible in the $E_\infty$-page.
Lemma \ref{lem:perm-D^2d1} rules out one possible source for the
differential.  The only remaining possibilities are that 
$d_3(\tau^2 d_0 B_5)$ equals
$\D^2 h_0 d_0 e_0$ or 
$\D^2 h_0 d_0 e_0 + \tau^3 \D h_1 e_0^2 g$.
\end{proof} 

\begin{remark}
We are grateful to Dexter Chua for pointing out an error in a previous
version of Lemma \ref{lem:d3-t^2d0B5}.  In fact, Chua has determined
that the value of $d_3(\tau^2 d_0 B_5)$ is 
$\D^2 h_0 d_0 e_0 + \tau^3 \D h_1 e_0^2 g$ \cite{Chua21}.
\end{remark}
}

\begin{lemma}
\label{lem:d3-h2h4h6}
\mbox{}
\begin{enumerate}
\item
\revdeg{81, 3, 42}
$d_3(h_2 h_4 h_6) = 0$.
\item
\revdeg{82, 10, 42}
$d_3(P^2 h_2 h_6) = 0$.
\end{enumerate}
\end{lemma}

\begin{proof}
The value of $d_3(h_2 h_4 h_6)$ is not $h_2 h_6 d_0$ nor
$h_2 h_6 d_0 + \tau h_1 x_1$ by comparison to the Adams
spectral sequence for $C\tau$.

It remains to show that $d_3(h_2 h_4 h_6)$ cannot equal $\tau h_1 x_1$.
Suppose that the differential did occur.
Then there would be no possible targets for a 
hidden $\tau$ extension on $h_1 x_1$,
so the $\eta$ extension from $h_1 x_1$ to $h_1^2 x_1$ would be detected
by projection from $C\tau$ to the top cell.  But there 
is no such $\eta$ extension in the homotopy groups of $C\tau$.
This establishes the first formula.

The proof of the second formula is essentially the same,
using that the $\eta$ extension
from $\D^2 h_1 d_1$ to $\D^2 h_1^2 d_1$ cannot be
detected by projection from $C\tau$ to the top cell.
\end{proof}

\begin{lemma}
\label{lem:d3-D^2p}
\revdeg{81, 12, 42}
$d_3(\D^2 p) = 0$.
\end{lemma}

\begin{proof}
Suppose that $d_3(\D^2 p)$ were equal to $\tau^3 h_1 M e_0^2$.
In the Adams $E_4$-page,
the Massey product $\langle \tau^2 M g, \tau h_1 d_0, d_0 \rangle$
would equal $\tau^2 M \D h_2^2 g$, with no indeterminacy,
because of the Adams differential $d_3(\D h_2^2) = \tau h_1 d_0^2$
and because $d_0 \cdot \D^2 p = 0$.
By Moss's higher Leibniz rule \ref{thm:Leibniz},
$d_4(\tau^2 M \D h_2^2 g)$ would be a linear
combination of multiples of $\tau^2 M g$ and $d_0$.
But Table \ref{tab:Adams-d4} shows that $d_4(\tau^2 M \D h_2^2 g)$
equals $M P \D h_0^2 e_0$, which is not such a linear combination
in the Adams $E_4$-page.
\end{proof}

\begin{lemma}
\label{lem:d3-th6g}
\revdeg{83, 5, 43}
$d_3(\tau h_6 g + \tau h_2 e_2) = 0$.
\end{lemma}

\begin{proof}
In the Adams $E_3$-page, we have the matric Massey product
\[
\tau h_6 g + \tau h_2 e_2 =
\left\langle 
\left[
\begin{array}{cc}
\tau g & \tau h_2
\end{array}
\right],
\left[
\begin{array}{c}
h_5^2 \\ x_1
\end{array}
\right],
h_0
\right\rangle
\]
because of the Adams differentials
$d_2(h_6) = h_0 h_5^2$ and
$d_2(e_2) = h_0 x_1$, as well as the relation
$\tau g \cdot h_5^2 + \tau h_2 x_1$ in the Adams $E_2$-page.
Moss's higher Leibniz rule \ref{thm:Leibniz}
implies that 
$d_3(\tau h_6 g + \tau h_2 e_2)$ 
belongs to
\[
\left\langle
\left[
0 \enskip 0
\right],
\left[
\begin{array}{c}
h_5^2 \\ x_1
\end{array}
\right],
h_0 \right\rangle +
\left\langle
\left[
\tau g \enskip \tau h_2
\right],
\left[
\begin{array}{c}
0 \\ \tau h_1 m_1
\end{array}
\right],
h_0 \right\rangle +
\left\langle
\left[
\tau g \enskip \tau h_2
\right],
\left[
\begin{array}{c}
h_5^2 \\ x_1
\end{array}
\right],
0 \right\rangle
\]
since $d_3(x_1) = \tau h_1 m_1$,
where the Massey products are formed in the Adams $E_3$-page using
the $d_2$ differential.
This expression simplifies to
$\left\langle
\left[
\tau g \enskip \tau h_2
\right],
\left[
\begin{array}{c}
0 \\ \tau h_1 m_1
\end{array}
\right],
h_0 \right\rangle$,
which equals $\{0, \tau h_0^2 h_4 Q_3 \}$.

Table \ref{tab:nu-extn} shows that there is a hidden $\nu$
extension from $h_0^2 h_4 Q_3$ to $P h_1 x_{76,6}$.
The element $\tau P h_1 x_{76,6}$ is non-zero in the Adams
$E_\infty$-page.  Therefore,
$h_0^2 h_4 Q_3$ supports a (hidden or not hidden) 
$\tau$ extension whose target is in Adams filtration at most $10$.
The only possibility is that
$\tau h_0^2 h_4 Q_3$ is non-zero in the Adams 
$E_\infty$-page.
\end{proof}

\rev{
\begin{lemma}
\label{lem:d3-f2}
\revdeg{84, 4, 44}
$d_3(f_2) = \tau h_1 h_4 Q_3$.
\end{lemma}

\begin{proof}
Table \ref{tab:nu-extn} shows that 
$\tau h_1 Q_3$ detects $\nu^2 \theta_5$, and 
Table \ref{tab:Toda} shows that
$\tau h_1 h_4 Q_3$ detects
$\langle \nu^2 \theta_5, 2, \sigma^2 \rangle$,
with indeterminacy in strictly higher Adams filtration.
This bracket contains
$\nu^2 \langle \theta_5, 2, \sigma \rangle$, so
$\tau h_1 h_4 Q_3$ detects a multiple of $\nu^2$.
The only possibility is that $\tau h_1 h_4 Q_3$ is a multiple
of $h_2^2$ in the $E_\infty$-page.
This implies that 
$d_3(f_2)$ equals $\tau h_1 h_4 Q_3$ or $\tau h_1 h_4 Q_3 + h_0^2 h_6 g$.
The latter possibility is ruled out by comparison to
$C\tau$.
\end{proof}
}

\rev{
\begin{lemma}
\label{lem:d3-tx85,6+h0^3c3}
\revdeg{85, 6, 45}
$d_3(\tau x_{85,6} + h_0^3 c_3) = 0$.
\end{lemma}

\begin{proof}
Let $\alpha$ be an element of $\pi_{66,35}$ that is detected by
$\tau h_2 C'$.  Then $\nu \alpha$ is detected by $\tau h_2^2 C'$,
and $\tau \nu \alpha$ is zero.

Let $\ol{\alpha}$
be an element of $\pi_{70,36} C\tau$ that is detected by $h_0^2 h_3 h_6$.
Projection from $C\tau$ to the top cell takes $\ol{\alpha}$ to
$\nu \alpha$.
Moreover, in the homotopy of $C\tau$, the Toda bracket
$\langle 2, \sigma^2, \ol{\alpha} \rangle$
is detected by $h_0^3 c_3$.

Now projection from $C\tau$ to the top cell takes
$\langle 2, \sigma^2, \ol{\alpha} \rangle$ to 
$\langle 2, \sigma^2, \nu \alpha \rangle$, which equals zero
by Lemma \ref{lem:2,sigma^2,th2^2C'}.
Therefore, $h_0^3 c_3$ maps to zero under projection to the top cell
of $C\tau$,
so it must be in the image of inclusion of the bottom cell.

There are two possibilities.
First, $\tau x_{85,6} + h_0^3 c_3$ could survive, and it could
map to $h_0^3 c_3$ under inclusion of the bottom cell of $C\tau$.
Second,
$\tau h_1 f_2$ could map to $h_0^3 c_3$ under inclusion of the 
bottom cell.  This could only occur if $d_{10}(h_1 f_2)$ equaled
$M \D h_1 d_0$
and $d_9(\tau x_{85,6} + h_0^3 c_3)$ equaled $\tau M \D h_1 d_0$.

In either case, $d_3(\tau x_{85,6} + h_0^3 c_3)$ is zero.
\end{proof}
}

\rev{
\begin{remark}
\label{rem:d3-tx85,6+h0^3c3}
In the proof of Lemma \ref{lem:d3-tx85,6+h0^3c3}, we have used that
$d_5(\tau p_1 + h_0^2 h_3 h_6)$ equals $\tau^2 h_2^2 C'$ in order
to conclude that $\tau \nu \alpha$ is zero.  This differential
depends on work in preparation \cite{BIX}.  

However, we can
also prove Lemma \ref{lem:d3-tx85,6+h0^3c3} independently of \cite{BIX}.
Lemma \ref{lem:d5-tp1+h0^2h3h6} shows that the other possible value
of $d_5(\tau p_1 + h_0^2 h_3 h_6)$ is 
$\tau^2 h_2^2 C' + \tau h_3(\D e_1 + C_0)$.  In this case,
let $\beta$ be an element of $\pi_{62,33}$ that is detected by
$\D e_1 + C_0$.  Then $\nu \alpha + \sigma \beta$ is detected by 
$\tau h_2^2 C' + h_3(\D e_1 + C_0)$, 
and $\tau (\nu \alpha + \sigma \beta)$ is zero.

Projection from $C\tau$ to the top cell takes 
$\ol{\alpha}$ to $\nu \alpha + \sigma \beta$, and takes
$\langle 2, \sigma^2, \ol{\alpha} \rangle$ to 
$\langle 2, \sigma^2, \nu \alpha + \sigma \beta \rangle$, 
which equals zero
by Lemmas \ref{lem:2,sigma^2,th2^2C'} and 
\ref{lem:2,sigma^2,h3(De1+C0)}.
As in the proof of Lemma \ref{lem:d3-tx85,6+h0^3c3},
$h_0^3 c_3$ maps to zero under projection to the top cell of $C\tau$,
so it must be in the image of inclusion of the bottom cell.
\end{remark}
}

\rev{
\begin{lemma}
\label{lem:d3-t^2Mh0d0k}
\revdeg{88, 18, 46}
$d_3(\tau^2 M h_0 d_0 k) = P \D^2 h_0 d_0 e_0 + \tau^3 \D h_1 d_0^2 e_0^2$.
\end{lemma}

\begin{proof}
Table \ref{tab:Adams-d4} shows that
$d_4(\tau^2 M h_1 e_0) = M P^2 h_1$.
Multiply by $\tau d_0^2$ to see that 
$d_4(\tau^3 M h_1 d_0^2 e_0) = \tau M P^2 h_1 d_0^2$.
We have the relation $h_2 \cdot \tau^2 M h_0 d_0 k = \tau^3 M h_1 d_0^2 e_0$,
but $\tau M P^2 h_1 d_0^2$ is not divisible by $h_2$.
Therefore,
$\tau^2 M h_0 d_0 k$ cannot survive to the $E_4$-page.
By comparison to $\mmf$, there is only one possible value for
$d_3(\tau^2 M h_0 d_0 k)$.
\end{proof}

\begin{remark}
We are grateful to Dexter Chua for pointing out a small error in a previous
version of Lemma \ref{lem:d3-t^2Mh0d0k}.
\end{remark}
}

\begin{lemma}
\label{lem:d3-h2B5g}
\revdeg{89, 15, 50}
$d_3(h_2 B_5 g) = M h_1 c_0 e_0^2$.
\end{lemma}

\begin{proof}
\revv{We will first establish the relation
$d_0 \cdot h_2 B_5 g = \tau M h_1 e_0 g^2$.
We use the map $p_*$ of \cite{Isaksen19}*{Theorem 1.1}.
We have that $p_*(d_0 \cdot h_2 B_5 g) = p_*(\tau M h_1 e_0 g^2)$.
Therefore, 
$d_0 \cdot h_2 B_5 g$ equals $\tau M h_1 e_0 g^2$, modulo a possible
error term $P h_1^7 h_6 c_0 e_0$ in the kernel of $p_*$.
However, multiplication by $h_1$ eliminates the error term.
}

Table \ref{tab:Adams-d3} shows that $d_3(\tau e_0 g^2) = c_0 d_0 e_0^2$.
Therefore, $d_3(\tau M h_1 e_0 g^2)$ equals
$M h_1 c_0 d_0 e_0^2$.  Observing
that $M h_1 c_0 d_0 e_0^2$ is in fact non-zero in the Adams $E_3$-page,
we conclude that $d_3(h_2 B_5 g)$ must equal $M h_1 c_0 e_0^2$.
\end{proof}

\begin{lemma}
\label{lem:perm-M^2}
\revdeg{90, 8, 48}
The element $M^2$ is a permanent cycle.
\end{lemma}

\begin{proof}
Table \ref{tab:Massey} shows that the Massey product 
$\langle M h_1, h_0, h_0^2 g_2 \rangle$ equals $M^2 h_1$.
Therefore,
$M^2 h_1$ detects the Toda bracket
$\langle \eta \theta_{4.5}, 2, \sigma^2 \theta_4 \rangle$.
The indeterminacy consists entirely
of multiples of $\eta \theta_{4.5}$.
The Toda bracket contains
$\theta_4 \langle \eta \theta_{4.5}, 2, \sigma^2 \rangle$.
Now 
$\langle \eta \theta_{4.5}, 2, \sigma^2 \rangle$ is zero because
$\pi_{61,33}$ is zero.

We have now shown that $M^2 h_1$ detects a multiple
of $\eta$.  In fact, it detects a non-zero multiple of $\eta$
because $M^2 h_1$ cannot be hit by a differential by 
comparison to the Adams spectral sequence for $C\tau$.

Therefore, there exists a non-zero element of $\pi_{90,48}$
that is detected in Adams fitration at most $12$.
The only possibility is that $M^2$ survives.
\end{proof}

\begin{lemma}
\label{lem:d3-Dh2^2h6}
\revdeg{93, 7, 48}
$d_3(\D h_2^2 h_6) = \tau h_1 h_6 d_0^2$.
\end{lemma}

\begin{proof}
In the Adams $E_3$-page, 
$\D h_2^2 h_6$ equals $\langle \D h_2^2, h_5^2, h_0 \rangle$,
with no indeterminacy, because of the Adams differential
$d_2(h_6) = h_0 h_5^2$.
Using that $d_3(\D h_2^2) = \tau h_1 d_0^2$, 
Moss's higher Leibniz rule \ref{thm:Leibniz}
implies that $d_3(\D h_2^2 h_6)$ is contained in
\[
\langle \tau h_1 d_0^2, h_5^2, h_0 \rangle +
\langle \D h_2^2, 0, h_0 \rangle +
\langle \D h_2^2, h_5^2, 0 \rangle.
\]
All of these brackets have no indeterminacy, and the last two
equal zero.
The first bracket equals $\tau h_1 h_6 d_0^2$, using the
Adams differential $d_2(h_6) = h_0 h_5^2$.
\end{proof}

\begin{lemma}
\label{lem:d3-P^2h6d0}
\revdeg{93, 13, 48}
$d_3(P^2 h_6 d_0) = 0$.
\end{lemma}

\begin{proof}
In the Adams $E_3$-page, the element $P^2 h_6 d_0$ equals
the Massey product
$\langle P^2 d_0, h_5^2, h_0 \rangle$, with no indeterminacy,
because of the Adams differential $d_2(h_6) = h_0 h_5^2$.
Moss's higher Leibniz rule \ref{thm:Leibniz} implies
that $d_3(P^2 h_6 d_0)$ is a linear combination of multiples of
$h_0$ and of $P^2 d_0$.  The only possibility is that
$d_3(P^2 h_6 d_0)$ is zero.
\end{proof}

\rev{
\begin{lemma}
\label{lem:d3-t^2MPh0d0j}
\revdeg{93, 22, 48}
$d_3(\tau^2 M P h_0 d_0 j) = P^2 \D^2 h_0 d_0^2 + 
\tau^3 P \D h_1 d_0^3 e_0$.
\end{lemma}
}
\begin{proof}
Table \ref{tab:Adams-d4} shows that
$d_4(\tau^2 P d_0 e_0) = P^3 d_0$.
Multiplication by $\tau M P h_1$ shows that
$d_4(\tau^3 M P^2 h_1 d_0 e_0)$ equals
$\tau M P^4 h_1 d_0$.
But $\tau^3 M P^2 h_1 d_0 e_0$ equals
$h_0 \cdot \tau^2 M P h_0 d_0 j$, 
while $\tau M P^4 h_1 d_0$ is not divisible by $h_0$.
Therefore,
$\tau^2 M P h_0 d_0 j$ cannot survive to the $E_4$-page.

The possible values for
$d_3(\tau^2 M P h_0 d_0 j)$ are
\revv{the non-zero}
linear combinations of 
$P^2 \D^2 h_0 d_0^2$ and
$\tau^3 P \D h_1 d_0^3 e_0$.
\revv{The map to the Adams spectral sequence for $\tmf$
takes both $\tau^2 M P h_0 d_0 j$ and 
$P^2 \D^2 h_0 d_0^2 + \tau^3 P \D h_1 d_0^3 e_0$ to zero, 
but it takes each of
$P^2 \D^2 h_0 d_0^2$ and
$\tau^3 P \D h_1 d_0^3 e_0$ to the unique non-zero element
in the appropriate degree.
Therefore,
$d_3(\tau^2 M P h_0 d_0 j)$ cannot equal either
$P^2 \D^2 h_0 d_0^2$ or $\tau^3 P \D h_1 d_0^3 e_0$.
}
\end{proof}

\begin{lemma}
\label{lem:perm-P^3h6c0}
\revdeg{95, 16, 49}
The element $P^3 h_6 c_0$ is a permanent cycle.
\end{lemma}

\begin{proof}
Table \ref{tab:eta-extn} shows that
$P^3 c_0$ detects the product $\eta \rho_{31}$.
Using the Moss Convergence Theorem \ref{thm:Moss} 
and the Adams differential $d_2(h_6) = h_0 h_5^2$,
the element $P^3 h_6 c_0$ must survive to detect the
Toda bracket
$\langle \eta \rho_{31}, 2, \theta_5 \rangle$.
\end{proof}

\begin{remark}
\label{rem:perm-P^3h6c0}
We suspect that $P^3 h_6 c_0$ detects the product
$\eta_6 \rho_{31}$.  However, the argument of
Lemma \ref{lem:rho15-h1h6} cannot be completed because
the Toda bracket
$\langle \eta \rho_{31}, 2, \theta_5 \rangle$ might
have indeterminacy in lower Adams filtration.
\end{remark}

\section{The Adams $d_4$ differential}
\label{sctn:Adams-d4}

\revv{
Table \ref{tab:Adams-d4} lists the multiplicative generators 
of the Adams $E_4$-page through the 95-stem
whose $d_4$ differentials are non-zero, or whose $d_4$ differentials
are zero for non-obvious reasons.
}

\begin{thm}
\label{thm:Adams-d4}
Table \ref{tab:Adams-d4} lists some values of the
Adams $d_4$ differential on multiplicative generators.
Through the 95-stem, 
the Adams $d_4$ differential is zero on all multiplicative
generators not listed in the table.
\end{thm}

\begin{proof}
The $d_4$ differential on many multiplicative generators is zero.
A few of these multiplicative generators appear
in Table \ref{tab:Adams-d4} because their proofs require further
explanation.
For the remaining majority of such multiplicative generators,
the $d_4$ differential is zero
because there are no possible non-zero values, 
or because of comparison to the Adams spectral sequences for $C\tau$,
$\tmf$, or $\mmf$.
In a few cases, the multiplicative generator
is already known to be a permanent cycle as shown in
Table \ref{tab:Adams-perm}.
These cases do not appear in Table \ref{tab:Adams-d4}.

The last column of Table \ref{tab:Adams-d4} gives information on the
proof of each differential.
Most follow immediately by comparison to
the Adams spectral sequence for $C\tau$, or
by comparison to the classical Adams spectral sequence for $\tmf$, or
by comparison to the $\C$-motivic Adams spectral sequence for $\mmf$.

If an element is listed in the last column of Table \ref{tab:Adams-d4}, 
then the corresponding differential can be deduced from a straightforward
argument using a multiplicative relation.  For example, 
\[
d_4( d_0 \cdot \tau^2 e_0 g^2 ) =
d_4( e_0^2 \cdot \tau^2 e_0 g) = 
e_0^2 \cdot P d_0^2 = d_0^5,
\]
so $d_4(\tau^2 e_0 g^2)$ must equal $d_0^4$.

The remaining more difficult computations are carried out in the
following lemmas.
\end{proof}

\begin{lemma}
\label{lem:d4-th1.Dx}
\revdeg{62, 10, 32}
$d_4(\tau h_1 \cdot \D x) = \tau^2 \D h_2^2 d_0 e_0$.
\end{lemma}

\revv{This differential was previously proved in
\cite{WangXu17}*{Remark 11.2}.  We repeat the argument here
for completeness.
}

\begin{proof}
Table \ref{tab:Adams-d4} shows that $\tau^3 \D h_2^2 g^2$ supports a $d_4$ differential, and 
Table \ref{tab:Adams-perm} shows that 
$\tau \D^2 h_1^2 g + \tau^3 \D h_2^2 g^2$ is a permanent cycle.
Therefore,
$\tau \D^2 h_1^2 g$ also supports a $d_4$ differential.

On the other hand, we have
\[
h_1 \cdot \tau \D^2 h_1 g = P h_1 \cdot \D x =
\D x \langle h_1, h_0^3 h_3, h_0 \rangle.
\]
This expression equals $\langle h_1 \cdot \D x, h_0^3 h_3, h_0 \rangle$
by inspection of indeterminacies.
Therefore, the Toda bracket
$\langle \{ \tau h_1 \cdot \D x \}, 8 \sigma, 2 \rangle$
cannot be well-formed, since otherwise it would be detected
by $\tau \D^2 h_1^2 g$.
The only possibility is that 
$\tau h_1 \cdot \D x$ is not a permanent cycle, and the
only possible differential is that
$d_4(\tau h_1 \cdot \D x)$ equals $\tau^2 \D h_2^2 d_0 e_0$.
\end{proof}

\begin{lemma}
\label{lem:d4-D^2h3^2}
\revdeg{62, 10, 32}
$d_4(\D^2 h_3^2) = 0$.
\end{lemma}

\begin{proof}
Table \ref{tab:Adams-d5} shows that 
$d_5(\tau h_1^2 \cdot \D x)$ equals
$\tau^3 d_0^2 e_0^2$.  
The element $\tau^3 d_0^2 e_0^2$ is not divisible by
$h_1$ in the $E_5$-page, so
$\tau h_1^2 \cdot \D x$ 
cannot be divisible by $h_1$ in the $E_4$-page.

If $d_4(\D^2 h_3^2)$ equaled 
$\tau^2 \D h_2^2 d_0 e_0$, then
$\D^2 h_3^2 + \tau h_1 \cdot \D x$ would survive to the $E_5$-page, 
and $\tau^2 h_1^2 \cdot \D x$ would be divisible by $h_1$ in the $E_5$-page.
\end{proof}

\begin{lemma}
\label{lem:d4-tX2}
\revdeg{63, 7, 33}
$d_4(\tau X_2) = \tau M h_2 d_0$.
\end{lemma}

\begin{proof}
Table \ref{tab:Adams-d4} shows that $d_4(C')$ equals $M h_2 d_0$.
Therefore, either $\tau X_2$ or $\tau X_2 + \tau C'$ is non-zero on the
$E_\infty$-page.  The inclusion of the bottom cell into $C\tau$
takes this element to $\ol{h_5 d_0 e_0}$.

In the homotopy of $C\tau$, there is a $\nu$ extension from
$\ol{h_5 d_0 e_0}$ to $\tau B_5$, and
inclusion of the bottom cell into $C\tau$ takes
$\tau h_2 C'$ to $\tau B_5$.

It follows that there must be a $\nu$ extension with target
$\tau h_2 C'$.  The only possibility is that
$\tau X_2 + \tau C'$ is non-zero on the $E_\infty$-page,
and therefore $d_4(\tau X_2)$ equals $d_4(\tau C')$.
\end{proof}

\begin{lemma}
\label{lem:d4-h0d2}
\revdeg{68, 5, 36}
$d_4(h_0 d_2) = X_3$.
\end{lemma}

\begin{proof}
The element $X_3$ is a permanent cycle.  The only possible target for 
a differential is $\tau^2 d_0 e_0 m$, but this is ruled out by 
comparison to $\tmf$.

\revv{
In the Adams $E_\infty$-page for $C\tau$, there are several elements in stem 67 and weight 36.  However, they all have filtration less than 9.
Since $X_3$ has filtration $9$, it must map to zero under inclusion
of the bottom cell into $C\tau$.}
Therefore, $X_3$ is the target of a hidden
$\tau$ extension, or it is hit by a differential.

The only possible hidden $\tau$ extension would have source
$h_1 \cdot \D_1 h_3^2$.  In $C\tau$, there is an $\eta$ extension from
$h_0 d_2$ to $\ol{h_1^2 \cdot \D_1 h_3^2}$.  
Since $\ol{h_1^2 \cdot \D_1 h_3^2}$ maps non-trivially 
(to $h_1^2 \cdot \D_1 h_3^2$) under projection to the top cell of
$C\tau$, it follows that
$h_0 d_2$ also maps non-trivially under projection.
\revv{For degree reasons,}
the only possibility is that $h_0 d_2$ maps to
$h_1 \cdot \D_1 h_3^2$, and therefore
$h_1 \cdot \D_1 h_3^2$ does not support a hidden 
$\tau$ extension.

Therefore, $X_3$ must be hit by a differential, and there is just
one possibility.
\end{proof}

\begin{lemma}
\label{lem:d4-Mh2g}
\revdeg{68, 11, 38}
$d_4(M h_2 g) = 0$.
\end{lemma}

\begin{proof}
Table \ref{tab:Massey} shows that the Massey product
$\langle h_2 g, h_0^3, g_2 \rangle$ equals $M h_2 g$.
The Moss Convergence Theorem \ref{thm:Moss} shows that
$M h_2 g$ must survive to detect the Toda bracket
$\langle \{h_2 g\}, 8, \kappabar_2 \rangle$.
\end{proof}

\begin{lemma}
\label{lem:d4-h2^2G0}
\revdeg{72, 9, 40}
$d_4(h_2^2 G_0) = \tau g^2 n$.
\end{lemma}

\begin{proof}
Table \ref{tab:2-extn} shows that
there is a hidden $2$ extension from 
$h_0 h_3 g_2$ to $\tau g n$.  Therefore, 
$\tau g n$ detects $4 \sigma \kappabar_2$.

Table \ref{tab:Massey} shows that
$\langle h_1^3 h_4, h_1, \tau g n \rangle$ consists of the
two elements
$\tau g^2 n$ and $\tau g^2 n + M h_2 g \cdot h_0 h_2$.
Then the Toda bracket
$\langle \eta^2 \eta_4, \eta, 4 \sigma \kappabar_2 \rangle$
is detected by either
$\tau g^2 n$ or $\tau g^2 n + M h_2 g \cdot h_0 h_2$.
But $M h_2 g \cdot h_0 h_2$ is hit by an Adams $d_2$ differential,
so $\tau g^2 n$ detects the Toda bracket.

The Toda bracket has no indeterminacy, so it equals
$\langle \eta^2 \eta_4, \eta, 2 \rangle 2 \sigma \kappabar_2$.
This last expression must be zero.

We have shown that $\tau g^2 n$ must be hit by some differential.
The only possibility is that $d_4(h_2^2 G_0) = \tau g^2 n$.
\end{proof}

\begin{lemma}
\label{lem:d4-Dh0^2h3g2}
\revdeg{75, 11, 40}
$d_4(\D h_0^2 h_3 g_2) = \tau M h_1 d_0^2$.
\end{lemma}

\begin{proof}
Table \ref{tab:Adams-d5} shows that 
$d_5(A') = \tau M h_1 d_0$.
Now $d_0 A'$ is zero in the $E_5$-page, so
$\tau M h_1 d_0^2$ must also be zero in the $E_5$-page.
\end{proof}

\begin{lemma}
\label{lem:d4-D^2h1h3g}
\revdeg{76, 14, 41}
$d_4(\D^2 h_1 h_3 g) = \tau \D h_2^2 d_0^2 e_0$.
\end{lemma}

\begin{proof}
Table \ref{tab:Toda} shows that the element
$\D h_2^2 d_0^2 e_0$ detects the Toda bracket
$\langle \tau \eta \kappa \kappabar^2, \eta, \eta^2 \eta_4 \rangle$.
Now shuffle to obtain
\[
\tau \langle \tau \eta \kappa \kappabar^2, \eta, \eta^2 \eta_4 \rangle =
\langle \tau, \tau \eta \kappa \kappabar^2, \eta \rangle \eta^2 \eta_4.
\]
Table \ref{tab:Toda} shows that
$\langle \tau, \tau \eta \kappa \kappabar^2, \eta \rangle$ is detected
by $h_0 h_2 h_5 i$.
It follows that the expression
$\langle \tau, \tau \eta \kappa \kappabar^2, \eta \rangle \eta^2 \eta_4$
is zero, so $\tau \D h_2^2 d_0^2 e_0$ must be hit by some differential.
The only possibility is that 
$d_4(\D^2 h_1 h_3 g)$ equals 
$\tau \D h_2^2 d_0^2 e_0$.
\end{proof}

\begin{lemma}
\label{lem:d4-h0e2}
\revdeg{80, 5, 42}
$d_4(h_0 e_2) = \tau h_1^3 x_{76,6}$.
\end{lemma}

\begin{proof}
Table \ref{tab:misc-extn}
shows that $\sigma^2 \theta_5$ is detected by
$h_0 h_4 A$ \rev{or $h_0 h_4 A + \tau^2 d_1 g_2$.
Note that both $h_2 \cdot h_0 h_4 A$ and 
$h_2 (h_0 h_4 A + \tau^2 d_1 g_2)$ equal 
$\tau h_1^3 x_{76,6}$.
}
Since $\nu \sigma = 0$, the element
$\tau h_1^3 x_{76,6}$ must be hit
by a differential.
The only possibility is that $d_4(h_0 e_2)$ equals
$\tau h_1^3 x_{76,6}$.
\end{proof}

\begin{lemma}
\label{lem:d4-D^3h1^2h3}
\revdeg{81, 15, 42}
$d_4(\D^3 h_1^2 h_3) = \tau^4 d_0 e_0^2 l$.
\end{lemma}

\begin{proof}
Table \ref{tab:eta-extn} shows that there is a hidden
$\eta$ extension from $\tau^2 \D h_1 g^2$ to $\tau^2 d_0 e_0 m$.
\revv{Multiply by $d_0$ to see that}
there is also a hidden $\eta$ extension
from $\tau^2 \D h_1 e_0^2 g$ to $\tau^2 d_0 e_0^2 l$.

Also, $\tau^2 \D h_1 e_0^2 g$ 
detects an element in $\pi_{79,43}$ that is annihilated by $\tau^2$.
Therefore, $\tau^4 d_0 e_0^2 l$ must be hit by some
differential.
Moreover, comparison to $\mmf$ shows that 
$\tau^3 d_0 e_0^2 l$ is not hit by a differential.

The hidden $\eta$ extension from $\tau^3 \D h_1 e_0^2 g$
to $\tau^3 d_0 e_0^2 l$ is detected by projection from
$C\tau$ to the top cell.
The only possibility is that this hidden $\eta$ extension is
the image of the $h_1$ extension
from $\D^3 h_1 h_3$ to $\D^3 h_1^2 h_3$
in the Adams $E_\infty$-page for $C\tau$.

Therefore, $\D^3 h_1^2 h_3$ maps non-trivially under projection
from $C\tau$ to the top cell.
Consequently,
$\D^3 h_1^2 h_3$ cannot be a permanent cycle in the
Adams spectral sequence for the sphere.
\end{proof}

\begin{lemma}
\label{lem:d4-Dj1}
\revdeg{83, 11, 45}
$d_4(\D j_1) = \tau M h_0 e_0 g$.
\end{lemma}

\begin{proof}
Otherwise, both $\D j_1$ and $\tau g C'$ would survive to the
$E_\infty$-page, and neither could be the target of a
hidden $\tau$ extension.  They would both map non-trivially
under inclusion of the bottom cell into $C\tau$.
But there are not enough elements in $\pi_{83,45} C\tau$
for this to occur.
\end{proof}

\rev{
\begin{lemma}
\label{lem:d4-h1f2}
\revdeg{85, 5, 45}
$d_4(h_1 f_2) = 0$.
\end{lemma}

\begin{proof}
Table \ref{tab:nu-extn} shows that there is a hidden
$\nu$ extension from $h_2 g D_3$ to $B_6 d_1$.
If $d_4(h_1 f_2)$ equaled $\tau h_2 g D_3$,
then this $\nu$ extension would be in the image of
projection from $C\tau$ to the top cell, since $h_2 g D_3$
cannot support a hidden $\tau$ extension.
However, there is no such $\nu$ extension in the homotopy 
of $C\tau$.
\end{proof}
}

\rev{
\begin{lemma}
\label{lem:d4-tx85,6+h0^3c3}
\revdeg{85, 6, 44}
$d_4(\tau x_{85,6} + h_0^3 c_3) = 0$.
\end{lemma}

\begin{proof}
We showed in Lemma \ref{lem:d3-tx85,6+h0^3c3} that
$h_0^3 c_3$ is in the image of inclusion of the bottom cell of $C\tau$.
Therefore, $P x_{76,6}$ cannot be in the image of projection from
$C\tau$ to the top cell.  Since $P x_{76,6}$ cannot support a hidden
$\tau$ extension, there can be no differential whose value is
$\tau P x_{76,6}$.
\end{proof}
}

\begin{lemma}
\label{lem:d4-h1c3}
\revdeg{86, 4, 45}
$d_4(h_1 c_3) = \tau h_0 h_2 h_4 Q_3$.
\end{lemma}

\begin{proof}
Lemma \ref{lem:compound-(epsilon+etasigma)-theta5}
shows that there exists an element
$\alpha$ in $\pi_{67,36}$ that is detected by
$h_0 Q_3 + h_0 n_1$ such that
$\tau \nu \alpha$ equals
$(\eta \sigma + \epsilon) \theta_5$.

Table \ref{tab:Toda} shows that the Toda bracket
$\langle \nu, \sigma, 2\sigma \rangle$ is detected by $h_2 h_4$, so
the element
$\tau h_0 h_2 h_4 Q_3$ detects
$\tau \alpha \langle \nu, \sigma, 2 \sigma \rangle$,
which is contained in
$\langle \tau \nu \alpha, \sigma, 2 \sigma \rangle$.
The indeterminacy in these expressions
is zero because
$\tau \nu \alpha \cdot \pi_{15,8}$ and
$2 \sigma \cdot \pi_{78,41}$ are both zero.

We now know that
$\tau h_0 h_2 h_4 Q_3$ detects the Toda bracket
$\langle (\epsilon + \eta \sigma) \theta_5, \sigma, 2 \sigma \rangle$.
This bracket contains
$\theta_5 \langle \epsilon + \eta \sigma, \sigma, 2 \sigma \rangle$.
Lemma \ref{lem:nubar,sigma,2sigma} shows that 
the bracket
$\langle \epsilon + \eta \sigma, \sigma, 2 \sigma \rangle$ contains
$0$, so
$\theta_5 \langle \epsilon + \eta \sigma, \sigma, 2 \sigma \rangle$
equals zero.

Finally, we have shown that $\tau h_0 h_2 h_4 Q_3$
detects zero, so it must be hit by some differential.
\end{proof}

\begin{lemma}
\label{lem:d4-x87,7}
\revdeg{87, 7, 45}
$d_4(x_{87,7}) = 0$.
\end{lemma}

\begin{proof}
Consider the exact sequence
\[
\pi_{87,45} \map \pi_{87,45} C\tau \map \pi_{86,46}.
\]
The middle term $\pi_{87,45} C\tau$ is isomorphic to $(\Z/2)^4$.
The elements of $\pi_{87,45}$ that are not divisible by
$\tau$ are detected by $P^2 h_6 c_0$, and possibly
$x_{87,7}$ and $\tau \D h_1 H_1$.
On the other hand, the elements of $\pi_{86,46}$ that are
annihilated by $\tau$ are detected by
$\tau^3 \D c_0 e_0^2 g$ and possibly $M \D h_0^2 e_0$.

In order for the possibility $M \D h_0^2 e_0$ to occur,
either $x_{87,7}$ or $\tau \D h_1 H_1$ would have to support
a differential hitting $\tau M \D h_0^2 e_0$, in which case
one of those possibilities could not occur.

If $d_4(x_{87,7})$ equaled $\tau^3 g G_0$, then there would not
be enough elements to make the above sequence exact.
\end{proof}

\begin{lemma}
\label{lem:d4-tDh1H1}
\revdeg{87, 10 45}
$d_4(\tau \D h_1 H_1) = 0$.
\end{lemma}

\begin{proof}
The element $\D^2 h_2^2 d_1$ is a permanent cycle
that cannot be hit by any differential because
$h_2 \cdot \D^2 h_2^2 d_1$ cannot be hit by a differential.
The element $\D^2 h_2^2 d_1$ cannot be in the image
of projection from $C\tau$ to the top cell,
and it cannot support a hidden $\tau$ extension.
Therefore,
$\tau \D^2 h_2^2 d_1$ cannot be hit by a differential.
\end{proof}

\begin{lemma}
\label{lem:d4-th2B5g}
\revdeg{89, 15, 49}
$d_4(\tau h_2 B_5 g) = M h_1 d_0^3$.
\end{lemma}

\begin{proof}
Table \ref{tab:misc-extn} shows that
$M d_0$ detects $\kappa \theta_{4.5}$.
Therefore, $M d_0^3$ detects $\kappa^3 \theta_{4.5}$,
which equals $\eta^2 \kappabar^2 \theta_{4.5}$ because
Table \ref{tab:eta-extn} shows that there is a hidden
$\eta$ extension from $\tau^2 h_1 g^2$ to $d_0^3$.

Now $\eta^2 \kappabar^2 \theta_{4.5}$ is zero because
$\eta^2 \kappabar \theta_{4.5}$ is zero.
Therefore, $M d_0^3$ and $M h_1 d_0^3$ must both be hit
by differentials.

There are several possible differentials that can hit
$M h_1 d_0^3$.  
The element $h_1 x_{88,10}$ cannot be the source of this differential
because Table \ref{tab:Adams-perm} shows that
$x_{88,10}$ is a permanent cycle.
The element $\tau h_2^2 g C'$ cannot be the source of the differential
because $h_2^2 g C'$ is a permanent cycle by comparison to $\mmf$.
The element $\D h_1 g_2 g$ cannot be the source 
because it equals $h_3 (\D e_1 + C_0) g$.
The only remaining possibility is that
$d_4(\tau h_2 B_5 g)$ equals $M h_1 d_0^3$.
\end{proof}

\begin{lemma}
\label{lem:d4-Dh2^2A'}
\revdeg{91, 12, 48}
$d_4(\D h_2^2 A') = 0$.
\end{lemma}

\begin{proof}
In the Adams $E_4$-page, the element $\D h_2^2 A'$ equals the 
Massey product $\langle A', h_1, \tau d_0^2 \rangle$,
with no indeterminacy
because of the Adams differential $d_3(\D h_2^2) = \tau h_1 d_0^2$.
Moss's higher Leibniz rule \ref{thm:Leibniz}
implies that $d_4(\D h_2^2 A')$ is contained in
\[
\langle 0, h_1, \tau d_0^2 \rangle + \langle A', 0, \tau d_0^2 \rangle +
\langle A', h_1, 0 \rangle,
\]
so it is
a linear combination
of multiples of $A'$ and $\tau d_0^2$.
The only possibility is that $d_4(\D h_2^2 A')$ is zero.
\end{proof}

\begin{lemma}
\label{lem:d4-h4^2h6}
\revdeg{93, 3, 48}
$d_4(h_4^2 h_6) = h_0^3 g_3$.
\end{lemma}

\begin{proof}
By comparison to 
the Adams spectral sequence for $C\tau$, the value of 
$d_4(h_4^2 h_6)$ is either
$h_0^3 g_3$ or $h_0^3 g_3 + \tau h_1 h_4^2 D_3$.

Table \ref{tab:misc-extn} shows that 
$h_0^2 g_3$ detects the product $\theta_4 \theta_5$.
Since $2 \theta_4 \theta_5$ equals zero, 
$h_0^3 g_3$ must be hit by a differential.
\end{proof}

\rev{
\begin{lemma}
\label{lem:d4-MD^2h1^2}
\revdeg{95, 16, 50}
$d_4(M \D^2 h_1^2) = M P \D h_0^2 e_0$.
\end{lemma}

\begin{proof}
Table \ref{tab:Toda} shows that
$M \D^2 h_1^2 + \tau^2 M \D h_2^2 g$ detects the Toda bracket
$\langle \eta, \tau \kappa^2, \tau \theta_{4.5} \kappabar \rangle$.
Therefore,
$d_4(M \D^2 h_1^2)$ equals $d_4(\tau^2 M \D h_2^2 g)$.
\end{proof}
}

\section{The Adams $d_5$ differential}
\label{sctn:Adams-d5}

\revv{
Table \ref{tab:Adams-d5} lists the multiplicative generators 
of the Adams $E_5$-page through the \rev{92}-stem
whose $d_5$ differentials are non-zero, or whose $d_5$ differentials
are zero for non-obvious reasons.
}

\begin{thm}
\label{thm:Adams-d5}
Table \ref{tab:Adams-d5} lists some values of the
Adams $d_5$ differential on multiplicative generators.
Through the \rev{92}-stem, 
the Adams $d_5$ differential is zero on all multiplicative
generators not listed in the table.
\end{thm}

\begin{proof}
The $d_5$ differential on many multiplicative generators is zero.
For the majority of such multiplicative generators,
the $d_5$ differential is zero
because there are no possible non-zero
values, or by comparison to the Adams spectral sequence for $C\tau$,
or by comparison to $\tmf$ or $\mmf$.
In a few cases, the multiplicative generator
is already known to be a permanent cycle; $h_1 h_6$ is one
such example.
A few additional cases appear
in Table \ref{tab:Adams-d5} because their proofs require further
explanation.

The last column of Table \ref{tab:Adams-d5} gives information on the
proof of each differential.
Many computations follow immediately by comparison to
the Adams spectral sequence for $C\tau$.

If an element is listed in the last column of Table \ref{tab:Adams-d5}, 
then the corresponding differential can be deduced from a straightforward
argument using a multiplicative relation.  For example, 
\[
d_5( \tau \cdot g A' ) =
d_5( \tau g \cdot A') =
\tau g \cdot \tau M h_1 d_0 = \tau^2 M h_1 e_0^2,
\]
so $d_5(g A')$ must equal $\tau M h_1 e_0^2$.

A few of the more difficult computations appear in \cite{BIX}.
The remaining more difficult computations are carried out in the
following lemmas.
\end{proof}

\begin{lemma}
\label{lem:d5-th1^2.Dx}
\revdeg{63, 11, 33}
$d_5(\tau h_1^2 \cdot \D x) = \tau^3 d_0^2 e_0^2$.
\end{lemma}

\begin{proof}
The element $\tau^2 d_0^2 e_0^2$ cannot be hit by a differential.
There is a hidden $\eta$ extension
from $\tau \D h_2^2 d_0 e_0$ to $\tau^2 d_0^2 e_0^2$
because of the hidden $\tau$ extensions
from $\tau h_1 g^3 + h_1^5 h_5 c_0 e_0$ to
$\D h_2^2 d_0 e_0$ and
from $h_1^6 h_5 c_0 e_0$ to $d_0^2 e_0^2$.
This shows that $\tau^3 d_0^2 e_0^2$ must be hit by some differential.

This hidden $\eta$ extension is detected by
projection from $C\tau$ to the top cell.
Since $\ol{P h_5 c_0 d_0}$ in $C\tau$ maps to
$\tau \D h_2^2 d_0 e_0$ under projection to the top cell,
it follows that $\ol{P h_1 h_5 c_0 d_0}$ in $C\tau$
maps to $\tau^2 d_0^2 e_0^2$ under projection to the top cell.

If $\tau h_1^2 \cdot \D x$ survived, then it could not be the target of
a hidden $\tau$ extension and it could not be hit by a differential.
Also, it could not map non-trivially under inclusion of
the bottom cell into $C\tau$, since the only possible value
$\ol{P h_1 h_5 c_0 d_0}$ has already been accounted for 
in the previous paragraph.
\end{proof}

\begin{lemma}
\label{lem:d5-h5d0i}
\revdeg{68, 12, 36}
$d_5(h_5 d_0 i) = \tau \D h_1 d_0^3$.
\end{lemma}

\begin{proof}
We showed in Lemma \ref{lem:d3-tDg2.h0^3}
that $P h_2 h_5 j$ cannot be divisible by $h_0$ in the $E_\infty$-page.
Therefore, $h_5 d_0 i$ must support a differential.
\end{proof}

\begin{lemma}
\label{lem:d5-tp1+h0^2h3h6}
\revdeg{70, 4, 36}
$d_5(\tau p_1 + h_0^2 h_3 h_6)$ equals either
$\tau^2 h_2^2 C'$ or $\tau^2 h_2^2 C' + \tau h_3(\D e_1 + C_0)$.
\end{lemma}

\begin{proof}
Projection to the top cell of $C\tau$ takes
$h_4 D_2$ to $\tau^3 d_1 g^2$.
Moreover,
there is a $\nu$ extension in the homotopy of $C\tau$
from $h_0^2 h_3 h_6$ to $h_4 D_2$.
Therefore, this $\nu$ extension must be in the image of
projection to the top cell.

Table \ref{tab:nu-extn} shows that there is a hidden
$\nu$ extension from
$\tau h_2^2 C'$ to $\tau^3 d_1 g^2$.
Therefore, either $\tau h_2^2 C'$ or
$\tau h_2^2 C' + h_3(\D e_1 + C_0)$ is in the image of
projection to the top cell, so
$\tau^2 h_2^2 C'$ or $\tau^2 h_2^2 C' + \tau h_2 (\D e_1 + C_0)$
is hit by a differential.
The element $\tau p_1 + h_0^2 h_3 h_6$ is the 
only possible source for this differential.
\end{proof}

\begin{lemma}
\label{lem:d5-h1x71,6}
\revdeg{72, 7, 39}
$d_5(h_1 x_{71,6}) = 0$.
\end{lemma}

\begin{proof}
Table \ref{tab:tau-extn} shows that 
there is a hidden $\tau$ extension from 
$M h_1^2 h_3 g$ to $M h_1 d_0^2$.
Therefore, $M h_2^2 g$ must also support a $\tau$ extension.
This shows that $\tau M h_2^2 g$ cannot be the target of a differential.
\end{proof}

\begin{lemma}
\label{lem:d5-h4D2}
\revdeg{73, 7, 38}
$d_5(h_4 D_2) = \tau^4 d_1 g^2$.
\end{lemma}

\begin{proof}
Suppose for sake of contradiction that $h_4 D_2$ survived, and
let $\alpha$ be an element of $\pi_{73,38}$ that is detected by it.
Table \ref{tab:tau-extn} shows that there is a hidden $\tau$ extension
from $h_1^2 h_6 c_0$ to $h_0 h_4 D_2$.  Therefore,
$h_0 h_4 D_2$ detects both $2 \alpha$ and $\tau \eta \epsilon \eta_6$.
However, it is possible that the difference between these two elements
is detected by $\tau^2 M d_0^2$ or by $\tau^3 \D h_1 d_0 e_0^2$.
We will handle of each of these cases.

First, suppose that $2 \alpha$ equals $\tau \eta \epsilon \eta_6$.
Then the Toda bracket
\[
\left\langle
\eta,
\left[
\begin{array}{cc}
2 & \tau \eta \epsilon 
\end{array}
\right],
\left[
\begin{array}{c}
\alpha \\ \eta_6 
\end{array}
\right]
\right\rangle
\]
is well-defined.  Inclusion of the bottom cell into $C\tau$ takes this
bracket to
\[
\left\langle
\eta,
\left[
\begin{array}{cc}
2 & 0
\end{array}
\right],
\left[
\begin{array}{c}
\alpha \\ \eta_6 
\end{array}
\right]
\right\rangle =
\langle \eta, 2, \alpha \rangle,
\]
so $\langle \eta, 2, \alpha \rangle$ is in the image of inclusion of 
the bottom cell.

On the other hand,
in the homotopy of $C\tau$, the bracket 
$\langle \eta, 2, \alpha \rangle$ is detected by
$\ol{h_1^3 h_6 c_0}$, with indeterminacy generated by
$\ol{h_1^2 h_4 Q_2}$.  
These elements map non-trivially under projection to the top cell,
which contradicts that they are in the image of inclusion of the 
bottom cell.  

Next, suppose that $2 \alpha + \tau \eta \epsilon \eta_6$
is detected by $\tau^3 \D h_1 d_0 e_0^2$.
Then the Toda bracket
\[
\left\langle
\eta,
\left[
\begin{array}{ccc}
2 & \tau \eta \epsilon & \tau^2 \beta
\end{array}
\right],
\left[
\begin{array}{c}
\alpha \\ \eta_6 \\ \kappabar
\end{array}
\right]
\right\rangle
\]
is well-defined, where $\beta$ is an element of $\pi_{53,29}$
that is detected by $\D h_1 d_0^2$.
The same argument involving inclusion of the bottom cell into
$C\tau$ applies to this Toda bracket.

Finally, assume that $2 \alpha + \tau \eta \epsilon \eta_6$
is detected by $\tau^2 M d_0^2$.
Table \ref{tab:misc-extn} shows that $M d_0$ detects
$\kappa \theta_{4.5}$, so $\tau^2 M d_0^2$ detects 
$\tau^2 \kappa^2 \theta_{4.5}$.
Then $2 \alpha + \tau \eta \epsilon \eta_6$ equals either
$\tau^2 \kappa^2 \theta_{4.5}$ or 
$\tau^2 \kappa^2 \theta_{4.5} + \tau^2 \beta \kappabar$.
We can apply the same argument to the Toda bracket
\[
\left\langle
\eta,
\left[
\begin{array}{ccc}
2 & \tau \eta \epsilon & \tau^2 \kappa \theta_{4.5}
\end{array}
\right],
\left[
\begin{array}{c}
\alpha \\ \eta_6 \\ \kappa
\end{array}
\right]
\right\rangle,
\]
or to the Toda bracket
\[
\left\langle
\eta,
\left[
\begin{array}{cccc}
2 & \tau \eta \epsilon & \tau^2 \kappa \theta_{4.5} & \tau^2 \beta
\end{array}
\right],
\left[
\begin{array}{c}
\alpha \\ \eta_6 \\ \kappa \\ \kappabar
\end{array}
\right]
\right\rangle.
\]

We have now shown by contradiction that $h_4 D_2$ does not survive.
After ruling out other possibilities by comparison to $C\tau$
and to $\mmf$, the only remaining possibility is that 
$d_5(h_4 D_2)$ equals $\tau^4 d_1 g^2$.
\end{proof}

\begin{lemma}
\label{lem:d5-t^3gG0}
\revdeg{86, 11, 45}
$d_5(\tau^3 g G_0) = \tau M \D h_1^2 d_0$.
\end{lemma}

\begin{proof}
Suppose for sake of contradiction that the element
$\tau^3 g G_0$ survived.
It cannot be the target of a hidden $\tau$ extension,
and it cannot be hit by a differential.
Therefore, it maps non-trivially under inclusion of the bottom
cell into $C\tau$, and the only possible image is 
$\D^2 e_1 + \ol{\tau \D h_2 e_1 g}$.

Let $\alpha$ be an element of $\pi_{86,45}$ that is detected
by $\tau^3 g G_0$.  
Consider the Toda bracket
$\langle \alpha, 2\nu, \nu \rangle$.
Lemma \ref{lem:t^3gG0,h0h2,h2} implies that this Toda bracket
is detected by $e_0 x_{76,9}$, or is detected in higher Adams filtration.

On the other hand, 
under inclusion of the bottom cell into $C\tau$,
the Toda bracket is detected by $\D^2 h_1 g_2$.
This is inconsistent with the conclusion of the previous paragraph,
since inclusion of the bottom cell can only increase
Adams filtrations.

We now know that $\tau^3 g G_0$ does not survive.
After eliminating other possibilities by comparison to $\mmf$,
the only remaining possibility is that $d_5(\tau^3 g G_0)$ equals
$\tau M \D h_1^2 d_0$.
\end{proof}

\begin{lemma}
\revdeg{92, 4, 48}
\label{lem:d5-g3}
$d_5(g_3) = h_6 d_0^2$.
\end{lemma}

\begin{proof}
Table \ref{tab:Toda} shows that $h_1 h_6$ detects the Toda bracket
$\langle \eta, 2, \theta_5 \rangle$.  Therefore,
$h_1 h_6 d_0^2$ detects $\kappa^2 \langle \eta, 2, \theta_5 \rangle$.
Now consider the shuffle
\[
\tau \kappa^2 \langle \eta, 2, \theta_5 \rangle =
\langle \tau \kappa^2, \eta, 2 \rangle \theta_5.
\]
Lemma \ref{lem:tkappa^2,eta,2} shows that the last bracket is zero.
Therefore, $h_1 h_6 d_0^2$ does not support a hidden
$\tau$ extension, so it is either hit by a differential or 
in the image of projection from $C\tau$ to the top cell.

In the Adams spectral sequence for $C\tau$, the element 
$h_0^3 h_4^2 h_6$ detects the Toda bracket
$\langle \theta_4, 2, \theta_5 \rangle$.
Therefore, $h_0^3 h_4^2 h_6$ must be in the image of
inclusion of the bottom cell into $C\tau$.
In particular,
$h_0^3 h_4^2 h_6$ cannot map to $h_1 h_6 d_0^2$
under projection from $C\tau$ to the top cell.

Now $h_1 h_6 d_0^2$ cannot be in the image of
projection from $C\tau$ to the top cell, so it
must be hit by some differential.
The only possibility is that $d_5(h_1 g_3)$ equals $h_1 h_6 d_0^2$.
\end{proof}

\rev{
\begin{lemma}
\label{lem:d5-D^2g2}
\revdeg{92, 12, 48}
$d_5(\D^2 g_2) = 0$.
\end{lemma}

\begin{proof}
The only possible values for $d_2(\D^2 g_2)$ are the linear combinations
of $\tau M \D c_0 d_0 + \tau^2 M d_0 l$ 
and $\tau^2 \D^2 h_2 g^2$.
The possibilities
$\tau M \D c_0 d_0 + \tau^2 M d_0 l$ and
$\tau M \D c_0 d_0 + \tau^2 M d_0 l + \tau^2 \D^2 h_2 g^2$ are
ruled out by $d_0$ extensions.
More specifically,
$d_0 (\tau M \D c_0 d_0 + \tau^2 M d_0 l)$
and
$d_0(\tau M \D c_0 d_0 + \tau^2 M d_0 l + \tau^2 \D^2 h_2 g^2)$
equal the non-zero element $\tau^2 M d_0^2 l$
in the $E_5$-page, while
$d_0 \cdot \D^2 g_2$ is zero already in the $E_2$-page.

If $\tau^2 \D^2 h_2 g^2$ were the value of a differential,
then the $2$ extension from $\tau \D^2 h_2 g^2$ to $\tau \D^2 h_0 h_2 g^2$
would be detected by the top cell of $C\tau$.
However, there is no such $2$ extension in the homotopy of $C\tau$.
\end{proof}
}

\begin{lemma}
\label{lem:d5-e0x76,9}
\revdeg{93, 13, 50}
$d_5(e_0 x_{76,9}) = M \D h_1 c_0 d_0$.
\end{lemma}

\begin{proof}
If $M \D h_1 c_0 d_0$ were a permanent non-zero cycle, then it
could not support a hidden $\tau$ extension because
Lemma \ref{lem:hit-MPDh1d0} shows that $M P \D h_1 d_0$
is hit by some differential.
Therefore, it would lie in the image of projection from
$C\tau$ to the top cell, and the only possible pre-image is the
element $\D^2 h_1 g_2$ in the $E_\infty$-page of the
Adams spectral sequence for $C\tau$.

There is a $\sigma$ extension from $\D^2 e_1 + \ol{\tau \D h_2 e_1 g}$
to $\D^2 h_1 g_2$ in the Adams spectral sequence for $C\tau$.
Then $M \D h_1 c_0 d_0$ would also have to be the target of a 
$\sigma$ extension.
The only possible source for this extension would be $M \D h_1^2 d_0$.

Table \ref{tab:eta-extn} shows that $M h_1$
detects $\eta \theta_{4.5}$, so
$M \D h_1^2 d_0$ detects $\eta \theta_{4.5} \{\D h_1 d_0\}$.
The product $\eta \sigma \theta_{4.5} \{\D h_1 d_0\}$ equals zero
because $\sigma \{\D h_1 d_0\}$ is zero.
Therefore, $M \D h_1^2 d_0$ cannot support a hidden $\sigma$
extension to $M \D h_1 c_0 d_0$.

We have now shown that $M \D h_1 c_0 d_0$ must be hit by some
differential, and the only possibility is that equals 
$d_5(e_0 x_{76,9})$.
\end{proof}

\section{Higher differentials}
\label{sctn:Adams-higher}

\revv{
Table \ref{tab:Adams-higher} lists the multiplicative generators 
of the Adams $E_r$-page, for $r \geq 6$, through the 90-stem
whose $d_r$ differentials are non-zero, or whose $d_r$ differentials
are zero for non-obvious reasons.
}

\begin{thm}
\label{thm:Adams-higher}
Table \ref{tab:Adams-higher} lists some values of the
Adams $d_r$ differential on multiplicative generators of the
$E_r$-page, for $r \geq 6$.
For $r \geq 6$, 
the Adams $d_r$ differential is zero on all multiplicative
generators of the $E_r$-page not listed in the table.
The list is complete through the 90-stem, except that:
\rev{
\begin{enumerate}
\item
$d_{10}(h_1 f_2)$ might equal $M \D h_1 d_0$.
\item
$d_9(\tau x_{85,6} + h_0^3 c_3)$ might equal $\tau M \D h_1 d_0$.
\item
$d_9(h_4^2 D_3)$ might equal $\tau M \D h_1 g$.
\end{enumerate}
}
\end{thm}

\begin{proof}
The $d_r$ differential on many multiplicative generators is zero.
For the majority of such multiplicative generators,
the $d_r$ differential is zero
because there are no possible non-zero
values, or by comparison to the Adams spectral sequence for $C\tau$,
or by comparison to $\tmf$ or $\mmf$.
In a few cases, the multiplicative generator
is already known to be a permanent cycle, as shown in 
Table \ref{tab:Adams-perm}.
A few additional cases appear
in Table \ref{tab:Adams-higher} because their proofs require further
explanation.

Some of the more difficult computations appear in \cite{BIX}.
The remaining more difficult computations are carried out in the
following lemmas.
\end{proof}

\begin{lemma}
\label{lem:d6-tQ3+tn1}
\mbox{}
\begin{enumerate}
\item
\revdeg{67, 5, 35}
$d_6(\tau Q_3 + \tau n_1) = 0$.
\item
\revdeg{87, 9, 48}
$d_6(g Q_3) = 0$
\end{enumerate}
\end{lemma}

\begin{proof}
Several possible differentials on these elements are eliminated
by comparison to the
Adams spectral sequences for $C\tau$ and for $\tmf$.
The only remaining possibility is that $d_6(\tau Q_3 + \tau n_1)$
might equal $\tau^2 M h_1 g$, and
that $d_6(g Q_3)$ might equal $\tau M h_1 g^2$.

The element $M \D h_0^2 e_0$ is not hit by any differential because
Table \ref{tab:Adams-perm} shows that $h_1^2 c_3$ is a permanent cycle,
and Table \ref{tab:Toda} shows that $\tau^2 g Q_3 = h_4^2 Q_2$
must survive to detect the Toda bracket
$\langle \theta_4, \tau \kappabar, \{ t\} \rangle$.

Lemma \ref{lem:tetakappabar^2,2,4kappabar2}
shows that $M \D h_0^2 e_0$ 
detects the Toda bracket
$\langle \tau \eta \kappabar^2, 2, 4 \kappabar_2 \rangle$, which
contains
$\tau \kappabar^2 \langle \eta, 2, 4 \kappabar_2 \rangle$.
Lemma \ref{lem:eta,2,4kappabar2} shows that
this expression contains zero.
We now know that $M \D h_0^2 e_0$ detects an element
in the indeterminacy of the bracket
$\langle \tau \eta \kappabar^2, 4, 2 \kappabar_2 \rangle$.
In fact, it must detect a multiple
of $\tau \eta \kappabar^2$ since
$2 \kappabar_2 \cdot \pi_{42,22}$ is zero.

The only possibility is that $M \D h_0^2 e_0$
detects $\kappabar$ times an element detected by $\tau^2 M h_1 g$.
Therefore, $\tau^2 M h_1 g$ cannot be hit by a differential.
This shows that $\tau Q_3 + \tau n_1$ is a permanent cycle.

We also know that $M \D h_0^2 e_0$ is the target
of a hidden $\tau$ extension, since it detects a multiple of $\tau$.
The element $\tau^2 M h_1 g^2$
is the only possible source of this hidden $\tau$ extension,
so it cannot be hit by a differential.
This shows that $d_6(g Q_3)$ cannot equal $\tau M h_1 g^2$.
\end{proof}

\begin{lemma}
\label{lem:d6-h2^2H1}
\mbox{}
\begin{enumerate}
\item
\revdeg{68, 7, 37}
$d_6(h_2^2 H_1) = M c_0 d_0$.
\item
\revdeg{68, 7, 36}
$d_7(\tau h_2^2 H_1) = M P d_0$.
\end{enumerate}
\end{lemma}

\begin{proof}
Table \ref{tab:Massey} shows that 
$M P d_0$ equals the Massey product $\langle P d_0, h_0^3, g_2 \rangle$.
This implies that
$M P d_0$ detects the
Toda bracket $\langle \tau \eta^2 \kappabar, 8, \kappabar_2 \rangle$.
Lemma \ref{lem:teta^2kappabar,8,kappabar2} shows that this
Toda bracket consists entirely of multiples of
$\tau \eta^2 \kappabar$.

We now know that $M P d_0$ detects a multiple of 
$\tau \eta^2 \kappabar$.
The only possibility is that 
$M P d_0$ detects $\eta$ times an element detected by
$\tau^2 M h_1 g$.

We will show in Lemma \ref{lem:nu-th1H1}
that $\tau^2 M h_1 g$ is the target of a $\nu$ extension,
so $\tau^2 M h_1 g$ cannot support a hidden $\eta$ extension.
Therefore, $M P d_0$ must be hit by some differential.
The only possibility is that 
$d_7(\tau h_2^2 H_1)$ equals $M P d_0$.
Then $h_2^2 H_1$ cannot survive to the $E_7$-page,
so $d_6(h_2^2 H_1)$ equals $M c_0 d_0$.
\end{proof}

\begin{lemma}
\label{lem:perm-th1p1}
\revdeg{71, 5, 37}
The element $\tau h_1 p_1$ is a permanent cycle.
\end{lemma}

\begin{proof}
Lemma \ref{lem:d5-tp1+h0^2h3h6}, together with results of
\cite{BIX}, show that $\tau h_1 p_1$ survives to the $E_6$-page.
We must eliminate possible higher differentials.

Table \ref{tab:tau-extn} shows that there is a hidden
$\tau$ extension from $\tau h_2^2 C''$ to $\D^2 h_1^2 h_4 c_0$.
This means that $\tau h_2 C'' + h_1 h_3(\D e_1 + C_0)$ must 
also support a hidden $\tau$ extension.

The two possible targets for this hidden $\tau$ extension
are $\D^2 h_2 c_1$ and $\tau \D^2 h_1^2 g + \tau^3 \D h_2^2 g^2$.
The second possibility is ruled out by comparison to $\tmf$, so
$\D^2 h_2 c_1$ cannot be hit by a differential.
\end{proof}

\begin{lemma}
\label{lem:perm-Ph0h2h6}
\revdeg{74, 7, 38}
The element $P h_0 h_2 h_6$ is a permanent cycle.
\end{lemma}

\begin{proof}
First note that projection from $C\tau$ to the top cell
takes $P h_2 h_6$ to a non-zero element.
If $P h_0 h_2 h_6$ were not a permanent cycle in the
Adams spectral sequence for the sphere,
then projection from $C\tau$ to the top cell would
also take $P h_0 h_2 h_6$ to a non-zero element.
Then the $2$ extension from $P h_2 h_6$ to $P h_0 h_2 h_6$
in $\pi_{74,38} C\tau$ would project to a $2$ extension
in $\pi_{73,39}$.  However, there are no possible
$2$ extensions in $\pi_{73,39}$.
\end{proof}

\begin{lemma}
\label{lem:d7-m1}
\revdeg{77, 7, 42}
$d_7(m_1) = 0$.
\end{lemma}

\begin{proof}
The only other possibility is that $d_7(m_1)$ equals $\tau^2 g^2 t$.
If that were the case, then the $\nu$ extension from
$\tau g^2 t$ to $\tau^2 c_1 g^3$ would be detected by projection
from $C\tau$ to the top cell.  However, the homotopy groups of
$C\tau$ have no such $\nu$ extension.
\end{proof}

\begin{lemma}
\label{lem:d8-h1x1}
\revdeg{80, 6, 43}
$d_8(h_1 x_1) = 0$.
\end{lemma}

\begin{proof}
Table \ref{tab:Adams-perm} shows that $\tau h_1 x_1$ is a permanent
cycle.
Then $d_8(\tau h_1 x_1)$ cannot equal 
$\tau^2 M e_0^2$, and $d_8(h_1 x_1)$ cannot equal
$\tau M e_0^2$.
\end{proof}

\begin{lemma}
\label{lem:d6-h2h4h6}
\revdeg{81, 3, 42}
$d_6(h_2 h_4 h_6) = 0$.
\end{lemma}

\begin{proof}
Table \ref{tab:Adams-perm} shows that $h_2^2 h_4 h_6$ is a permanent
cycle.  Therefore, the Adams differential
$d_6(h_2^2 h_4 h_6)$ does not equal
$\tau h_2 c_1 A'$, and $d_6(h_2 h_4 h_6)$ does not equal
$\tau c_1 A'$.
\end{proof}

\rev{
\begin{lemma}
\label{lem:d10-h2h6g+h1^2f2}
\revdeg{86, 6, 46}
$d_{10}(h_2 h_6 g + h_1^2 f_2) = 0$.
\end{lemma}

\begin{proof}
If $d_{10}(h_2 h_6 g + h_1^2 f_2)$ equaled
$M \D h_1^2 d_0$, then the $\nu$ extension from
$\tau h_6 g + \tau h_2 e_2$ to $\tau h_2 h_6 g + \tau h_1^2 f_2$
would be detected by the bottom cell of $C\tau$.
However, there is no such $\nu$ extension in the homotopy 
of $C\tau$.
\end{proof}
}

\begin{lemma}
\label{lem:d7-x87,7}
\revdeg{87, 7, 45}
$d_7(x_{87,7}) = 0$.
\end{lemma}

\begin{proof}
If $\tau \D^2 h_2^2 d_1$ were hit by a differential, then
the $\nu$ extension from $\D^2 h_2^2 d_1$ to $\D^2 h_1^2 h_3 d_1$
would be detected by projection from $C\tau$ to the top cell.
But the homotopy of $C\tau$ has no such $\nu$ extension.
\end{proof}

\rev{
\begin{lemma}
\label{lem:d6-tDh1H1}
\revdeg{87, 10, 45}
$d_6(\tau \D h_1 H_1) = \tau M \D h_0^2 e_0$.
\end{lemma}

\begin{proof}
Suppose for sake of contradiction that
$\tau \D h_1 H_1$ survives.
This element cannot be hit, nor can it be the target of a hidden
$\tau$ extension.  Therefore, it would have non-zero image 
under inclusion of the bottom cell into $C\tau$, and it would
map to $\ol{\D h_1 B_7}$.

In the homotopy of $C\tau$, the Toda bracket
$\langle \ol{\D h_1 B_7}, h_0, h_2^2 \rangle$
is detected by the element $M \D^2 h_1$.
Beware that this bracket has indeterminacy in lower filtration
since $h_3 \cdot \ol{\D h_1 B_7} = \ol{h_1^2 x_{91,11}}$.

This implies that the Toda bracket
$\langle \{\tau \D h_1 H_1\}, 2, \nu^2 \rangle$
would be non-zero in $\pi_{94,49}$, and all of its elements 
would be detected in Adams filtration at most $15$.
(Beware that this Toda bracket would have indeterminacy
detected by $\tau \D h_1 h_3 H_1$.)

On the other hand, the Moss Convergence Theorem \ref{thm:Moss}
would imply that the Toda bracket is detected in filtration
at least $12$.
However, there are no possible elements in filtrations $12$ through $15$.

We have now shown that $\tau \D h_1 H_1$ cannot survive.
There is only one possible value for a differential on
$\tau \D h_1 H_1$.

The previous argument assumed that $2 \{\tau \D h_1 H_1\}$ is zero
in order to form the Toda bracket
$\langle \{\tau \D h_1 H_1\}, 2, \nu^2 \rangle$.
However, it is possible that $\tau \D h_1 H_1$ supports a hidden
$2$ extension to $\tau \D^2 h_3 d_1$ or to $\tau^2 \D^2 c_1 g$.
Therefore, $2 \{\tau \D h_1 H_1\}$ might equal
$\tau \sigma \{\D^2 d_1\}$, $\tau \nu \{\D^2 t\}$, or their sum.
In those cases, we would need to consider the matric Toda brackets
\[
\left\langle 
\left[
\begin{array}{cc}
\{ \tau \D h_1 H_1\} & \tau \{\D^2 d_1\}
\end{array}
\right],
\left[
\begin{array}{c}
2 \\ \sigma
\end{array}
\right],
\nu^2
\right\rangle,
\]
\[
\left\langle 
\left[
\begin{array}{cc}
\{ \tau \D h_1 H_1\} & \nu
\end{array}
\right],
\left[
\begin{array}{c}
2 \\ \tau \{\D^2 t\}
\end{array}
\right],
\nu^2
\right\rangle,
\]
or
\[
\left\langle 
\left[
\begin{array}{ccc}
\{ \tau \D h_1 H_1\} & \tau \{\D^2 d_1\} & \nu
\end{array}
\right],
\left[
\begin{array}{c}
2 \\ \sigma \\ \tau \{\D^2 t \}
\end{array}
\right],
\nu^2
\right\rangle
\]
respectively.
Under inclusion of the bottom cell into $C\tau$, these three brackets
would map to
\[
\left\langle 
\left[
\begin{array}{cc}
\ol{\D h_1 B_7} & 0
\end{array}
\right],
\left[
\begin{array}{c}
h_0 \\ h_3
\end{array}
\right],
h_2^2
\right\rangle,
\]
\[
\left\langle 
\left[
\begin{array}{cc}
\ol{\D h_1 B_7} & h_2
\end{array}
\right],
\left[
\begin{array}{c}
h_0 \\ 0
\end{array}
\right],
h_2^2
\right\rangle,
\]
or
\[
\left\langle 
\left[
\begin{array}{ccc}
\ol{\D h_1 B_7} & 0 & h_2
\end{array}
\right],
\left[
\begin{array}{c}
h_0 \\ h_3 \\ 0
\end{array}
\right],
h_2^2
\right\rangle
\]
respectively.
All three of these brackets in $C\tau$ equal
$\langle \ol{\D h_1 B_7}, h_0, h_2^2 \rangle$.
Beware that the last two could have larger indeterminacy,
but in fact do not.
\end{proof}
}

\begin{lemma}
\label{lem:perm-x88,10}
\revdeg{88, 10, 48}
The element $x_{88,10}$ is a permanent cycle.
\end{lemma}

\begin{proof}
In the Adams spectral sequence for $C\tau$, there is a hidden
$\eta$ extension from $h_1^2 x_{85,6}$ to $x_{88,10}$.
Therefore, $x_{88,10}$ lies in the image of inclusion of the
bottom cell into $C\tau$.
The only possible pre-image is the element $x_{88,10}$ in the
Adams spectral sequence in the sphere, so $x_{88,10}$ must survive.
\end{proof}

\begin{lemma}
\label{lem:d6-h2^2gH1}
\mbox{}
\begin{enumerate}
\item
\revdeg{88, 11, 49}
$d_6(h_2^2 g H_1) = M c_0 e_0^2$.
\item
\revdeg{88, 11, 48}
$d_7(\tau h_2^2 g H_1) = 0$.
\end{enumerate}
\end{lemma}

\begin{proof}
If $M c_0 e_0^2$ is non-zero in the $E_\infty$-page, then it detects
an element that is annihilated by $\tau$ because Lemma \ref{lem:d6-Dg2g}
shows that the only possible target of such an extension is hit
by a differential.
Then $M c_0 e_0^2$ would be in the image of projection 
from $C\tau$ to the top cell.
The only possible pre-image would be the element $\D g_2 g$ of the Adams
spectral sequence for $C\tau$.

In the Adams spectral sequence for $C\tau$, there is a $\sigma$
extension from $g A'$ to $\D g_2 g$.  
Projection from $C\tau$ to the top cell would imply that there
is a hidden $\sigma$ extension in the homotopy groups of the sphere,
from $M h_1 e_0^2$ to $M c_0 e_0^2$, because
$g A'$ maps to $M h_1 e_0^2$ under projection from $C\tau$
to the top cell.

But $M h_1 e_0^2$ detects $\eta \theta_{4.5} \{ e_0^2 \}$,
which cannot support a $\sigma$ extension.  This establishes the
first formula.

For the second formula, 
if $d_7(\tau h_2^2 g H_1)$ were equal to $\tau^2 \D h_2^2 e_0 g^2$,
then the same argument would apply,
with $\tau \D h_2^2 e_0 g^2$ substituted for $M c_0 e_0^2$.
\end{proof}

\begin{lemma}
\label{lem:d6-Dg2g}
\revdeg{88, 12, 48}
$d_6(\D g_2 g) = M d_0^3$.
\end{lemma}

\begin{proof}
The proof of Lemma \ref{lem:d4-th2B5g} shows that
$M d_0^3$ must be hit by a differential.
The only possibility is that $d_6(\D g_2 g)$ equals
$M d_0^3$.

Alternatively, Lemma \ref{lem:d6-h2^2H1} shows that
$d_7(\tau h_2^2 H_1) = M P d_0$.
Note that $\tau g \cdot \tau h_2^2 H_1 = 0$ in the $E_7$-page.
Therefore, $\tau M d_0^3 = \tau g \cdot M P d_0$ must already
be zero in the $E_7$-page.  The only possibility is that
$d_6(\tau \D g_2 g) = \tau M d_0^3$, and then
$d_6(\D g_2 g) = M d_0^3$.
\end{proof}

\rev{
\begin{remark}
\label{rem:d6-D^2f1}
Table \ref{tab:Adams-higher} shows that
$d_6(\D^2 f_1)$ equals $\tau^2 M d_0^3$.
The proof relies on $d_6(\tau \D h_1 H_1) = \tau M \D h_0^2 e_0$
and uses techniques similar to the ones in \cite{Chua21}.
\end{remark}
}

\begin{lemma}
\label{lem:perm-h0g3}
\revdeg{92, 5, 48}
The element $h_0 g_3$ is a permanent cycle.
\end{lemma}

\begin{proof}
In the homotopy of $C\tau$, the product $\theta_4 \theta_5$ is detected
by $h_0^2 g_3$.  In the sphere, the product $\theta_4 \theta_5$ 
is therefore non-zero and detected in Adams filtration at most 6.

Table \ref{tab:Toda} shows that 
the Toda bracket $\langle 2, \theta_4, \theta_4, 2 \rangle$
contains $\theta_5$.
Therefore, the product $\theta_4 \theta_5$ is contained in
\[
\theta_4 \langle 2, \theta_4, \theta_4, 2 \rangle =
\langle \theta_4, 2, \theta_4, \theta_4 \rangle 2.
\]
(Note that the sub-bracket $\langle \theta_4, \theta_4, 2 \rangle$
is zero because $\pi_{61,32}$ is zero.)
Therefore, $\theta_4 \theta_5$ is divisible by $2$.
It follows that
$\theta_4 \theta_5$ is detected by $h_0^2 g_3$, and
$h_0 g_3$ is a permanent cycle that detects
$\langle \theta_4, 2, \theta_4, \theta_4 \rangle$.
\end{proof}

\begin{lemma}
\label{lem:d6-D1h1^2e1}
\revdeg{92, 10, 51}
$d_6(\D_1 h_1^2 e_1) = 0$.
\end{lemma}

\begin{proof}
Consider the element $\ol{\tau M h_2^2 g^2}$
in the Adams spectral sequence for $C\tau$.
This element cannot be in the image of inclusion of the bottom
cell into $C\tau$.  Therefore,
it must map non-trivially under projection from $C\tau$
to the top cell.
The only possibility is that
$\tau M h_2^2 g^2$ is the image.
Therefore, $\tau M h_2^2 g^2$ cannot be the target of a differential.
\end{proof}

\begin{lemma}
\label{lem:d7-x92,10}
\revdeg{92, 10, 48}
$d_7(x_{92,10})$ does not equal $\tau^2 \D^2 h_2 g^2$.
\end{lemma}

\begin{proof}
If $\tau^2 \D^2 h_2 g^2$ were hit by a differential, then
the $2$ extension from $\tau \D^2 h_2 g^2$ to $\tau \D^2 h_0 h_2 g^2$
would be detected by projection from $C\tau$ to the top cell.
But the homotopy of $C\tau$ has no such $2$ extension.
\end{proof}

\begin{lemma}
\label{lem:d8-D1h2e1}
\revdeg{92, 10, 51}
$d_8(\D_1 h_2 e_1) = 0$.
\end{lemma}

\begin{proof}
Consider the element $e_0 x_{76,9}$ in the
Adams $E_\infty$-page for $C\tau$.
It cannot be in the image of inclusion of the bottom cell
into $C\tau$, so it must project to a non-zero element in the
top cell.  The only possible image is $M \D h_1^3 g$.
Therefore, $M \D h_1^3 g$ cannot be the target of a differential.
\end{proof}

\begin{lemma}
\label{lem:hit-MPDh1d0}
The element $M P \D h_1 d_0$ is hit by some differential.
\end{lemma}

\begin{proof}
Table \ref{tab:tau-extn} shows that
there is a hidden $\tau$ extension from
$\D h_1 c_0 d_0$ to $P \D h_1 d_0$.  Therefore,
$P \D h_1 d_0$ detects $\tau \epsilon \{\D h_1 d_0\}$.
On the other hand, Tables \ref{tab:eta-extn} and
\ref{tab:misc-extn} show that
$P \D h_1 d_0$ also detects $\tau \eta \kappa \{\D h_1 h_3\}$.
Since there are no elements in higher Adams filtration, we 
have that $\tau \epsilon \{\D h_1 d_0\}$ equals
$\tau \eta \kappa \{\D h_1 h_3\}$.

Table \ref{tab:misc-extn} shows that
$M P$ detects $\tau \epsilon \theta_{4.5}$,
so $M P \D h_1 d_0$ detects 
$\tau \epsilon \{\D h_1 d_0\} \theta_{4.5}$, 
which equals 
$\tau \eta \kappa \{\D h_1 h_3\} \theta_{4.5}$.
But $\tau \eta \kappa \theta_{4.5}$ is zero because
all elements of $\pi_{60,32}$ are detected by $\tmf$.
This shows that $M P \D h_1 d_0$ detects zero, so it
must be hit by a differential.
\end{proof}

\rev{
\begin{remark}
\label{rem:d6-te0x76,9}
Lemma \ref{lem:hit-MPDh1d0} does not specify the differential that
hits the element $M P \D h_1 d_0$.  In fact, $d_6(\tau e_0 x_{76,9})$ equals
$M P \D h_1 d_0$ \cite{BIX}.
\end{remark}
}

%% file: more-stable-stems-Toda.tex
\chapter{Toda brackets}
\label{ch:Toda}

The purpose of this chapter is to establish various Toda brackets
that are used elsewhere in this manuscript.
\rev{Tables \ref{tab:Toda} and \ref{tab:Toda-null} collect
all of this information in one place.}
Many Toda brackets can be easily computed from
the Moss Convergence Theorem \ref{thm:Moss}.
These are summarized in the tables without further
discussion.  
However, some brackets require more complicated arguments.
Those arguments are collected in this chapter.
\revv{
For easy reference, the lemmas in this chapter are labelled with degrees that match the degrees given in the tables.}

We will need the
following 
$\C$-motivic version of a 
theorem of Toda \cite{Toda62}*{Theorem 3.6} that applies to
symmetric Toda brackets.

\begin{thm}
\label{thm:Toda-symmetric}
Let $\alpha$ be an element of $\pi_{s,w}$, with $s$ even.
There exists an element $\alpha^*$ in $\pi_{2s+1,2w}$ such that
$\langle \alpha, \beta, \alpha \rangle$ contains
the product $\beta \alpha^*$ for all $\beta$ such that
$\alpha \beta$ \revv{equals zero}.
\end{thm}

\begin{cor}
\label{cor:2-symmetric}
If $2 \beta = 0$, then $\langle 2, \beta, 2 \rangle$
contains $\tau \eta \beta$.
\end{cor}

\begin{proof}
Apply Theorem \ref{thm:Toda-symmetric} to $\alpha = 2$.
We need to find the value of $\alpha^*$. 
Table \ref{tab:Massey} shows that
the Massey product $\langle h_0, h_1, h_0 \rangle$
equals $\tau h_1^2$.
The Moss Convergence Theorem \ref{thm:Moss} then shows that 
$\langle 2, \eta, 2 \rangle$ equals $\tau \eta^2$.
It follows that $\alpha^*$ equals $\tau \eta$.
\end{proof}

\begin{thm}
\label{thm:Toda}
Tables \ref{tab:Toda} \rev{and \ref{tab:Toda-null}} list 
some Toda brackets in the 
$\C$-motivic stable homotopy groups.
\end{thm}

\begin{proof}
The fourth column of the table gives information about the proof
of each Toda bracket.  

If the fourth column shows a Massey product, then the 
Toda bracket follows from the Moss Convergence Theorem \ref{thm:Moss}.
If the fourth column shows an Adams differential, then
the Toda bracket follows from the Moss Convergence Theorem \ref{thm:Moss},
using the mentioned differential.

A few Toda brackets are established elsewhere in the literature; 
specific citations are given in these cases.

Additional more difficult cases are established in the following lemmas.
\end{proof}

Tables \ref{tab:Toda} lists information about some Toda brackets
\rev{that do not contain zero,
while Table \ref{tab:Toda-null} lists information about some Toda brackets
that do contain zero.}
The third columns of the tables give elements of the
Adams $E_\infty$-page that detect elements of the Toda brackets.
The fourth columns of the tables give partial
information about indeterminacies, again by giving detecting elements
of the Adams $E_\infty$-page.  We have not completely analyzed the
indeterminacies of all brackets when the details are inconsequential
for our purposes.  The fifth columns indicate the proofs of 
the Toda brackets, and the sixth columns shows where each specific
Toda bracket is used in the manuscript.

\begin{lemma}
\label{lem:kappa,2,eta}
\revdeg{16, 9}
The Toda bracket $\langle \kappa, 2, \eta \rangle$ contains
zero, with indeterminacy generated by $\eta \rho_{15}$.
\end{lemma}

\begin{proof}
Using the Adams differential $d_3(h_0 h_4) = h_0 d_0$, 
the Moss Convergence Theorem \ref{thm:Moss} shows that the
Toda bracket is detected in filtration at least $3$.  The
only element in sufficiently high filtration is $P c_0$, 
which detects the product $\eta \rho_{15}$.  This product
lies in the indeterminacy, so the bracket must contain
zero.
\end{proof}

\begin{lemma}
\label{lem:kappa,2,eta,nu}
\revdeg{20, 11}
The Toda bracket $\langle \kappa, 2, \eta, \nu \rangle$
is detected by $\tau g$.
\end{lemma}

\begin{proof}
The subbracket $\langle 2, \eta, \nu \rangle$ is strictly zero,
since $\pi_{5,3}$ is zero.  
The subbracket 
$\langle \kappa, 2, \eta \rangle$ contains zero by Lemma 
\ref{lem:kappa,2,eta}.
Therefore, the fourfold bracket
$\langle \kappa, 2, \eta, \nu \rangle$
is well-defined.

Shuffle to obtain
\[
\langle \kappa, 2, \eta, \nu \rangle \eta^2 = 
\kappa \langle 2, \eta, \nu, \eta^2 \rangle.
\]
Table \ref{tab:Toda} shows that $\epsilon$ is contained
in the Toda bracket $\langle \eta^2, \nu, \eta, 2 \rangle$, so
the latter expression equals $\epsilon \kappa$,
which is detected by $c_0 d_0$.
It follows that 
$\langle \kappa, 2, \eta, \nu \rangle$
must be detected by $\tau g$.
\end{proof}

\begin{lemma}
\label{lem:nubar,sigma,2sigma}
\revdeg{23, 13}
The Toda bracket
$\langle \epsilon + \eta \sigma, \sigma, 2 \sigma \rangle$
contains zero, with indeterminacy generated
by $4 \nu \kappabar$ in $\{ P h_1 d_0 \}$.
\end{lemma}

\begin{proof}
Consider the shuffle
\[
\langle \epsilon + \eta \sigma, \sigma, 2 \sigma \rangle \eta =
(\epsilon + \eta \sigma) \langle \sigma, 2 \sigma, \eta \rangle.
\]
Table \ref{tab:Toda} shows that $h_1 h_4$ detects
$\langle \sigma, 2 \sigma, \eta \rangle$, so
$h_1 h_4 c_0$ detects the product
$\epsilon \langle \sigma, 2 \sigma, \eta \rangle$.
On the other hand, Table \ref{tab:misc-extn} shows that
\revv{there is a hidden $\sigma$ extension from $h_1 h_4$ to $h_4 c_0$.
Therefore,}
$h_1 h_4 c_0$ also detects 
$\eta \sigma \langle \sigma, 2 \sigma, \eta \rangle$.
It follows that
$(\epsilon + \eta \sigma) \langle \sigma, 2 \sigma, \eta \rangle$
is detected in filtration greater than $5$.

\revv{Consider the shuffle
\[
2 \langle \epsilon + \eta \sigma, \sigma, 2 \sigma \rangle =
\langle 2, \epsilon + \eta \sigma, \sigma \rangle 2 \sigma.
\]
The latter expression is zero since $2 \sigma$ annihilates
all elements of $\pi_{16,9}$.
}
This shows that no elements of the Toda bracket
can be detected by $\tau h_2 g$ or $\tau h_0 h_2 g$.

The element $4 \nu \kappabar$ generates the indeterminacy
because it equals $\tau \eta \kappa (\epsilon + \eta \sigma)$.
\end{proof}

\begin{lemma}
\label{lem:tkappa^2,eta,2}
\revdeg{30, 16}
The Toda bracket $\langle \tau \kappa^2, \eta, 2 \rangle$
equals zero, with no indeterminacy.
\end{lemma}

\begin{proof}
The Adams differential $d_3(\D h_2^2) = \tau h_1 d_0^2$
implies that the bracket is detected by $h_0 \cdot \D h_2^2$,
which equals zero \revv{in the $E_\infty$-page}.  
Therefore, the Toda bracket is detected 
in Adams filtration at least $7$, but there are no 
elements in the Adams $E_\infty$-page in sufficiently high
filtration.

The indeterminacy can be computed by inspection.
\end{proof}

\begin{lemma}
\label{lem:eta^2,theta4,eta^2}
\revdeg{35, 20}
The Toda bracket $\langle \eta^2, \theta_4, \eta^2 \rangle$
contains zero, with indeterminacy generated by
$\eta^3 \eta_5$.
\end{lemma}

\begin{proof}
If the bracket were detected by $h_2 d_1$, then
\[
\nu \langle \eta^2, \theta_4, \eta^2 \rangle =
\langle \nu, \eta^2, \theta_4 \rangle \eta^2
\]
would be detected by $h_2^2 d_1$.
However, $h_2^2 d_1$ does not detect a multiple of $\eta^2$.

The bracket cannot be detected by $\tau h_1 e_0^2$ by 
comparison to $\tmf$.

By inspection, the only remaining possibility is that
the bracket contains zero.  The indeterminacy can be computed
by inspection.
\end{proof}

\begin{lemma}
\label{lem:t,eta^2kappa1,eta}
\revdeg{36, 20}
The Toda bracket $\langle \tau, \eta^2 \kappa_1, \eta \rangle$
is detected by $t$, with indeterminacy generated by
$\eta^3 \mu_{33}$.
\end{lemma}

\begin{proof}
There is a relation
$h_1 \cdot \ol{h_1^2 d_1} = t$ in the homotopy of $C\tau$.
Using the connection between Toda brackets and cofibers
as described in \cite{Isaksen14c}*{Section 3.1.1},
this shows that $t$ detects the Toda bracket.

The indeterminacy is computed by inspection.
\end{proof}

\rev{
The third author presents the following Lemma~\ref{lem:theta4,2,sigma^2+kappa} as a correction to \cite{Xu16}*{Theorem 2.1}, where it states that the said Toda bracket contains 0.

\begin{lemma}
\label{lem:theta4,2,sigma^2+kappa}
\revdeg{45}
The classical Toda bracket $\langle \theta_4, 2, \sigma^2 + \kappa \rangle$
contains $0$ or $\eta \kappabar_2$.  Its indeterminacy
is generated by $\rho_{15} \theta_4$, which is detected by $h_0^2 h_5 d_0$.
\end{lemma}

\begin{proof}
The gap originated in \cite{Xu16}*{Remark~3.3},
where it was claimed that $\langle \theta_4, 2, \sigma^2 \rangle$ contains an order 2 element of the form $2 \alpha + \beta$, where $\alpha$ is detected by $h_4^3$ and $\beta$ is detected by $h_5 d_0$. In fact, since $h_5 d_0$ and $h_1 g_2$ are in the same filtration, we can only conclude that $\beta$ is detected by $h_5 d_0$ or $h_5 d_0 + h_1 g_2$, therefore the missed possibility in the statement of the lemma.
\end{proof}
}

\rev{
\begin{remark}
\label{rem:theta4,2,sigma^2+kappa}
In fact, we have evidence that this classical Toda bracket\\
$\langle \theta_4, 2, \sigma^2 + \kappa \rangle$ contains $\eta \kappabar_2$.  However, the argument depends on computations
as far as the 110-stem.
\end{remark}
}

\begin{lemma}
\label{lem:eta,2,4kappabar2}
\revdeg{46, 25}
The Toda bracket
$\langle \eta, 2, 4 \kappabar_2 \rangle$
contains zero.
\end{lemma}

\begin{proof}
The Massey product $M h_1 = \langle h_1, h_0, h_0^2 g_2 \rangle$
shows that $M h_1$ detects the Toda bracket.
Table \ref{tab:eta-extn} shows that
$M h_1$, $\D h_2 c_1$, and $\tau d_0 l + \D c_0 d_0$ are all
targets of hidden $\eta$ extensions.
(Beware that the
hidden $\eta$ extension from $h_3^2 h_5$ to $M h_1$ is a crossing
extension in the sense of Section \ref{sctn:assoc-graded},
but that does not matter.)
Therefore, $M h_1$ detects only multiples of $\eta$,
so the Toda bracket contains a multiple of $\eta$.
This implies that it contains zero, since multiples of $\eta$
belong to the indeterminacy.
\end{proof}

\begin{lemma}
\label{lem:tkappabar2,sigma^2,2}
\revdeg{59, 31}
The Toda bracket $\langle \tau \kappabar_2, \sigma^2, 2 \rangle$
equals zero.
\end{lemma}

\begin{proof}
No elements of the bracket can
be detected by $\tau^2 \D h_1 d_0 g$ by comparison to $\tmf$.

Consider the shuffle
\[
\langle \tau \kappabar_2, \sigma^2, 2 \rangle \kappa =
\tau \kappabar_2 \langle \sigma^2, 2, \kappa \rangle.
\]
The bracket $\langle \sigma^2, 2, \kappa \rangle$ is zero
because it is contained in $\pi_{29,16} = 0$.
On the other hand,
$\{\tau M d_0\} \kappa$ is non-zero and detected by
$\tau M d_0^2$.
Therefore, no elements of
$\langle \tau \kappabar_2, \sigma^2, 2 \rangle$
can be can be detected by $\tau M d_0$.
\end{proof}

\rev{
The third author presents the following Lemma~\ref{lem:etakappabar2,2sigma,sigma}, which is needed in the proof of Lemma~\ref{lem:theta4-h4^2}.

\begin{lemma}
\label{lem:etakappabar2,2sigma,sigma}
\revdeg{60}
The classical Toda bracket $\langle \eta \kappabar_2, 2 \sigma, \sigma \rangle$
equals zero.
\end{lemma}

\begin{proof}
Due to $d_3(e_1) = h_2^2 n$ and that $h_1g_2 = h_3e_1$, we have the following Massey product in the $E_4$-page
$$h_1g_2 = \langle h_2, h_2n, h_3 \rangle$$
with zero indeterminacy. Since there are no crossing differentials, we conclude that $h_1g_2$ detects a homotopy class in the Toda bracket $\langle \nu, \nu\{n\}, \sigma \rangle$. We claim that $\eta \kappabar_2$ is contained in this bracket. In fact, they might differ by classes detected in filtration 6 or higher: $h_0h_5d_0, \ h_0^2 h_5d_0, \ w$. The first two detect $\sigma$-multiples so they are in the indeterminacy. The homotopy class $\{w\}$ is detected by tmf, but $\kappabar_2$ and $\{n\}$ are not, so $w$ can be ruled out too. 

Therefore, we have 
$$
\langle \eta \kappabar_2, 2\sigma, \sigma \rangle \subseteq \langle \langle \nu, \nu\{n\}, \sigma \rangle, 2\sigma, \sigma \rangle
\supseteq \nu \langle \nu\{n\}, \sigma, 2\sigma, \sigma \rangle = 0.
$$
Here $\langle \sigma, 2\sigma, \sigma \rangle = 0$, and $ \langle \nu\{n\}, \sigma, 2\sigma \rangle$ contains 0 since coker J in $\pi_{49}$ is 0. So the 4-fold bracket $\langle \nu\{n\}, \sigma, 2\sigma, \sigma \rangle$ in $\pi_{57}$ is well-defined. By comparison with $\pi_{57}tmf$, we know it is 0.

We remain to show the indeterminacy of $\langle \langle \nu, \nu\{n\}, \sigma \rangle, 2\sigma, \sigma \rangle$ is 0. In fact,
\begin{itemize}
\item
$\pi_{15} \cdot \langle \nu, \nu \{n\}, \sigma \rangle = \langle \pi_{15},  \nu, \nu \{n\} \rangle \sigma  \subseteq \sigma \pi_{53} = 0$.  (Lemma~2.3 in \cite{Xu16}.)

\item $\langle \nu \cdot \pi_{42}, 2\sigma, \sigma \rangle = \langle 0, 2\sigma, \sigma \rangle + \langle \rho_{15}\theta_4, 2\sigma, \sigma \rangle = 0$.  (Lemmas~2.3 and 2.4 in \cite{Xu16}.)

\item $\langle \sigma \cdot \pi_{38}, 2\sigma, \sigma \rangle \supseteq \pi_{38} \cdot \langle \sigma, 2\sigma, \sigma \rangle = 0$.
\end{itemize}
This completes the proof.
\end{proof}
}

\begin{lemma}
\label{lem:2,sigma^2,theta4.5}
\revdeg{60, 32}
For every $\alpha$ that is detected by $h_3^2 h_5$,
the Toda bracket $\langle 2, \sigma^2, \alpha \rangle$
contains zero.  The indeterminacy is generated by 
$2 \tau \kappabar^3$, which is detected by $\tau^2 d_0^2 l$.
\end{lemma}

\begin{proof}
\revv{
Let $\alpha$ be detected by $h_3^2 h_5$.
For degree reasons,
the only elements that could detect $\sigma^2 \alpha$ either support
$\eta$ extensions or are detected by $\tmf$.
Therefore, $\sigma^2 \alpha$ is zero.
}
Hence the bracket is defined.

By comparison to $\tmf$, the bracket cannot be detected
by $\tau^4 g^3$. 
Table \ref{tab:2-extn} shows that
$\tau^2 d_0^2 l$ is the target of a hidden $2$ extension, so it
detects an element in the indeterminacy.
Since there are no other possibilities, the bracket must
contain zero.
\end{proof}

\begin{remark}
This result is consistent with Table 23 of \cite{Isaksen14c},
which claims that the bracket
$\langle 2, \sigma^2, \theta_{4.5} \rangle$
contains an element that is detected by $B_3$.
The element $B_3$ is now known to be zero in the Adams $E_\infty$-page,
so this just means that the bracket contains an element detected
in Adams filtration strictly greater than the filtration of $B_3$.
\end{remark}

\begin{lemma}
\label{lem:theta4,eta^2,theta4}
deg{63, 34}
The Toda bracket $\langle \theta_4, \eta^2, \theta_4 \rangle$
equals zero.
\end{lemma}

\begin{proof}
Theorem \ref{thm:Toda-symmetric} says that there exists an element
$\theta_4^*$ in $\pi_{61,32}$ such that
$\langle \theta_4, \eta^2, \theta_4 \rangle$ contains
$\eta^2 \theta_4^*$.
The group $\pi_{61,32}$ is zero, so $\theta_4^*$ must be zero,
and the bracket must contain zero.

In order to compute the indeterminacy of 
$\langle \theta_4, \eta^2, \theta_4 \rangle$, we must consider
the product of $\theta_4$ with elements of $\pi_{33,18}$.
There are several cases to consider.

First consider $\{\D h_1^2 h_3\}$.  The product
$\theta_4 \{\D h_1^2 h_3\}$ is detected in Adams filtration at least
$10$, but there are no elements in sufficiently high filtration.

Next consider $\nu \theta_4$ detected by $p$.
The product $\theta_4^2$ is zero \cite{Xu16},
so $\nu \theta_4^2$ is also zero.

Finally, consider $\eta \eta_5$ detected by $h_1^2 h_5$.
Table \ref{tab:Toda} shows that $\langle \eta, 2, \theta_4 \rangle$
detects $\eta_5$.
Shuffle to obtain
\[
\eta \eta_5 \theta_4 = \eta \langle \eta, 2, \theta_4 \rangle \theta_4 =
\eta^2 \langle 2, \theta_4, \theta_4 \rangle.
\]
The bracket $\langle 2, \theta_4, \theta_4 \rangle$ is zero
because it is contained in $\pi_{61,32} = 0$.
\end{proof}

\begin{lemma}
\label{lem:eta^2,theta4,eta^2,theta4}
\revdeg{66, 36}
The Toda bracket $\langle \eta^2, \theta_4, \eta^2, \theta_4 \rangle$
is detected by $\D_1 h_3^2$.
\end{lemma}

\begin{proof}
Table \ref{tab:Massey} shows that 
$\D_1 h_3^2$ equals $\langle h_1^2, h_4^2, h_1^2, h_4^2 \rangle$.
Therefore, $\D_1 h_3^2$ detects 
$\langle \eta^2, \theta_4, \eta^2, \theta_4 \rangle$,
if the Toda bracket is well-defined.

In order to show that the Toda bracket is well-defined, we 
need to know that the 
subbrackets $\langle \eta^2, \theta_4, \eta^2 \rangle$
and $\langle \theta_4, \eta^2, \theta_4 \rangle$ contain zero.
These are handled by
Lemmas \ref{lem:eta^2,theta4,eta^2} and \ref{lem:theta4,eta^2,theta4}.
\end{proof}

\begin{lemma}
\label{lem:teta^2kappabar,8,kappabar2}
\revdeg{67, 36}
The Toda bracket 
$\langle \tau \eta^2 \kappabar, 8, \kappabar_2 \rangle$
contains zero,
and its indeterminacy is generated by multiples of
$\tau \eta^2 \kappabar$.
\end{lemma}

\begin{proof}
The bracket
$\langle \tau \eta^2 \kappabar, 8, \kappabar_2 \rangle$
contains
$\tau \eta \kappabar \langle \eta, 2, 4 \kappabar_2 \rangle$.
Lemma \ref{lem:eta,2,4kappabar2} shows that
this expression contains zero.

It remains to show that $\kappabar_2 \cdot \pi_{23, 12}$ equals zero.
There are several cases to consider.

First, the product $\tau \sigma \eta_4 \kappabar_2$ in $\pi_{60,32}$ 
could only be detected by $\tau^4 g^3$ or $\tau^2 d_0^2 l$.
Comparison to $\tmf$ rules out both possibilities.
Therefore,
$\tau \sigma \eta_4 \kappabar_2$ is zero.

Second, the product $\kappabar \kappabar_2$ in $\pi_{64,35}$ must be
detected in filtration at least $9$, since $\tau g g_2$ equals zero,
so it could only be detected by $h_1^2 (\D e_1 + C_0)$.
This implies that $\tau \nu \kappabar \kappabar_2$ is zero.

Third, we must consider the product $\rho_{23} \kappabar_2$.
Table \ref{tab:Toda} shows that
the Toda bracket $\langle \sigma, 16, 2 \rho_{15} \rangle$
detects $\rho_{23}$.
Then
$\rho_{23} \kappabar_2$
is contained in
\[
\langle \sigma, 16, 2 \rho_{15} \rangle \kappabar_2 =
\sigma \langle 16, 2 \rho_{15}, \kappabar_2 \rangle.
\]
The latter bracket is contained in $\pi_{60,32}$.
As above, comparison to $\tmf$ shows that the expression
is zero.
\end{proof}

\begin{lemma}
\label{lem:eta,nu,ttheta4.5kappabar}
\revdeg{70, 37}
The Toda bracket $\langle \eta, \nu, \tau \theta_{4.5} \kappabar \rangle$
is detected by $\tau h_1 D'_3$.
\end{lemma}

\begin{proof}
Table \ref{tab:Adams-d3}
shows that $d_3(\tau D'_3)$
equals
$\tau^2 M h_2 g$.
The Moss Convergence Theorem \ref{thm:Moss} implies that
$\tau h_1 D'_3$ detects the Toda bracket.
\end{proof}

\begin{lemma}
\label{lem:eta,nu,t^2h2C'}
\revdeg{71, 37}
There exists an element $\alpha$ in $\pi_{66,34}$
detected by $\tau^2 h_2 C'$ such that
the Toda bracket $\langle \eta, \nu, \alpha \rangle$
is defined and detected by $\tau h_1 p_1$.
\end{lemma}

\begin{proof}
The differential $d_5(\tau p_1 + h_0^2 h_3 h_6) = \tau^2 h_2^2 C'$
and the Moss Convergence Theorem \ref{thm:Moss} 
establish that the Toda bracket is detected by $\tau h_1 p_1$, provided
that the Toda bracket is well-defined.

Let $\alpha$ be an element of $\pi_{66,34}$ that is detected
by $\tau^2 h_2 C'$.  Then $\nu \alpha$ does not necessarily 
equal zero; it could be detected in higher filtration by
$\tau^2 h_2 B_5 + h_2 D'_2$.
Then we can adjust our choice of $\alpha$ by an element detected
by $\tau^2 B_5 + D'_2$ to ensure that $\nu \alpha$ is zero.
\end{proof}

\begin{lemma}
\label{lem:sigma^2,2,T,tkappabar}
\revdeg{72, 38}
The Toda bracket
$\langle \sigma^2, 2, \{t\}, \tau \kappabar \rangle$
is detected by $h_4 Q_2 + h_3^2 D_2$.
\end{lemma}

\begin{proof}
The subbracket $\langle \sigma^2, 2, \{ t \} \rangle$
contains zero by comparison to $C\tau$, and its indeterminacy
is generated by $\sigma^3 \theta_4 = 4 \sigma \kappabar_2$
detected by $\tau g n$.
The subbracket $\langle 2, \{t\}, \tau \kappabar \rangle$
is strictly zero because it cannot be detected by
$h_0 h_2 h_5 i$ by comparison to $\tmf$.
This shows that the desired four-fold Toda bracket is well-defined.

Consider the relation
\[
\eta \langle \sigma^2, 2, \{t\}, \tau \kappabar \rangle \subseteq
\langle \langle \eta, \sigma^2, 2 \rangle, \{t\}, \tau \kappabar \rangle.
\]
Let $\alpha$ be any element of $\langle \eta, \sigma^2, 2 \rangle$.
Table \ref{tab:Toda} shows that $\alpha$ 
is detected by $h_1 h_4$ and 
equals either $\eta_4$ or $\eta_4 + \eta \rho_{15}$.
By inspection, the indeterminacy of
$\langle \alpha, \{t\}, \tau \kappabar \rangle$
equals $\tau \kappabar \cdot \pi_{53,29}$, which 
is detected in Adams filtration at least 14.  (In fact, 
the indeterminacy is non-zero, since it contains both
$\tau \kappabar \cdot \{M c_0\}$ detected by
$\tau M d_0^2$ and also
$\tau \kappabar \cdot \{\D h_1 d_0^2\}$ detected by
$\tau^2 \D h_1 d_0 e_0^2$.)

Table \ref{tab:Toda} shows that
$\langle \alpha, \{t\}, \tau \kappabar \rangle$ is detected
by $h_1 h_4 Q_2$.
Together with the partial analysis of the indeterminacy in the
previous paragraph, this shows that 
$\langle \alpha, \{t\}, \tau \kappabar \rangle$ does
not contain zero.

Then 
$\eta \langle \sigma^2, 2, \{t\}, \tau \kappabar \rangle$
also does not contain zero, and the only possibility is that
$\langle \sigma^2, 2, \{t\}, \tau \kappabar \rangle$ 
is detected by $h_4 Q_2 + h_3^2 D_2$.
\end{proof}

\begin{lemma}
\label{lem:theta4,theta4,kappa}
\revdeg{75, 40}
The Toda bracket $\langle \theta_4, \theta_4, \kappa \rangle$
equals zero.
\end{lemma}

\begin{proof}
The Massey product $\langle h_4^2, h_4^2, d_0 \rangle$ equals zero,
since
\[
h_1^2 \langle h_4^2, h_4^2, d_0 \rangle =
\langle h_1^2, h_4^2, h_4^2 \rangle d_0 = 0,
\]
while $h_1^2 x_{75,7}$ is not zero.
The Moss Convergence Theorem \ref{thm:Moss} then implies
that $\langle \theta_4, \theta_4, \kappa \rangle$
is detected in Adams filtration at least $8$.

The only element in sufficiently high filtration is $P h_1^4 h_6$.
However,
\[
\eta^2 \langle \theta_4, \theta_4, \kappa \rangle =
\langle \eta^2, \theta_4, \theta_4 \rangle \kappa = 0,
\]
while $h_1^2 \cdot P h_1^4 h_6$ is not zero.
Then 
$\langle \theta_4, \theta_4, \kappa \rangle$ must contain zero because
there are no remaining possibilities.

The indeterminacy can be computed by inspection, using that
$\theta_4 \theta_{4.5}$ is zero by comparison to $C\tau$.
\end{proof}

\begin{lemma}
\label{lem:kappa,2,theta5}
\revdeg{77, 40}
The Toda bracket
$\langle \kappa, 2, \theta_5 \rangle$
is detected by $h_6 d_0$.
\end{lemma}

\begin{proof}
The differential $d_3(h_0 h_4) = h_0 d_0$ implies that
$\langle \kappa, 2, \theta_5 \rangle$
is detected by $h_0 h_4 \cdot h_5^2 = 0$ in filtration 4.
In other words, the Toda bracket is detected in Adams filtration
at least 5.

The element $h_1 h_6 d_0$ detects
$\langle \eta \kappa, 2, \theta_5 \rangle$, 
using the Adams differential $d_2(h_6) = h_0 h_5^2$.
This expression contains
$\eta \langle \kappa, 2, \theta_5 \rangle$, which shows that
$\langle \kappa, 2, \theta_5 \rangle$ is detected 
in filtration at most 5.

The only possibility is that the Toda bracket is detected by $h_6 d_0$. 
\end{proof}

\rev{
\begin{lemma}
\label{lem:tm1,eta,2}
\revdeg{79, 42}
There exists an element $\mu$ in $\pi_{77,41}$ that is detected by
$\tau m_1$ such that $\eta \mu$ is zero and $\mu$ is not divisible by $\tau$.
Moreover, the Toda bracket $\langle \mu, \eta, 2 \rangle$ 
contains zero or is detected by $\tau^2 M e_0^2$, and
its indeterminacy is detected by $h_0 h_2 x_{76,6}$.
\end{lemma}

\begin{proof}
Let $\mu'$ be an element of $\pi_{77,42}$ that is detected by
$m_1$.  Then $\tau \mu'$ is detected by $\tau m_1$, and 
$\eta \mu'$ is detected by $h_1 m_1$.
Table \ref{tab:tau-extn} shows that there is a hidden
$\tau$ extension from $h_1 m_1$ to $M \D h_1^2 h_3$.
Therefore, $\tau \eta \mu'$ is detected by $M \D h_1^2 h_3$.

Now let $\mu''$ be an element of $\pi_{77,41}$ that is detected by
$M \D h_1 h_3$.  Then $\eta \mu''$ is also detected by $M \D h_1^2 h_3$.
This shows that $\eta (\tau \mu' + \mu'')$ is zero because there 
are no possible detecting elements in higher filtration.

Choose $\mu$ to be $\tau \mu' + \mu''$.
Note that $\mu''$ is not divisible by $\tau$ because
inclusion of the bottom cell of $C\tau$ takes $M \D h_1 h_3$ to
a non-zero element.  Therefore, $\mu$ is also not divisible
by $\tau$.

Now that $\mu$ is defined, it remains to study the Toda bracket.
We begin with an analysis of its indeterminacy, which is generated
by $\tau \eta^2 \cdot \mu$ and the multiples of $2$ in 
$\pi_{79,42}$.  The first expression is zero by the construction
of $\mu$. 
Let $\alpha$ be an element of $\pi_{79,42}$ that is detected by
$h_2 x_{76,6}$, so $2 \alpha$ is detected by $h_0 h_2 x_{76,6}$.
Tables \ref{tab:2-extn} and \ref{tab:2-extn-possible}
show that there are no hidden $2$ extensions in the $79$-stem with
weight $42$.  Therefore, the indeterminacy is generated by $2 \alpha$.

Inclusion of the bottom cell of $C\tau$ takes the bracket to
$\langle M \D h_1 h_3, h_1, h_0 \rangle$.
Machine-generated data \cite{Wang19} shows that this bracket 
equals $\{0, h_0 h_2 x_{76,6} \}$ 
in $C\tau$.  

Let $\beta$ be any element of $\langle \mu, \eta, 2 \rangle$.
It is possible that $\beta$ maps to $h_0 h_2 x_{76,6}$ under
inclusion of the bottom cell of $C\tau$.  In that case,
$\beta + 2 \alpha$ also belongs to $\langle \mu, \eta, 2 \rangle$
and must map to zero under inclusion of the bottom cell of $C\tau$.

In either case, the original Toda bracket
contains an element 
that maps to zero under inclusion of the bottom cell of $C\tau$,
and that element is therefore divisible by $\tau$.
By inspection, the only possible detecting elements are
$\tau^2 M e_0^2$ and $\tau^3 \D h_1 e_0^2 g$.
The latter option is ruled out by comparison to $\mmf$.
\end{proof}
}

\begin{lemma}
\label{lem:2,eta,th1^2x76,6}
\revdeg{80, 42}
The Toda bracket
$\langle 2, \eta, \tau \eta \{h_1 x_{76,6} \} \rangle$ is detected by
$\tau h_1 x_1$.
\end{lemma}

\begin{proof}
Let $\alpha$ be an element of $\pi_{77,41}$ that is detected by
$h_1 x_{76,6}$.
First we must show that the Toda bracket is well-defined.

Note that $2 \alpha$ is zero because there are no 
$2$ extensions in $\pi_{77,41}$ in sufficiently high
Adams filtration.
Now consider the shuffle
\[
\tau \eta^2 \alpha = \langle 2, \eta, 2 \rangle \alpha =
2 \langle \eta, 2, \alpha \rangle.
\]
Table \ref{tab:Toda} shows that $\langle \eta, 2, \alpha \rangle$
is detected by $h_0 h_2 x_{76,6}$, but this element does
not support a hidden $2$ extension.  This shows that
$\tau \eta^2 \alpha$ is zero and that the Toda bracket
is well-defined.

Finally, use the Adams differential $d_4(h_0 e_2) = \tau h_1^3 x_{76,6}$
and the relation $h_0 \cdot h_0 e_2 = \tau h_1 x_1$ to compute
the Toda bracket.
\end{proof}

\rev{
\begin{lemma}
\label{lem:2,eta,h2x76,6}
\revdeg{81, 43}
There exists an element $\alpha$ in $\pi_{79,42}$ that is detected by
$h_2 x_{76,6}$ such that
$\eta \alpha$ is zero.  Moreover, the Toda bracket
$\langle 2, \eta, \alpha \rangle$ is zero, with no indeterminacy.
\end{lemma}

\begin{proof}
There is no hidden $\eta$ extension on $h_2 x_{76,6}$ 
because the possible targets $\tau^3 d_0 e_0^2 l$ and
$\tau^5 g^4$ are ruled out by comparison to $\mmf$.
Therefore, $\alpha$ exists.

The Massey product $\langle h_0, h_1, h_2 x_{76,6} \rangle$
has no indeterminacy by inspection.  Consequently,
\[
\langle h_0, h_1, h_2 x_{76,6} \rangle =
\langle h_0, h_1, h_2 \rangle x_{76,6} = 0.
\]
The Moss Convergence Theorem \ref{thm:Moss} implies that
the Toda bracket $\langle 2, \eta, \alpha \rangle$ is detected
in filtration $9$ or greater.
The possible detecting elements are $P h_1^2 h_6 c_0$ and
$\D^2 h_1 d_1$.  In either of these cases, the Toda bracket
would be detected by inclusion of the bottom cell of $C\tau$,
but the corresponding bracket is zero in $C\tau$.

The indeterminacy is generated by $\tau \eta^2 \alpha$
and the multiples of $2$ in $\pi_{81,43}$.  The first expression
is zero by the choice of $\alpha$.  Tables \ref{tab:2-extn}
and \ref{tab:2-extn-possible} show that there are no
multiples of $2$ in $\pi_{81,43}$.
\end{proof}
}

\begin{lemma}
\label{lem:2,sigma^2,th2^2C'}
\revdeg{84, 45}
The Toda bracket $\langle 2, \sigma^2, \{\tau h_2^2 C'\} \rangle$
equals zero, with no indeterminacy.
\end{lemma}

\begin{proof}
Let $\alpha$ be an element of $\pi_{66,35}$ that is detected by
$\tau h_2 C'$, so $\nu \alpha$ is the unique element
that is detected by $\tau h_2^2 C'$.
We consider the Toda bracket
$\langle 2, \sigma^2, \nu \alpha \rangle$.
By inspection, the indeterminacy is zero, so
the bracket equals
$\langle 2, \sigma^2, \nu \rangle \alpha$, which equals
$\langle \alpha, 2, \sigma^2 \rangle \nu$.

Apply the Moss Convergence Theorem \ref{thm:Moss}
with the Adams $d_2$ differential to see
that the Toda bracket $\langle \alpha, 2, \sigma^2 \rangle$
is detected by $0$ in Adams filtration $9$, but it could be detected
by a non-zero element in higher filtration.  However, this shows
that $\langle \alpha, 2, \sigma^2 \rangle \nu$ is zero by inspection.
\end{proof}

\begin{lemma}
\label{lem:2,sigma^2,h3(De1+C0)}
\revdeg{84, 45}
The Toda bracket $\langle 2, \sigma^2, \{h_3(\D e_1 + C_0)\} \rangle$
equals zero, with no indeterminacy.
\end{lemma}

\begin{proof}
Let $\beta$ be an element of $\pi_{62,33}$ that is detected by
$\D e_1 + C_0$, so $\sigma \beta$ is the unique element that is
detected by $h_3(\D e_1 + C_0)$.
We consider the Toda bracket
$\langle 2, \sigma^2, \sigma \beta \rangle$.
By inspection, the indeterminacy is zero, so
the bracket equals
$\langle 2, \sigma^2, \beta \rangle \sigma$.

Apply the Moss Convergence Theorem \ref{thm:Moss} 
with the Adams $d_2$ differential to see that the Toda bracket
$\langle 2, \sigma^2, \beta \rangle$ is detected by $0$ in 
Adams filtration $9$, but it could be detected by a non-zero
element in higher filtration. 
Then the only possible non-zero value for
$\langle 2, \sigma^2, \beta \rangle \sigma$ is
$\{M \D h_1 h_3\} \sigma$.
Table \ref{tab:misc-extn} shows that $M \D h_1 h_3$ detects
$\{\D h_1 h_3\} \theta_{4.5}$, so
$\sigma \{M \D h_1 h_3\}$ equals $\sigma \{\D h_1 h_3\} \theta_{4.5}$,
which equals zero.
\end{proof}

\begin{lemma}
\label{lem:tetakappabar^2,2,4kappabar2}
\revdeg{86, 46}
The Toda bracket $\langle \tau \eta \kappabar^2, 2, 4 \kappabar_2 \rangle$
is detected by $M \D h_0^2 e_0$.
\end{lemma}

\begin{proof}
Table \ref{tab:Massey} shows that the Massey product
$\langle \D h_0^2 e_0, h_0^2, h_0 g_2 \rangle$ equals the element
$M \D h_0^2 e_0$.
Now apply the Moss Convergence Theorem \ref{thm:Moss}, using that
Table \ref{tab:eta-extn} shows that
$\D h_0^2 e_0$ detects $\tau \eta \kappabar^2$.
\end{proof}

\begin{lemma}
\label{lem:th0Q3,nu4,eta}
\revdeg{87, 46}
There exists an element $\alpha$ in $\pi_{67,36}$ that is detected
by $h_0 Q_3 + h_0 n_1$ such that 
$h_1^2 c_3$ detects the Toda bracket
$\langle \tau \alpha, \nu_4, \eta \rangle$.
\end{lemma}

\begin{proof}
A consequence of the proof of Lemma \ref{lem:d4-h1c3} is that
there exists $\alpha$ in $\pi_{67,36}$ that is detected by
$h_0 Q_3 + h_0 n_1$ such that
the product $\tau \nu_4 \alpha$ is zero.
Therefore,
$h_1^2 c_3$ detects the Toda bracket
$\langle \tau \alpha, \nu_4, \eta \rangle$ because
of the Adams differential $d_4(h_1 c_3) = \tau h_0 h_2 h_4 Q_3$.
\end{proof}

%% file: more-stable-stems-hidden-extn.tex
\chapter{Hidden extensions}
\label{ch:hidden}

In this chapter, we will discuss hidden extensions
in the $E_\infty$-page of the Adams spectral sequence.
We methodically explore hidden extensions by
$\tau$, $2$, $\eta$, and $\nu$, and we study other miscellaneous
hidden extensions that are relevant for specific purposes.
\revv{
For easy reference, the lemmas in this chapter are labelled with degrees that match the degrees given in the tables.}

\section{Hidden $\tau$ extensions}
\label{sctn:tau-extn}

In order to study hidden $\tau$ extensions,
we will use the long exact sequence
\begin{equation}
\label{eq:LES-Ctau}
\xymatrix@1{
\cdots \ar[r] & \pi_{p,q+1} \ar[r]^\tau & \pi_{p,q} \ar[r] & 
\pi_{p,q} C\tau \ar[r] & 
\pi_{p-1,q+1} \ar[r]^\tau & \pi_{p-1,q} \ar[r] & \cdots
}
\end{equation}
extensively.
This sequence governs hidden $\tau$ extensions in the following sense.
An element $\alpha$ in $\pi_{p,q}$ is divisible by $\tau$
if and only if it maps to zero in $\pi_{p,q} C\tau$,
and an element $\alpha$ in $\pi_{p-1,q+1}$ supports a $\tau$ extension
if and only if it is not in the image of $\pi_{p,q} C\tau$.
Therefore, we need to study the maps
$\pi_{*,*} \map \pi_{*,*} C\tau$ and
$\pi_{*,*} C\tau \map \pi_{*-1,*+1}$
induced by inclusion of the bottom cell into $C\tau$
and by projection from $C\tau$ to the top cell.

The $E_\infty$-pages of the Adams spectral sequences for $S^{0,0}$ and
$C\tau$ give associated graded objects for the homotopy groups that 
are the sources and targets of these maps.  Naturality of the
Adams spectral sequence induces maps on associated graded objects.

These maps on associated graded objects often detect the values
of the maps on homotopy groups.  For example, the element
$h_0$ in the Adams spectral sequence for the sphere is mapped to the
element $h_0$ in the Adams spectral sequence for $C\tau$.
In homotopy groups, this means that inclusion of the bottom cell into
$C\tau$ takes the element $2$ in $\pi_{0,0}$ to the
element $2$ in $\pi_{0,0} C\tau$.

On the other side, the element $\ol{h_1^4}$ in the 
Adams spectral sequence for $C\tau$ is mapped to the
element $h_1^4$ in Adams spectral sequence for the sphere.
In homotopy groups, this means that projection from $C\tau$ to
the top cell takes the element $\{ \ol{h_1^4} \}$ in $\pi_{5,3}C\tau$
to the element $\eta^4$ in $\pi_{4,4}$.

However, some values of the maps on homotopy groups can be hidden
in the map of associated graded objects.  This situation is rare in
low stems but becomes more and more common in higher stems.
The first such example occurs in the 30-stem.
The element $\D h_2^2 $ is a permanent cycle in the Adams spectral
sequence for $C\tau$, so $\{\D h_2^2 \}$ is an element in $\pi_{30,16} C\tau$.
Now $\D h_2^2$ maps to zero in the $E_\infty$-page of the Adams spectral
sequence for the sphere, but $\{\D h_2^2 \}$ does not map to zero
in $\pi_{29,17}$.  In fact $\{\D h_2^2 \}$ maps to $\eta \kappa^2$,
which is detected by $h_1 d_0^2$.
This demonstrates that projection from $C\tau$ to the top cell
has a hidden value.

We refer the reader to Section \ref{sctn:assoc-graded}
for a precise discussion of these issues.

\begin{thm}
\label{thm:hid-incl-proj}
\mbox{}
\begin{enumerate}
\item
Through the 90-stem,
Table \ref{tab:hid-incl} lists all
hidden values of inclusion of the bottom cell into $C\tau$,
except that:
\rev{
\begin{enumerate}
\item
If $h_1 f_2$ does not survive but $\tau h_1 f_2$ does survive,
then $\tau h_1 f_2$ maps to $h_0^3 c_3$.
\item
If $h_1^2 f_2$ does not survive, then $\tau h_1^2 f_2$
maps to $\ol{\tau h_1^3 h_4 Q_3}$ or
$\D^2 e_1 + \ol{\tau \D h_2 e_1 g}$.
\item
$\tau h_1 x_{85,6}$
maps to $\ol{\tau h_1^3 h_4 Q_3}$ or
$\D^2 e_1 + \ol{\tau \D h_2 e_1 g}$.
\end{enumerate}
}
\item
Through the 90-stem, Table \ref{tab:hid-proj} lists all
hidden values of projection from $C\tau$ to the top cell,
except that:
\rev{
\begin{enumerate}
\item
If $h_1 f_2$ does not survive, then $h_1 f_2$ maps to $h_1^2 h_4 Q_3$ 
or $\D h_1 j_1$.
\item
$x_{85,6}$ maps to $h_1^2 h_4 Q_3$ or $\D h_1 j_1$.
\item
If $h_1^2 f_2$ does not survive, then $h_1^2 f_2$ maps to
$h_1^3 h_4 Q_3$ or $\tau M h_0 g^2$.
\item
$h_1 x_{85,6}$ maps to $h_1^3 h_4 Q_3$ or $\tau M h_0 g^2$.
\item
If $h_1^2 f_2$ survives, then 
$\ol{\tau h_1^3 h_4 Q_3}$ or $\D^2 e_1 + \ol{\tau \D h_2 e_1 g}$
maps to $M \D h_1^2 d_0$.
\end{enumerate}
}
\end{enumerate}
\end{thm}

\begin{proof}
The values of inclusion of the bottom cell and projection
to the top cell are almost entirely determined by inspection
of Adams $E_\infty$-pages.
Taking into account the multiplicative structure,
there are no other combinatorial possibilities.
For example, consider the exact sequence
\[
\pi_{30,16} \map \pi_{30,16} C\tau \map \pi_{29,17}.
\]
In the Adams $E_\infty$-page for $C\tau$, 
$h_4^2$ and $\D h_2^2$
are the only two elements in the 30-stem with weight 16.
In the Adams $E_\infty$-page for the sphere,
$h_4^2$ is the only element in the 30-stem with weight 16,
and $h_1 d_0^2$ is the only element in the 29-stem with weight 17.
The only possibility is that $h_4^2$ maps to $h_4^2$ under
inclusion of the bottom cell, and
$\D h_2^2$ maps to $h_1 d_0^2$ under projection to the top cell.

One case, given below in Lemma \ref{lem:t-d1e1}, 
requires a more complicated argument.
\end{proof}

\begin{remark}
Through the 90-stem, inclusion of the bottom cell into $C\tau$
has only one hidden value with target indeterminacy.
Namely, $h_2 c_1 A'$ is the hidden value of $\ol{h_1 g B_7}$,
with target indeterminacy generated by $\D j_1$.
Through the 90-stem, projection from $C\tau$ to the top cell
has no hidden values with target indeterminacy.
\end{remark}

\begin{remark}
Through the 90-stem,
inclusion of the bottom cell into $C\tau$ has no crossing values.
On the other hand, projection from $C\tau$ to the top cell
does have crossing values in this range.  These occurrences
are described in the fourth column of 
Table \ref{tab:hid-proj}.
Each can be verified by direct inspection.
\end{remark}

\begin{thm}
\label{thm:tau-extn}
Through the 90-stem,
Table \ref{tab:tau-extn}
lists all hidden $\tau$ extensions in $\C$-motivic stable homotopy groups,
except that:
\rev{
\begin{enumerate}
\item
if $M \D h_1 d_0$ is not hit by a differential, then
there is a hidden $\tau$ extension from $\D h_1 j_1$ to
$M \D h_1 d_0$.
\item
if $M \D h_1^2 d_0$ is not hit by a differential, then
there is a hidden $\tau$ extension from $\tau M h_0 g^2$ to
$M \D h_1^2 d_0$.
\end{enumerate}
}
In this range, the only crossing extension is:
\begin{enumerate}
\item
the hidden $\tau$ extension from $h_1^2 h_6 c_0$ to $h_0 h_4 D_2$,
and the not hidden $\tau$ extension on $\tau h_2^2 Q_3$.
\end{enumerate}
\end{thm}

\begin{proof}
Almost all of these hidden $\tau$ extensions follow immediately
from the values of the maps in the long exact sequence
(\ref{eq:LES-Ctau})
given in Tables \ref{tab:hid-incl} and \ref{tab:hid-proj}.

For example, 
consider the element $P d_0$
in the Adams $E_\infty$-page for the sphere, which 
belongs to the 22-stem with weight 12. 
Now $\pi_{22,12} C\tau$ is zero because there are 
no elements in that degree in the Adams $E_\infty$-page
for $C\tau$, so
inclusion of the bottom cell takes $\{ P d_0 \}$ to zero.
Therefore, $\{ P d_0 \}$ must be in the image of multiplication
by $\tau$.
The only possibility is that there is a hidden $\tau$
extension from $c_0 d_0$ to $P d_0$.
\end{proof}

\begin{remark}
\label{rem:t-Dh1j1}
If $M \D h_1 d_0$ and $M \D h_1^2 d_0$ are not hit by differentials,
then a straightforward analysis of the sequence
(\ref{eq:LES-Ctau})
shows that the possible $\tau$ extensions on $\D h_1 j_1$ and 
$\tau M h_0 g^2$ must occur.  Thus these uncertainties are entirely
determined by corresponding uncertainties in values of the Adams 
differentials.
\end{remark}

\begin{lemma}
\label{lem:t-d1e1}
\mbox{}
\begin{enumerate}
\item
\revdeg{70, 10, 38}
The element 
$h_1 h_3 (\D e_1 + C_0) + \tau h_2 C''$ maps to
$\ol{h_1^4 c_0 Q_2}$ 
under inclusion of the bottom cell into $C\tau$.
\item
\revdeg{70, 8, 39}
There is a hidden $\tau$ extension from $d_1 e_1$ to 
$h_1 h_3 (\D e_1 + C_0)$.
\end{enumerate}
\end{lemma}

\begin{proof}
Consider the exact sequence
$\pi_{70, 38} \map \pi_{70, 38} C\tau \map \pi_{69, 39}$.
For combinatorial reasons, one of the following two possibilities
must occur:
\begin{enumerate}[(a)]
\item
the element $h_1 h_3 (\D e_1 + C_0) + \tau h_2 C''$ maps to
$\ol{h_1^4 c_0 Q_2}$ 
under inclusion of the bottom cell into $C\tau$, and
there is a hidden $\tau$ extension from $d_1 e_1$ to 
$h_1 h_3 (\D e_1 + C_0)$.
\item
the element $h_1 h_3 (\D e_1 + C_0)$
maps to $\ol{h_1^4 c_0 Q_2}$ 
under inclusion of the bottom cell into $C\tau$, and
there is a hidden $\tau$ extension from $d_1 e_1$ to 
$h_1 h_3 (\D e_1 + C_0) + \tau h_2 C''$.
\end{enumerate}
We will show that there cannot be a hidden
$\tau$ extension from $d_1 e_1$ to
$h_1 h_3 (\D e_1 + C_0) + \tau h_2 C''$.

Lemma \ref{lem:eta-sigma-k1} shows that
$\tau \nu \{d_1 e_1\}$ equals
$\tau \eta \sigma \{k_1\}$.
Since there is no hidden $\tau$ extension on $h_1 k_1$,
there must exist an element $\alpha$ in $\{k_1\}$ such that
$\tau \eta \alpha = 0$.
Therefore,
$\tau \nu \{d_1 e_1\}$ must be zero.

If there were a $\tau$ extension from
$d_1 e_1$ to 
$h_1 h_3 (\D e_1 + C_0) + \tau h_2 C''$, then
$\tau \nu \{d_1 e_1\}$ would be detected by
\[
h_2 \cdot (h_1 h_3 (\D e_1 + C_0) + \tau h_2 C'') =
\tau h_2^2 C'',
\]
and in particular would be non-zero.
\end{proof}

\section{Hidden $2$ extensions}
\label{sctn:2-extn}

\begin{thm}
\label{thm:2-extn}
Tables \ref{tab:2-extn} \rev{and \ref{tab:2-extn-null}} list some hidden extensions by $2$.
\end{thm}

\begin{proof}
Many of the hidden extensions follow by comparison to 
$C\tau$.  For example, there is a hidden $2$ extension
from $h_0 h_2 g$ to $h_1 c_0 d_0$ in the Adams spectral
sequence for $C\tau$.
Pulling back along inclusion of the bottom cell into $C\tau$,
there must also be a hidden $2$ extension from $h_0 h_2 g$
to $h_1 c_0 d_0$ in the Adams spectral sequence for the sphere.
This type of argument is indicated by the notation
$C\tau$ in the fourth column of Table \ref{tab:2-extn}.

Next, Table \ref{tab:tau-extn} shows a hidden $\tau$
extension from $h_1 c_0 d_0$ to $P h_1 d_0$.
Therefore, there is also a hidden $2$ extension from
$\tau h_0 h_2 g$ to $P h_1 d_0$.
This type of argument is indicated by the notation $\tau$
in the fourth column of Table \ref{tab:2-extn}. 

Many cases require more complicated arguments.  In stems up to
approximately dimension 62, see 
\cite{Isaksen14c}*{Section 4.2.2 and Tables 27--28}
\cite{WangXu18}, and \cite{Xu16}.
The higher-dimensional cases are handled in the following lemmas.
\end{proof}

\begin{remark}
Through the 90-stem, there are no crossing $2$ extensions.
\end{remark}

\begin{remark}
The hidden $2$ extension from $h_0 h_3 g_2$ to $\tau g n$
is proved in \cite{WangXu18}, which uses on the
``$\R P^\infty$-method" to establish a hidden
$\sigma$ extension from $\tau h_3 d_1$ to $\D h_2 c_1$
and a hidden $\eta$ extension from $\tau h_1 g_2$ to $\D h_2 c_1$.
We now have easier proofs for these $\eta$ and $\sigma$ extensions,
using the hidden $\tau$ extension from $h_1^2 g_2$ to $\D h_2 c_1$
given in Table \ref{tab:tau-extn}, as well as the relation
$h_3^2 d_1 = h_1^2 g_2$.
\end{remark}

\begin{remark}
\label{rem:2-h0h5i}
Comparison to synthetic homotopy gives additional
information about some
possible hidden $2$ extensions, including:
\begin{enumerate}
\item
there is a hidden $2$ extension from $h_0 h_5 i$ to
$\tau^4 e_0^2 g$.
\item
there is no hidden $2$ extension from $P x_{76,6}$ 
to $M \D h_1 d_0$.
\end{enumerate}
See \rev{\cite{Burklund21} and} \cite{BIX} for more details.
We are grateful to John Rognes for pointing out a mistake
in \cite{Isaksen14c}*{Lemma 4.56 and Table 27} 
concerning the hidden $2$ extension
on $h_0 h_5 i$.  
Lemma \ref{lem:2-h0h5i} shows that the extension occurs but
does not determine its target precisely.  
\end{remark}

\begin{remark}
The first correct proof of the relation $2 \theta_5 = 0$ 
appeared in \cite{Xu16}.  Earlier claims in \cite{Lin01}
and \cite{Kochman90} were based upon a mistaken understanding of the
Toda bracket $\langle \sigma^2, 2, \theta_4 \rangle$.
See \cite{Isaksen14c}*{Table 23} for the correct value of
this bracket.
\end{remark}

\begin{remark}
\label{rem:2-tMh0g^2}
If $M \D h_1^2 d_0$ is non-zero in the $E_\infty$-page, then
there is a hidden $\tau$ extension from $\tau M h_0 g^2$ to
$M \D h_1^2 d_0$.  This implies that there must be a hidden
$2$ extension from $\tau^2 M g^2$ to $M \D h_1^2 d_0$.
\end{remark}

\rev{
\begin{remark}
\label{rem:2-x87,7}
Table \ref{tab:2-extn} shows that there is a hidden $2$
extension from $x_{87,7}$ to $\tau^3 g Q_3$.
This follows from data recently produced by Dexter Chua on the
$d_2$ differentials in the Adams spectral sequence for the cofiber 
of $2$.
\end{remark}
}

\begin{thm}
\label{thm:2-extn-possible}
Table \ref{tab:2-extn-possible} lists all unknown hidden $2$
extensions, through the $90$-stem.
\end{thm}

\begin{proof}
Many possibilities are eliminated by comparison to $C\tau$, to
$\tmf$, or to $\mmf$.  For example,
there cannot be a hidden $2$ extension from $h_2^2 h_4$ to $\tau h_1 g$
by comparison to $C\tau$.

Many additional possibilities are eliminated by consideration of other
parts of the multiplicative structure.  For example,
there cannot be a hidden $2$ extension from $P h_1 h_5$
to $\tau^3 g^2$ because $\tau^3 g^2$ supports an $h_1$ extension
and $2 \eta$ equals zero.

Several cases are a direct consequence of Proposition
\ref{prop:2-<alpha,2,theta5>}.

Some possibilities are eliminated by more complicated arguments.
These cases are handled in the following lemmas.
\end{proof}

\begin{remark}
\label{rem:2-Px76,6}
If $M \D h_1^2 d_0$ is not zero in the $E_\infty$-page, then
$M \D h_1 d_0$ supports an $h_1$ multiplication, and there cannot
be a hidden $2$ extension from $P x_{76,6}$ to $M \D h_1 d_0$.
\end{remark}

\begin{prop}
\label{prop:2-<alpha,2,theta5>}
Suppose that $2 \alpha$ and $\tau \eta \alpha$ are both zero.
Then $2 \langle \alpha, 2, \theta_5 \rangle$ is zero.
\end{prop}

\begin{proof}
Consider the shuffle
\[
2 \langle \alpha, 2, \theta_5 \rangle =
\langle 2, \alpha, 2 \rangle \theta_5.
\]
Since $2 \theta_5$ is zero, this expression has no indeterminacy.
Corollary \ref{cor:2-symmetric} implies that it equals
$\tau \eta \alpha \theta_5$, which is zero by assumption.
\end{proof}

\begin{remark}
\label{rem:2-<alpha,2,theta5>}
Proposition \ref{prop:2-<alpha,2,theta5>} eliminates possible
hidden $2$ extensions on several elements, including $h_2^2 h_6$,
$h_0^3 h_3 h_6$, $h_3^2 h_6$, $h_6 c_1$, $h_2^2 h_4 h_6$,
$h_0^5 h_6 i$, and $h_2^2 h_6 g$.
\end{remark}

\begin{lemma}
\label{lem:2-h0h5i}
\revdeg{54, 9, 28}
There is a hidden $2$ extension from $h_0 h_5 i$ to either
$\tau M P h_1$ or to $\tau^4 e_0^2 g$.
\end{lemma}

\begin{proof}
Table \ref{tab:unit-mmf} shows that
$h_0 h_5 i$ maps to $\D^2 h_2^2$ in the homotopy of $\tmf$.
The element $\D^2 h_2^2$ supports a hidden $2$ extension,
so $h_0 h_5 i$ must support a hidden $2$ extension as well.
\end{proof}

\begin{lemma}
\label{lem:2-th1H1}
\mbox{}
\begin{enumerate}
\item
\revdeg{63, 6, 33}
There is a hidden $2$ extension from
$\tau h_1 H_1$ to $\tau h_1 (\D e_1 + C_0)$.
\item
\revdeg{63, 7, 33}
There is no hidden $2$ extension on $\tau X_2 + \tau C'$.
\item
\revdeg{70, 7, 37}
There is a hidden $2$ extension from
$\tau h_1 h_3 H_1$ to $\tau h_1 h_3 (\D e_1 + C_0)$.
\end{enumerate}
\end{lemma}

\begin{proof}
Table \ref{tab:eta-extn} shows that there is an $\eta$
extension from $\tau h_1 H_1$ to $h_3 Q_2$.
Let $\alpha$ be any element of $\pi_{63,33}$ that is
detected by $\tau h_1 H_1$.
Then
$\tau \eta^2 \alpha$ is non-zero and detected by 
$\tau h_1 h_3 Q_2$.
Note that $\tau h_1 h_3 Q_2$ cannot be the target of a hidden
$2$ extension because there are no possibilities.

If $2 \alpha$ were zero, then we would have the shuffling
relation
\[
\tau \eta^2 \alpha =
\langle 2, \eta, 2 \rangle \alpha = 
2 \langle \eta, 2, \alpha \rangle.
\]
But this would contradict the previous paragraph.

We now know that $2 \alpha$ must be non-zero for every possible
choice of $\alpha$.
The only possibility is that there is a hidden
$2$ extension from $\tau h_1 H_1$ to $\tau h_1 (\D e_1 + C_0)$,
and that there is no hidden $2$ extension on $\tau X_2 + \tau C'$.
This establishes the first two parts.

The third part follows immediately from the first part by
multiplication by $h_3$.
\end{proof}

\begin{lemma}
\label{lem:2-h1h6}
\revdeg{64, 2, 33}
There is a hidden $2$ extension from $h_1 h_6$ to $\tau h_1^2 h_5^2$.
\end{lemma}

\begin{proof}
Table \ref{tab:Toda} shows that 
$h_1 h_6$ detects $\langle \eta, 2, \theta_5 \rangle$.
Now shuffle to obtain
\[
2 \langle \eta, 2, \theta_5 \rangle =
\langle 2, \eta, 2 \rangle \theta_5 = \tau \eta^2 \theta_5.
\]
\end{proof}

\begin{lemma}
\label{lem:2-D1h3^2}
\revdeg{66, 6, 36}
There is no hidden $2$ extension on $\D_1 h_3^2$.
\end{lemma}

\begin{proof}
Table \ref{tab:Toda} shows that $\D_1 h_3^2$ detects the
Toda bracket $\langle \eta^2, \theta_4, \eta^2, \theta_4 \rangle$.
We have
\[
2 \langle \eta^2, \theta_4, \eta^2, \theta_4 \rangle \subseteq
\langle \langle 2, \eta^2, \theta_4 \rangle, \eta^2, \theta_4 \rangle.
\]
Table \ref{tab:Toda} shows that
\[
\nu \theta_4 = \langle 2, \eta, \eta \theta_4 \rangle =
\langle 2, \eta^2, \theta_4 \rangle,
\]
so we must compute $\langle \nu \theta_4, \eta^2, \theta_4 \rangle$.

This bracket contains $\nu \langle \theta_4, \eta^2, \theta_4 \rangle$,
which equals zero by Lemma \ref{lem:theta4,eta^2,theta4}.
Therefore, we only need to compute the indeterminacy of 
$\langle \nu \theta_4, \eta^2, \theta_4 \rangle$.

The only possible non-zero element in the indeterminacy is the 
product $\theta_4 \{ t \}$.
Table \ref{tab:Toda} shows that 
$\{t \} = \langle \nu, \eta, \eta \theta_4 \rangle$.  Now
\[
\theta_4 \{ t\} = \langle \nu, \eta, \eta \theta_4 \rangle \theta_4 =
\nu \langle \eta, \eta \theta_4, \theta_4 \rangle.
\]
This last expression is well-defined because
$\theta_4^2$ is zero \cite{Xu16}, and it 
must be zero because 
$\pi_{63,34}$ consists entirely of multiples of $\eta$.
\end{proof}

\begin{lemma}
\label{lem:2-h0Q3}
\revdeg{67, 6, 36}
There is no hidden $2$ extension on $h_0 Q_3 + h_2^2 D_3$.
\end{lemma}

\begin{proof}
By comparison to the homotopy of $C\tau$,
there is no hidden extension with value $h_2^2 A'$.
Table \ref{tab:eta-extn} shows that
$\tau^2 \D h_2^2 e_0 g$ supports a hidden $\eta$ extension.
Therefore, it cannot be the target of a $2$ extension.
\end{proof}

\begin{lemma}
\label{lem:2-h3A'}
\revdeg{68, 7, 36}
There is no hidden $2$ extension on $h_3 A'$.
\end{lemma}

\begin{proof}
Table \ref{tab:Toda} shows that $h_3 A'$ detects the Toda bracket
$\langle \sigma, \kappa, \tau \eta \theta_{4.5} \rangle$.
Shuffle to obtain
\[
\langle \sigma, \kappa, \tau \eta \theta_{4.5} \rangle 2 =
\sigma \langle \kappa, \tau \eta \theta_{4.5}, 2 \rangle.
\]
The bracket $\langle \kappa, \tau \eta \theta_{4.5}, 2 \rangle$
is zero because it is contained in $\pi_{61,32} = 0$.
\end{proof}

\begin{lemma}
\label{lem:2-p'}
\revdeg{69, 4, 36}
There is no hidden $2$ extension on $p'$.
\end{lemma}

\begin{proof}
Table \ref{tab:misc-extn} shows that $p'$
detects the product $\sigma \theta_5$, and 
$2 \theta_5$ is already known to be zero
\cite{Xu16}.
\end{proof}

\begin{lemma}
\label{lem:2-th1D3'}
\revdeg{70, 9, 37}
There is no hidden $2$ extension on $\tau h_1 D_3'$.
\end{lemma}

\begin{proof}
Table \ref{tab:Toda} shows that 
$\tau h_1 D'_3$ detects the Toda bracket 
$\langle \eta, \nu, \tau \theta_{4.5} \kappabar \rangle$.
Now shuffle to obtain
\[
2 \langle \eta, \nu, \tau \theta_{4.5} \kappabar \rangle =
\langle 2, \eta, \nu \rangle \tau \theta_{4.5} \kappabar,
\]
which equals zero because $\langle 2, \eta, \nu \rangle$
is contained in in $\pi_{5,3} = 0$.
\end{proof}

\begin{lemma}
\label{lem:2-h1h3h6}
\revdeg{71, 3, 37}
There is no hidden $2$ extension on $h_1 h_3 h_6$.
\end{lemma}

\begin{proof}
Table \ref{tab:Toda} shows that
$h_1 h_6$ detects the Toda bracket $\langle \eta, 2, \theta_5 \rangle$.
Let $\alpha$ be an element of this bracket.
Then $h_1 h_3 h_6$ detects $\sigma \alpha$, and
\[
2 \sigma \alpha = 2 \sigma \langle \eta, 2, \theta_5 \rangle =
\sigma \langle 2, \eta, 2 \rangle \theta_5 = 
\tau \eta^2 \sigma \theta_5.
\]

Table \ref{tab:Toda} also shows that $h_6 c_0$ detects 
the Toda bracket $\langle \epsilon, 2, \theta_5 \rangle$.
Let $\beta$ be an element of this bracket.
As in the proof of Lemma \ref{lem:2-h6c0}, we compute
that $2 \beta$ equals $\tau \eta \epsilon \theta_5$.

Now consider the element $\sigma \alpha + \beta$, which is also
detected by $h_1 h_3 h_6$.  Then
\[
2(\sigma \alpha + \beta) = 
\tau \eta^2 \sigma \theta_5 + \tau \eta \epsilon \theta_5 =
\tau \nu^3 \theta_5,
\]
using Toda's relation $\eta^2 \sigma + \nu^3 = \eta \epsilon$
\cite{Toda62}.

Table \ref{tab:nu-extn} shows that there is a hidden
$\nu$ extension from $h_2 h_5^2$ to $\tau h_1 Q_3$.
Therefore, $\tau h_1 Q_3$ detects $\nu^2 \theta_5$.

This does not yet imply that $\nu^3 \theta_5$ is zero,
because $\nu^2 \theta_5 + \eta \{\tau Q_3 + \tau n_1\}$
might be detected $h_3 A'$ or $P h_2 h_5 j$ in higher filtration.
However, $h_3 A'$ does not support a hidden
$\nu$ extension by Lemma \ref{lem:nu-h3A'}.
Also, Table \ref{tab:unit-mmf} shows that
$P h_2 h_5 j$ maps non-trivially to $\tmf$,
while $\nu^2 \theta_5 + \eta \{\tau Q_3 + \tau n_1\}$
maps to zero.
This is enough to conclude that $\nu^3 \theta_5$ is zero.

We have now shown that $2(\sigma \alpha + \beta)$ is zero
in $\pi_{71,37}$.  Since $h_1 h_3 h_6$
detects $\sigma \alpha + \beta$, it follows that 
$h_1 h_3 h_6$ does not support a hidden $2$ extension.
\end{proof}

\begin{lemma}
\label{lem:2-h6c0}
\revdeg{71, 4, 37}
There is a hidden $2$ extension from $h_6 c_0$ to $\tau h_1^2 p'$.
\end{lemma}

\begin{proof}
Table \ref{tab:Toda} shows that $h_6 c_0$ detects the Toda bracket
$\langle \epsilon, 2, \theta_5 \rangle$.
Now shuffle to obtain
\[
2 \langle \epsilon, 2, \theta_5 \rangle = 
\langle 2, \epsilon, 2 \rangle \theta_5 = 
\tau \eta \epsilon \theta_5.
\]
Finally, $\tau \eta \epsilon \theta_5$ is detected
by $\tau h_1^2 p'$ because of the relation
$h_1^2 p' = h_1 h_5^2 c_0$.
\end{proof}

\begin{lemma}
\label{lem:2-th1p1}
\revdeg{71, 5, 37}
There is no hidden $2$ extension on $\tau h_1 p_1$.
\end{lemma}

\begin{proof}
Lemma \ref{lem:eta,nu,t^2h2C'} shows that $\tau h_1 p_1$
detects $\langle \eta, \nu, \alpha \rangle$ for some
$\alpha$ detected by $\tau^2 h_2 C'$.
Now shuffle to obtain
\[
2 \langle \eta, \nu, \alpha \rangle = 
\langle 2, \eta, \nu \rangle \alpha,
\]
which is zero because $\langle 2, \eta, \nu \rangle$ is contained
in $\pi_{5,3} = 0$.
\end{proof}

\begin{lemma}
\label{lem:2-h1^2h3H1}
\revdeg{71, 8, 39}
There is a hidden $2$ extension from $h_2^3 H_1$ to $\tau M h_2^2 g$.
\end{lemma}

\begin{proof}
Table \ref{tab:nu-extn} shows that there are hidden $\nu$ extensions
from $h_2^3 H_1$ to $h_3 C''$, and from
$\tau M h_2^2 g$ to $M h_1 d_0^2$.
Table \ref{tab:2-extn} shows that there is also a hidden
$2$ extension from $h_3 C''$ to $M h_1 d_0^2$.
The only possibility is that there must also be a hidden $2$
extension on $h_1^2 h_3 H_1$.
\end{proof}

\begin{lemma}
\label{lem:2-Ph1h6}
\revdeg{72, 6, 37}
There is no hidden $2$ extension on $P h_1 h_6$.
\end{lemma}

\begin{proof}
Table \ref{tab:Toda} shows that $P h_1 h_6$ detects the Toda bracket
$\langle \mu_9, 2, \theta_5 \rangle$.
Shuffle to obtain
\[
2 \langle \mu_9, 2, \theta_5 \rangle =
\langle 2, \mu_9, 2 \rangle \theta_5 =
\tau \eta \mu_9 \theta_5.
\]
Table \ref{tab:Toda} also shows that $\mu_9$ is contained in the
Toda bracket $\langle \eta, 2, 8 \sigma \rangle$.
Shuffle again to obtain
\[
\tau \eta \mu_9 \theta_5 =
\langle \eta, 2, 8 \sigma \rangle \tau \eta \theta_5 =
\tau \eta^2 \langle 2, 8 \sigma, \theta_5 \rangle.
\]
Table \ref{tab:Toda} shows that
$h_0^3 h_3 h_6$ detects $\langle 2, 8 \sigma, \theta_5 \rangle$.

By inspection, the product $\eta^2 \{h_0^3 h_3 h_6\}$
can only be detected by $\D^2 h_1 h_4 c_0$.
However, this cannot occur by comparison to $C\tau$.
Therefore,
$\eta^2 \{h_0^3 h_3 h_6\}$, and also
$\tau \eta^2 \{h_0^3 h_3 h_6\}$, must be zero.
\end{proof}

\begin{lemma}
\label{lem:2-h4Q2+h3^2D2}
\revdeg{72, 8, 38}
There is no hidden $2$ extension on $h_4 Q_2 + h_3^2 D_2$.
\end{lemma}

\begin{proof}
Table \ref{tab:Toda} shows that the element
$h_4 Q_2 + h_3^2 D_2$ detects the Toda bracket
$\langle \sigma^2, 2, \{t\}, \tau \kappabar \rangle$.
Consider the relation
\[
2 \langle \sigma^2, 2, \{t\}, \tau \kappabar \rangle \subseteq
\langle \langle 2, \sigma^2, 2 \rangle, \{t\}, \tau \kappabar \rangle.
\]
Corollary \ref{cor:2-symmetric} shows 
that the Toda bracket $\langle 2, \sigma^2, 2 \rangle$ contains
zero since $\eta \sigma^2$ is zero.  Therefore, it
consists of even multiples of $\rho_{15}$;
let $2 k \rho_{15}$ be any such element in $\pi_{15,8}$.

The Toda bracket
$\langle 2 k \rho_{15}, \{t\}, \tau \kappabar \rangle$
contains $k \rho_{15} \langle 2, \{t\}, \tau \kappabar \rangle$,
which equals zero as discussed in the proof
of Lemma \ref{lem:sigma^2,2,T,tkappabar}.
Moreover, its indeterminacy
is equal to $\tau \kappabar \cdot \pi_{52,28}$,
which is detected in Adams filtration at least 12.
This implies that
$\langle 2 k \rho_{15}, \{t\}, \tau \kappabar \rangle$
is detected in Adams filtration at least 12,
and that
the target of a hidden $2$ extension on $h_4 Q_2 + h_3^2 D_2$
must have Adams filtration at least 12.

The remaining possible targets with Adams filtration at least 12
are eliminated by comparison to $C\tau$ or to $\mmf$.
\end{proof}

\begin{remark}
\label{rem:2-h4Q2+h3^2D2}
The proof of Lemma \ref{lem:2-h4Q2+h3^2D2} might be simplified by
considering the shuffle
\[
2 \langle \sigma^2, 2, \{t\}, \tau \kappabar \rangle =
\langle 2, \sigma^2, 2, \{t\} \rangle \tau \kappabar.
\]
However, the latter four-fold bracket may not exist, since
both three-fold subbrackets have indeterminacy.
See \cite{Isaksen14b} for a discussion of the analogous 
difficulty with Massey products.
\end{remark}

\begin{lemma}
\label{lem:2-h2^2Q3}
\revdeg{73, 7, 40}
There is no hidden $2$ extension on $h_2^2 Q_3$.
\end{lemma}

\begin{proof}
The element $\tau h_2^2 Q_3$ detects
$\nu^2 \{\tau Q_3 + \tau n_1\}$, so it cannot support a 
hidden $2$ extension.
This rules out all possible $2$ extensions on
$h_2^2 Q_3$.
\end{proof}

\begin{lemma}
\label{lem:2-h0h4D2}
\revdeg{73, 8, 38}
There is no hidden $2$ extension on $h_0 h_4 D_2$.
\end{lemma}

\begin{proof}
Table \ref{tab:tau-extn} shows that there is a hidden
$\tau$ extension from $h_1^2 h_6 c_0$ to $h_0 h_4 D_2$.
Therefore, $h_0 h_4 D_2$ detects either
$\tau \eta \epsilon \eta_6$ or
$\tau \eta \epsilon \eta_6 + \nu^2 \{\tau Q_3 + \tau n_1\}$,
because of the presence of $\tau h_2^2 Q_3$ in higher
filtration.
In either case, $h_0 h_4 D_2$ cannot support a hidden $2$ extension.
\end{proof}

\begin{lemma}
\label{lem:2-h3(tQ3+tn1)}
\revdeg{74, 6, 39}
There is a hidden $2$ extension from $h_3(\tau Q_3+ \tau n_1)$
to $\tau x_{74,8}$.
\end{lemma}

\begin{proof}
Table \ref{tab:misc-extn} shows that
$\tau x_{74,8}$ detects $\tau \kappabar_2 \theta_4$.
Table \ref{tab:Toda} shows that $\theta_4$ equals
the Toda bracket $\langle \sigma^2, 2, \sigma^2, 2 \rangle$.

Now consider the shuffle
\[
\tau \kappabar_2 \theta_4 = 
\tau \kappabar_2 \langle \sigma^2, 2, \sigma^2, 2 \rangle =
\langle \tau \kappabar_2, \sigma^2, 2, \sigma^2 \rangle 2.
\]
Lemma \ref{lem:tkappabar2,sigma^2,2} shows that the latter
bracket is well-defined.
This implies that $\tau x_{74,8}$ is the target of a hidden $2$
extension, and $h_3(\tau Q_1 + \tau n_1)$ is the only possible
source.
\end{proof}

\begin{lemma}
\label{lem:2-h6d0}
\revdeg{77, 5, 40}
There is no hidden $2$ extension on $h_6 d_0$.
\end{lemma}

\begin{proof}
Table \ref{tab:Toda} shows that $h_6 d_0$ detects the Toda bracket
$\langle \kappa, 2, \theta_5 \rangle$.  Now shuffle to obtain
\[
2 \langle \kappa, 2, \theta_5 \rangle = 
\langle 2, \kappa, 2 \rangle \theta_5 = \tau \eta \kappa \theta_5.
\]
Lemma \ref{lem:sigma^2+kappa-h5^2} shows that this product
equals \rev{either $\tau \eta \sigma^2 \theta_5$ or
$\tau \eta \sigma^2 \theta_5 + \tau^3 \eta \kappa_1 \kappabar_2$.
Both possibilities are zero because
$\eta \sigma^2$ and $\tau \eta \kappa_1$ are zero.}
\end{proof}

\begin{lemma}
\label{lem:2-e0A'}
\revdeg{78, 10, 42}
There is a hidden $2$ extension from $e_0 A'$ to $M \D h_1^2 h_3$.
\end{lemma}

\begin{proof}
Let $\alpha$ be an element of $\pi_{76,40}$ that is detected
by $x_{76,9}$.
Table \ref{tab:eta-extn} shows that there is a hidden
$\eta$ extension from $x_{76,9}$ to $M \D h_1 h_3$,
so $\tau \eta^2 \alpha$ is detected by $\tau M \D h_1^2 h_3$.
Now shuffle to obtain
\[
\tau \eta^2 \alpha = \langle 2, \eta, 2 \rangle \alpha =
2 \langle \eta, 2, \alpha \rangle.
\]
This shows that $\tau M \D h_1^2 h_3$ must be the target
of a hidden $2$ extension.

Moreover, the source of this hidden $2$ extension must be in Adams
filtration at least $10$, since the Adams differential
$d_2(\tau x_{77,8}) = h_0 x_{76,9}$ implies that
$\langle \eta, 2, \alpha \rangle$ is detected by
$h_1 x_{77,8} = 0$ in filtration $9$.
The only possible source is $e_0 A'$.
\end{proof}

\begin{lemma}
\label{lem:2-h1h4h6}
\revdeg{79, 3, 41}
There is no hidden $2$ extension on $h_1 h_4 h_6$.
\end{lemma}

\begin{proof}
Table \ref{tab:Toda} shows that $h_1 h_4 h_6$ detects the Toda bracket
$\langle \eta_4, 2, \theta_5 \rangle$.
Now shuffle to obtain
\[
2 \langle \eta_4, 2, \theta_5 \rangle =
\langle 2, \eta_4, 2 \rangle \theta_5,
\]
which equals $\tau \eta \eta_4 \theta_5$ by Table \ref{tab:Toda}.
We will show that this product is zero.

There are several elements in the Adams $E_\infty$-page
that might detect $\eta_4 \theta_5$.  The possibilities
$h_1 h_6 d_0$ and $x_{78,9}$ are ruled out by comparison to $C\tau$.
The possibility $\tau e_0 A'$ is ruled out because
Table \ref{tab:2-extn} shows that $e_0 A'$ supports a hidden
$2$ extension.

Two possibilities remain.  If $\eta_4 \theta_5$ is detected
by $\tau M \D h_1^2 h_3$, then $\tau \eta \eta_4 \theta_5$
must be zero because there are no elements in sufficiently
high Adams filtration.

Finally, suppose that $\eta_4 \theta_5$ is detected by
$\tau h_1^2 x_{76,6}$.  Let $\alpha$ be an element of $\pi_{77,41}$
that is detected by $h_1 x_{76,6}$.
If $\eta_4 \theta_5 + \tau \eta \alpha$ is not zero,
then it is detected in
higher filtration.  
\rev{It cannot be detected by $x_{78,9}$ by comparison to $C\tau$, and}
it cannot be detected by $\tau e_0 A'$
because of the hidden $2$ extension on $e_0 A'$.
If it is detected by $\tau M \D h_1^2 h_3$, then we may
change the choice of $\alpha$ to ensure that
$\eta_4 \theta_5 + \tau \eta \alpha$ is zero.

We have now shown that $\tau \eta \eta_4 \theta_5$
equals $\tau^2 \eta^2 \alpha$.
Shuffle to obtain
\[
\tau^2 \eta^2 \alpha =
\tau \alpha \langle 2, \eta, 2 \rangle =
\langle \alpha, 2, \eta \rangle 2 \tau.
\]
\rev{Here we are using that $2 \alpha$ is zero; the possible
$2$ extensions on $h_1 x_{76,6}$ are easily eliminated by the presence
of $\eta$ extensions and by comparison to $\mmf$.}

Table \ref{tab:Toda} shows that
$\langle \alpha, 2, \eta \rangle$ is detected
by $h_0 h_2 x_{76,6}$, and Lemma \ref{lem:2-h0h2x76,6}
shows that this element does not support a hidden $2$ extension.
Therefore,
$\langle \alpha, 2, \eta \rangle 2 \tau$ is zero.
\end{proof}

\begin{lemma}
\label{lem:2-h0h2x76,6}
\revdeg{79, 8, 42}
There is no hidden $2$ extension on $h_0 h_2 x_{76,6}$.
\end{lemma}

\begin{proof}
Let $\alpha$ be an element of $\pi_{76,40}$ that is detected by
$h_4 A$.  Then $\nu \alpha$ is detected by $h_0 h_2 x_{76,6}$,
\rev{and $2 \alpha$ is detected by $h_0 h_4 A$.
We will show that $2 \nu \alpha$ is zero.}

Table \ref{tab:misc-extn} shows that
$h_0 h_4 A$ also detects 
\rev{either $\sigma^2 \theta_5$ or $\sigma^2 \theta_5 + \tau^2 \kappa_1 \kappabar_2$.
Then
$2 \alpha + \sigma^2 \theta_5$ or $2 \alpha + \sigma^2 \theta_5 + \tau^2 \kappa_1 \kappabar_2$ could be detected in higher filtration.
However, only $x_{76,9}$ could detect this error term, and
inclusion of the bottom cell into $C\tau$ rules it out.

We now know that
$\sigma^2 \theta_5 + 2 \alpha$ is either zero or
$\tau^2 \kappa_1 \kappabar_2$.
Multiply by $\nu$ to conclude that
$2 \nu \alpha$ is either zero or $\tau^2 \nu \kappa_1 \kappabar_2$.
As in the proof of Lemma \ref{lem:eta-h1h6d0}, this last
expression is also zero.
}
\end{proof}

\begin{lemma}
\label{lem:2-Ph6c0}
\revdeg{79, 8, 41}
There is no hidden $2$ extension on $P h_6 c_0$.
\end{lemma}

\begin{proof}
Table \ref{tab:misc-extn} shows that
$P h_6 c_0$ detects the product $\rho_{15} \eta_6$.
Table \ref{tab:Toda} shows that 
$\eta_6$ is contained in the Toda bracket
$\langle \eta, 2, \theta_5 \rangle$.
Now shuffle to obtain
\[
2 \rho_{15} \eta_6 =
2 \rho_{15} \langle \eta, 2, \theta_5 \rangle =
\rho_{15} \theta_5 \langle 2, \eta, 2 \rangle,
\]
which equals $\tau \eta^2 \rho_{15} \theta_5$ by 
Table \ref{tab:Toda}.

Table \ref{tab:misc-extn} shows that
$\rho_{15} \theta_5$ is detected by either
$h_0 x_{77,7}$ or $\tau^2 m_1$.
First suppose that it is detected by $h_0 x_{77,7}$.
Table \ref{tab:2-extn} shows that $h_0 x_{77,7}$ is the target
of a $2$ extension.
Then $\rho_{15} \theta_5$ equals $2 \alpha$ modulo
higher filtration.
In any case,
$\tau \eta^2 \rho_{15} \theta_5$ is zero.

Next suppose that $\rho_{15} \theta_5$ is detected by
$\tau^2 m_1$.
Then $\rho_{15} \theta_5$ equals $\tau^2 \alpha$ modulo
higher filtration for some element $\alpha$ detected by $m_1$.
Table \ref{tab:tau-extn}
shows that there is a hidden $\tau$ extension from $h_1 m_1$
to $M \D h_1^2 h_3$.  This implies that
$\tau \eta \alpha$ is detected by $M \D h_1^2 h_3$.
Finally,
$\tau^3 \eta^2 \alpha = \tau \eta^2 \rho_{15} \theta_5$
must be zero.
\end{proof}

\begin{lemma}
\label{lem:2-DB6}
\revdeg{79, 11, 42}
There is no hidden $2$ extension on $\D B_6$.
\end{lemma}

\begin{proof}
Table \ref{tab:eta-extn} shows that there is a hidden $\eta$ extension
from $h_0^6 h_4 h_6$ to $\tau \D B_6$.
Therefore, $\tau \D B_6$ cannot be the source of a hidden $2$
extension, so there cannot be a hidden $2$ extension
from $\D B_6$ to $\tau^2 M e_0^2$.
\end{proof}

\rev{
\begin{lemma}
\label{lem:2-h5^2g}
\revdeg{82, 6, 44}
There is no hidden $2$ extension on $h_5^2 g$.
\end{lemma}

\begin{proof}
The element $\tau h_5^2 g$ detects the product
$\kappabar \theta_5$, so it cannot support a hidden $2$
extension since $2 \theta_5$ is zero.

If there were a hidden $2$ extension from
$h_5^2 g$ to $\tau (\D e_1 + C_0) g$, then
the hidden 
$\tau$ extension from $\tau (\D e_1 + C_0) g$ to $\D^2 h_2 n$
would imply that
there is a hidden $2$ extension from 
$\tau h_5^2 g$ to $\D^2 h_2 n$.
\end{proof}
}

\begin{lemma}
\label{lem:2-h2^2x76,6}
\revdeg{82, 8, 44}
There is no hidden $2$ extension on $h_2^2 x_{76,6}$.
\end{lemma}

\begin{proof}
As in the proof of Lemma \ref{lem:2,eta,h2x76,6},
let $\alpha$ be an element in $\pi_{79,42}$ that is detected by
$h_2 x_{76,6}$ such that $\eta \alpha$ is zero.
Then $\nu \alpha$ is detected by $h_2^2 x_{76,6}$, and we wish to show
that $2 \nu \alpha$ is zero.

Table \ref{tab:Toda} shows that $2 \nu$ is contained in
$\langle \eta, 2, \eta \rangle$.
Consider the shuffle
\[
2 \nu \alpha = \langle \eta, 2, \eta \rangle \alpha =
\eta \langle 2, \eta, \alpha \rangle.
\]
Table \ref{tab:Toda} shows that the last Toda bracket is zero.
\end{proof}

\begin{lemma}
\label{lem:2-h0^2h6g}
\revdeg{83, 7, 44}
There is no hidden $2$ extension on $h_0^2 h_6 g$.
\end{lemma}

\begin{proof}
The element $h_0^2 h_6 g$ equals
$h_2^2 h_6 d_0$, so it detects $\nu^2 \{h_6 d_0\}$.
\end{proof}

\begin{lemma}
\label{lem:2-th2h4Q3}
\revdeg{85, 7, 45}
There is no hidden $2$ extension on $\tau h_2 h_4 Q_3$.
\end{lemma}

\begin{proof}
There cannot be a hidden $2$ extension from $\tau h_2 h_4 Q_3$
to $\tau P h_1 x_{76,6}$ because there is no hidden
$\tau$ extension from $h_0 h_2 h_4 Q_3$ to $P h_1 x_{76,6}$.

Table \ref{tab:eta-extn} shows that 
$\tau^3 M g^2$ supports a hidden $\eta$ extension.  Therefore, it
cannot be the target of a hidden $2$ extension.
\end{proof}

\begin{lemma}
\label{lem:2-h6hc0d0}
\mbox{}
\begin{enumerate}
\item
\revdeg{85, 8, 45}
There is no hidden $2$ extension on $h_6 c_0 d_0$.
\item
\revdeg{85, 9, 44}
There is no hidden $2$ extension on $P h_6 d_0$.
\end{enumerate}
\end{lemma}

\begin{proof}
Table \ref{tab:eta-extn} shows that both elements
are targets of hidden $\eta$ extensions.
\end{proof}

\begin{lemma}
\label{lem:2-h4h6c0}
\revdeg{86, 5, 45}
There is no hidden $2$ extension on $h_4 h_6 c_0$.
\end{lemma}

\begin{proof}
Table \ref{tab:misc-extn} shows that
$h_4 h_6 c_0$ detects the product $\sigma \{h_1 h_4 h_6\}$,
and the element $h_1 h_4 h_6$ does not support a hidden $2$ extension
by Lemma \ref{lem:2-h1h4h6}.
\end{proof}

\begin{lemma}
\label{lem:2-th2gC'}
\revdeg{86, 12, 47}
There is no hidden $2$ extension on $\tau h_2 g C'$.
\end{lemma}

\begin{proof}
The possible target $\tau^3 e_0^3 m$ is ruled out by comparison
to $\mmf$.  The possible target $P h_1^7 h_6$ is ruled out by 
comparison to $C\tau$.

It remains to eliminate the possible target $\tau^2 M h_1 g^2$.
Table \ref{tab:tau-extn} shows that there are hidden $\tau$
extensions from
$\tau h_2 g C'$ and $\tau^2 M h_1 g^2$ to
$\D^2 h_2^2 d_1$ and $M \D h_0^2 e_0$ respectively.
However, there is no hidden $2$ extension from
$\D^2 h_2^2 d_1$ to $M \D h_0^2 e_0$, so there cannot be a
$2$ extension from $\tau h_2 g C'$ to $\tau^2 M h_1 g^2$.
\end{proof}

\begin{lemma}
\label{lem:2-h1^2c3}
\revdeg{87, 5, 46}
There is no hidden $2$ extension on $h_1^2 c_3$.
\end{lemma}

\begin{proof}
Table \ref{tab:Toda} shows that the Toda
bracket $\langle \tau \{h_0 Q_3 + h_0 n_1\}, \nu_4, \eta \rangle$
is detected by $h_1^2 c_3$.
Shuffle to obtain
\[
\langle \tau \{h_0 Q_3 + h_0 n_1\}, \nu_4, \eta \rangle 2 =
\tau \{h_0 Q_3 + h_0 n_1\} \langle \nu_4, \eta, 2 \rangle.
\]
These expressions have no indeterminacy because
$\tau \{h_0 Q_3 + h_0 n_1\}$ does not support a $2$ extension.
Finally, the bracket $\langle \nu_4, \eta, 2 \rangle$
contains zero by comparison to $\tmf$.
\end{proof}

\begin{lemma}
\label{lem:2-P^2h6c0}
\revdeg{87, 12, 45}
There is no hidden $2$ extension on $P^2 h_6 c_0$.
\end{lemma}

\begin{proof}
Table \ref{tab:misc-extn} shows that $P^2 h_6 c_0$ detects
the product $\rho_{23} \eta_6$.
Table \ref{tab:Toda} shows that $\eta_6$ is contained in the
Toda bracket $\langle \eta, 2, \theta_5 \rangle$.
Shuffle to obtain
\[
2 \rho_{23} \eta_6 =
2 \rho_{23} \langle \eta, 2, \theta_5 \rangle =
\langle 2, \eta, 2 \rangle \rho_{23} \theta_5 =
\tau \eta^2 \rho_{23} \theta_5.
\]
The product $\rho_{23} \theta_5$ is detected in Adams
filtration at least 13, and then $\tau \eta^2 \rho_{23} \theta_5$
is detected in filtration at least 16.  This rules out all possible
targets for a hidden $2$ extension on $P^2 h_6 c_0$.
\end{proof}

\begin{lemma}
\label{lem:2-M^2}
\revdeg{90, 12, 48}
There is no hidden $2$ extension on $M^2$.
\end{lemma}

\begin{proof}
Table \ref{tab:misc-extn} shows that 
$M^2$ detects $\theta_{4.5}^2$.
Graded commutativity implies that $2 \theta_{4.5}^2$ is zero.
\end{proof}

\section{Hidden $\eta$ extensions}
\label{sctn:eta-extn}

\begin{thm}
\label{thm:eta-extn}
Tables \ref{tab:eta-extn} \rev{and \ref{tab:eta-extn-null}} list some hidden extensions by $\eta$.
\end{thm}

\begin{proof}
Many of the hidden extensions follow by comparison to 
$C\tau$.  For example, there is a hidden $\eta$ extension
from $\tau h_1 g$ to $c_0 d_0$ in the Adams spectral
sequence for $C\tau$.
Pulling back along inclusion of the bottom cell into $C\tau$,
there must also be a hidden $\eta$ extension from $\tau h_1 g$
to $c_0 d_0$ in the Adams spectral sequence for the sphere.
This type of argument is indicated by the notation
$C\tau$ in the fourth column of Table \ref{tab:eta-extn}.

Next, Table \ref{tab:tau-extn} shows a hidden $\tau$
extension from $c_0 d_0$ to $P d_0$.
Therefore, there is also a hidden $\eta$ extension from
$\tau^2 h_1 g$ to $P d_0$.
This type of argument is indicated by the notation $\tau$
in the fourth column of Table \ref{tab:eta-extn}.

The proofs of several of the extensions in Table \ref{tab:eta-extn}
rely on analogous extensions in $\mmf$.  
Extensions in $\mmf$ have not been rigorously analyzed \cite{Isaksen18}.
However, the specific extensions from $\mmf$ that we need
are easily deduced from extensions in $\tmf$, together
with the multiplicative structure.
For example, there is a hidden $\eta$ extension in $\tmf$
from $a n$ to $\tau d_0^2$.
Therefore, there is a hidden $\eta$ extension in $\mmf$
from $a n g$ to $\tau d_0^2 g$, and also a hidden
$\eta$ extension from $\D h_2^2 e_0$ to $\tau d_0 e_0^2$
in the homotopy groups of the sphere spectrum.
Note that $\mmf$ really is required here, since 
$a n g$ and $d_0^2 g$ equal
zero in the homotopy of $\tmf$.

Many cases require more complicated arguments.  In stems up to
approximately dimension 62, see 
\cite{Isaksen14c}*{Section 4.2.3 and Tables 29--30} and
\cite{WangXu18}.
The higher-dimensional cases are handled in the following lemmas.
\end{proof}

\begin{remark}
\label{rem:eta-tC}
The hidden $\eta$ extension from $\tau C$ to $\tau^2 g n$
is proved in \cite{WangXu18}, which uses on the
``$\R P^\infty$-method" to establish a hidden
$\sigma$ extension from $\tau h_3 d_1$ to $\D h_2 c_1$
and a hidden $\eta$ extension from $\tau h_1 g_2$ to $\D h_2 c_1$.
We now have easier proofs for these $\eta$ and $\sigma$ extensions,
using the hidden $\tau$ extension from $h_1^2 g_2$ to $\D h_2 c_1$
given in Table \ref{tab:tau-extn}, as well as the relation
$h_3^2 d_1 = h_1^2 g_2$.
\end{remark}

\rev{
\begin{remark}
\label{rem:eta-tDj1+t^2gC'}
If $h_1 f_2$ survives, then there is a hidden
$\tau$ extension from $\D h_1 j_1$ to $M \D h_1 d_0$.
It follows that there must be a hidden
$\eta$ extension from $\tau \D j_1 + \tau^2 g C'$ to
$M \D h_1 d_0$.
\end{remark}
}

\begin{remark}
The last column of Table \ref{tab:eta-extn}
indicates the crossing $\eta$ extensions.
\end{remark}

\begin{thm}
\label{thm:eta-extn-possible}
Table \ref{tab:eta-extn-possible}
lists all unknown hidden $\eta$ extensions, through
the $90$-stem. 
\end{thm}

\begin{proof}
Many possible extensions can be eliminated by
comparison to $C\tau$, to $\tmf$, or to $\mmf$.
For example,
there cannot be a hidden $\eta$ extension
from $\tau M d_0$ to $\tau^4 g^3$ because
$\tau^4 g^3$ maps to a non-zero element
in $\pi_{60} \tmf$ that is not divisible by $\eta$.

Other possibilities are eliminated by consideration of other
parts of the multiplicative structure.  For example,
there cannot be a hidden $\eta$ extension whose target supports
a multiplication by $2$, since $2 \eta$ equals zero.

Many cases are eliminated by more complicated arguments.
These are handled in the following lemmas.
\end{proof}

\rev{
\begin{remark}
\label{rem:eta-h2D3}
There is no hidden $\eta$ extension on $h_2 D_3$.
The possible target $\tau k_1$ is eliminated by computer
data recently produced by Dexter Chua on $d_2$ differentials
the Adams spectral sequence for the cofiber of $\eta$.
\end{remark}
}

\begin{remark}
\label{rem:eta-tgD3}
There is a hidden $\tau$ extension from
$\tau (\D e_1 + C_0) g$ to $\D^2 h_2 n$.
The possible extension from $\tau g D_3$ to 
$\tau (\D e_1 + C_0) g$ occurs if and only if
the possible extension from $\tau^2 g D_3$ to
$\D^2 h_2 n$ occurs.
\end{remark}

\rev{
\begin{remark}
\label{rem:eta-to-gQ3}
Computer
data recently produced by Dexter Chua on $d_2$ differentials in
the Adams spectral sequence for the cofiber of $\eta$
shows that the classical element $g Q_3$ must be the target
of a hidden $\eta$ extension.  Therefore,
there is either a hidden $\eta$ extension from
$h_1^2 f_2$ to $\tau g Q_3$,
or from $\tau h_1 x_{85,6}$ to $\tau^2 g Q_3$.
\end{remark}
}

\begin{lemma}
\label{lem:eta-th1Q2}
\revdeg{58, 8, 30}
There is no hidden $\eta$ extension on $\tau h_1 Q_2$.
\end{lemma}

\begin{proof}
There cannot be a hidden $\eta$ extension from
$\tau h_1 Q_2$ to $\tau^2 \D h_1 d_0 g$ by comparison to $\tmf$.
It remains to show that there cannot be a hidden
$\eta$ extension from $\tau h_1 Q_2$ to $\tau M d_0$.

Note that $h_1 d_0 Q_2 = \tau^3 d_1 g^2$,
so $\kappa \{h_1 Q_2\}$ is detected by
$\tau^3 d_1 g^2$.  Therefore,
$\kappa \{h_1 Q_2\} + \tau \kappa_1 \kappabar^2$
is detected in higher filtration.
The only possibility is $\tau^3 e_0 g m$, but that cannot occur
by comparison to $\mmf$.
Therefore,
$\kappa \{h_1 Q_2\} + \tau \kappa_1 \kappabar^2$ is zero.

Now $\tau \eta \kappa_1 \kappabar^2$ is zero because
$\tau \eta \kappa_1 \kappabar$ cannot be detected by
$\D h_1 d_0^2$ by comparison to $\tmf$.
Therefore,
$\eta \kappa \{h_1 Q_2\}$ is zero,
so $\tau \eta \kappa \{h_1 Q_2\}$ is also zero.

On the other hand, $\tau \kappa \{M d_0\}$ is non-zero
because it is detected by $\tau M d_0^2$.
Therefore $\tau \eta \{h_1 Q_2\}$ cannot be detected
by $\tau M d_0$.
\end{proof}

\begin{lemma}
\label{lem:eta-th1^2h5^2}
\revdeg{64, 4, 33}
There is no hidden $\eta$ extension on $\tau h_1^2 h_5^2$.
\end{lemma}

\begin{proof}
Table \ref{tab:2-extn} shows that $\tau h_1^2 h_5^2$ is the
target of a hidden $2$ extension.
\end{proof}

\begin{lemma}
\label{lem:eta-t^2h1X2}
\revdeg{64, 8, 33}
There is a hidden $\eta$ extension from $\tau^2 h_1 X_2$
to $\tau^2 M h_0 g$.
\end{lemma}

\begin{proof}
Table \ref{tab:eta-extn} shows that there is a hidden
$\eta$ extension from $\tau h_1 X_2$ to $c_0 Q_2$.
Since $c_0 Q_2$ does not support a hidden $\tau$ extension,
there exists an element $\beta$ in $\pi_{65,35}$ that is 
detected by $c_0 Q_2$ such that $\tau \beta = 0$.

Projection from $C\tau$ to the top cell takes
$\ol{c_0 Q_2}$ and $P(A+A')$ to $c_0 Q_2$ and $\tau M h_0 h_2 g$
respectively.  Since $h_2 \cdot \ol{c_0 Q_2} = P(A+A')$
in the Adams spectral sequence for $C\tau$, it follows that
$\nu \beta$ is non-zero and detected by $\tau M h_0 h_2 g$.

Let $\alpha$ be an element of $\pi_{63,33}$ that is detected
by $\tau X_2 + \tau C'$, and
consider the sum $\eta^2 \alpha + \beta$.
Both terms are detected by $c_0 Q_2$, but the sum could be detected
in higher filtration.  In fact, the sum is non-zero because
$\nu (\eta^2 \alpha + \beta)$ is non-zero.

It follows that $\eta^2 \alpha + \beta$ is detected by $\tau M h_0 g$,
and that $\tau \eta^2 \alpha$ is detected by $\tau^2 M h_0 g$.
\end{proof}

\begin{lemma}
\label{lem:eta-th1^3h6}
\revdeg{66, 4, 34}
There is no hidden $\eta$ extension on $\tau h_1^3 h_6$.
\end{lemma}

\begin{proof}
The element $\tau \eta^2 \eta_6$ is detected by
$\tau h_1^3 h_6$.  Table \ref{tab:Toda} shows that
$\eta_6$ is contained in the Toda bracket
$\langle \eta, 2, \theta_5 \rangle$
Now shuffle to obtain
\[
\eta \cdot \tau \eta^2 \eta_6 = 4 \nu \eta_6 = 
4 \nu \langle \eta, 2, \theta_5 \rangle =
4 \langle \nu, \eta, 2 \rangle \theta_5,
\]
which equals zero because $2 \theta_5$ is zero.
\end{proof}

\begin{lemma}
\label{lem:eta-tD1h3^2}
\revdeg{66, 6, 35}
There is no hidden $\eta$ extension from
$\tau \D_1 h_3^2$ to $h_2^2 A'$.
\end{lemma}

\begin{proof}
Table \ref{tab:nu-extn} shows that $h_2^2 A'$
supports a hidden $\nu$ extension, so it cannot be the
target of a hidden $\eta$ extension.
\end{proof}

\begin{lemma}
\label{lem:eta-th1Q3}
\revdeg{68, 6, 36}
There is no hidden $\eta$ extension on $\tau h_1 Q_3$.
\end{lemma}

\begin{proof}
Table \ref{tab:nu-extn} shows that $\tau h_1 Q_3$
is the target of a hidden $\nu$ extension.
Therefore, it cannot be the source of a hidden $\eta$
extension.
\end{proof}

\begin{lemma}
\label{lem:eta-h3A'}
\revdeg{68, 7, 36}
There is a hidden $\eta$ extension from $h_3 A'$ to
$h_3 (\D e_1 + C_0)$.
\end{lemma}

\begin{proof}
Comparison to $C\tau$ shows that there is a hidden
$\eta$ extension from $h_3 A'$ to either
$\tau h_2^2 C' + h_3(\D e_1 + C_0)$ or
$h_3 (\D e_1 + C_0)$.
Table \ref{tab:nu-extn} shows that
$\tau h_2^2 C' + h_3(\D e_1 + C_0)$ supports a hidden
$\nu$ extension.  Therefore, it cannot be the target of a
hidden $\eta$ extension.
\end{proof}

\begin{lemma}
\label{lem:eta-h2^2h6}
\revdeg{69, 3, 36}
There is a hidden $\eta$ extension from $h_2^2 h_6$ to
$\tau h_0 h_2 Q_3$.
\end{lemma}

\begin{proof}
Table \ref{tab:Massey} gives the Massey product
$h_0 h_2 = \langle h_1, h_0, h_1 \rangle$.
Therefore,
\[
\langle \tau h_1 Q_3, h_0, h_1 \rangle = 
\{ \tau h_0 h_2 Q_3, \tau h_0 h_2 Q_3 + \tau h_1 h_3 H_1 \}.
\]
Table \ref{tab:nu-extn} shows that
there is a hidden $\nu$ extension from
$h_2 h_5^2$ to $\tau h_1 Q_3$, so
$\nu^2 \theta_5$ is detected by $\tau h_1 Q_3$.
Therefore,
the Toda bracket $\langle \nu^2 \theta_5, 2, \eta \rangle$
is detected by $\tau h_0 h_2 Q_3$ or by 
$\tau h_0 h_2 Q_3 + \tau h_1 h_3 H_1$.

Now $\langle \nu^2 \theta_5, 2, \eta \rangle$
contains $\nu^2 \langle \theta_5, 2, \eta \rangle$. This 
expression 
equals $\nu \theta_5 \langle 2, \eta, \nu \rangle$, which
equals zero because $\langle 2, \eta, \nu \rangle$ is contained
in $\pi_{5,3} = 0$.

We now know that $\langle \nu^2 \theta_5, 2, \eta \rangle$
equals its own determinacy, so
$\tau h_0 h_2 Q_3$ or $\tau h_0 h_2 Q_3 + \tau h_1 h_3 H_1$
detects a multiple of $\eta$.
The only possibility is that there is a hidden $\eta$
extension on $h_2^2 h_6$.

The target of this extension cannot be
$\tau h_0 h_2 Q_3 + \tau h_1 h_3 H_1$ by comparison 
to $C\tau$.
\end{proof}

\begin{lemma}
\label{lem:eta-h0^3h3h6}
\revdeg{70, 5, 36}
There is no hidden $\eta$ extension on $h_0^3 h_3 h_6$.
\end{lemma}

\begin{proof}
There are several possible targets for a hidden
$\eta$ extension on $h_0^3 h_3 h_6$.
The element $\tau \D^2 h_2 g$ is ruled out
because it supports an $h_2$ extension.
The element $\D^2 h_4 c_0$ is ruled out by comparison to $C\tau$.
The elements $\tau h_3^2 Q_2$ and $\tau d_0 Q_2$ are ruled out
because Table \ref{tab:eta-extn}
shows that they are targets of hidden $\eta$ extensions from
$\tau^2 h_1 h_3 H_1$ and $\tau^2 h_1 D'_3$ respectively.

The only remaining possibility is $\tau^2 l_1$.  This case is
more complicated.

Table \ref{tab:Toda} shows that $h_0^3 h_3 h_6$ detects the Toda bracket
$\langle 8\sigma, 2, \theta_5 \rangle$.
Now shuffle to obtain
\[
\eta \langle 8\sigma, 2, \theta_5 \rangle =
\langle \eta, 8 \sigma, 2 \rangle \theta_5.
\]
Table \ref{tab:Toda} shows that
$\langle \eta, 8 \sigma, 2 \rangle$ contains $\mu_9$
and has indeterminacy generated by
$\tau \eta^2 \sigma$ and $\tau \eta \epsilon$.
Thus the expression
$\langle \eta, 8 \sigma, 2 \rangle \theta_5$ contains
at most four elements.

The product $\mu_9 \theta_5$ is detected in filtration at least $8$,
so it is not detected by $\tau^2 l_1$.
The product $(\mu_9 + \tau \eta^2 \sigma) \theta_5$
is detected by $\tau h_1^2 p'$ because
Table \ref{tab:misc-extn}
shows that there is a hidden $\sigma$ extension
from $h_5^2$ to $p'$.
The product
$(\mu_9 + \tau \eta \epsilon) \theta_5$ is also 
detected by $\tau h_1^2 p' = \tau h_1 h_5^2 c_0$.
Finally, the product
$(\mu_9 + +\tau \eta^2 \sigma + \tau \eta \epsilon) \theta_5$
equals $(\mu_9 + \tau \nu^3) \theta_5$, which also
must be detected in filtration at least $8$. 
\end{proof}

\begin{lemma}
\label{lem:eta-h2Q3}
\revdeg{70, 6, 38}
There is no hidden $\eta$ extension on $h_2 Q_3$.
\end{lemma}

\begin{proof}
There cannot be a hidden $\eta$ extension on $\tau h_2 Q_3$
because it is a multiple of $h_2$.  Therefore, the possible
targets for an $\eta$ extension on $h_2 Q_3$ must be annihilated
by $\tau$.

The element $h_1^3 h_3 H_1$ cannot be the target because
Table \ref{tab:tau-extn} shows that it supports a hidden
$\tau$ extension.
The element $\tau M h_2^2 g$ cannot be the target because
Table \ref{tab:nu-extn} shows that it supports a hidden
$\nu$ extension to $M h_1 d_0^2$.
\end{proof}

\begin{lemma}
\label{lem:eta-th1h3H1}
\revdeg{70, 7, 37}
There is a hidden $\eta$ extension from $\tau h_1 h_3 H_1$ to
$h_3^2 Q_2$.
\end{lemma}

\begin{proof}
Table \ref{tab:eta-extn} shows that there is a hidden $\eta$ extension
from $\tau h_1 H_1$ to $h_3 Q_2$.  Now multiply by $h_3$.
\end{proof}

\begin{lemma}
\label{lem:eta-h1h3(De1+C0)}
\mbox{}
\begin{enumerate}
\item
\revdeg{70, 10, 38}
There is no hidden $\eta$ extension on $h_1 h_3 (\D e_1 + C_0)$.
\item
\revdeg{70, 10, 38}
There is no hidden $\eta$ extension on 
$\tau h_2 C'' + h_1 h_3 (\D e_1 + C_0)$.
\end{enumerate}
\end{lemma}

\begin{proof}
The element $\tau M h_2^2 g$ is the 
only possible target for such hidden $\eta$ extensions.
However, Table \ref{tab:nu-extn} shows that there is a hidden
$\nu$ extension from $\tau M h_2^2 g$ to $M h_1 d_0^2$.
\end{proof}

\begin{lemma}
\label{lem:eta-th1^2p'}
\revdeg{71, 6, 37}
There is no hidden $\eta$ extension on $\tau h_1^2 p'$.
\end{lemma}

\begin{proof}
The element $\tau h_1^2 p'$ detects 
$\tau \eta^2 \sigma \theta_5$ because Table \ref{tab:misc-extn}
shows that there is a hidden $\sigma$ extension from
$h_5^2$ to $p'$.
Then $\tau \eta^3 \sigma \theta_5$ is zero since
$\tau \eta^3 \sigma$ is zero.
\end{proof}

\begin{lemma}
\label{lem:eta-h1^2h3H1}
\revdeg{71, 8, 39}
There is no hidden $\eta$ extension on $h_2^3 H_1$.
\end{lemma}

\begin{proof}
Table \ref{tab:misc-extn} shows that $M d_0$ detects the product
$\kappa \theta_{4.5}$.  Then
Table \ref{tab:Toda} shows that $h_2^3 H_1$ detects the
Toda bracket $\langle \nu, \epsilon, \kappa \theta_{4.5} \rangle$.
Now shuffle to obtain
\[
\eta \langle \nu, \epsilon, \kappa \theta_{4.5} \rangle =
\langle \eta, \nu, \epsilon \rangle \kappa \theta_{4.5},
\]
which is zero because $\langle \eta, \nu, \epsilon \rangle$
is contained in $\pi_{13,8} = 0$.
\end{proof}

\begin{lemma}
\label{lem:eta-th1h6c0}
\revdeg{72, 5, 37}
There is a hidden $\eta$ extension from $\tau h_1 h_6 c_0$ 
to $\tau^2 h_2^2 Q_3$.
\end{lemma}

\begin{proof}
The hidden $\tau$ extension from $h_1^2 h_6 c_0$ to
$h_0 d_0 D_2$ implies that $\tau h_1 h_6 c_0$ must support a hidden
$\eta$ extension.  However, this hidden $\tau$ extension crosses
the $\tau$ extension from $\tau h_2^2 Q_3$ to $\tau^2 h_2^2 Q_3$.
Therefore, the target of the hidden $\eta$ extension is either
$\tau^2 h_2^2 Q_3$ or $h_0 d_0 D_2$.

The element $\tau h_1 h_6 c_0$ detects the product 
$\tau \eta_6 \epsilon$, so we want to compute 
$\tau \eta \eta_6 \epsilon$.
Table \ref{tab:Toda} shows that 
$\eta_6$ belongs to $\langle \theta_5, 2, \eta \rangle$.
Shuffle to obtain
\[
\tau \eta \eta_6 \epsilon = 
\langle \theta_5, 2, \eta \rangle \tau \eta \epsilon =
\theta_5 \langle 2, \eta, \tau \eta \epsilon \rangle.
\]
Table \ref{tab:Toda} shows that
$\langle 2, \eta, \tau \eta \epsilon \rangle$
contains $\zeta_{11}$.
Finally, $\theta_5 \zeta_{11}$ is detected by
$\tau^2 h_2^2 Q_3 = h_5^2 \cdot P h_2$.
\end{proof}

\begin{lemma}
\label{lem:eta-h1^3p'}
\revdeg{72, 7, 39}
There is no hidden $\eta$ extension on $h_1^3 p'$.
\end{lemma}

\begin{proof}
The element $h_1^3 p'$ does not support a hidden $\tau$ extension,
while Table \ref{tab:tau-extn} shows that
there is a hidden $\tau$ extension from $\tau h_2^2 C''$ to
$\D^2 h_1^2 h_4 c_0$.  Therefore, there cannot be a hidden
$\eta$ extension from $h_1^3 p'$ to $\tau h_2^2 C''$.
\end{proof}

\begin{lemma}
\label{lem:eta-h0d0D2}
\revdeg{72, 11, 38}
There is a hidden $\eta$ extension from $h_0 d_0 D_2$ to
$\tau M d_0^2$.
\end{lemma}

\begin{proof}
Table \ref{tab:Toda} shows that $h_0 d_0 D_2$ detects the
Toda bracket $\langle \tau \kappabar \theta_{4.5}, 2 \nu, \nu \rangle$.
Now shuffle to obtain
\[
\langle \tau \kappabar \theta_{4.5}, 2 \nu, \nu \rangle \eta =
\tau \kappabar \theta_{4.5} \langle 2 \nu, \nu, \eta \rangle.
\]
Table \ref{tab:Toda} shows that 
the Toda bracket $\langle 2 \nu, \nu, \eta \rangle$
contains $\epsilon$.
Finally, $\tau \kappabar \theta_{4.5} \epsilon$
is detected by $\tau M d_0^2$ because
Table \ref{tab:misc-extn} shows that there 
is a hidden $\epsilon$ extension from
$\tau M g$ to $M d_0^2$.
\end{proof}

\begin{lemma}
\label{lem:eta-h0h3d2}
\revdeg{75, 6, 40}
There is a hidden $\eta$ extension from $h_0 h_3 d_2$ 
to $\tau d_1 g_2$.
\end{lemma}

\begin{proof}
Table \ref{tab:Toda} shows that the Toda bracket
$\langle \eta, \sigma^2, \eta, \sigma^2 \rangle$
equals $\kappa_1$.
We would like to consider the shuffle
\[
\langle \eta, \sigma^2, \eta, \sigma^2 \rangle \tau \kappabar_2 =
\eta \langle \sigma^2, \eta, \sigma^2, \tau \kappabar_2 \rangle,
\]
but we must show that the Toda bracket
$\langle \eta, \sigma^2, \tau \kappabar_2 \rangle$ is well-defined and
contains zero.
It is well-defined because $\sigma^2 \kappabar_2$ is detected by
$h_3^2 g_2$ in $\pi_{58,32}$, and there are no $\tau$ extensions
on this group.
The bracket contains zero by comparison to $\tmf$, since
all non-zero elements of $\pi_{60,32}$ are detected
by $\tmf$.

We have now shown that $\tau \kappa_1 \kappabar_2$ is divisible
by $\eta$.  The only possibility is that 
there is a hidden $\eta$ extension from $h_0 h_3 d_2$ to
$\tau d_1 g_2$.
\end{proof}

\begin{lemma}
\label{lem:eta-tm1}
\mbox{}
\begin{enumerate}
\item
\revdeg{77, 3, 40}
There is no hidden $\eta$ extension on $h_3^2 h_6$.
\item
\revdeg{77, 7, 41}
There is no hidden $\eta$ extension on $\tau m_1$.
\end{enumerate}
\end{lemma}

\begin{proof}
Table \ref{tab:2-extn} shows that $e_0 A'$ and $\tau e_0 A'$ support hidden $2$ extensions, so they cannot be the targets of hidden
$\eta$ extensions.
\end{proof}

\begin{lemma}
\label{lem:eta-h0x77,7}
\revdeg{77, 8, 40}
There is no hidden $\eta$ extension on $h_0 x_{77,7}$.
\end{lemma}

\begin{proof}
Table \ref{tab:2-extn} shows that $h_0 x_{77,7}$ is the target
of a hidden $2$ extension.
\end{proof}

\begin{lemma}
\label{lem:eta-h1h6d0}
\revdeg{78, 6, 41}
There is no hidden $\eta$ extension on $h_1 h_6 d_0$.
\end{lemma}

\begin{proof}
Table \ref{tab:Toda} shows that $h_1 h_6$ detects the Toda bracket
$\langle \theta_5, 2, \eta \rangle$, so
$h_1 h_6 d_0$ detects
$\langle \theta_5, 2, \eta \rangle \kappa$.
Now shuffle to obtain
\[
\langle \theta_5, 2, \eta \rangle \eta \kappa = 
\theta_5 \langle 2, \eta, \eta \kappa \rangle.
\]
Table \ref{tab:Toda} shows that the Toda bracket
$\langle 2, \eta, \eta \kappa \rangle$ equals $\nu \kappa$.
Thus we need to compute the product
$\nu \kappa \theta_5$.
Lemma \ref{lem:sigma^2+kappa-h5^2} shows that this
product equals 
\rev{either $\nu \sigma^2 \theta_5$ or
$\nu (\sigma^2 \theta_5 + \tau^2 \kappa_1 \kappabar_2)$.
These expressions equal $0$ and $\tau^2 \nu \kappa_1 \kappabar_2$
respectively since $\nu \sigma = 0$.

It remains to show that $\tau^2 \nu \kappa_1 \kappabar_2$ is zero.
The proof of Lemma \ref{lem:eta-h0h3d2} shows that
$\tau \kappa_1 \kappabar_2$ is divisible by $\eta$.
Therefore, $\tau^2 \nu \kappa_1 \kappabar_2$ is zero
since $\eta \nu = 0$.
}
\end{proof}

\begin{lemma}
\label{lem:eta-th1^2x76,6}
\revdeg{78, 8, 41}
There is no hidden $\eta$ extension on $\tau h_1^2 x_{76,6}$.
\end{lemma}

\begin{proof}
Let $\alpha$ be an element of $\pi_{77,41}$ that is detected by
$h_1 x_{76,6}$.
Then $\tau h_1^2 x_{76,6}$ detects $\tau \eta \alpha$.
Now consider the shuffle
\[
\tau \eta^2 \alpha = \langle 2, \eta, 2 \rangle \alpha =
2 \langle \eta, 2, \alpha \rangle.
\]
Note that $2 \alpha$ is zero because there are no $2$ extensions
in $\pi_{77,41}$, so the second bracket is well-defined.

Finally, 
$2 \langle \eta, 2, \alpha \rangle$ must be zero because there are 
no $2$ extensions in $\pi_{79,42}$ in sufficiently high filtration.
\end{proof}

\rev{
\begin{lemma}
\label{lem:eta-h0^6h4h6}
\revdeg{78, 8, 40}
There is a hidden $\eta$ extension from
$h_0^6 h_4 h_6$ to $\tau \D B_6$.
\end{lemma}

\begin{proof}
In the homotopy of $C\tau$, there is a hidden $\eta$
extension from $h_0^6 h_4 h_6$ to $\D^2 n$.  Therefore,
$h_0^6 h_4 h_6$ must support a hidden $\eta$ extension
whose target lies in Adams filtration 13 or less.
However, $\D^2 n$ is not the target because it supports an
$h_2$ extension.  The only remaining possible
target is $\tau \D B_6$.
\end{proof}

}

\rev{
\begin{lemma}
\label{lem:eta-e0A'}
\revdeg{78, 10, 42}
There is a hidden $\eta$ extension from $e_0 A'$ to $\tau M e_0^2$.
\end{lemma}

\begin{proof}
The classical relation
$g C' = e_0 G_0$ implies the $\C$-motivic relation
$\tau g \cdot C' = e_0 \cdot \tau G_0$ modulo the possible
error term $\D j_1$.  The
error term does not appear because of $h_1^2$ extensions.

Table \ref{tab:Massey} shows that
$\tau G_0$ equals $\langle A', h_1, h_2 \rangle$.
Therefore, we have
\[
\tau g C' = e_0 \langle A', h_1, h_2 \rangle =
\langle e_0 A', h_1, h_2 \rangle.
\]
The second equality holds because there is no indeterminacy by 
inspection.

Let $\alpha$ be an element of $\pi_{78,44}$ that is detected by
$e_0 A'$.  If the product $\eta \alpha$ were zero, then the
Moss Convergence Theorem would imply that
$\tau g C'$ is a permanent cycle that detects the 
Toda bracket
$\langle \alpha, \eta, \nu \rangle$.
However, $\tau g C'$ supports a $d_4$ differential and does not survive.

We now know that $e_0 A'$ supports a hidden $\eta$ extension.
After ruling out $\tau^2 \D h_1 e_0^2 g$ by comparison to $\mmf$,
the only remaininng possible target is $\tau M e_0^2$.
\end{proof}
}

\begin{lemma}
\label{lem:eta-h2h4h6}
\revdeg{81, 3, 42}
If $h_2 h_4 h_6$ supports a hidden $\eta$ extension, then its
target is not $\tau h_2^2 x_{76,6}$.
\end{lemma}

\begin{proof}
Table \ref{tab:nu-extn} shows that
$\tau h_2^2 x_{76,6}$ supports a hidden $\nu$ extension, so
it cannot be the target of a hidden $\eta$ extension.
\end{proof}

\rev{
\begin{lemma}
\label{lem:eta-h1^3h4h6}
\revdeg{81, 5, 43}
There is no hidden $\eta$ extension on $h_1^3 h_4 h_6$.
\end{lemma}

\begin{proof}
The element $\tau h_1^3 h_4 h_6$ is a multiple of $h_0$, so
it cannot support a hidden $\eta$ extension.
This eliminates all possible targets except for $\tau (\D e_1 + C_0) g$.

However, 
$\tau (\D e_1 + C_0) g$ supports a hidden $\tau$ extension.
As in the previous paragraph, this eliminates
$\tau (\D e_1 + C_0) g$ as a possible target.
\end{proof}
}

\begin{lemma}
\label{lem:eta-h3^2n1}
\revdeg{81, 7, 44}
There is no hidden $\eta$ extension on $h_3^2 n_1$.
\end{lemma}

\begin{proof}
The element $\tau h_3^2 n_1 = h_3^2 (\tau Q_3 + \tau n_1)$
detects $\sigma^2 \{\tau Q_3 + \tau n_1\}$.  Then
$\eta \sigma^2 \{\tau Q_3 + \tau n_1\}$ is zero because
$\eta \sigma^2$ is zero.
\end{proof}

\begin{lemma}
\label{lem:eta-D^2p}
\revdeg{81, 12, 42}
There is no hidden $\eta$ extension on $\D^2 p$.
\end{lemma}

\begin{proof}
Table \ref{tab:nu-extn} shows that
$\D^2 p$ is the target of a hidden $\nu$ extension, so it cannot
be the source of an $\eta$ extension.
\end{proof}

\begin{lemma}
\label{lem:eta-h6c1}
\revdeg{82, 4, 43}
There is no hidden $\eta$ extension on $h_6 c_1$.
\end{lemma}

\begin{proof}
Table \ref{tab:Toda} shows that $h_6 c_1$ detects the Toda
bracket $\langle \sigmabar, 2, \theta_5 \rangle$.  By inspection,
all possible indeterminacy is in higher Adams filtration,
so $h_6 c_1$ detects every element of the Toda bracket.

Shuffle to obtain
\[
\eta \langle \sigmabar, 2, \theta_5 \rangle =
\langle \eta, \sigmabar, 2 \rangle \theta_5.
\]
The Toda bracket $\langle \eta, \sigmabar, 2 \rangle$
is detected in filtration at least $5$ since
the Massey product $\langle h_1, c_1, h_0 \rangle$
is zero.
Therefore, the Toda bracket equals $\{0, \eta \kappabar\}$.

We now know that $\eta \langle \sigmabar, 2, \theta_5 \rangle$
contains zero, and therefore $h_6 c_1$ does not support
a hidden $\eta$ extension.
\end{proof}

\begin{lemma}
\label{lem:eta-h0h6g}
\revdeg{83, 6, 44}
There is no hidden $\eta$ extension on $h_0 h_6 g$.
\end{lemma}

\begin{proof}
Table \ref{tab:Toda} shows that $h_0 h_6 g$ detects the Toda bracket
$\langle \nu, \eta, \eta_6 \kappa \rangle$.
Shuffle to obtain
\[
\eta \langle \nu, \eta, \eta_6 \kappa \rangle =
\langle \eta, \nu, \eta \rangle \eta_6 \kappa.
\]
Table \ref{tab:Toda} shows that $\langle \eta, \nu, \eta \rangle$
equals $\nu^2$.
Finally, 
\[
\nu^2 \eta_6 \kappa = \nu^2 \kappa \langle \eta, 2, \theta_5 \rangle =
\nu \theta_5 \kappa \langle \nu, \eta, 2 \rangle,
\]
which equals zero because $\langle \nu, \eta, 2 \rangle$ is contained
in $\pi_{5,3} = 0$.
\end{proof}

\begin{lemma}
\label{lem:eta-h2h4Q3}
\revdeg{85, 7, 46}
There is no hidden $\eta$ extension on $h_2 h_4 Q_3$.
\end{lemma}

\begin{proof}
We must eliminate $\tau h_2 g C'$ as a possible target.
One might hope to use the homotopy of $C\tau$ in order to do this,
but the homotopy of $C\tau$ has an $\eta$ extension in the
relevant degree that could possibly detect a hidden extension
from $h_2 h_4 Q_3$ to $\tau h_2 g C'$.

If there were a hidden $\eta$ extension from $h_2 h_4 Q_3$
to $\tau h_2 g C'$, then the hidden $\tau$ extension
from $\tau h_2 g C'$ to $\D^2 h_2^2 d_1$ would imply
that there is a hidden $\eta$ extension from
$\tau h_2 h_4 Q_3$ to $\D^2 h_2^2 d_1$.
However, $\tau h_2 h_4 Q_3$ detects the product
$\nu_4 \{\tau Q_3 + \tau n_1\}$, and $\eta \nu_4$ is zero.
Therefore,
$\tau h_2 h_4 Q_3$ cannot support a hidden $\eta$ extension.
\end{proof}

\begin{lemma}
\label{lem:eta-h1h6c0d0}
\mbox{}
\begin{enumerate}
\item
\revdeg{86, 9, 46}
There is no hidden $\eta$ extension on $h_1 h_6 c_0 d_0$.
\item
\revdeg{86, 10, 45}
There is no hidden $\eta$ extension on $P h_1 h_6 d_0$.
\end{enumerate}
\end{lemma}

\begin{proof}
Table \ref{tab:2-extn} shows that $h_1 h_6 c_0 d_0$ and
$P h_1 h_6 d_0$ are targets of hidden $2$ extensions, 
so they cannot be the sources of hidden $\eta$ extensions.
\end{proof}

\begin{lemma}
\label{lem:eta-h0^3h6i}
\revdeg{86, 11, 44}
There is a hidden $\eta$ extension from $h_0^3 h_6 i$ to
$\tau^2 \D^2 c_1 g$.
\end{lemma}

\begin{proof}
The Adams differential $d_2(\D^3 h_3^2) = \D^2 h_0^3 x$
implies that
$\tau^2 \D^2 c_1 g = h_1 \cdot \D^3 h_3^2$ detects the Toda bracket
$\langle \eta, 2, \{\D^2 h_0^2 x\} \rangle$.
However, the later Adams differential $d_5(h_0^2 h_6 i) = \D^2 h_0^2 x$
implies that $0$ belongs to $\{\D^2 h_0^2 x\}$.
Therefore,
$\tau^2 \D^2 c_1 g$ detects $\langle \eta, 2, 0 \rangle$,
so $\tau^2 \D^2 c_1 g$ detects a multiple of $\eta$.
The only possibility is that there is a hidden $\eta$ extension
from $h_0^3 h_6 i$ to $\tau^2 \D^2 c_1 g$.
\end{proof}

\begin{lemma}
\label{lem:eta-B6d1}
\revdeg{87, 11, 48}
There is no hidden $\eta$ extension on $B_6 d_1$.
\end{lemma}

\begin{proof}
Table \ref{tab:2-extn} shows that $B_6 d_1$ is the target of a 
hidden $2$ extension, so it cannot be the source of a hidden
$\eta$ extension.
\end{proof}

\begin{lemma}
\label{lem:eta-h1^2h4h6c0}
\revdeg{88, 7, 47}
There is no hidden $\eta$ extension on $h_1^2 h_4 h_6 c_0$.
\end{lemma}

\begin{proof}
Table \ref{tab:nu-extn} shows that $h_1^2 h_6 h_6 c_0$
is the target of a hidden $\nu$ extension, so it cannot support
a hidden $\eta$ extension.
\end{proof}

\begin{lemma}
\label{lem:eta-D^2h1f1}
\revdeg{89, 13, 47}
There is a hidden $\eta$ extension from $\D^2 h_1 f_1$ to
the element $\tau \D^2 h_2 c_1 g$.
\end{lemma}

\begin{proof}
The element $\tau \D^2 h_2 c_1 g$ detects the product
$\nu^2 \{\D^2 t\}$.
Table \ref{tab:Toda} shows that $\nu^2$ equals the Toda bracket
$\langle \eta, \nu, \eta \rangle$.
Shuffle to obtain
\[
\langle \eta, \nu, \eta \rangle \{\D^2 t\} =
\eta \langle \nu, \eta, \{\D^2 t\} \rangle.
\]
This shows that $\tau \D^2 h_2 c_1 g$ is the target of a hidden
$\eta$ extension.  The only possible source for this extension
is $\D^2 h_1 f_1$.
\end{proof}

\section{Hidden $\nu$ extensions}
\label{sctn:nu-extn}

\begin{thm}
\label{thm:nu-extn}
Tables \ref{tab:nu-extn} \rev{and \ref{tab:nu-extn-null}} list some hidden extensions by $\nu$.
\end{thm}

\begin{proof}
Many of the hidden extensions follow by comparison to 
$C\tau$.  For example, there is a hidden $\nu$ extension
from $ h_0^2 g$ to $h_1 c_0 d_0$ in the Adams spectral
sequence for $C\tau$.
Pulling back along inclusion of the bottom cell into $C\tau$,
there must also be a hidden $\nu$ extension from $h_0^2 g$
to $h_1 c_0 d_0$ in the Adams spectral sequence for the sphere.
This type of argument is indicated by the notation
$C\tau$ in the fourth column of Table \ref{tab:eta-extn}.

Next, Table \ref{tab:tau-extn} shows a hidden $\tau$
extension from $h_1 c_0 d_0$ to $P h_1 d_0$.
Therefore, there is also a hidden $\nu$ extension from
$\tau h_0^2 g$ to $P h_1 d_0$.
This type of argument is indicated by the notation $\tau$
in the fourth column of Table \ref{tab:eta-extn}.

Some extensions can be resolved by comparison to
$\tmf$ or to $\mmf$.
For example,
Table \ref{tab:unit-mmf} shows that 
the classical unit map $S \map \tmf$
takes $\{\D h_1 h_3\}$ in $\pi_{32}$ to a non-zero element
$\alpha$ of $\pi_{32} \tmf$ 
such that $\nu \alpha = \eta \kappa \kappabar$
in $\pi_{35} \tmf$.
Therefore, there must be a hidden $\nu$ extension
from $\D h_1 h_3$ to $\tau h_1 e_0^2$.

The proofs of several of the extensions in Table \ref{tab:nu-extn}
rely on analogous extensions in $\mmf$.  
Extensions in $\mmf$ have not been rigorously analyzed \cite{Isaksen18}.
However, the specific extensions from $\mmf$ that we need
are easily deduced from extensions in $\tmf$, together
with the multiplicative structure.
For example, there is a hidden $\nu$ extension in $\tmf$
from $\D h_1$ to $\tau d_0^2$.
Therefore, there is a hidden $\nu$ extension in $\mmf$
from $\D h_1 g$ to $\tau d_0^2 g$, and also a hidden
$\nu$ extension from $\tau \D h_1 g$ to $\tau^2 d_0 e_0^2$
in the homotopy groups of the sphere spectrum.
Note that $\mmf$ really is required here, since $d_0^2 g$ equals
zero in the homotopy of $\tmf$.

Many cases require more complicated arguments.  In stems up to
approximately dimension 62, see 
\cite{Isaksen14c}*{Section 4.2.4 and Tables 31--32} and
\cite{WangXu18}.
The higher-dimensional cases are handled in the following lemmas.
\end{proof}

\begin{remark}
\label{rem:nu-extn-notes}
The last column of Table \ref{tab:nu-extn} indicates which
$\nu$ extensions are crossing, as well as which extensions
have indeterminacy in the sense of Section \ref{subsctn:hid-indet}.
\end{remark}

\begin{remark}
\label{rem:nu-h2h5d0}
The hidden $\nu$ extension from $h_2 h_5 d_0$ to $\tau g n$
is proved in \cite{WangXu18}, which relies on the
``$\R P^\infty$-method" to establish a hidden
$\sigma$ extension from $\tau h_3 d_1$ to $\D h_2 c_1$
and a hidden $\eta$ extension from $\tau h_1 g_2$ to $\D h_2 c_1$.
We now have easier proofs for these $\eta$ and $\sigma$ extensions,
using the hidden $\tau$ extension from $h_1^2 g_2$ to $\D h_2 c_1$
given in Table \ref{tab:tau-extn}, as well as the relation
$h_3^2 d_1 = h_1^2 g_2$.
\end{remark}

\begin{remark}
\label{rem:nu-t(De1+C0)g}
If $M \D h_1^2 d_0$ is not hit by a differential, then there is a 
hidden $\tau$ extension from $\tau M h_0 g^2$ from $M \D h_1^2 d_0$.
This implies that there must be a hidden
$\nu$ extension from $\tau (\D e_1 + C_0) g$ to $M \D h_1^2 d_0$.
\end{remark}

\begin{thm}
\label{thm:nu-extn-possible}
Table \ref{tab:nu-extn-possible}
lists all unknown hidden $\nu$ extensions, through
the $90$-stem. 
\end{thm}

\begin{proof}
Many possible extensions can be eliminated by
comparison to $C\tau$, to $\tmf$, or to $\mmf$.
For example,
there cannot be a hidden $\nu$ extension from $h_0 h_2 h_4$
to $\tau h_1 g$ by comparison to $C\tau$.

Other possibilities are eliminated by consideration of other
parts of the multiplicative structure.  For example,
there cannot be a hidden $\nu$ extension whose target supports
a multiplication by $\eta$, since $\eta \nu$ equals zero.

Many cases are eliminated by more complicated arguments.
These are handled in the following lemmas.
\end{proof}

\begin{remark}
\label{rem:h2-th1p1}
Comparison to synthetic homotopy eliminates several
possible hidden $\nu$ extensions, including:
\begin{enumerate}
\item
from $\tau h_1 p_1$ to $\tau x_{74,8}$.
\item
from $\D^2 p$ to $\tau M \D h_1 d_0$.
\end{enumerate}
See \cite{BIX} for more details.
\end{remark}

\begin{remark}
\label{rem:h0h2h4h6}
If $M \D h_1^2 d_0$ is not hit by a differential, then
$M \D h_1 d_0$ supports an $h_1$ extension, and
there cannot be a hidden $\nu$ extension from
$h_0 h_2 h_4 h_6$ to $M \D h_1 d_0$.
\end{remark}

\begin{lemma}
\label{lem:nu-De1+C0}
\revdeg{62, 8, 33}
There is a hidden $\nu$ extension from $\D e_1 + C_0$ to
$\tau M h_0 g$.
\end{lemma}

\begin{proof}
Table \ref{tab:Toda} shows that
$2 \kappabar$ is contained in 
$\tau \langle \nu, \eta, \eta \kappa \rangle$.
Shuffle to obtain that
\[
\nu \langle \eta, \eta \kappa, \tau \theta_{4.5} \rangle =
\langle \nu, \eta, \eta \kappa \rangle \tau \theta_{4.5},
\]
so $2 \kappabar \theta_{4.5}$ is divisible by $\nu$.

Table \ref{tab:misc-extn} shows that 
$\tau M g$ detects $\kappabar \theta_{4.5}$, so
$\tau M h_0 g$ detects 
$2 \kappabar \theta_{4.5}$.
Now we know that there is a hidden $\nu$ extension whose target is
$\tau M h_0 g$, and the only possible source is $\D e_1 + C_0$.
\end{proof}

\begin{remark}
\label{rem:eta,etakappa,ttheta4.5}
One consequence of the proof of Lemma \ref{lem:nu-De1+C0}
is that $\D e_1 + C_0$ detects the Toda bracket
$\langle \eta, \eta \kappa, \tau \theta_{4.5} \rangle$.
\end{remark}

\begin{lemma}
\label{lem:nu-th1H1}
\revdeg{63, 6, 33}
There is a hidden $\nu$ extension from $\tau h_1 H_1$ to
$\tau^2 M h_1 g$.
\end{lemma}

\begin{proof}
Lemma \ref{lem:kappa,2,eta} shows that 
the bracket $\langle \kappa, 2, \eta \rangle$
contains zero with indeterminacy generated by $\eta \rho_{15}$.
The bracket $\langle \tau \eta \theta_{4.5}, \kappa, 2 \rangle$
equals zero since $\pi_{61,32}$ is zero.
Therefore, the Toda bracket 
$\langle \tau \eta \theta_{4.5}, \kappa, 2, \eta \rangle$ is well-defined.

Table \ref{tab:Toda} shows that $\tau g $ detects 
$\langle \kappa, 2, \eta, \nu \rangle$.
Therefore, $\tau^2 M h_1 g$ detects
\[
\tau \eta \theta_{4.5} \langle \kappa, 2, \eta, \nu \rangle =
\langle \tau \eta \theta_{4.5}, \kappa, 2, \eta \rangle \nu.
\]
This shows that
$\tau^2 M h_1 g$ is the target of a $\nu$ extension,
and the only possible source is $\tau h_1 H_1$.
\end{proof}

\begin{remark}
\label{rem:tetatheta4.5,kappa,2,eta}
The proof of Lemma \ref{lem:nu-th1H1} shows that
$\tau h_1 H_1$ detects 
the Toda bracket
$\langle \tau \eta \theta_{4.5}, \kappa, 2, \eta \rangle$.
\end{remark}

\begin{lemma}
\label{lem:nu-h1h6}
\revdeg{64, 2, 33}
There is no hidden $\nu$ extension on $h_1 h_6$.
\end{lemma}

\begin{proof}
Table \ref{tab:Toda} shows that $h_1 h_6$ detects the Toda bracket
$\langle \eta, 2, \theta_5 \rangle$.  Shuffle to obtain
\[
\nu \langle \eta, 2, \theta_5 \rangle =
\langle \nu, \eta, 2 \rangle \theta_5 = 0,
\]
since $\langle \nu, \eta, 2 \rangle$ is contained in $\pi_{5,3} = 0$.
\end{proof}

\begin{lemma}
\label{lem:nu-h3Q2}
\revdeg{64, 8, 34}
There is no hidden $\nu$ extension on $h_3 Q_2$.
\end{lemma}

\begin{proof}
Table \ref{tab:eta-extn} shows that 
$\tau^2 \D h_2^2 e_0 g$ supports a hidden $\eta$ extension.
Therefore, it cannot be the target of a $\nu$ extension.
\end{proof}

\begin{lemma}
\label{lem:nu-h2^2A'}
\revdeg{67, 8, 36}
There is a hidden $\nu$ extension from
the element $h_2^2 A'$ to $h_1 h_3 (\D e_1 + C_0)$.
\end{lemma}

\begin{proof}
By comparison to $C\tau$, 
There cannot be a hidden $\nu$ extension from
$h_2^2 A'$ to $\tau h_2 C'' + h_1 h_3 (\D e_1 + C_0)$

Table \ref{tab:Toda} shows that
$\D e_1 + C_0$ detects the Toda bracket
$\langle \eta, \eta \kappa, \tau \theta_{4.5} \rangle$,
and $h_2 A'$ detects the Toda bracket
$\langle \nu, \eta, \tau \kappa \theta_{4.5} \rangle$.
Note that $h_2 A'$ also detects
$\langle \nu, \eta \kappa, \tau \theta_{4.5} \rangle$.

Now shuffle to obtain
\[
(\eta \sigma + \epsilon) \langle \eta, \eta \kappa, 
\tau \theta_{4.5} \rangle +
\nu^2 \langle \nu, \eta \kappa, \tau \theta_{4.5} \rangle =
\left\langle
\left[
\begin{array}{cc}
\eta \sigma + \epsilon & \nu^2
\end{array}
\right],
\left[
\begin{array}{c}
\eta \\
\nu 
\end{array}
\right],
\eta \kappa \right\rangle \tau \theta_{4.5}.
\]
The matric Toda bracket
$\displaystyle
\left\langle
\left[
\begin{array}{cc}
\eta \sigma + \epsilon & \nu^2
\end{array}
\right],
\left[
\begin{array}{c}
\eta \\
\nu 
\end{array}
\right],
\eta \kappa \right\rangle$
must equal $\{ 0, \nu^2 \sigmabar \}$,
since $\nu^2 \sigmabar = \{ h_1^2 h_4 c_0 \}$
is the only non-zero element of $\pi_{25,15}$,
and that element belongs to the indeterminacy because it is a
multiple of $\nu^2$.

Next observe that $\tau \nu^2 \sigmabar \theta_{4.5}$ is zero
because all possible values of
$\sigmabar \theta_{4.5}$ are multiples of $\eta$.
This shows that
\[
(\eta \sigma + \epsilon) \alpha + \nu^2 \beta = 0,
\]
for some $\alpha$ and $\beta$ detected by $\D e_1 + C_0$
and $h_2 A'$ respectively.
The product
$(\eta \sigma + \epsilon) \alpha$ is detected by
$h_1 h_3 (\D e_1 + C_0)$, so there must be a hidden $\nu$
extension from $h_2^2 A'$ to $h_1 h_3 (\D e_1 + C_0)$.
\end{proof}

\begin{lemma}
\label{lem:nu-h3A'}
\revdeg{68, 7, 36}
There is no hidden $\nu$ extension on $h_3 A'$.
\end{lemma}

\begin{proof}
Table \ref{tab:Toda} shows that $h_3 A'$ detects the Toda bracket
$\langle \sigma, \kappa, \tau \eta \theta_{4.5} \rangle$.
Now shuffle to obtain
\[
\nu \langle \sigma, \kappa, \tau \eta \theta_{4.5} \rangle = 
\langle \nu, \sigma, \kappa \rangle \tau \eta \theta_{4.5} =
\langle \eta, \nu, \sigma \rangle \tau \kappa \theta_{4.5}.
\]
The Toda bracket
$\langle \eta, \nu, \sigma \rangle$ is zero because
it is contained in $\pi_{12,7} = 0$.
\end{proof}

\begin{lemma}
\label{lem:nu-p'}
\revdeg{69, 4, 36}
There is no hidden $\nu$ extension on $p'$.
\end{lemma}

\begin{proof}
Table \ref{tab:misc-extn} shows that 
$p'$ detects the product $\sigma \theta_5$.
Therefore, it cannot support a hidden $\nu$ extension.
\end{proof}

\begin{lemma}
\label{lem:nu-h2^2C'}
\revdeg{69, 9, 38}
There is a hidden $\nu$ extension from $h_2^2 C'$ to $\tau^2 d_1 g^2$.
\end{lemma}

\begin{proof}
Let $\alpha$ be an element of $\pi_{63,33}$ that is detected by
$\tau X_2 + \tau C'$.
Table \ref{tab:misc-extn} shows that 
$\epsilon \alpha$ is detected by $d_0 Q_2$, so
$\eta \epsilon \alpha$ is detected by $\tau^3 d_1 g^2$.
On the other hand,
$\eta \sigma \alpha$ is zero by comparison to $C\tau$.

Now consider the relation $\eta^2 \sigma + \nu^3 = \eta \epsilon$.
This shows that $\nu^3 \alpha$ is detected by $\tau^3 d_1 g^2$.
Since $\nu^2 \alpha$ is detected by $\tau h_2^2 C'$,
there must be a hidden $\nu$ extension from
$h_2^2 C'$ to $\tau^2 d_1 g^2$.
\end{proof}

\begin{lemma}
\label{lem:nu-th1D3'}
\revdeg{70, 9, 37}
There is a hidden $\nu$ extension from $\tau h_1 D'_3$ to
$\tau M d_0^2$.
\end{lemma}

\begin{proof}
Table \ref{tab:Toda} shows that $\tau h_1 D'_3$ detects the 
Toda bracket $\langle \eta, \nu, \tau \kappabar \theta_{4.5} \rangle$.
Now shuffle to obtain
\[
\nu \langle \eta, \nu, \tau \kappabar \theta_{4.5} \rangle = 
\langle \nu, \eta, \nu \rangle \tau \kappabar \theta_{4.5}.
\]
The bracket $\langle \nu, \eta, \nu \rangle$
equals $\eta \sigma + \epsilon$ \cite{Toda62}.

Now we must compute 
$(\eta \sigma + \epsilon) \tau \kappabar \theta_{4.5}$.
The product $\sigma \kappabar$ is zero, and
Table \ref{tab:misc-extn} shows that
$\epsilon \kappabar \theta_{4.5}$ is detected by $M d_0^2$.
These two observations imply that
$(\eta \sigma + \epsilon) \tau \kappabar \theta_{4.5}$
is detected by $\tau M d_0^2$.
\end{proof}

\begin{lemma}
\label{lem:nu-h6c0}
\revdeg{71, 4, 37}
There is no hidden $\nu$ extension on $h_6 c_0$.
\end{lemma}

\begin{proof}
Table \ref{tab:Toda} shows that $h_6 c_0$ detects the Toda bracket
$\langle \epsilon, 2, \theta_5 \rangle$.  Now shuffle to obtain
\[
\nu \langle \epsilon, 2, \theta_5 \rangle = 
\langle \nu, \epsilon, 2 \rangle \theta_5.
\]
Finally, the Toda bracket $\langle \nu, \epsilon, 2 \rangle$ is 
zero because it is contained in $\pi_{12,7} = 0$.
\end{proof}

\begin{lemma}
\label{lem:nu-h2^2C''}
\revdeg{73, 11, 41} 
There is a hidden $\nu$ extension from $h_2^2 C''$ to $\tau g^2 t$.
\end{lemma}

\begin{proof}
Let $\alpha$ be an element of $\pi_{53,30}$ that is detected by $i_1$.
Table \ref{tab:nu-extn} shows $g t$ detects $\nu \alpha$.
Therefore $\tau g^2 t$ detects $\nu \kappa \alpha$,
so $\tau g^2 t$ must be the target of a hidden $\nu$ extension.
The element $h_2^2 C''$ is the only
possible source for this extension.
\end{proof}

\begin{lemma}
\label{lem:nu-Mh1h3g}
\revdeg{73, 12, 41}
There is no hidden $\nu$ extension on $M h_1 h_3 g$.
\end{lemma}

\begin{proof}
If there were a hidden $\nu$ extension from $M h_1 h_3 g$ to
$\tau g^2 t$, then there would also be a hidden $\nu$ extension
with target $\tau^2 g^2 t$.  But there is no possible source
for such an extension.
\end{proof}

\begin{lemma}
\label{lem:nu-h0h3d2}
\revdeg{75, 6, 40}
If there is a hidden $\nu$ extension on $h_0 h_3 d_2$, then
its target is $M \D h_1^2 h_3$.
\end{lemma}

\begin{proof}
The only other possible target is $e_0 A'$.
However, Table \ref{tab:2-extn} shows that $e_0 A'$ supports
a hidden $2$ extension, while $h_0 h_3 d_2$ does not.
\end{proof}

\begin{lemma}
\label{lem:nu-td1g2}
\revdeg{76, 8, 41}
There is no hidden $\nu$ extension on $\tau d_1 g_2$.
\end{lemma}

\begin{proof}
Table \ref{tab:eta-extn} shows that $\tau d_1 g_2$ is the target
of a hidden $\eta$ extension.  Therefore, it cannot be the source
of a hidden $\nu$ extension.
\end{proof}

\begin{lemma}
\label{lem:nu-h0h4A}
\revdeg{76, 8, 40}
There is no hidden $\nu$ extension on $h_0 h_4 A$.
\end{lemma}

\begin{proof}
\rev{
Table \ref{tab:misc-extn} shows that
$h_0 h_4 A$ detects either $\sigma^2 \theta_5$
or $\sigma^2 \theta_5 + \tau^2 \kappa_1 \kappabar_2$.
As in the proof of Lemma \ref{lem:eta-h1h6d0}, both possibilities
are annihilated by $\nu$.
}
\end{proof}

\begin{lemma}
\label{lem:nu-h3^2h6}
\revdeg{77, 3, 40}
There is a hidden $\nu$ extension from $h_3^2 h_6$ to $\tau h_1 x_1$.
\end{lemma}

\begin{proof}
Table \ref{tab:Toda} shows that
$h_3^2 h_6$ detects
$\langle \theta_5, 2, \sigma^2 \rangle$.
Let $\alpha$ be an element of $\pi_{77,40}$ that is contained
in this Toda bracket.
Then $\nu \alpha$ is an element of
\[
\langle \theta_5, 2, \sigma^2 \rangle \nu =
\langle \theta_5, 2 \sigma, \sigma \rangle \nu =
\theta_5 \langle 2 \sigma, \sigma, \nu \rangle \subseteq
\langle 2 \sigma, \sigma \theta_5, \nu \rangle.
\]

Table \ref{tab:misc-extn} shows that $p'$ detects 
$\sigma \theta_5$.
Therefore, the Toda bracket
$\langle 2 \sigma, \sigma \theta_5, \nu \rangle$
is detected by an element of the Massey product
$\langle h_0 h_3, p', h_2 \rangle$.
Table \ref{tab:Massey} shows that $h_0 e_2$ equals the
Massey product $\langle h_3, p', h_2\rangle$. 
By inspection of indeterminacy,
the Massey product $\langle h_0 h_3, p', h_2 \rangle$
contains $h_0^2 e_2 = \tau h_1 x_1$ with indeterminacy generated by
$h_0 h_6 e_0$.

We have now shown that 
$\nu \alpha$ is detected by either
$\tau h_1 x_1$ or $\tau h_1 x_1 + h_0 h_6 e_0$.
But $h_0 h_6 e_0 = h_2 h_6 d_0$ is a multiple of $h_2$,
so we may add an element in higher Adams filtration to $\alpha$,
if necessary, to conclude that $\nu \alpha$ is detected by
$\tau h_1 x_1$. 
\end{proof}

\begin{lemma}
\label{lem:nu-h0^7h4h6}
\revdeg{78, 9, 40}
There is a hidden $\nu$ extension from the element $h_0^7 h_4 h_6$ to
$\tau \D^2 h_1 d_1$.
\end{lemma}

\begin{proof}
Table \ref{tab:nu-extn} shows that there is a hidden
$\nu$ extension from $h_0^6 h_4 h_6$ to $\D^2 p$.
Therefore, there is also a hidden $\nu$ extension from
$h_0^7 h_4 h_6$ to $h_0 \cdot \D^2 p = \tau \D^2 h_1 d_1$.
\end{proof}

\begin{lemma}
\label{lem:nu-e0A'}
\revdeg{78, 10, 42}
There is no hidden $\nu$ extension on $e_0 A'$.
\end{lemma}

\begin{proof}
A possible hidden $\nu$ extension from $e_0 A'$ to $\D h_1^2 B_6$
would be detected by $C\tau$, but we have to be careful with the
analysis of the homotopy of $C\tau$ because of the $h_2$ extension
from $\ol{\D h_1 d_1 g}$ to $\D h_1^2 B_6$ in the Adams
$E_\infty$-page for $C\tau$.

Let $\alpha$ be an element of $\pi_{75,40} C\tau$ that is detected
by $\ol{h_3 C'}$.  Then $\nu \alpha$ is detected by
$e_0 A'$, and $\nu \alpha$ maps to zero under projection to the 
top cell because $h_3 C'$ does not support a $\nu$ extension
in the homotopy of the sphere.

Therefore, $\nu \alpha$ lies in the image of 
$e_0 A'$ under inclusion of the bottom
cell.  Since $\nu^2 \alpha$ is zero, $e_0 A'$ cannot support
a hidden $\nu$ extension to $\D h_1^2 B_6$.
\end{proof}

\begin{lemma}
\label{lem:nu-h1h4h6}
\revdeg{79, 3, 41}
There is no hidden $\nu$ extension on $h_1 h_4 h_6$.
\end{lemma}

\begin{proof}
Table \ref{tab:Toda} shows that $h_1 h_4 h_6$
detects the Toda bracket $\langle \eta_4, 2, \theta_5 \rangle$.
Shuffle to obtain
\[
\nu \langle \eta_4, 2, \theta_5 \rangle =
\langle \nu, \eta_4, 2 \rangle \theta_5.
\]
Finally, $\langle \nu, \eta_4, 2 \rangle$ must contain zero
in $\pi_{20,11}$ because $\tmf$ detects every element of
$\pi_{20,11}$.
\end{proof}

\begin{lemma}
\label{lem:nu-h3^2n1}
\revdeg{81, 7, 44}
There is no hidden $\nu$ extension on $h_3^2 n_1$.
\end{lemma}

\begin{proof}
The element $h_2 g D_3$ cannot be the target of a hidden
$\nu$ extension by comparison to $C\tau$.

The element $\tau h_3^2 n_1 = h_3 \cdot h_3 (\tau Q_3 + \tau n_1)$
detects a multiple of $\sigma$, so it cannot support a hidden
$\nu$ extension.  This rules out $h_2 g A'$ as a possible target.
\end{proof}

\begin{lemma}
\label{lem:nu-te1g2}
\revdeg{82, 8, 44}
There is no hidden $\nu$ extension on $\tau e_1 g_2$.
\end{lemma}

\begin{proof}
After eliminating other possibilities by comparison to
$\tmf$, comparison to $\mmf$, and by inspection of $h_1$
multiplications, the only possible target for a hidden
$\nu$ extension is $P h_1 x_{76,6}$.  

Let $\alpha$ be an element of $\pi_{82,45}$ that is detected
by $e_1 g_2$.  Then $\nu \alpha$ is detected by
$h_2 e_1 g_2 = h_1^3 h_4 Q_3$.
Choose an element $\beta$ of $\pi_{83,45}$ that is detected
by $h_1 h_4 Q_3$ such that $\tau \beta$ is zero.
Then $\eta^2 \beta$ is also detected by $h_1^3 h_4 Q_3$.
However, $\nu \alpha + \eta^2 \beta$ is not necessarily zero; 
it could be detected in Adams filtration at least $13$.
In any case,
$\tau \nu \alpha$ equals $\tau \eta^2 \beta = 0$ modulo
filtration $13$.
In particular, $\tau \nu \alpha$ cannot be detected by
$P h_1 x_{76,6}$ in filtration $11$.
\end{proof}

\begin{lemma}
\label{lem:nu-Ph0h2h6}
\revdeg{82, 11, 42}
There is no hidden $\nu$ extension on $P^2 h_0 h_2 h_6$.
\end{lemma}

\begin{proof}
Table \ref{tab:nu-extn} shows that there is a hidden
$\nu$ extension from $P^2 h_2 h_6$ to $\D^2 h_0 x$.
The target of a hidden $\nu$ extension
on $P^2 h_0 h_2 h_6$ must have Adams filtration greater than
the filtration of $\D^2 h_0 x$.
The only possibilities are ruled out by comparison to $\tmf$.
\end{proof}

\begin{lemma}
\label{lem:nu-(De1+C0)g}
\revdeg{82, 12, 45}
There is a hidden $\nu$ extension
from $(\D e_1 + C_0) g$ to $\tau M h_0 g^2$.
\end{lemma}

\begin{proof}
Let $\alpha$ be an element of $\pi_{62,33}$ that is detected
by $\D e_1 + C_0$.
Table \ref{tab:Toda} shows that
$(\D e_1 + C_0) g$ detects $\langle \alpha, \eta^3, \eta_4 \rangle$.
Then
\[
\nu \langle \alpha, \eta^3, \eta_4 \rangle =
\langle \nu \alpha, \eta^3, \eta_4 \rangle
\]
by inspection of indeterminacies.
Table \ref{tab:nu-extn} shows that $\tau M h_0 g$ detects $\nu \alpha$.
The Toda bracket
$\langle \nu \alpha, \eta^3, \eta_4 \rangle$ is detected by
the Massey product
\[
\langle \tau M h_0 g, h_1^3, h_1 h_4 \rangle =
\langle \tau M h_0 g, h_1^4, h_4 \rangle = 
M h_0 g \langle \tau, h_1^4, h_4 \rangle =
\tau M h_0 g^2.
\]
\end{proof}

\begin{lemma}
\label{lem:nu-h2c1A'}
\revdeg{83, 10, 45}
There is no hidden $\nu$ extension on $h_2 c_1 A'$.
\end{lemma}

\begin{proof}
Table \ref{tab:Toda} shows that 
$\tau h_2 c_1 A'$ detects 
$\langle \tau \theta_{4.5} \kappa, \eta, \nu \rangle \tau \sigmabar$.
Shuffle to obtain
\[
\langle \tau \theta_{4.5} \kappa, \eta, \nu \rangle \tau \sigmabar \nu =
\tau \theta_{4.5} \kappa \langle \eta, \nu, \tau \nu \sigmabar \rangle.
\]
The Toda bracket $\langle \eta, \nu, \tau \nu \sigmabar \rangle$
is zero because $\pi_{27, 15}$ contains only a $v_1$-periodic element
detected by $P^3 h_1^3$.

We now know that $\tau h_2 c_1 A'$ does not support a hidden
$\nu$ extension.  In particular, there cannot be a hidden
$\nu$ extension from $\tau h_2 c_1 A'$ to $M \D h_0^2 e_0$.
The hidden $\tau$ extension from $\tau^2 M h_1 g^2$ to 
$M \D h_0^2 e_0$ implies that there cannot be a hidden
$\nu$ extension from $h_2 c_1 A'$ to $\tau^2 M h_1 g^2$.

Additional cases are ruled out by comparison to $C\tau$ and
to $\mmf$.
\end{proof}

\begin{lemma}
\label{lem:nu-Dj1+tgC'}
\mbox{}
\begin{enumerate}
\item
\revdeg{83, 11, 45}
There is a hidden $\nu$ extension from $\D j_1 + \tau g C'$ to
$\tau^2 M h_1 g^2$.
\item
\revdeg{83, 11, 44}
There is a hidden $\nu$ extension from $\tau^2 g C'$ to $M \D h_0^2 e_0$.
\end{enumerate}
\end{lemma}

\begin{proof}
Table \ref{tab:nu-extn} shows that there exists an element
$\alpha$ in $\pi_{63,33}$ detected by $\tau h_1 H_1$ such that
$\nu$ is detected by $\tau^2 M h_1 g$.
(Beware that there is a crossing extension here, so not every
element detected by $\tau h_1 H_1$ has the desired property.)
Table \ref{tab:misc-extn} shows that 
$\tau^2 M h_1 g$ also detects $\tau \theta_{4.5} \eta \kappabar$.
However, $\nu \alpha$ does not necessarily equal
$\tau \theta_{4.5} \eta \kappabar$ because the difference could
be detected in higher filtration by $\D^2 h_1^3 h_4$.
In any case,
$\nu \kappabar \alpha$ equals $\tau \theta_{4.5} \eta \kappabar^2$.

The product $\theta_{4.5} \eta \kappabar^2$
is detected by $\tau^2 M h_1 g^2$.
The hidden $\tau$ extension from $\tau^2 M h_1 g^2$ to
$M \D h_0^2 e_0$ then implies that
$\nu \kappabar \alpha = \tau \theta_{4.5} \eta \kappabar^2$
is detected by $M \D h_0^2 e_0$.

We now know that $M \D h_0^2 e_0$ is the target of a hidden
$\nu$ extension.
The only possible source is $\tau^2 g C'$. (Lemma \ref{lem:nu-h2c1A'}
eliminates another possible source.)
This establishes the second extension.
The first extension follows from onsideration of $\tau$ extensions.
\end{proof}

\begin{remark}
\label{rem:nu-Dj1+tgC'}
The proof of Lemma \ref{lem:nu-Dj1+tgC'} shows that
$\nu \kappabar \alpha$ is detected by $M \D h_0^2 e_0$,
where $\alpha$ is detected by $\tau h_1 H_1$.
Note that $\kappabar \alpha$ is detected by $\tau^2 h_1 H_1 g =
\tau h_2 c_1 A'$.  But this does not show that 
$\tau h_2 c_1 A'$ supports a hidden $\nu$ extension.  Rather,
it shows that the source of the hidden $\nu$ extension is
either $\tau h_2 c_1 A'$, or a non-zero element in higher filtration.
\end{remark}

\begin{lemma}
\label{lem:nu-h2^2h4h6}
\revdeg{84, 4, 44}
There is no hidden $\nu$ extension on $h_2^2 h_4 h_6$.
\end{lemma}

\begin{proof}
Table \ref{tab:Toda} shows that $h_2^2 h_4 h_6$ detects the 
Toda bracket $\langle \nu \nu_4, 2, \theta_5 \rangle$.
Shuffle to obtain
\[
\nu \langle \nu \nu_4, 2, \theta_5 \rangle = 
\langle \nu, \nu \nu_4, 2 \rangle \theta_5.
\]
The Toda bracket $\langle \nu, \nu \nu_4, 2 \rangle$
is zero because $\pi_{25,14}$ consists only of a $v_1$-periodic
element detected by $P^2 h_1 c_0$.
\end{proof}

\rev{
\begin{lemma}
\label{lem:nu-tx85,6+h0^3c3}
\revdeg{85, 6, 44}
If $\tau x_{85,6} + h_0^3 c_3$ survives, then
it supports a hidden $\nu$ extension to
$h_1 x_{87,7} + \tau^2 g_2^2$.
\end{lemma}

\begin{proof}
By comparison to $C\tau$, there must be a hidden $\nu$ extension
whose target is either $h_1 x_{87,7}$ or $h_1 x_{87,7} + \tau^2 g_2^2$.

Table \ref{tab:nu-extn} shows that there is a hidden
$\nu$ extension from $\tau^2 h_2 h_4 Q_3$ to $\tau^2 h_0 g_2^2$.
This implies that the target of the $\nu$ extension on
$\tau x_{85,6} + h_0^3 c_3$ must be $h_1 x_{87,7} + \tau^2 g_2^2$.
\end{proof}
}

\begin{lemma}
\label{lem:nu-P^2h6c0}
\revdeg{87, 12, 45}
There is no hidden $\nu$ extension on $P^2 h_6 c_0$.
\end{lemma}

\begin{proof}
Table \ref{tab:misc-extn} shows that $P^2 h_6 c_0$ 
detects the product $\rho_{23} \eta_6$, and
$\nu \rho_{23} \eta_6$ is zero.
\end{proof}

\begin{lemma}
\label{lem:nu-h2^2gA'}
\revdeg{87, 12, 48}
There is a hidden $\nu$ extension from
$h_2^2 g A'$ to $\D h_1^2 g_2 g$.
\end{lemma}

\begin{proof}
Comparison to $C\tau$ shows that $h_2^2 g A'$ supports a hidden
$\nu$ extension whose target is either $\D h_1^2 g_2 g$
or $\D h_1^2 g_2 g + \tau h_2 g C''$.

Let $\alpha$ be an element of $\pi_{84,46}$ that is detected
by $h_2 g A'$.  Since $h_2 g A'$ does not support a hidden
$\eta$ extension, we may choose $\alpha$ such that $\eta \alpha$ is
zero.  Note that $h_2^2 g A'$ detects $\nu \alpha$.

Shuffle to obtain
\[
\nu^2 \alpha = \langle \eta, \nu, \eta \rangle \alpha =
\eta \langle \nu, \eta, \alpha \rangle.
\]
This shows that $\nu^2 \alpha$ must be divisible by $\eta$.
Consequently, the hidden $\nu$ extension on 
$h_2^2 g A'$ must have target $\D h_1^2 g_2 g$.
\end{proof}

\begin{remark}
\label{rem:nu,eta,h2gA'}
The proof of Lemma \ref{lem:nu-h2^2gA'} shows that
$\D h_1 g_2 g$ detects the Toda bracket
$\langle \nu, \eta, \{h_2 g A'\} \rangle$.
\end{remark}

\section{Miscellaneous hidden extensions}
\label{sctn:misc-extn}

\rev{
\begin{thm}
\label{thm:misc-extn}
Tables \ref{tab:misc-extn} \rev{and \ref{tab:misc-extn-null}} list some miscellaneous hidden extensions.
\end{thm}

\begin{proof}
Similarly to Theorems \ref{thm:2-extn}, \ref{thm:eta-extn}, and \ref{thm:nu-extn}, some of the extensions follow by comparison to $C\tau$ or to $\tmf$.
The more difficult cases are handled in the following lemmas.
\end{proof}
}

\rev{
Based on the corrected statement of Lemma~\ref{lem:theta4,2,sigma^2+kappa} (\cite{Xu16}*{Theorem 2.1}) and Lemma~\ref{lem:etakappabar2,2sigma,sigma}, the third author presents the proof of the following Lemma~\ref{lem:theta4-h4^2} (\cite{Xu16}*{Theorem 1.2}), fixing a gap in its original proof. Note that as in the original proof, it only uses classical knowledge back then up to the 60-stem.

\begin{lemma}
\label{lem:theta4-h4^2}
\revdeg{30, 2}
Classically, there is no hidden $\theta_4$ extension on $h_4^2$.
In other words $\theta_4^2$ is zero in $\pi_{60}$.
\end{lemma}

\begin{proof}
We have
$$\theta_4^2 = \theta_4 \langle 2, \sigma^2 + \kappa, 2\sigma, \sigma \rangle \subseteq \langle \langle \theta_4, 2, \sigma^2 + \kappa \rangle, 2\sigma, \sigma \rangle.$$
Using \cite{Xu16}*{Theorem~2.2} and Lemma \ref{lem:theta4,2,sigma^2+kappa}, the last expression is contained in the union of 
$$\langle 0, 2\sigma, \sigma \rangle, \ \langle \eta\kappabar_2, 2\sigma, \sigma \rangle, \ \langle \rho_{15}\theta_4, 2\sigma, \sigma \rangle.$$
By \cite{Xu16}*{Lemmas~2.3 and 2.4} and Lemma~\ref{lem:etakappabar2,2sigma,sigma}, all three brackets contain a single element 0. Therefore, $\theta_4^2 = 0$.
\end{proof}
}

\begin{lemma}
\label{lem:epsilon-h3^2h5}
\mbox{}
\begin{enumerate}
\item
\revdeg{45, 3, 24}
There is a hidden $\epsilon$ extension from $h_3^2 h_5$ to $M c_0$.
\item
\revdeg{45, 3, 23}
There is a hidden $\epsilon$ extension from $\tau h_3^2 h_5$ to $M P$.
\end{enumerate}
\end{lemma}

\begin{proof}
Table \ref{tab:eta-extn} shows that
$M h_1$ detects the product $\eta \theta_{4.5}$.
Then $M h_1 c_0$ detects $\eta \epsilon \theta_{4.5}$.
This implies that $M c_0$ detects $\epsilon \theta_{4.5}$.

This only shows that $M c_0$ is the target of a hidden
$\epsilon$ extension, whose source could be $h_3^2 h_5$
or $h_5 d_0$.  However, Lemma \ref{lem:epsilon-h5d0}
rules out the latter case.
This establishes the first hidden extension.

Table \ref{tab:tau-extn} shows that there is a hidden
$\tau$ extension from $M c_0$ to $M P$.  Then the first
hidden extension implies the second one.
\end{proof}

\begin{remark}
We claimed in \cite{Isaksen14c}*{Table 33} that there is a
hidden $\epsilon$ extension from $h_3^2 h_5$ to $M c_0$.
However, the argument given in \cite{Isaksen14c}*{Lemma 4.108}
only implies that $M c_0$ is the target of a hidden extension
from either $h_3^2 h_5$ or $h_5 d_0$.
\end{remark}

\begin{lemma}
\label{lem:kappa-h3^2h5}
\revdeg{45, 3, 24}
There is a hidden $\kappa$ extension from $h_3^2 h_5$ to $M d_0$.
\end{lemma}

\begin{proof}
Table \ref{tab:eta-extn} shows that
$M h_1$ detects the product $\eta \theta_{4.5}$.
Then $M h_1 d_0$ detects the product $\eta \kappa \theta_{4.5}$,
so $M d_0$ must detect the product $\kappa \theta_{4.5}$.
This shows that $M d_0$ is the target of a hidden $\kappa$ extension
whose source is either $h_3^2 h_5$ or $h_5 d_0$.

We showed in Lemma \ref{lem:epsilon-h5d0}
that $\epsilon \alpha$ is zero for some element 
$\alpha$ of $\pi_{45,24}$ that is detected by $h_5 d_0$.
Then $\epsilon \kappabar \alpha$ is also zero.
Table \ref{tab:misc-extn} shows that
$\epsilon \kappabar$ equals $\kappa^2$.
Therefore,
$\kappa^2 \alpha$ is zero.
If $\kappa \alpha$ were detected by $M d_0$, 
then $\kappa^2 \alpha$ would be detected by $M d_0^2$.
It follows that there is no hidden $\kappa$ extension
from $h_5 d_0$ to $M d_0$.
\end{proof}

\begin{remark}
We showed in \cite{Isaksen14c}*{Table 33} that there is a 
hidden $\kappa$ extension from either $h_3^2 h_5$ or $h_5 d_0$
to $M d_0$. Lemma \ref{lem:kappa-h3^2h5} settles this uncertainty.
\end{remark}

\begin{lemma}
\label{lem:kappabar-h3^2h5}
\revdeg{45, 3, 24}
There is a hidden $\kappabar$ extension from $h_3^2 h_5$ to
$\tau M g$.
\end{lemma}

\begin{proof}
Table \ref{tab:eta-extn} shows that
$M h_1$ detects the product $\eta \theta_{4.5}$.
Then $\tau M h_1 g$ detects the product $\eta \kappabar \theta_{4.5}$,
so $\tau M g$ must detect the product $\kappabar \theta_{4.5}$.
This shows that $\tau M g$ is the target of a hidden 
$\kappabar$ extension
whose source is either $h_3^2 h_5$ or $h_5 d_0$.

We showed in Lemma \ref{lem:epsilon-h5d0}
that $\epsilon \alpha$ is zero for some element 
$\alpha$ of $\pi_{45,24}$ that is detected by $h_5 d_0$.
If $\kappabar \alpha$ were detected by $\tau M g$,
then $\epsilon \kappabar \alpha$ would be detected by
$M d_0^2$ because Table \ref{tab:misc-extn} shows that
there is a hidden $\epsilon$ extension from $\tau M g$
to $M d_0^2$.
Therefore, there is no hidden $\kappabar$ extension
from $h_5 d_0$ to $\tau M g$.
\end{proof}

\begin{lemma}
\label{lem:Dh1h3-h3^2h5}
\revdeg{45, 3, 24}
There is a hidden $\{\D h_1 h_3\}$ extension from
$h_3^2 h_5$ to $M \D h_1 h_3$.
\end{lemma}

\begin{proof}
Table \ref{tab:eta-extn} shows that $M h_1$ detects the product
$\eta \theta_{4.5}$.  Therefore, the element
$M \D h_1^2 h_3$ detects $\eta \{\D h_1 h_3\} \theta_{4.5}$.
This shows that $M \D h_1 h_3$ is the target of a hidden
$\{\D h_1 h_3\}$ extension.
Lemma \ref{lem:Dh1h3-h5d0} rules out $h_5 d_0$ as a possible source.
The only remaining possible source is $h_3^2 h_5$.
\end{proof}

\begin{lemma}
\label{lem:theta4.5-h3^2h5}
\revdeg{45, 3, 24}
There is a hidden $\theta_{4.5}$ extension from
$h_3^2 h_5$ to $M^2$.
\end{lemma}

\begin{proof}
The proof of Lemma \ref{lem:perm-M^2} shows that
$M^2 h_1$ detects a multiple of $\eta \theta_{4.5}$.
Therefore, it detects either $\eta \theta_{4.5}^2$
or $\eta \theta_{4.5} \{h_5 d_0\}$.

Now $h_1 h_5 d_0$ detects $\eta \{h_5 d_0\}$, which also
detects $\eta_4 \theta_4$ by Table \ref{tab:misc-extn}.
In fact, the proof of \cite{Isaksen14c}*{Lemma 4.112} shows that
these two products are equal.
Then $\eta \theta_{4.5} \{h_5 d_0\}$ equals
$\eta_4 \theta_4 \theta_{4.5}$.
Next,
$\eta_4 \theta_{4.5}$ lies in $\pi_{61,33}$.
The only non-zero element of $\pi_{61,33}$ is detected by
$\mmf$, so the product $\eta_4 \theta_{4.5}$ must be zero.

We have now shown that $M^2 h_1$ detects $\eta \theta_{4.5}^2$.
This implies that $M^2$ detects $\theta_{4.5}^2$.
\end{proof}

\begin{lemma}
\label{lem:epsilon-h5d0}
\revdeg{45, 5, 24}
There is no hidden $\epsilon$ extension on $h_5 d_0$.
\end{lemma}

\begin{proof}
Table \ref{tab:Toda} shows that
$h_5 d_0$ detects the Toda bracket $\langle 2, \theta_4, \kappa \rangle$.
Now shuffle to obtain
\[
\epsilon \langle 2, \theta_4, \kappa \rangle =
\langle \epsilon, 2, \theta_4 \rangle \kappa.
\]
Table \ref{tab:Toda} shows that
$h_5 c_0$ detects the Toda bracket
$\langle \epsilon, 2, \theta_4 \rangle$, and there is no indeterminacy.
Let $\alpha$ in $\pi_{39,21}$ be the unique element of this Toda bracket.
We wish to compute $\alpha \kappa$.

Table \ref{tab:Toda} shows that
$h_5 c_0$ also detects the Toda bracket 
$\langle \eta_5, \nu, 2 \nu \rangle$, 
with indeterminacy generated by $\sigma \eta_5$.
Let $\beta$ in $\pi_{39,21}$ be an element of
this Toda bracket.
Then $\alpha$ and $\beta$ are equal, modulo
$\sigma \eta_5$ and modulo elements in higher filtration.
Both $\tau h_3 d_1$ and $\tau^2 c_1 g$ detect multiples of
$\sigma$.
Also, the difference between $\alpha$ and $\beta$ cannot be
detected by $\D h_1 d_0$ by comparison to $\tmf$.

This implies that $\alpha$ equals $\beta + \sigma \gamma$
for some element $\gamma$ in $\pi_{31,17}$.
Then
\[
\alpha \kappa = (\beta + \sigma \gamma) \kappa = \beta \kappa
\]
because $\sigma \kappa$ is zero.

Now shuffle to obtain
\[
\beta \kappa = \langle \eta_5, \nu, 2 \nu \rangle \kappa =
\eta_5 \langle \nu, 2\nu, \kappa \rangle.
\]
Table \ref{tab:Toda} shows that
$\langle \nu, 2\nu, \kappa \rangle$
contains $\eta \kappabar$, and its indeterminacy is generated by
$\nu \nu_4$.
We now need to compute $\eta_5 \eta \kappabar$.

The product $\eta_5 \kappabar$ is detected by
$\tau h_1 h_3 g_2 = \tau h_1 h_5 g$, so
$\eta_5 \kappabar$ equals $\tau \eta \sigma \kappabar_2$,
modulo elements of higher filtration.
But these elements of higher filtration are either annihilated
by $\eta$ or detected by $\tmf$, so 
$\eta_5 \eta \kappabar$ equals
$\tau \eta^2 \sigma \kappabar_2$.
By comparison to $\tmf$, this latter expression must be zero.
\end{proof}

\begin{lemma}
\label{lem:Dh1h3-h5d0}
\revdeg{45, 5, 24}
There is no hidden $\{\D h_1 h_3\}$ extension on 
$h_5 d_0$.
\end{lemma}

\begin{proof}
Table \ref{tab:Toda} shows that $h_5 d_0$ detects the Toda bracket
$\langle \kappa, \theta_4, 2 \rangle$.  By inspection of 
indeterminacies, we have
\[
\{\D h_1 h_3\} \langle \kappa, \theta_4, 2 \rangle =
\langle \{\D h_1 h_3\} \kappa, \theta_4, 2 \rangle.
\]
Table \ref{tab:misc-extn} shows that
$\tau d_0 l + \D c_0 d_0$ detects the product
$\{\D h_1 h_3 \} \kappa$.
Now apply the Moss Convergence Theorem \ref{thm:Moss}
with the Adams differential $d_2(h_5) = h_0 h_4^2$
to determine that the Toda bracket
$\langle \{\D h_1 h_3\} \kappa, \theta_4, 2 \rangle$ is 
detected in Adams filtration at least $13$.

The only element in sufficiently high filtration is $\tau^5 e_0 g^3$,
but comparison to $\mmf$ rules this out.
Thus the Toda bracket
$\langle \{\D h_1 h_3\} \kappa, \theta_4, 2 \rangle$ 
contains zero.
\end{proof}

\begin{lemma}
\label{lem:rho15-h5^2}
\revdeg{62, 2, 32}
There is a hidden $\rho_{15}$ extension from $h_5^2$ to
either $h_0 x_{77,7}$ or $\tau^2 m_1$.
\end{lemma}

\begin{proof}
Table \ref{tab:Toda} shows that the Toda bracket
$\langle 8, 2 \sigma, \sigma \rangle$ contains
$\rho_{15}$.
Then $\rho_{15}$ is also contained in 
$\langle 2, 8 \sigma, \sigma \rangle$, although the indeterminacy
increases.

Now shuffle to obtain
\[
\rho_{15} \theta_5 = \theta_5 \langle 2, 8\sigma, \sigma \rangle =
\langle \theta_5, 2, 8 \sigma \rangle \sigma.
\]
Table \ref{tab:Toda} shows that $h_0^3 h_3 h_6$
detects $\langle \theta_5, 2, 8\sigma \rangle$.
Also, there is a $\sigma$ extension from $h_0^3 h_3 h_6$
to $h_0 x_{77,7}$ in the homotopy of $C\tau$.

This implies that $\rho_{15} \theta_5$ is non-zero in
$\pi_{77,40}$, and that it is detected in filtration at most $8$.
Moreover, it is detected in filtration at least $7$, since
$\rho_{15}$ and $\theta_5$ are detected in filtrations $4$ and $2$
respectively.

There are several elements
in filtration $7$ that could detect $\rho_{15} \theta_5$.
The element $x_{77,7}$ (if it survives to the $E_\infty$-page)
is ruled out by comparison to $C\tau$.
The element $\tau h_1 x_{76,6}$ is ruled out because
$\eta \rho_{15} \theta_5$ is detected in filtration at least $10$,
since $\eta \rho_{15}$ is detected in filtration $7$.

The only remaining possibilities are $h_0 x_{77,7}$
and $\tau^2 m_1$.
\end{proof}

\begin{lemma}
\label{lem:rho15-h1h6}
\mbox{}
\begin{enumerate}
\item
\revdeg{64, 2, 33}
There is a hidden $\rho_{15}$ extension from
$h_1 h_6$ to $P h_6 c_0$.
\item
\revdeg{64, 2, 33}
There is a hidden $\rho_{23}$ extension from
$h_1 h_6$ to $P^2 h_6 c_0$.
\end{enumerate}
\end{lemma}

\begin{proof}
Table \ref{tab:Toda} shows that
$h_1 h_6$ detects $\langle \eta, 2, \theta_5 \rangle$.
Then
\[
\rho_{15} \langle \eta, 2, \theta_5 \rangle \subseteq
\langle \eta \rho_{15}, 2, \theta_5 \rangle.
\]
The last bracket is detected by $P h_6 c_0$ because
$d_2(h_6) = h_0 h_5^2$ and because
$P c_0$ detects $\eta \rho_{15}$.
Also, its indeterminacy is in Adams filtration greater than $8$.
This establishes the first hidden extension.

The proof for the second extension is essentially the same,
using that 
$P^2 c_0$ detects $\eta \rho_{23}$ and that
the indeterminacy of
$\langle \eta \rho_{23}, 2, \theta_5 \rangle$
is in Adams filtration greater than $12$.
\end{proof}

\begin{lemma}
\label{lem:epsilon-tMg}
\revdeg{65, 10, 35}
There is a hidden $\epsilon$ extension from $\tau M g$ to
$M d_0^2$.
\end{lemma}

\begin{proof}
First, we have the relation $c_0 \cdot h_1^2 X_2 = M h_1 h_3 g$
in the Adams $E_2$-page, which is detected in the homotopy of
$C\tau$.
Table \ref{tab:tau-extn} shows that there are hidden $\tau$
extensions from $h_1^2 X_2$ and $M h_1 h_3 g$
to $\tau M g$ and $M d_0^2$ respectively.
\end{proof}

\begin{lemma}
\label{lem:epsilon-MDh1h3}
\revdeg{77, 12, 41}
If $M \D h_1^2 d_0$ is non-zero in the $E_\infty$-page, then
there is a hidden $\epsilon$ extension from $M \D h_1 h_3$
to $M \D h_1^2 d_0$.
\end{lemma}

\begin{proof}
Table \ref{tab:eta-extn} shows that $M h_1$ detects the 
product $\eta \theta_{4.5}$.
Since $M \D h_1^2 d_0$ equals $\D h_1 d_0 \cdot M h_1$, it detects
$\eta \{\D h_1 d_0\} \theta_{4.5}$.

Table \ref{tab:misc-extn} shows that 
$\eta \{\D h_1 d_0\}$ equals 
$\epsilon \{\D h_1 h_3\}$, since they are both
detected by $\D h_1^2 d_0$ and there are no elements in higher
Adams filtration.
Therefore, the product 
$\epsilon \{\D h_1 h_3\} \theta_{4.5}$ is detected
by $M \D h_1^2 d_0$.
In particular, $\{\D h_1 h_3\} \theta_{4.5}$ is non-zero, and
it can only be detected by $M \D h_1 h_3$.
\end{proof}

\section{Additional relations}

\begin{lemma}
\label{lem:nubar-h5^2}
The product $(\eta \sigma + \epsilon) \theta_5$ is detected by
$\tau h_0 h_2 Q_3$.
\end{lemma}

\begin{proof}
Table \ref{tab:Toda} indicates a hidden
$\eta$ extension from $h_2^2 h_6$ to $\tau h_0 h_2 Q_3$.
Therefore, there exists an element $\alpha$ in $\pi_{69,36}$
such that $\tau h_0 h_2 Q_3$ detects $\eta \alpha$.
(Beware of the crossing extension from $p'$ to $h_1 p'$.
This means that it is possible to choose such an $\alpha$, but not
any element detected by $h_2^2 h_6$ will suffice.)

Table \ref{tab:Toda} shows that $h_2^2 h_6$ detects the
Toda bracket $\langle \nu^2, 2, \theta_5 \rangle$.
Let $\beta$ be an element of this Toda bracket.
Since $\alpha$ and $\beta$ are both detected by $h_2^2 h_6$, 
the difference $\alpha - \beta$ is
detected in Adams filtration at least $4$.

Table \ref{tab:misc-extn} shows that
$p'$ detects $\sigma \theta_5$, which belongs to the 
indeterminacy of $\langle \nu^2, 2, \theta_5 \rangle$.
Therefore, we may choose $\beta$ such that
the difference $\alpha - \beta$
is detected in filtration at least 9.
Since $\eta \alpha$ is detected by $\tau h_0 h_2 Q_3$ in filtration 7,
it follows that $\eta \beta$ is also detected by $\tau h_0 h_2 Q_3$.

We now have an element $\beta$ contained in
$\langle \nu^2, 2, \theta_5 \rangle$ such that
$\eta \beta$ is detected by $\tau h_0 h_2 Q_3$.
Now consider the shuffle
\[
\eta \langle \nu^2, 2, \theta_5 \rangle =
\langle \eta, \nu^2, 2 \rangle \theta_5.
\]
Table \ref{tab:Toda} shows that the last bracket equals
$\{ \epsilon, \epsilon + \eta \sigma \}$.
Therefore, either
$\epsilon \theta_5$ or
$(\epsilon + \eta \sigma) \theta_5$ is detected by
$\tau h_0 h_2 Q_3$.
But $\epsilon \theta_5$ is detected by $h_1 p' = h_5^2 c_0$.
\end{proof}

\begin{lemma}
\label{lem:compound-(epsilon+etasigma)-theta5}
There exists an element $\alpha$ of $\pi_{67,36}$ that is 
detected by $h_0 Q_3 + h_0 n_1$ such that
$\tau \nu \alpha$ equals $(\eta \sigma + \epsilon) \theta_5$.
\end{lemma}

\begin{proof}
Lemma \ref{lem:nubar-h5^2} shows that
$\tau h_0 h_2 Q_3$ detects $(\epsilon + \eta \sigma) \theta_5$.
The element $\tau h_0 h_2 Q_3$ also detects $\tau \nu \alpha$.
Let $\beta$ be the difference
$\tau \nu \alpha - (\epsilon + \eta \sigma) \theta_5$,
which is detected in higher Adams filtration.
We will show that $\beta$ must equal zero.

First, $\tau h_1 D'_3$ cannot detect $\beta$ because
$\eta^2 \beta$ is zero, while 
Table \ref{tab:eta-extn} shows that $\tau^3 d_1 g^2$ detects
$\eta^2 \{\tau h_1 D'_3\}$.
Second, Table \ref{tab:nu-extn} shows that $\tau h_1 h_3(\D e_1 + C_0)$
is the target of a hidden $\nu$ extension.
Therefore, we may alter the choice of $\alpha$
to ensure that $\beta$ is not detected by $\tau h_1 h_3(\D e_1 + C_0)$.
Third, $\D^2 h_2 c_1$ is also the target of a $\nu$ extension.
Therefore, we may alter the choice of $\alpha$
to ensure that $\beta$ is not detected by $\D^2 h_2 c_1$.
Finally, comparison to $\tmf$ implies that
$\beta$ is not detected by 
$\tau \D^2 h_1^2 g + \tau^3 \D h_2^2 g^2$. 
\end{proof}

\rev{
\begin{lemma}
\label{lem:sigma^2+kappa-h5^2}
The product $(\sigma^2 + \kappa) \theta_5$ is
zero, or it is equal to $\tau^2 \kappa_1 \kappabar_2$ detected
by $\tau^2 d_1 g_2$.
\end{lemma}

\begin{proof}
The product
$(\sigma^2 + \kappa) \theta_5$ maps to zero under
inclusion of the bottom cell of $C\tau$.
Therefore, $(\sigma^2 + \kappa) \theta_5$ is divisible
by $\tau$.  The only two possibilities are $0$ and
$\tau^2 \kappa_1 \kappabar_2$.
\end{proof}
}

\begin{lemma}
\label{lem:eta-sigma-k1}
$\eta \sigma \{k_1\} + \nu \{d_1 e_1\} = 0$ in
$\pi_{73,41}$.
\end{lemma}

\begin{proof}
We have the relation $h_1 h_3 k_1 + h_2 d_1 e_1 = 0$
in the Adams $E_\infty$-page, but
$\eta \sigma \{k_1\} + \nu\{d_1 e_1\}$ could possibly be 
detected in higher Adams filtration.
However, it cannot be detected by
$h_2^2 C''$ or $M h_1 h_3 g$ by comparison to $C\tau$.
Also, it cannot be detected by $\D h_1 d_0 e_0^2$ by comparison
to $\mmf$. 
\end{proof}

%% file: more-stable-stems-tables.tex
\chapter{Tables}
\label{ch:table}

Table \ref{tab:notation} gives some notation for elements in
$\pi_{*,*}$.  The fourth column gives partial information that
reduces the indeterminacies in the definitions, but does not
completely specify a unique element in all cases.
See Section \ref{sctn:notation} for further discussion.

Table \ref{tab:unit-mmf} gives hidden values of the unit map
$\pi_{*,*} \map \pi_{*,*} \mmf$.
The elements in the third column belong to the
Adams $E_\infty$-page for $\mmf$ \cite{Isaksen09} \cite{Isaksen18}.
See Section \ref{sctn:mmf} for further discussion.

Table \ref{tab:Massey} lists information about some Massey products.
The fifth column indicates the proof.  When a differential appears
in this column, it indicates the May differential that can be used
with the May Convergence Theorem (see Remark \ref{rem:May-convergence})
to compute the bracket.  The sixth column
shows where each specific Massey product is used in the manuscript.
See Chapter \ref{ch:Massey} for more discussion.

Table \ref{tab:Adams-d2} lists all of the multiplicative generators 
of the Adams $E_2$-page through the 95-stem.  The third column indicates
the value of the $d_2$ differential, if it is non-zero.
A blank entry in the third column indicates that the $d_2$ differential
is zero.
The fourth column indicates the proof. A blank entry in the fourth column
indicates that there are no possible values for the differential.
The fifth column gives alternative names for the element, as used 
in \cite{Bruner97}, \cite{Isaksen14c}, or \cite{Tangora70a}.
See Sections \ref{sctn:notation} and
\ref{sctn:Adams-d2} for further discussion.

Table \ref{tab:Adams-perm} lists some elements in the 
Adams spectral sequence that are known to be permanent cycles.
The third column indicates the proof.  When a Toda bracket
appears in the third column, the Moss Convergence Theorem
\ref{thm:Moss} applied to that Toda bracket implies that the
element is a permanent cycle (see Table \ref{tab:Toda} for
more information).  When a product appears in the third column,
the element must survive to detect that product.

Table \ref{tab:Adams-d3} lists the multiplicative generators 
of the Adams $E_3$-page through the 95-stem
whose $d_3$ differentials are non-zero, or whose $d_3$ differentials
are zero for non-obvious reasons.
See Section \ref{sctn:Adams-d3} for further discussion.

Table \ref{tab:Adams-d4} lists the multiplicative generators 
of the Adams $E_4$-page through the 95-stem
whose $d_4$ differentials are non-zero, or whose $d_4$ differentials
are zero for non-obvious reasons.
See Section \ref{sctn:Adams-d4} for further discussion.

Table \ref{tab:Adams-d5} lists the multiplicative generators 
of the Adams $E_5$-page through the 95-stem
whose $d_5$ differentials are non-zero, or whose $d_5$ differentials
are zero for non-obvious reasons.
See Section \ref{sctn:Adams-d5} for further discussion.

Table \ref{tab:Adams-higher} lists the multiplicative generators 
of the Adams $E_r$-page, for $r \geq 6$, through the 90-stem
whose $d_r$ differentials are non-zero, or whose $d_r$ differentials
are zero for non-obvious reasons.
See Section \ref{sctn:Adams-higher} for further discussion.

Table \ref{tab:Toda} lists information about some Toda brackets.
\rev{
Whenever possible, we use Greek letter names to refer to specific homotopy
elements.  An expression of the form $\{ x\}$ means that the Toda
bracket computation applies to any homotopy element detected by
the element $x$ of the Adams $E_\infty$-page.  An expression
of the form $[ x ]$ means that the Toda bracket computation
applies to at least one homotopy element that is detected by $x$.
}
The third column of Table \ref{tab:Toda} gives an element of the
Adams $E_\infty$-page that detects an element of the Toda bracket.
The fourth column of Table \ref{tab:Toda} gives partial
information about indeterminacies, again by giving detecting elements
of the Adams $E_\infty$-page.  We have not completely analyzed the
indeterminacies of some brackets when the details are inconsequential
for our purposes; this is indicated by a blank entry in the fourth 
column.  The fifth column indicates the proof of 
the Toda bracket, and the sixth column shows where each specific
Toda bracket is used in the manuscript.
See Chapter \ref{ch:Toda} for further discussion.

Tables \ref{tab:hid-incl} and \ref{tab:hid-proj} gives
hidden values of the inclusion $\pi_{*,*} \map \pi_{*,*} C\tau$ of the
bottom cell, and of the projection $\pi_{*,*} C\tau \map \pi_{*-1,*+1}$
to the top cell.
See Section \ref{sctn:tau-extn} for further discussion.

Table \ref{tab:tau-extn} lists hidden $\tau$ extensions in the
$E_\infty$-page of the $\C$-motivic Adams spectral sequence.
See Section \ref{sctn:tau-extn} for further discussion.

Tables \ref{tab:2-extn}, \ref{tab:eta-extn}, and \ref{tab:nu-extn}
list hidden extensions by $2$, $\eta$, and $\nu$.  
The fourth column indicates the proof of each extension.
The fifth column gives additional information about each extension,
including whether it is a crossing extension and whether it
has indeterminacy in the sense of Section \ref{subsctn:hid-indet}.
See Sections \ref{sctn:2-extn}, \ref{sctn:eta-extn}, and \ref{sctn:nu-extn}
for further discussion.

Tables \ref{tab:2-extn-possible}, \ref{tab:eta-extn-possible}, 
and \ref{tab:nu-extn-possible} list possible hidden extensions
by $2$, $\eta$, and $\nu$ that we have not yet resolved.

Finally, Table \ref{tab:misc-extn} gives some various hidden
extensions by elements other than $2$, $\eta$, and $\nu$.
See Section \ref{sctn:misc-extn} for further discussion.

\newpage

\begin{longtable}{llll}
\caption{Notation for $\pi_{*,*}$  \label{tab:notation}
} \\
\toprule
$(s,w)$ & element & detected by & definition \\
\midrule \endfirsthead
\caption[]{Notation for $\pi_{*,*}$} \\
\toprule
$(s,w)$ & element & detected by & definition \\
\midrule \endhead
\bottomrule \endfoot
$(0,-1)$ & $\tau$ & $\tau$ & \\
$(0, 0)$ & $2$ & $h_0$ & \\
$(1, 1)$ & $\eta$ & $h_1$ & \\
$(3, 2)$ & $\nu$ & $h_2$ & \\
$(7, 4)$ & $\sigma$ & $h_3$ & \\
$(8, 5)$ & $\epsilon$ & $c_0$ & \\
$(9, 5)$ & $\mu_9$ & $P h_1$ & \\
$(14, 8)$ & $\kappa$ & $d_0$ & \\
$(15, 8)$ & $\rho_{15}$ & $h_0^3 h_4$ & \\
$(16, 9)$ & $\eta_4$ & $h_1 h_4$ & \\
$(17, 9)$ & $\mu_{17}$ & $P^2 h_1$ & \\
$(19, 11)$ & $\sigmabar$ & $c_1$ & \\
$(20, 11)$ & $\kappabar$ & $\tau g$ & 
	$\langle \kappa, 2, \eta, \nu \rangle$ \\
$(23, 12)$ & $\rho_{23}$ & $h_0^2 i + \tau P h_1 d_0$ & \\
$(25, 13)$ & $\mu_{25}$ & $P^3 h_1$ & \\
$(30,16)$ & $\theta_4$ & $h_4^2$ & \\
$(32, 17)$ & $\eta_5$ & $h_1 h_5$ & 
	in $\langle \eta, 2, \theta_4 \rangle$ \\
$(32, 18)$ & $\kappa_1$ & $d_1$ & \\
$(44, 24)$ & $\kappabar_2$ & $g_2$ & \\
$(45, 24)$ & $\theta_{4.5}$ & $h_4^3$ & 
	$\eta \theta_{4.5} \in \{M h_1\}$ \\
$(62, 32)$ & $\theta_5$ & $h_5^2$ & \\
$(63,32)$ & $\eta_6$ & $h_1 h_6$ & 
	in $\langle \eta, 2, \theta_5 \rangle$ \\
\end{longtable}

\newpage

\begin{longtable}{lll}
\caption{Some hidden values of the unit map of $\mmf$ \label{tab:unit-mmf}} \\
\toprule
$(s,f,w)$ & element & image \\
\midrule \endfirsthead
\caption[]{Some hidden values of the unit map of $\mmf$} \\
\toprule
$(s,f,w)$ & element & image  \\
\midrule \endhead
\bottomrule \endfoot
$(28,6,17)$ & $h_1 h_3 g$ & $c g$ \\
$(29,7,18)$ & $h_1^2 h_3 g$ & $h_1 c g$ \\
$(32,6,17)$ & $\D h_1 h_3$ & $\D c + \tau a g$ \\
$(33,7,18)$ & $\D h_1^2 h_3$ & $\D h_1 c$ \\
$(35,8,18)$ & $\tau^3 h_1 e_0^2$ & $P a n$ \\
$(40,10,21)$ & $\tau \D h_1^2 d_0$ & $P(\D c + \tau a g)$ \\
$(48,10,29)$ & $h_1 h_3 g^2$ & $c g^2$ \\
$(49,11,30)$ & $h_1^2 h_3 g^2$ & $h_1 c g^2$ \\
$(52,10,29)$ & $\D h_1 h_3 g$ & $(\D c + \tau a g) g$ \\
$(53,11,30)$ & $\D h_1^2 h_3 g$ & $\D h_1 c g$ \\
$(54,9,28)$ & $h_0 h_5 i$ & $\D^2 h_2^2$ \\
$(54,11,32)$ & $h_1^6 h_5 e_0$ & $d g^2$ \\
$(55,12,33)$ & $h_1^7 h_5 e_0$ & $h_1 d g^2$ \\
$(57,10,30)$ & $h_0 h_2 h_5 i$ & $\D h_1 (\D c + \tau a g)$ \\
$(59,12,33)$ & $P h_1^3 h_5 e_0$ & $\D h_1 d g$ \\
$(60,13,34)$ & $\tau^2 h_0 g^3$ & $\D h_1^2 d g$ \\
$(62,14,37)$ & $h_1^6 h_5 c_0 e_0$ & $c d g^2$ \\
$(63,15,38)$ & $h_1^7 h_5 c_0 e_0$ & $h_1 c d g^2$ \\
$(65,12,34)$ & $P h_5 j$ & $\D^2 h_2 d$ \\
$(66,14,37)$ & $P h_1^2 h_5 c_0 e_0$ & $(\D c + \tau a g) d g$ \\
$(67,15,38)$ & $P h_1^3 h_5 c_0 e_0$ & $\D h_1 c d g$ \\
$(68,13,37)$ & $P h_2 h_5 j$ & $\D^2 h_2^2 d$ \\
$(68,14,41)$ & $h_1 h_3 g^3$ & $c g^3$ \\
$(69,15,42)$ & $h_1^2 h_3 g^3$ & $h_1 c g^3$ \\
$(71,15,38)$ & $\D^2 h_0^2 h_2 g$ & $\D h_1 (\D c + \tau a g) d$ \\
$(72,14,41)$ & $\D h_1 h_3 g^2$ & $(\D c + \tau a g) g^2$ \\
$(73,15,42)$ & $\D h_1^2 h_3 g^2$ & $\D h_1 c g^2$ \\
$(77,15,42)$ & $\D^2 h_2^3 g$ & $\D h_1 (\D c + \tau a g) g$ \\
$(88,18,53)$ & $h_1 h_3 g^4$ & $c g^4$ \\
$(89,19,54)$ & $h_1^2 h_3 g^4$ & $h_1 c g^4$ \\
\end{longtable}

\begin{landscape}

\begin{longtable}{llllll}
\caption{Some Massey products in $\Ext_\C$ \label{tab:Massey}} \\
\toprule
$(s,f,w)$ & bracket & contains & indeterminacy & proof & used for \\
\midrule \endfirsthead
\caption[]{Some Massey products in $\Ext_\C$} \\
\toprule
$(s,f,w)$ & bracket & contains & indeterminacy & proof & used for \\
\midrule \endhead
\bottomrule \endfoot
$(2, 2, 1)$ & $\langle h_0, h_1, h_0 \rangle$ & $\tau h_1^2$ & $0$ & Theorem \ref{thm:Massey-symmetric} & $\langle 2, \eta, 2 \rangle$ \\
\revrow
$(3, 2, 2)$ & $\langle h_1, h_0, h_1 \rangle$ & $h_0 h_2$ & $0$ & Theorem \ref{thm:Massey-symmetric} & $\langle \eta, 2, \eta \rangle$, \ref{lem:eta-h2^2h6} \\
& & & & & $\langle \{h_1 x_{76,6} \}, 2, \eta \rangle$  \\
$(6, 2, 4)$ & $\langle h_1, h_2, h_1 \rangle$ & $h_2^2$ & $0$ & Theorem \ref{thm:Massey-symmetric} & $\langle \eta, \nu, \eta \rangle$ \\
$(8, 2, 5)$ & $\langle h_2, h_1, h_2 \rangle$ & $h_1 h_3$ & $0$ & Theorem \ref{thm:Massey-symmetric} & $\langle \nu, \eta, \nu \rangle$ \\
$(8, 3, 5)$ & $\langle h_1^2, h_0, h_1, h_2 \rangle$ & $c_0$ & $0$ & $d_1(h_{20}) = h_0 h_1$ & $\langle \eta^2, 2, \eta, \nu \rangle$ \\
 & & & &  $d_1(h_{21}) = h_1 h_2$ \\
$(8, 3, 5)$ & $\langle h_1, h_2, h_0 h_2 \rangle$ & $c_0$ & $0$ & $d_2(h_0(1) = h_0 h_2^2$ & $\langle \eta, \nu, 2 \nu \rangle$, $\langle \eta_5, \nu, 2 \nu \rangle$ \\
\end{longtable}

\end{landscape}

\begin{longtable}{lllll}
\caption{Generators of the $\C$-motivic Adams $E_2$-page
\label{tab:Adams-d2}
} \\
\toprule
$(s,f,w)$ & element & $d_2$ & proof & other names \\
\midrule \endfirsthead
\caption[]{Generators of the $\C$-motivic Adams $E_2$-page
} \\
\toprule
$(s,f,w)$ & element & $d_2$ & proof & other names \\
\midrule \endhead
\bottomrule \endfoot
$(0, 1, 0)$ & $h_0$ & $$ & \\
$(1, 1, 1)$ & $h_1$ & $$ & \\
$(3, 1, 2)$ & $h_2$ & $$ & \\
$(7, 1, 4)$ & $h_3$ & $$ & \\
$(8, 3, 5)$ & $c_0$ & $$ & \\
$(9, 5, 5)$ & $P h_1$ & $$ & \\
$(11, 5, 6)$ & $P h_2$ & $$ & \\
$(14, 4, 8)$ & $d_0$ & $$ & \\
\end{longtable}

\begin{longtable}{lll}
\caption{Some $\C$-motivic permanent cycles
\label{tab:Adams-perm}
} \\
\toprule
$(s,f,w)$ & element & proof \\
\midrule \endfirsthead
\caption[]{Some $\C$-motivic permanent cycles
} \\
\toprule
$(s,f,w)$ & element & proof \\
\midrule \endhead
\bottomrule \endfoot
$(36,6,20)$ & $t$ & $\langle \tau, \eta^2 \kappa_1, \eta \rangle$ \\
$(64,2,33)$ & $h_1 h_6$ & $\langle \eta, 2, \theta_5 \rangle$ \\
$(68,7,36)$ & $h_3 A'$ & $\langle \sigma, \kappa, \tau \eta \theta_{4.5} \rangle$ \\
$(69,3,36)$ & $h_2^2 h_6$ & $\langle \nu^2, 2, \theta_5 \rangle$ \\
$(69,4,36)$ & $p'$ & $\sigma \theta_5$ \\
$(70,5,36)$ & $h_0^3 h_3 h_6$ & $\langle 8 \sigma, 2, \theta_5 \rangle$ \\
$(70,14,37)$ & $\tau \D^2 h_1^2 g + \tau^3 \D h_2^2 g^2$ &
	$\langle \eta, \tau \kappa^2, \tau \kappabar^2 \rangle$ \\
$(71,4,37)$ & $h_6 c_0$ & $\langle \epsilon, 2, \theta_5 \rangle$ \\
\end{longtable}

\newpage

\begin{longtable}{llll}
\caption{$\C$-motivic Adams $d_3$ differentials
\label{tab:Adams-d3}
} \\
\toprule
$(s,f,w)$ & element & $d_3$ & proof \\
\midrule \endfirsthead
\caption[]{$\C$-motivic Adams $d_3$ differentials
} \\
\toprule
$(s,f,w)$ & element & $d_3$ & proof \\
\midrule \endhead
\bottomrule \endfoot
$(15, 2, 8)$ & $h_0 h_4$ & $h_0 d_0$ & $C\tau$ \\
$(30, 6, 16)$ & $\D h_2^2$ & $\tau h_1 d_0^2$ & $\tmf$ \\
$(31, 4, 16)$ & $h_0^3 h_5$ & $h_0 \cdot \D h_2^2$ & $C\tau$ \\
$(31, 8, 17)$ & $\tau d_0 e_0$ & $P c_0 d_0$ & $d_4(\tau^2 d_0 e_0 + h_0^7 h_5)$ \\
$(34, 2, 18)$ & $h_2 h_5$ & $\tau h_1 d_1$ & Lemma \ref{lem:d3-h2h5} \\
$(37, 8, 21)$ & $\tau e_0 g$ & $c_0 d_0^2$ & $d_4(\tau^2 e_0 g)$ \\
$(38, 4, 21)$ & $e_1$ & $h_1 t$ & $C\tau$ \\
$(39, 12, 21)$ & $\tau P d_0 e_0$ & $P^2 c_0 d_0$ & $h_1$ \\
\end{longtable}

\begin{longtable}{llll}
\caption{$\C$-motivic Adams $d_4$ differentials
\label{tab:Adams-d4}
} \\
\toprule
$(s,f,w)$ & element & $d_4$ & proof \\
\midrule \endfirsthead
\caption[]{$\C$-motivic Adams $d_4$ differentials
} \\
\toprule
$(s,f,w)$ & element & $d_4$ & proof \\
\midrule \endhead
\bottomrule \endfoot
$(31,8,16)$ & $\tau^2 d_0 e_0 + h_0^7 h_5$ & $P^2 d_0$ & $C\tau$ \\
$(37,8,20)$ & $\tau^2 e_0 g$ & $P d_0^2$ & $\tmf$ \\
$(38,2,20)$ & $h_3 h_5$ & $h_0 x$ & $C\tau$ \\
$(39,12,20)$ & $\tau^2 P d_0 e_0$ & $P^3 d_0$ & $\tmf$ \\
$(42, 6, 22)$ & $P h_2 h_5$ & $0$ & $d_0$ \\
$(47,16,24)$ & $\tau^2 P^2 d_0 e_0$ & $P^4 d_0$ & $\tmf$ \\
$(50,6,27)$ & $C$ & $P h_1^2 h_5 c_0$ & $C\tau$ \\
$(50,10,26)$ & $\tau^2 \Delta h_2^2 g$ & $i j$ & $\tmf$ \\
\end{longtable}

\newpage

\begin{longtable}{llll}
\caption{$\C$-motivic Adams $d_5$ differentials
\label{tab:Adams-d5}
} \\
\toprule
$(s,f,w)$ & element & $d_5$ & proof \\
\midrule \endfirsthead
\caption[]{$\C$-motivic Adams $d_5$ differentials
} \\
\toprule
$(s,f,w)$ & element & $d_5$ & proof \\
\midrule \endhead
\bottomrule \endfoot
$(56, 9, 29)$ & $\tau P h_5 e_0$ & $\tau \D h_0^2 d_0 e_0$ &  \cite{Isaksen14c}*{Lemma 3.92}\\
$(61, 6, 32)$ & $A'$ & $\tau M h_1 d_0$ & \cite{WangXu17}*{Theorem 12.1} \\
$(63, 11, 33)$ & $\tau h_1^2 \cdot \D x$ & $\tau^3 d_0^2 e_0^2$ & Lemma \ref{lem:d5-th1^2.Dx} \\
$(63, 23, 32)$ & $h_0^{22} h_6$ & $P^6 d_0$ & $C\tau$ \\
$(67, 6, 36)$ & $h_0 Q_3 + h_2^2 D_3$ & $0$ & $C\tau$, $\tau$ \\
$(68, 12, 36)$ & $h_5 d_0 i$ & $\tau \D h_1 d_0^3$ & Lemma \ref{lem:d5-h5d0i} \\
$(70, 4, 36)$ & $\tau p_1 + h_0^2 h_3 h_6$ & $\tau^2 h_2^2 C'$ & Lemma \ref{lem:d5-tp1+h0^2h3h6}, \cite{BIX} \\
$(72, 7, 39)$ & $h_1 x_{71,6}$ & $0$ & Lemma \ref{lem:d5-h1x71,6} \\
$(73, 7, 38)$ & $h_4 D_2$ & $\tau^4 d_1 g^2$ & Lemma \ref{lem:d5-h4D2} \\
$(81, 10, 44)$ & $g A'$ & $\tau M h_1 e_0^2$ & $\tau$ \\
$(85,6,45)$ & $x_{85,6}$ & $0$ & $h_2$ \\
$(86, 10, 44)$ & $h_0^2 h_6 i$ & $\D^2 h_0^2 x$ & $C\tau$ \\
$(86, 11, 45)$ & $\tau^3 g G_0$ & $\tau M \D h_1^2 d_0$ & Lemma \ref{lem:d5-t^3gG0} \\
$(92, 4, 48)$ & $g_3$ & $h_6 d_0^2$ & Lemma \ref{lem:d5-g3} \\
\revrow
$(92, 12, 48)$ & $\D^2 g_2$ & $0$ & Lemma \ref{lem:d5-D^2g2}, \cite{Chua21} \\
$(93, 8, 48)$ & $h_0 \cdot \D h_2^2 h_6$ & $\D^2 h_0 g_2$ & $C\tau$ \\
$(93, 13, 50)$ & $e_0 x_{76,9}$ & $M \D h_1 c_0 d_0$ & Lemma \ref{lem:d5-e0x76,9} \\
\end{longtable}

\begin{longtable}{lllll}
\caption{$\C$-motivic higher Adams differentials
\label{tab:Adams-higher}
} \\
\toprule
$(s,f,w)$ & element & $r$ & $d_r$ & proof \\
\midrule \endfirsthead
\caption[]{$\C$-motivic higher Adams differentials
} \\
\toprule
$(s,f,w)$ & element & $r$ & $d_r$ & proof \\
\midrule \endhead
\bottomrule \endfoot
$(67, 5, 35)$ & $\tau Q_3 + \tau n_1$ & $6$ & $0$ & Lemma \ref{lem:d6-tQ3+tn1} \\
$(68, 7, 37)$ & $h_2^2 H_1$ & $6$ & $M c_0 d_0$ & 
	Lemma \ref{lem:d6-h2^2H1} \\
$(68, 7, 36)$ & $\tau h_2^2 H_1$ & $7$ & $M P d_0$ & 
	Lemma \ref{lem:d6-h2^2H1} \\
$(77, 7, 42)$ & $m_1$ & $7$ & $0$ & Lemma \ref{lem:d7-m1} \\
$(80, 6, 43)$ & $h_1 x_1$ & $8$ & $0$ & Lemma \ref{lem:d8-h1x1} \\
$(81, 3, 42)$ & $h_2 h_4 h_6$ & $6$ & $0$ 
	& Lemma \ref{lem:d6-h2h4h6} \\
\revrow
$(83, 5, 43)$ & $\tau h_6 g + \tau h_2 e_2$ & $9$ & $0$ & \cite{BIX} \\
\revrow
$(85, 5, 45)$ & $h_1 f_2$ & $10$ & $? M \D h_1 d_0$ & \\
\end{longtable}

\begin{landscape}

\begin{longtable}{llllll}
\caption{Some Toda brackets \label{tab:Toda}
} \\
\toprule
$(s,w)$ & bracket & contains & indeterminacy & proof & used in \\
\midrule \endfirsthead
\caption[]{Some Toda brackets} \\
\toprule
$(s,w)$ & bracket & detected by & indeterminacy & proof & used in \\
\midrule \endhead
\bottomrule \endfoot
$(2, 1)$ & $\langle 2, \eta, 2 \rangle$ & $\tau h_1^2$ & $0$ & $\langle h_0, h_1, h_0 \rangle$ & \ref{lem:2,eta,th1^2x76,6}, \ref{lem:2-th1H1}, \ref{lem:2-h1h6}, \ref{lem:2-h1h3h6}, \ref{lem:2-e0A'} \\
& & & & & \ref{lem:2-h1h4h6}, \ref{lem:2-Ph6c0}, 
\ref{lem:2-P^2h6c0}, \ref{lem:eta-th1^2x76,6} \\
\revrow
$(3, 2)$ & $\langle \eta, 2, \eta \rangle$ & $h_0 h_2$ & $\tau h_1^3$ & $\langle h_1, h_0, h_1 \rangle$ & \ref{lem:2-h2^2x76,6} \\
$(6, 4)$ & $\langle \eta, \nu, \eta \rangle$ & $h_2^2$ & $0$ & $\langle h_1, h_2, h_1 \rangle$ & \ref{lem:eta-h0h6g}, \ref{lem:eta-D^2h1f1} \\
$(8, 5)$ & $\langle \nu, \eta, \nu \rangle$ & $h_1 h_3$ & $0$ & $\langle h_2, h_1, h_2 \rangle$ & \ref{lem:nu-th1D3'} \\
$(8, 5)$ & $\langle \eta, \nu, 2 \nu \rangle$ & $c_0$ & $h_1 h_3$ & $\langle h_1, h_2, h_0 h_2 \rangle$ & \ref{lem:eta-h0d0D2}, \ref{lem:nubar-h5^2} \\
$(8, 5)$ & $\langle \eta^2, \nu, \eta, 2 \rangle$ & $c_0$ & $0$ & $\langle h_1^2, h_2, h_1, h_0 \rangle$ & \ref{lem:kappa,2,eta,nu} \\
$(9, 5)$ & $\langle 2, \epsilon, 2 \rangle$ & $\tau h_1 c_0$ & $0$ & Corollary \ref{cor:2-symmetric} & \ref{lem:2-h6c0} \\
\end{longtable}

\rev{
\begin{longtable}{llllll}
\caption{Some null Toda brackets \label{tab:Toda-null}
} \\
\toprule
$(s,w)$ & bracket & contains & indeterminacy & proof & used in \\
\midrule \endfirsthead
\caption[]{Some null Toda brackets} \\
\toprule
$(s,w)$ & bracket & detected by & indeterminacy & proof & used in \\
\midrule \endhead
\bottomrule \endfoot
$(16, 9)$ & $\langle \kappa, 2, \eta \rangle$ & $0$ & $P c_0$ & Lemma \ref{lem:kappa,2,eta} & \ref{lem:kappa,2,eta,nu}, \ref{lem:nu-th1H1} \\
$(23,13)$ & $\langle \epsilon + \eta \sigma, \sigma, 2 \sigma \rangle$ & $0$ & $P h_1 d_0$ & Lemma \ref{lem:nubar,sigma,2sigma} & \ref{lem:d4-h1c3} \\
$(30, 16)$ & $\langle \tau \kappa^2, \eta, 2 \rangle$ & $0$ & $0$ & Lemma \ref{lem:tkappa^2,eta,2} & \ref{lem:d5-g3} \\
$(35, 20)$ & $\langle \eta^2, \theta_4, \eta^2 \rangle$ & $0$ & $h_1^4 h_5$ & Lemma \ref{lem:eta^2,theta4,eta^2} & \ref{lem:eta^2,theta4,eta^2,theta4} \\
$(46, 25)$ & $\langle \eta, 2, 4 \kappabar_2 \rangle$ & $0$ & $h_1 h_5 d_0$, $M h_1$, $\D h_2 c_1$, $\tau d_0 l + \D c_0 d_0$ & Lemma \ref{lem:eta,2,4kappabar2} & \ref{lem:d6-tQ3+tn1}, \ref{lem:teta^2kappabar,8,kappabar2} \\
$(59, 31)$ & $\langle \tau \kappabar_2, \sigma^2, 2 \rangle$ & $0$ & $0$ & Lemma \ref{lem:tkappabar2,sigma^2,2} & \ref{lem:2-h3(tQ3+tn1)} \\
\revrow
$(60, 33)$ & $\langle \eta \kappabar_2, 2 \sigma, \sigma \rangle$ & $0$ & $0$ & Lemma \ref{lem:etakappabar2,2sigma,sigma} \\
$(60, 32)$ & $\langle 2, \sigma^2, \{ h_3^2 h_5 \} \rangle$ & $0$ & $\tau^2 d_0^2 l$ & Lemma \ref{lem:2,sigma^2,theta4.5} & \ref{lem:d3-tD3'} \\
\end{longtable}
}
\end{landscape}

\begin{longtable}{llll}
\caption{Hidden values of inclusion of the bottom cell into $C\tau$ \label{tab:hid-incl}
} \\
\toprule
$(s,f,w)$ & source & value & proof \\
\midrule \endfirsthead
\caption[]{Hidden values of inclusion of the bottom cell into $C\tau$} \\
\toprule
$(s,f,w)$ & source & value & proof \\
\midrule \endhead
\bottomrule \endfoot
$(50,6,26)$ & $\tau C$ & $\ol{P h_1^2 h_5 c_0}$ \\
$(57,10,30)$ & $h_0 h_2 h_5 i$ & $\D^2 h_1^2 h_3$ \\
$(63,6,33)$ & $\tau h_1 H_1$ & $\ol{h_1 B_7}$ \\
$(63,7,33)$ & $\tau X_2 + \tau C'$ & $\ol{h_5 d_0 e_0}$ \\
$(64,8,34)$ & $\tau h_1 X_2$ & $\ol{h_1 h_5 d_0 e_0}$ \\
$(66,8,35)$ & $\tau h_2 C'$ & $\tau B_5$ \\
$(70,7,37)$ & $\tau h_1 h_3 H_1$ & $\ol{h_1 h_3 B_7}$ \\
$(70,8,37)$ & $\tau h_1 D_3'$ & $\ol{h_1^2 X_3}$ \\
$(70,10,38)$ & $h_1 h_3 (\D e_1+C_0) + \tau h_2 C''$ & $\ol{h_1^4 c_0 Q_2}$ & Lemma \ref{lem:t-d1e1} \\
$(71,5,37)$ & $\tau h_1 p_1$ & $\ol{h_0^2 h_2 Q_3}$ \\
$(74,6,39)$ & $h_3 (\tau Q_3 + \tau n_1)$ & $\ol{h_1^4 p'}$ \\
$(75,11,41)$ & $h_1^3 h_4 Q_2$ & $\ol{\tau h_2 g^2 n}$ \\
$(77,6,40)$ & $\tau h_1 h_4 D_3$ & $x_{77,7}$ \\
$(81,8,43)$ & $\tau g D_3$ & $\ol{h_1^4 x_{76,6}}$ \\
$(83,10,45)$ & $h_2 c_1 A'$ & $\ol{h_1 g B_7}$ \\
\revrow
$(85,5,44)$ & $\tau h_1 f_2$ & $? h_0^3 c_3$ \\
\revrow
$(86,6,45)$ & $\tau h_1^2 f_2$ & $? \ol{\tau h_1^3 h_4 Q_3}$ \\
	\revrow
	& & $? \D^2 e_1 + \ol{\tau \D h_2 e_1 g}$ \\
\revrow
$(86,7,45)$ & $\tau h_1 x_{85,6}$ &  $? \ol{\tau h_1^3 h_4 Q_3}$ \\
	\revrow
	& & $? \D^2 e_1 + \ol{\tau \D h_2 e_1 g}$ \\
$(86,12,47)$ & $\tau h_2 g C'$ & $\tau B_5 g$ \\
$(88,11,48)$ & $\tau h_2^2 g H_1$ & $\D g_2 g$ \\
$(90,14,50)$ & $\tau h_2 g C''$ & $\ol{P h_1^{10} h_6 c_0}$ \\
$(90,19,49)$ & $\tau^3 g m^2$ & $\D^2 c_0 e_0^2$ \\
\end{longtable}

\newpage

\begin{longtable}{llll}
\caption{Hidden values of projection from $C\tau$ to the top cell \label{tab:hid-proj}
} \\
\toprule
$(s,f,w)$ & source & value & crossing source \\
\midrule \endfirsthead
\caption[]{Hidden values of projection from $C\tau$ to the top cell} \\
\toprule
$(s,f,w)$ & source & value & crossing source \\
\midrule \endhead
\bottomrule \endfoot
$(30,6,16)$ & $\D h_2^2$ & $h_1 d_0^2$ \\
$(34,2,18)$ & $h_2 h_5$ & $h_1 d_1$ \\
$(38,7,20)$ & $h_0 y$ & $\tau h_0 e_0 g$ \\
$(41,4,22)$ & $h_0 c_2$ & $h_1 h_3 d_1$ \\
$(44,10,24)$ & $\D h_2^2 d_0$ & $h_1 d_0^3$ \\
$(50,10,28)$ & $\D h_2^2 g$ & $h_1 d_0 e_0^2$ \\
$(55,7,30)$ & $B_6$ & $h_2 g n$ \\
$(56,10,29)$ & $\D^2 h_1 h_3$ & $\D h_0^2 d_0 e_0$ \\
\end{longtable}

\newpage

\begin{longtable}{llll}
\caption{Hidden $\tau$ extensions
\label{tab:tau-extn}
} \\
\toprule
$(s,f,w)$ & from & to & proof \\
\midrule \endfirsthead
\caption[]{Hidden $\tau$ extensions
} \\
\toprule
$(s,f,w)$ & from & to & proof \\
\midrule \endhead
\bottomrule \endfoot
$(22, 7, 13)$ & $c_0 d_0$ & $P d_0$ \\
$(23, 8, 14)$ & $h_1 c_0 d_0$ & $P h_1 d_0$ \\
$(28, 6, 17)$ & $h_1 h_3 g$ & $d_0^2$ \\
$(29, 7, 18)$ & $h_1^2 h_3 g$ & $h_1 d_0^2$ \\
$(40, 9, 23)$ & $\tau h_0 g^2$ & $\D h_1^2 d_0$ \\
$(41, 9, 23)$ & $\tau^2 h_1 g^2$ & $\D h_0^2 e_0$ \\
$(42, 11, 25)$ & $c_0 e_0^2$ & $d_0^3$ \\
$(43, 12, 26)$ & $h_1 c_0 e_0^2$ & $h_1 d_0^3$ \\
\end{longtable}

\newpage

\begin{longtable}{lllll}
\caption{Hidden $2$ extensions \label{tab:2-extn}
} \\
\toprule
$(s,f,w)$ & source & target & proof & notes \\
\midrule \endfirsthead
\caption[]{Hidden $2$ extensions} \\
\toprule
$(s,f,w)$ & source & target & proof & notes \\
\midrule \endhead
\bottomrule \endfoot
$(23, 6, 13)$ & $\tau h_0 h_2 g$ & $P h_1 d_0$ & $\tau$ \\
$(23, 6, 14)$ & $h_0 h_2 g$ & $h_1 c_0 d_0$ & $C\tau$ \\
$(40, 8, 22)$ & $\tau^2 g^2$ & $\D h_1^2 d_0$ & $\tau$ \\
$(43, 10, 25)$ & $\tau h_0 h_2 g^2$ & $h_1 d_0^3$ & $\tau$ \\
$(43, 10, 26)$ & $h_0 h_2 g^2$ & $h_1 c_0 e_0^2$ & $C\tau$ \\
$(47, 10, 25)$ & $\tau \D h_2^2 e_0$ & $P \D h_1 d_0$ & $\tau$ \\
$(47, 10, 26)$ & $\D h_2^2 e_0$ & $\D h_1 c_0 d_0$ & $C\tau$ \\
$(51, 6, 28)$ & $h_0 h_3 g_2$ & $\tau g n$ & \cite{WangXu18} \\
\end{longtable}

\newpage

\rev{
\begin{longtable}{lll}
\caption{Some null hidden $2$ extensions \label{tab:2-extn-null}
} \\
\toprule
$(s,f,w)$ & source & proof \\
\midrule \endfirsthead
\caption[]{Some null hidden $2$ extensions} \\
\toprule
$(s,f,w)$ & source & proof \\
\midrule \endhead
\bottomrule \endfoot
$(63, 7, 33)$ & $\tau X_2 + \tau C'$  & Lemma \ref{lem:2-th1H1} \\
$(66, 6, 36)$ & $\D_1 h_3^2$  & Lemma \ref{lem:2-D1h3^2} \\
$(67, 6, 36)$ & $h_0 Q_3 + h_2^2 D_3$  & Lemma \ref{lem:2-h0Q3} \\
$(68, 7, 36)$ & $h_3 A'$  & Lemma \ref{lem:2-h3A'} \\
$(69, 4, 36)$ & $p'$  & Lemma \ref{lem:2-p'} \\
$(70, 9, 37)$ & $\tau h_1 D_3'$  & Lemma \ref{lem:2-th1D3'} \\
$(71, 3, 37)$ & $h_1 h_3 h_6$  & Lemma \ref{lem:2-h1h3h6} \\
$(71, 5, 37)$ & $\tau h_1 p_1$  & Lemma \ref{lem:2-th1p1} \\
\end{longtable}
}

\begin{longtable}{lll}
\caption{Possible hidden $2$ extensions \label{tab:2-extn-possible}
} \\
\toprule
$(s,f,w)$ & source & target \\
\midrule \endfirsthead
\caption[]{Possible hidden $2$ extensions} \\
\toprule
$(s,f,w)$ & source & target \\
\midrule \endhead
\bottomrule \endfoot
$(59, 7, 33)$ & $j_1$ & $? \tau^2 c_1 g^2$ \\
$(72, 7, 39)$ & $h_1 x_{71,6}$ & $? \tau^3 d_1 g^2$ \\
$(79, 11, 45)$ & $g j_1$ & $? \tau^2 c_1 g^3$ \\
\revrow
$(85, 5, 45)$ & $h_1 f_2$ & $? \tau h_2 h_4 Q_3$ \\
	\revrow
	& & $? \tau P h_1 x_{76,6}$ \\
\revrow
$(85, 5, 44)$ & $\tau h_1 f_2$ & $? \tau^2 h_2 h_4 Q_3$ \\
	\revrow
	& & $? \tau^2 P h_1 x_{76,6}$ \\
	\revrow
	& & $? \tau^4 M g^2$ \\
\revrow
$(86, 6, 46)$ & $h_1^2 f_2$ & $? \D^2 h_2^2 d_1$ \\
\revrow
$(86, 6, 45)$ & $\tau h_1^2 f_2$ & $? \tau \D^2 h_2^2 d_1$ \\
\revrow
$(86, 7, 45)$ & $\tau h_1 x_{85,6}$ & $? \tau \D^2 h_2^2 d_1$ \\
	\revrow
	& & $? P h_1 h_6 d_0$ \\
\end{longtable}

\newpage

\begin{longtable}{lllll}
\caption{Hidden $\eta$ extensions \label{tab:eta-extn}
} \\
\toprule
$(s,f,w)$ & source & target & proof & notes \\
\midrule \endfirsthead
\caption[]{Hidden $\eta$ extensions} \\
\toprule
$(s,f,w)$ & source & target & proof & notes \\
\midrule \endhead
\bottomrule \endfoot
$(15, 4, 8)$ & $h_0^3 h_4$ & $P c_0$ & $C\tau$ \\
$(21, 5, 11)$ & $\tau^2 h_1 g$ & $P d_0$ & $\tau$ \\
$(21, 5, 12)$ & $\tau h_1 g$ & $c_0 d_0$ & $C\tau$ \\
$(23, 9, 12)$ & $h_0^2 i$ & $P^2 c_0$ & $C\tau$ & \\ 
$(31, 11, 16)$ & $h_0^{10} h_5$ & $P^3 c_0$ & $C\tau$ \\
$(38, 4, 20)$ & $h_0^2 h_3 h_5$ & $\tau^2 c_1 g$ & \cite{Isaksen14c}*{Table 29} \\
$(39, 17, 20)$ & $P^2 h_0^2 i$ & $P^4 c_0$ & $C\tau$ \\
$(40, 8, 21)$ & $\tau^3 g^2$ & $\D h_0^2 e_0$ & $\tau$ \\
\end{longtable}

\newpage

\rev{
\begin{longtable}{lll}
\caption{Some null hidden $\eta$ extensions \label{tab:eta-extn-null}
} \\
\toprule
$(s,f,w)$ & source & proof \\
\midrule \endfirsthead
\caption[]{Some null hidden $\eta$ extensions} \\ 
\toprule
$(s,f,w)$ & source & proof \\
\midrule \endhead
\bottomrule \endfoot
$(58, 8, 30)$ & $\tau h_1 Q_2$ & Lemma \ref{lem:eta-th1Q2} \\
$(64, 4, 33)$ & $\tau h_1^2 h_5^2$ & Lemma \ref{lem:eta-th1^2h5^2} \\
$(66, 4, 34)$ & $\tau h_1^3 h_6$ & Lemma \ref{lem:eta-th1^3h6} \\
$(68, 6, 36)$ & $\tau h_1 Q_3$ & Lemma \ref{lem:eta-th1Q3} \\
$(70, 5, 36)$ & $h_0^3 h_3 h_6$ & Lemma \ref{lem:eta-h0^3h3h6} \\
$(70, 6, 38)$ & $h_2 Q_3$ & Lemma \ref{lem:eta-h2Q3} \\
$(70, 10, 38)$ & $h_1 h_3 (\D e_1 + C_0)$ & Lemma \ref{lem:eta-h1h3(De1+C0)} \\
$(70, 10, 38)$ & $\tau h_2 C'' + h_1 h_3 (\D e_1 + C_0)$ & Lemma \ref{lem:eta-h1h3(De1+C0)} \\
$(71, 6, 37)$ & $\tau h_1^2 p'$ & Lemma \ref{lem:eta-th1^2p'} \\
$(71, 8, 39)$ & $h_2^3 H_1$ & Lemma \ref{lem:eta-h1^2h3H1} \\
$(72, 7, 39)$ & $h_1^3 p'$ & Lemma \ref{lem:eta-h1^3p'} \\
$(77, 3, 40)$ & $h_3^2 h_6$ & Lemma \ref{lem:eta-tm1} \\
$(77, 7, 41)$ & $\tau m_1$ & Lemma \ref{lem:eta-tm1} \\
$(77, 8, 40)$ & $h_0 x_{77,7}$ & Lemma \ref{lem:eta-h0x77,7} \\
$(78, 6, 41)$ & $h_1 h_6 d_0$ & Lemma \ref{lem:eta-h1h6d0} \\
$(78, 8, 41)$ & $\tau h_1^2 x_{76,6}$ & Lemma \ref{lem:eta-th1^2x76,6} \\
$(81, 5, 43)$ & $h_1^3 h_4 h_6$ & Lemma \ref{lem:eta-h1^3h4h6} \\
$(81, 7, 44)$ & $h_3^2 n_1$ & Lemma \ref{lem:eta-h3^2n1} \\
$(81, 12, 42)$ & $\D^2 p$ & Lemma \ref{lem:eta-D^2p} \\
$(82, 4, 43)$ & $h_6 c_1$ & Lemma \ref{lem:eta-h6c1} \\
$(83, 6, 44)$ & $h_0 h_6 g$ & Lemma \ref{lem:eta-h0h6g} \\
$(85, 7, 46)$ & $h_2 h_4 Q_3$ & Lemma \ref{lem:eta-h2h4Q3} \\
$(86, 9, 46)$ & $h_1 h_6 c_0 d_0$ & Lemma \ref{lem:eta-h1h6c0d0} \\
$(86, 10, 45)$ & $P h_1 h_6 d_0$ & Lemma \ref{lem:eta-h1h6c0d0} \\
$(87, 11, 48)$ & $B_6 d_1$ & Lemma \ref{lem:eta-B6d1} \\
$(88, 7, 47)$ & $h_1^2 h_4 h_6 c_0$ & Lemma \ref{lem:eta-h1^2h4h6c0} \\
\end{longtable}
}

\newpage

\begin{longtable}{llll}
\caption{Possible hidden $\eta$ extensions \label{tab:eta-extn-possible}
} \\
\toprule
$(s,f,w)$ & source & target & proof \\
\midrule \endfirsthead
\caption[]{Possible hidden $\eta$ extensions} \\
\toprule
$(s,f,w)$ & source & target & proof \\
\midrule \endhead
\bottomrule \endfoot
$(66, 6, 35)$ & $\tau \D_1 h_3^2$ & $? \tau^2 \D h_2^2 e_0 g$ & Lemma \ref{lem:eta-tD1h3^2} \\
$(66, 12, 35)$ & $\D^2 h_1^3 h_4$ & $? \tau^2 \D h_2^2 e_0 g$ & $$ \\
$(67, 6, 36)$ & $h_0 Q_3 + h_2^2 D_3$ & $? \tau M h_0 h_2 g$ & $$ \\
$(81,3,42)$ & $h_2 h_4 h_6$ & $? \tau h_5^2 g$ & Lemma \ref{lem:eta-h2h4h6} \\
&& $? \tau^2 e_1 g_2$ & $$ \\
&& $? \D^2 h_2 n$ & $$ \\
$(81,5,39)$ & $h_1^3 h_4 h_6$ & $? \tau (\D e_1 + C_0) g$ & Lemma \ref{lem:eta-h1^3h4h6} \\
$(81,8,42)$ & $\tau^2 g D_3$ & $? \D^2 h_2 n$ & $$ \\
$(81,8,43)$ & $\tau g D_3$ & $? \tau (\D e_1 + C_0) g$ & $$ \\
$(86,6,45)$ & $\tau h_1^2 f_2$ & $? \tau^2 g Q_3$ & $$ \\
&& $? \D^2 h_3 d_1$ & $$ \\
\revrow
$(86,6,46)$ & $h_1^2 f_2$ & $? h_1 h_4 x_{71,6}$ \\
	\revrow
	& & $? \tau g Q_3$ \\
	\revrow
	& & $? \tau h_2^2 g A'$ \\
\revrow
$(86,6,46)$ & $h_2 h_6 g + h_1^2 f_2$ & $? h_1 h_4 x_{71,6}$ & $$ \\
	\revrow
	&& $? \tau h_2^2 g A'$ & $$ \\
$(86,7,45)$ & $\tau h_1 x_{85,6}$ & $? \tau^2 g Q_3$ & \\
	&& $? \D^2 h_3 d_1$ & $$ \\
$(87,5,46)$ & $h_1^2 c_3$ & $? \tau h_0 g_2^2$ & $$ \\
$(87,6,45)$ & $\tau h_1 h_4 h_6 c_0$ & $? \tau^2 h_0 g_2^2$ & $$ \\
$(87,9,48)$ & $g Q_3$ & $? \tau M h_0 h_2 g^2$ & $$ \\
$(88,8,48)$ & $g_2^2$ & $? \tau h_2^2 g C'$ & $$ \\
&& $? \D h_1 g_2 g$ & $$ \\
\end{longtable}

\newpage

\begin{longtable}{lllll}
\caption{Hidden $\nu$ extensions \label{tab:nu-extn}
} \\
\toprule
$(s,f,w)$ & source & target & proof & notes \\
\midrule \endfirsthead
\caption[]{Hidden $\nu$ extensions} \\
\toprule
$(s,f,w)$ & source & target & proof & notes \\
\midrule \endhead
\bottomrule \endfoot
$(20, 6, 11)$ & $\tau h_0^2 g$ & $P h_1 d_0$ & $\tau$ \\
$(20, 6, 12)$ & $h_0^2 g$ & $h_1 c_0 d_0$ & $C\tau$ \\
$(22, 4, 13)$ & $h_2 c_1$ & $h_1^2 h_4 c_0$ & $C\tau$ \\
$(26, 6, 15)$ & $\tau h_2^2 g$ & $h_1 d_0^2$ & $\tau$ \\
$(30, 2, 16)$ & $h_4^2$ & $p$ & $C\tau$ \\
$(32, 6, 17)$ & $\D h_1 h_3$ & $\tau^2 h_1 e_0^2$ & $\tmf$ \\
$(39, 9, 21)$ & $\D h_1 d_0$ & $\tau d_0^3$ & $\tmf$ \\
$(40, 10, 23)$ & $\tau h_0^2 g^2$ & $h_1 d_0^3$ & $\tau$ \\
\end{longtable}

\newpage

\rev{
\begin{longtable}{lll}
\caption{Some null hidden $\nu$ extensions \label{tab:nu-extn-null}
} \\
\toprule
$(s,f,w)$ & source & proof \\
\midrule \endfirsthead
\caption[]{Some null hidden $\nu$ extensions} \\ 
\toprule
$(s,f,w)$ & source & proof \\
\midrule \endhead
\bottomrule \endfoot
$(64, 2, 33)$ & $h_1 h_6$ & Lemma \ref{lem:nu-h1h6} \\
$(64, 8, 34)$ & $h_3 Q_2$ & Lemma \ref{lem:nu-h3Q2} \\
$(68, 7, 36)$ & $h_3 A'$ & Lemma \ref{lem:nu-h3A'} \\
$(69, 4, 36)$ & $p'$ & Lemma \ref{lem:nu-p'} \\
$(71, 4, 37)$ & $h_6 c_0$ & Lemma \ref{lem:nu-h6c0} \\
$(71, 5, 37)$ & $\tau h_1 p_1$ & \cite{BIX} \\
$(73, 12, 41)$ & $M h_1 h_3 g$ & Lemma \ref{lem:nu-Mh1h3g} \\
$(76, 8, 41)$ & $\tau d_1 g_2$ & Lemma \ref{lem:nu-td1g2} \\
$(76, 8, 40)$ & $h_0 h_4 A$ & Lemma \ref{lem:nu-h0h4A} \\
$(78, 10, 42)$ & $e_0 A'$ & Lemma \ref{lem:nu-e0A'} \\
$(79, 3, 41)$ & $h_1 h_4 h_6$ & Lemma \ref{lem:nu-h1h4h6} \\
$(81, 7, 44)$ & $h_3^2 n_1$ & Lemma \ref{lem:nu-h3^2n1} \\
$(82, 8, 44)$ & $\tau e_1 g_2$ & Lemma \ref{lem:nu-te1g2} \\
$(82, 11, 42)$ & $P^2 h_0 h_2 h_6$ & Lemma \ref{lem:nu-Ph0h2h6} \\
$(83, 10, 45)$ & $h_2 c_1 A'$ & Lemma \ref{lem:nu-h2c1A'} \\
$(84, 4, 44)$ & $h_2^2 h_4 h_6$ & Lemma \ref{lem:nu-h2^2h4h6} \\
$(87, 12, 45)$ & $P^2 h_6 c_0$ & Lemma \ref{lem:nu-P^2h6c0} \\
\end{longtable}
}

\begin{longtable}{llll}
\caption{Possible hidden $\nu$ extensions \label{tab:nu-extn-possible}
} \\
\toprule
$(s,f,w)$ & source & target & proof  \\
\midrule \endfirsthead
\caption[]{Possible hidden $\nu$ extensions} \\
\toprule
$(s,f,w)$ & source & target & proof  \\
\midrule \endhead
\bottomrule \endfoot
$(70, 5, 36)$ & $h_0^3 h_3 h_6$ & $? h_0 h_4 D_2$  \\
\revrow
$(75, 6, 40)$ & $h_0 h_3 d_2$ & $? M \D h_1^2 h_3$ & Lemma \ref{lem:nu-h0h3d2}  \\
\revrow
$(81,12,42)$ & $\D^2 p$ & $? \tau M \D h_1 d_0$ \\
\revrow
$(85,5,45)$ & $h_1 f_2$ & $? \tau g_2^2$ \\
\revrow
$(85, 5, 44)$ & $\tau h_1 f_2$ & $? h_1 x_{87,7}$ or \\
	\revrow
	& & $? h_1 x_{87,7} + \tau^2 g_2^2$ \\
$(86,11,44)$ & $h_0^3 h_6 i$ & $? \tau \D^2 h_1 f_1$ \\
$(87,5,46)$ & $h_1^2 c_3$ & $? \tau M \D h_1 g$ \\
$(87,7,45)$ & $x_{87,7}$ & $? \tau^2 M \D h_1 g$ \\
\end{longtable}

\newpage

\begin{longtable}{lllll}
\caption{Miscellaneous hidden extensions \label{tab:misc-extn}	
} \\
\toprule
$(s,f,w)$ & type & source & target & proof \\
\midrule \endfirsthead
\caption[]{Miscellaneous hidden extensions} \\
\toprule
$(s,f,w)$ & type & source & target & proof \\
\midrule \endhead
\bottomrule \endfoot
$(16,2,9)$ & $\sigma$ & $h_1 h_4$ & $h_4 c_0$ & $C\tau$ \\
$(20,4,11)$ & $\epsilon$ & $\tau g$ & $d_0^2$ & \cite{Isaksen14c}*{Table 33} \\
$(30,2,16)$ & $\sigma$ & $h_4^2$ & $x$ & $C\tau$ \\
$(30,2,16)$ & $\eta_4$ & $h_4^2$ & $h_1 h_5 d_0$ & \cite{Isaksen14c}*{Table 33} \\
$(32,6,17)$ & $\epsilon$ & $\D h_1 h_3$ & $\D h_1^2 d_0$ & $\tmf$ \\
$(32,6,17)$ & $\kappa$ & $\D h_1 h_3$ & $\tau d_0 l + \D c_0 d_0$ & $\tmf$ \\
$(44,4,24)$ & $\theta_4$ & $g_2$ & $x_{74,8}$ & $C\tau$ \\
$(45,3,23)$ & $\epsilon$ & $\tau h_3^2 h_5$ & $M P$ & Lemma \ref{lem:epsilon-h3^2h5} \\
\end{longtable}

\rev{
\begin{longtable}{llll}
\caption{Some miscellaneous null hidden extensions \label{tab:misc-extn-null}
} \\
\toprule
$(s,f,w)$ & type & source & proof \\
\midrule \endfirsthead
\caption[]{Some null hidden $\nu$ extensions} \\ 
\toprule
$(s,f,w)$ & type & source & proof \\
\midrule \endhead
\bottomrule \endfoot
$(30, 2, 16)$ & $\theta_4$ & $h_4^2$ & Lemma \ref{lem:theta4-h4^2} \\
$(45, 5, 24)$ & $\epsilon$ & $h_5 d_0$ & Lemma \ref{lem:epsilon-h5d0} \\
$(45, 5, 24)$ & $\{\D h_1 h_3\}$ & $h_5 d_0$ & Lemma \ref{lem:Dh1h3-h5d0} \\

\end{longtable}
}

%% file: more-stable-stems.bbl
\begin{bibdiv}
\begin{biblist}

\bib{Adams60}{article}{
   author={Adams, J. F.},
   title={On the non-existence of elements of Hopf invariant one},
   journal={Ann. of Math. (2)},
   volume={72},
   date={1960},
   pages={20--104},
   issn={0003-486X},
   review={\MR{0141119 (25 \#4530)}},
}


\bib{BJM84}{article}{
   author={Barratt, M. G.},
   author={Jones, J. D. S.},
   author={Mahowald, M. E.},
   title={Relations amongst Toda brackets and the Kervaire invariant in
   dimension $62$},
   journal={J. London Math. Soc. (2)},
   volume={30},
   date={1984},
   number={3},
   pages={533--550},
   issn={0024-6107},
   review={\MR{810962 (87g:55025)}},
   doi={10.1112/jlms/s2-30.3.533},
}

\bib{BMT70}{article}{
   author={Barratt, M. G.},
   author={Mahowald, M. E.},
   author={Tangora, M. C.},
   title={Some differentials in the Adams spectral sequence. II},
   journal={Topology},
   volume={9},
   date={1970},
   pages={309--316},
   issn={0040-9383},
   review={\MR{0266215 (42 \#1122)}},
}


\bib{BF21}{article}{
	author={Beauvais-Feisthauer, Joey},
	title={Automated differential computation in the Adams spectral sequence},
	status={in preparation},
}

\bib{BBBCX}{article}{
   author={Beaudry, Agn\`es},
   author={Behrens, Mark},
   author={Bhattacharya, Prasit},
   author={Culver, Dominic},
   author={Xu, Zhouli},
   title={The telescope conjecture at height 2 and the tmf resolution},
   journal={J. Topol.},
   volume={14},
   date={2021},
   number={4},
   pages={1243--1320},
   issn={1753-8416},
   review={\MR{4332490}},
   doi={10.1112/topo.12208},
}


\bib{Behrens}{article}{
   author={Behrens, M.},
   author={Hill, M.},
   author={Hopkins, M. J.},
   author={Mahowald, M.},
   title={Detecting exotic spheres in low dimensions using coker $J$},
   journal={J. Lond. Math. Soc. (2)},
   volume={101},
   date={2020},
   number={3},
   pages={1173-1218},
   issn={0024-6107},
   review={\MR{4111938}},
   doi={10.1112/jlms.12301},
}

\bib{BMQ21}{article}{
	author={Behrens, M.},
	author={Mahowald, M.},
	author={Quigley, J.D.},
	title={The 2-primary Hurewicz image of tmf},
	status={to appear in Geometry and Topology},
	eprint={arXiv:2011.08956v4},
}

\bib{BK21}{article}{
	author={Belmont, Eva},
	author={Kong, Hana Jia},
	title={A Toda bracket convergence theorem for multiplicative spectral sequences},
	status={preprint},
	eprint={arXiv:2112.08689},
}

	

\bib{Bruner84}{article}{
   author={Bruner, Robert},
   title={A new differential in the Adams spectral sequence},
   journal={Topology},
   volume={23},
   date={1984},
   number={3},
   pages={271--276},
   issn={0040-9383},
   review={\MR{770563 (86c:55016)}},
   doi={10.1016/0040-9383(84)90010-7},
}

\bib{Bruner97}{article}{
   author={Bruner, Robert R.},
   title={The cohomology of the mod 2 Steenrod algebra: A computer calculation},
   journal={Wayne State University Research Report},
   volume={37},
   date={1997},
}

\bib{Bruner04}{article}{
   author={Bruner, Robert R.},
   title={Squaring operations in $\mathrm{Ext}_{\mathcal{A}}(F_2, F_2)$},
   status={preprint},
   date={2004},
}

\bib{BR20}{article}{
	author={Bruner, Robert R.},
	author={Rognes, John},
	title={The cohomology of the mod 2 Steenrod algebra [Data set]},
	doi={10.11582/2021.00077},
	journal={Norstore (2021)},
	date={{}},
}

\bib{BR21}{book}{
   author={Bruner, Robert R.},
   author={Rognes, John},
   title={The Adams spectral sequence for topological modular forms},
   series={Mathematical Surveys and Monographs},
   volume={253},
   publisher={American Mathematical Society, Providence, RI},
   date={[2021] \copyright 2021},
   pages={xix+690},
   isbn={978-1-4704-5674-0},
   review={\MR{4284897}},
   doi={10.1090/surv/253},
}

\bib{Burklund21}{article}{
   author={Burklund, Robert},
   title={An extension in the Adams spectral sequence in dimension 54},
   journal={Bull. Lond. Math. Soc.},
   volume={53},
   date={2021},
   number={2},
   pages={404--407},
   issn={0024-6093},
   review={\MR{4239183}},
   doi={10.1112/blms.12428},
}

\bib{BIX}{article}{
	author={Burklund, Robert},
	author={Isaksen, Daniel C.},
	author={Xu, Zhouli},
	title={Classical stable homotopy groups of spheres via $\F_2$-synthetic methods},
	status={in preparation},
}

\bib{Chua21}{article}{
	author={Chua, Dexter},
	title={Adams differentials via the secondary Steenrod algebra},
	status={preprint},
	eprint={arXiv:2105.07628},
}

\bib{Cohen68}{article}{
   author={Cohen, Joel M.},
   title={The decomposition of stable homotopy},
   journal={Ann. of Math. (2)},
   volume={87},
   date={1968},
   pages={305--320},
   issn={0003-486X},
   review={\MR{231377}},
   doi={10.2307/1970586},
}
	

\bib{Gheorghe18}{article}{
   author={Gheorghe, Bogdan},
   title={The motivic cofiber of $\tau$},
   journal={Doc. Math.},
   volume={23},
   date={2018},
   pages={1077--1127},
   issn={1431-0635},
   review={\MR{3874951}},
}


\bib{GIKR18}{article}{
   author={Gheorghe, Bogdan},
   author={Isaksen, Daniel C.},
   author={Krause, Achim},
   author={Ricka, Nicolas},
   title={$\Bbb{C}$-motivic modular forms},
   journal={J. Eur. Math. Soc. (JEMS)},
   volume={24},
   date={2022},
   number={10},
   pages={3597--3628},
   issn={1435-9855},
   review={\MR{4432907}},
   doi={10.4171/jems/1171},
}


\bib{GWX18}{article}{
   author={Gheorghe, Bogdan},
   author={Wang, Guozhen},
   author={Xu, Zhouli},
   title={The special fiber of the motivic deformation of the stable
   homotopy category is algebraic},
   journal={Acta Math.},
   volume={226},
   date={2021},
   number={2},
   pages={319--407},
   issn={0001-5962},
   review={\MR{4281382}},
   doi={10.4310/acta.2021.v226.n2.a2},
}

\bib{GI15}{article}{
   author={Guillou, Bertrand J.},
   author={Isaksen, Daniel C.},
   title={The $\eta$-local motivic sphere},
   journal={J. Pure Appl. Algebra},
   volume={219},
   date={2015},
   number={10},
   pages={4728--4756},
   issn={0022-4049},
   review={\MR{3346515}},
   doi={10.1016/j.jpaa.2015.03.004},
}

\bib{Hatcher}{article}{
	author={Hatcher, Allen},
	title={Pictures of stable homotopy groups of spheres},
	date={},
	eprint={pi.math.cornell.edu/~hatcher/stemfigs/stems.pdf},
}

\bib{HHR}{article}{
   author={Hill, M. A.},
   author={Hopkins, M. J.},
   author={Ravenel, D. C.},
   title={On the nonexistence of elements of Kervaire invariant one},
   journal={Ann. of Math. (2)},
   volume={184},
   date={2016},
   number={1},
   pages={1--262},
   issn={0003-486X},
   review={\MR{3505179}},
   doi={10.4007/annals.2016.184.1.1},
}

\bib{Hirsch55}{article}{
   author={Hirsch, Guy},
   title={Quelques propri\'{e}t\'{e}s des produits de Steenrod},
   language={French},
   journal={C. R. Acad. Sci. Paris},
   volume={241},
   date={1955},
   pages={923--925},
   issn={0001-4036},
   review={\MR{73182}},
}

\bib{Hovey}{article}{
   author={Hovey, Mark},
   title={Homotopy theory of comodules over a Hopf algebroid},
   conference={
      title={Homotopy theory: relations with algebraic geometry, group
      cohomology, and algebraic $K$-theory},
   },
   book={
      series={Contemp. Math.},
      volume={346},
      publisher={Amer. Math. Soc., Providence, RI},
   },
   date={2004},
   pages={261--304},
   review={\MR{2066503}},
   doi={10.1090/conm/346/06291},
}

\bib{HKO11}{article}{
   author={Hu, Po},
   author={Kriz, Igor},
   author={Ormsby, Kyle},
   title={Remarks on motivic homotopy theory over algebraically closed
   fields},
   journal={J. K-Theory},
   volume={7},
   date={2011},
   number={1},
   pages={55--89},
   issn={1865-2433},
   review={\MR{2774158}},
   doi={10.1017/is010001012jkt098},
}


\bib{Isaksen09}{article}{
   author={Isaksen, Daniel C.},
   title={The cohomology of motivic $A(2)$},
   journal={Homology Homotopy Appl.},
   volume={11},
   date={2009},
   number={2},
   pages={251--274},
   issn={1532-0073},
   review={\MR{2591921 (2011c:55034)}},
}

\bib{Isaksen14b}{article}{
   author={Isaksen, Daniel C.},
   title={When is a fourfold Massey product defined?},
   journal={Proc. Amer. Math. Soc.},
   volume={143},
   date={2015},
   number={5},
   pages={2235--2239},
   issn={0002-9939},
   review={\MR{3314129}},
   doi={10.1090/S0002-9939-2014-12387-2},
}

\bib{Isaksen14c}{article}{
   author={Isaksen, Daniel C.},
   title={Stable stems},
   journal={Mem. Amer. Math. Soc.},
   volume={262},
   date={2019},
   number={1269},
   pages={viii+159},
   issn={0065-9266},
   isbn={978-1-4704-3788-6},
   isbn={978-1-4704-5511-8},
   review={\MR{4046815}},
   doi={10.1090/memo/1269},
}

\bib{Isaksen18}{article}{
	author={Isaksen, Daniel C.},
	title={The homotopy of $\C$-motivic modular forms},
	journal={Zenodo},
	doi={10.5281/zenodo.6547197},
}


\bib{Isaksen19}{article}{
	author={Isaksen, Daniel C.},
	title={The Mahowald operator in the cohomology of the Steenrod algebra},
	status={to appear},
	journal={Tbilisi Math.\ J.},
	date={{}},
}

\bib{Isaksen14a}{article}{
   author={Isaksen, Daniel C.},
   author={Wang, Guozhen},
   author={Xu, Zhouli},
   title={Classical and $\C$-motivic Adams charts},
   date={2022},
   doi={doi.org/10.5281/zenodo.6987156},
}


\bib{IWX19}{article}{
	author={Isaksen, Daniel C.},
	author={Wang, Guozhen},
	author={Xu, Zhouli},
	title={Classical algebraic Novikov charts and $\C$-motivic Adams charts for the cofiber of $\tau$},
	date={2022},
	doi={doi.org/10.5281/zenodo.6987226},
}


\bib{IX15}{article}{
   author={Isaksen, Daniel C.},
   author={Xu, Zhouli},
   title={Motivic stable homotopy and the stable 51 and 52 stems},
   journal={Topology Appl.},
   volume={190},
   date={2015},
   pages={31--34},
   issn={0166-8641},
   review={\MR{3349503}},
   doi={10.1016/j.topol.2015.04.008},
}

\bib{KervaireMilnor}{article}{
   author={Kervaire, Michel A.},
   author={Milnor, John W.},
   title={Groups of homotopy spheres. I},
   journal={Ann. of Math. (2)},
   volume={77},
   date={1963},
   pages={504--537},
   issn={0003-486X},
   review={\MR{148075}},
   doi={10.2307/1970128},
}

\bib{Kochman78}{article}{
   author={Kochman, Stanley O.},
   title={A chain functor for bordism},
   journal={Trans. Amer. Math. Soc.},
   volume={239},
   date={1978},
   pages={167--196},
   issn={0002-9947},
   review={\MR{488031}},
   doi={10.2307/1997852},
}

\bib{Kochman90}{book}{
   author={Kochman, Stanley O.},
   title={Stable homotopy groups of spheres},
   series={Lecture Notes in Mathematics},
   volume={1423},
   note={A computer-assisted approach},
   publisher={Springer-Verlag},
   place={Berlin},
   date={1990},
   pages={viii+330},
   isbn={3-540-52468-1},
   review={\MR{1052407 (91j:55016)}},
}

\bib{Kochman96}{book}{
   author={Kochman, S. O.},
   title={Bordism, stable homotopy and Adams spectral sequences},
   series={Fields Institute Monographs},
   volume={7},
   publisher={American Mathematical Society, Providence, RI},
   date={1996},
   pages={xiv+272},
   isbn={0-8218-0600-9},
   review={\MR{1407034}},
   doi={10.1090/fim/007},
}

\bib{KM93}{article}{
   author={Kochman, Stanley O.},
   author={Mahowald, Mark E.},
   title={On the computation of stable stems},
   conference={
      title={The \v Cech centennial},
      address={Boston, MA},
      date={1993},
   },
   book={
      series={Contemp. Math.},
      volume={181},
      publisher={Amer. Math. Soc.},
      place={Providence, RI},
   },
   date={1995},
   pages={299--316},
   review={\MR{1320997 (96j:55018)}},
   doi={10.1090/conm/181/02039},
}

\bib{Krause18}{thesis}{
	author={Krause, Achim},
	title={Periodicity in motivic homotopy theory and over $BP_* BP$},
	type={Ph.D. thesis},
	date={2018},
	organization={Universit{\"a}t Bonn},
}

\bib{Lin01}{article}{
   author={Lin, Wen-Hsiung},
   title={A proof of the strong Kervaire invariant in dimension 62},
   conference={
      title={First International Congress of Chinese Mathematicians},
      address={Beijing},
      date={1998},
   },
   book={
      series={AMS/IP Stud. Adv. Math.},
      volume={20},
      publisher={Amer. Math. Soc., Providence, RI},
   },
   date={2001},
   pages={351--358},
   review={\MR{1830191}},
}

\bib{MT67}{article}{
   author={Mahowald, Mark},
   author={Tangora, Martin},
   title={Some differentials in the Adams spectral sequence},
   journal={Topology},
   volume={6},
   date={1967},
   pages={349--369},
   issn={0040-9383},
   review={\MR{0214072 (35 \#4924)}},
}

\bib{May69}{article}{
   author={May, J. Peter},
   title={Matric Massey products},
   journal={J. Algebra},
   volume={12},
   date={1969},
   pages={533--568},
   issn={0021-8693},
   review={\MR{0238929 (39 \#289)}},
}

\bib{May70}{article}{
   author={May, J. Peter},
   title={A general algebraic approach to Steenrod operations},
   conference={
      title={The Steenrod Algebra and its Applications (Proc. Conf. to
      Celebrate N. E. Steenrod's Sixtieth Birthday, Battelle Memorial Inst.,
      Columbus, Ohio, 1970)},
   },
   book={
      series={Lecture Notes in Mathematics, Vol. 168},
      publisher={Springer, Berlin},
   },
   date={1970},
   pages={153--231},
   review={\MR{0281196}},
}


\bib{Miller75}{book}{
   author={Miller, Haynes Robert},
   title={SOME ALGEBRAIC ASPECTS OF THE ADAMS-NOVIKOV SPECTRAL SEQUENCE},
   note={Thesis (Ph.D.)--Princeton University},
   publisher={ProQuest LLC, Ann Arbor, MI},
   date={1975},
   pages={103},
   review={\MR{2625232}},
}

\bib{Milnor}{article}{
   author={Milnor, John},
   title={Differential topology forty-six years later},
   journal={Notices Amer. Math. Soc.},
   volume={58},
   date={2011},
   number={6},
   pages={804--809},
   issn={0002-9920},
   review={\MR{2839925}},
}

\bib{Moss70}{article}{
   author={Moss, R. Michael F.},
   title={Secondary compositions and the Adams spectral sequence},
   journal={Math. Z.},
   volume={115},
   date={1970},
   pages={283--310},
   issn={0025-5874},
   review={\MR{0266216 (42 \#1123)}},
}

\bib{Novikov67}{article}{
   author={Novikov, S. P.},
   title={Methods of algebraic topology from the point of view of cobordism
   theory},
   language={Russian},
   journal={Izv. Akad. Nauk SSSR Ser. Mat.},
   volume={31},
   date={1967},
   pages={855--951},
   issn={0373-2436},
   review={\MR{0221509}},
}

\bib{Pstragowski18}{article}{
	author={Pstragowski, Piotr},
	title={Synthetic spectra and the cellular motivic category},
	eprint={arXiv:1803.01804},
	date={2018},
	status={preprint},
}

\bib{Ravenel86}{book}{
   author={Ravenel, Douglas C.},
   title={Complex cobordism and stable homotopy groups of spheres},
   series={Pure and Applied Mathematics},
   volume={121},
   publisher={Academic Press, Inc., Orlando, FL},
   date={1986},
   pages={xx+413},
   isbn={0-12-583430-6},
   isbn={0-12-583431-4},
   review={\MR{860042 (87j:55003)}},
}

\bib{Tangora70a}{article}{
   author={Tangora, Martin C.},
   title={On the cohomology of the Steenrod algebra},
   journal={Math. Z.},
   volume={116},
   date={1970},
   pages={18--64},
   issn={0025-5874},
   review={\MR{0266205 (42 \#1112)}},
}


\bib{Toda62}{book}{
   author={Toda, Hirosi},
   title={Composition methods in homotopy groups of spheres},
   series={Annals of Mathematics Studies, No. 49},
   publisher={Princeton University Press},
   place={Princeton, N.J.},
   date={1962},
   pages={v+193},
   review={\MR{0143217 (26 \#777)}},
}

\bib{tmf14}{collection}{
   title={Topological modular forms},
   series={Mathematical Surveys and Monographs},
   volume={201},
   note={Edited by Christopher L. Douglas, John Francis, Andr\'e G.
   Henriques and Michael A. Hill},
   publisher={American Mathematical Society, Providence, RI},
   date={2014},
   pages={xxxii+318},
   isbn={978-1-4704-1884-7},
   review={\MR{3223024}},
}

\bib{Voevodsky03b}{article}{
   author={Voevodsky, Vladimir},
   title={Motivic cohomology with ${\bf Z}/2$-coefficients},
   journal={Publ. Math. Inst. Hautes \'Etudes Sci.},
   number={98},
   date={2003},
   pages={59--104},
   issn={0073-8301},
   review={\MR{2031199 (2005b:14038b)}},
   doi={10.1007/s10240-003-0010-6},
}

\bib{Voevodsky10}{article}{
   author={Voevodsky, Vladimir},
   title={Motivic Eilenberg-Maclane spaces},
   journal={Publ. Math. Inst. Hautes \'Etudes Sci.},
   number={112},
   date={2010},
   pages={1--99},
   issn={0073-8301},
   review={\MR{2737977 (2012f:14041)}},
   doi={10.1007/s10240-010-0024-9},
}

\bib{Wang19}{article}{
   author={Wang, Guozhen},
   title={\emph{\texttt{github.com/pouiyter/morestablestems}}},
}

\bib{Wang20}{article}{
   author={Wang, Guozhen},
   title={Computations of the Adams-Novikov $E_2$-term},
   journal={Chinese Ann. Math. Ser. B},
   volume={42},
   date={2021},
   number={4},
   pages={551--560},
   issn={0252-9599},
   review={\MR{4289191}},
   doi={10.1007/s11401-021-0277-2},
}

\bib{WangXu17}{article}{
   author={Wang, Guozhen},
   author={Xu, Zhouli},
   title={The triviality of the 61-stem in the stable homotopy groups of
   spheres},
   journal={Ann. of Math. (2)},
   volume={186},
   date={2017},
   number={2},
   pages={501--580},
   issn={0003-486X},
   review={\MR{3702672}},
   doi={10.4007/annals.2017.186.2.3},
}


\bib{WangXu18}{article}{
   author={Wang, Guozhen},
   author={Xu, Zhouli},
   title={Some extensions in the Adams spectral sequence and the 51--stem},
   journal={Algebr. Geom. Topol.},
   volume={18},
   date={2018},
   number={7},
   pages={3887--3906},
   issn={1472-2747},
   review={\MR{3892234}},
   doi={10.2140/agt.2018.18.3887},
}

\bib{Xu16}{article}{
   author={Xu, Zhouli},
   title={The strong Kervaire invariant problem in dimension 62},
   journal={Geom. Topol.},
   volume={20},
   date={2016},
   number={3},
   pages={1611--1624},
   issn={1465-3060},
   review={\MR{3523064}},
   doi={10.2140/gt.2016.20.1611},
}

\end{biblist}
\end{bibdiv}
